\documentclass[DIV=12,numbers=enddot,leqno,bibliography=totoc]{scrartcl}
\usepackage{./q-HodgeStyle_minimal}
\title{\texorpdfstring{$q$}{q}-Hodge complexes over the Habiro ring}
\author{Ferdinand Wagner}

\newcommand{\pOmega}{\mathop{\widetilde{p}}\!\hspace{-0.1ex}\mhyph\hspace{-0.1ex}\Omega}

\newcommand{\pHodge}{\mathop{\widetilde{p}}\!\hspace{-0.1ex}\mhyph\hspace{-0.1ex}\Hodge}

\begin{document}
	\maketitle
	
	\begin{abstract}
		\textbf{Abstract. --- }
		Peter Scholze has raised the question whether some variant of the $q$-de Rham complex is already defined over the \emph{Habiro ring} $\Hh\coloneqq \limit_{m\in\IN}\IZ[q]_{(q^m-1)}^\complete$. Such a variant should then be called \emph{\embrace{algebraic} Habiro cohomology}.
		
		We show that algebraic Habiro cohomology exists whenever the $q$-de Rham complex can be equipped with a \emph{$q$-Hodge filtration}: a $q$-deformation of the Hodge filtration, subject to some reasonable conditions. To any such $q$-Hodge filtration we associate a small modification of the $q$-de Rham complex, which we call the \emph{$q$-Hodge complex}, and show that it descends canonically to the Habiro ring. This construction recovers and generalises the \emph{Habiro ring of a number field} from \cite{HabiroRingOfNumberField} and is closely related to the $q$-de Rham--Witt complexes from \cite{qWitt} as well as, conjecturally, to Scholze's \emph{analytic Habiro stack} \cite{HabiroCohomologyLecture}.
		
		While there's no canonical $q$-Hodge filtration in general, we show that it does exist in many cases of interest. For example, for a smooth scheme $X$ over $\IZ$, the $q$-de Rham complex $\qOmega_{X/\IZ}$ can be equipped with a canonical $q$-Hodge filtration as soon as one inverts all primes $p\leqslant \dim(X/\IZ)$.
	\end{abstract}

	\tableofcontents
	\renewcommand{\SectionPrefix}{\textrm{\S}}
	\renewcommand{\SubsectionPrefix}{\textrm{\S}}
	
	\newpage

	\section{Introduction}\label{sec:Intro}
	
	Throughout the introduction, we'll work over~$\IZ$ for simplicity. In the main body, we'll work instead relative to a $\Lambda$-ring $A$ which is \emph{perfectly covered} in the sense of \cref{par:Notations} below.
	
	\subsection{Habiro cohomology}
	
	In this paper we'll investigate the following question, which was raised by Peter Scholze:
	
	\begin{numpar}[Question.]\label{qst:Habiro}\itshape
		Is there a version of $q$-de Rham cohomology with coefficients not in the power series ring $\IZ\qpower$, but in the Habiro ring
		\begin{equation*}
			\Hh\coloneqq \limit_{m\in\IN}\IZ[q]_{(q^m-1)}^\complete\,?
		\end{equation*}
	\end{numpar}
	A cohomology theory that provides a positive answer to Question~\cref{qst:Habiro} would then deserve the name \emph{Habiro cohomology}. Let us first convince the reader that it should be worthwhile to search for such a cohomology theory.
	
	\begin{numpar}[Why Habiro cohomology?]\label{par:WhyHabiroCohomology}
		\emph{$q$-de Rham cohomology}, which was constructed by Bhatt and Scholze (\cite{Toulouse,Prismatic}; see also the review in \cref{appendix:GlobalqDeRham}) is already a very interesting cohomology theory: For every smooth scheme $X$ over $\IZ$, the \emph{$q$-de Rham complex} $\qOmega_{X/\IZ}$ is a $(q-1)$-complete object in the derived $\infty$-category $\Dd(X,\IZ\qpower)$ which \emph{$q$-deforms} the usual de Rham complex in the sense that $\qOmega_{X/\IZ}/(q-1)\simeq \Omega_{X/\IZ}^*$. Moreover, for any prime~$p$, the $p$-completion $(\qOmega_{X/\IZ})_p^\complete$ computes the prismatic cohomology of $X\times\Spf \IZ_p[\zeta_p]$ over the prism $(\IZ_p\qpower,[p]_q)$. This makes $q$-de Rham cohomology the only known case (besides de Rham cohomology) in which prismatic cohomology for all primes can be combined into a global object, and so it ought to be important.
		
		Now what do we hope to gain from a version of $q$-de Rham cohomology with coefficients in $\Hh$ instead of $\IZ\qpower$? Besides the general philosophy that whenever one has a deformation at $q=1$, one should look at the other roots of unity as well%
		\footnote{The Habiro ring can be regared, in a precise sense, as the ring of power series that have a Taylor expansion around each root of unity $\zeta$ with coefficients in $\IZ[\zeta]$. See \cref{rem:HTaylorSeries}.}%
		, here's our main motivation to pursue Question~\cref{qst:Habiro}:
		\begin{alphanumerate}
			\item A stacky approach to Habiro cohomology is expected to be much more interesting geometrically than a stacky approach to $q$-de Rham cohomology (to the extent that either one exists).\label{enum:MotivationStackyApproach}
			\item There's growing evidence that certain $3$-manifold invariants related to Chern--Simons theory take values in Habiro cohomology.\label{enum:MotivationChernSimons}%
			\footnote{In fact, the Habiro ring was introduced by Habiro himself as the target of certain invariants of hyperbolic knots and homology $3$-spheres \cite{Habiro,HabiroInvariants}.}
		\end{alphanumerate}
		We'll elaborate on both points below.
	\end{numpar}
	\begin{numpar}[Stacky approaches and Scholze's Habiro stack.]
		The term \emph{stacky approach} refers to the idea that for any reasonable cohomology theory $\R\Gamma_?(X)$ there should be a geometric construction $X\mapsto X^?$ in such a way that $\R\Gamma_?(X)\simeq \R\Gamma(X^?,\Oo)$ is the sheaf cohomology of the geometric object $X^?$. This allows to study algebraic properties of $\R\Gamma_?(X)$ via the geometry of $X^?$, which often holds much richer information.
		
		The first instance of a stacky approach is Simpson's \emph{de Rham stack} for smooth varieties over characteristic~$0$ fields \cite{SimpsonDeRhamStack}. The construction of a stacky approach to prismatic cohomology by Drinfeld and Bhatt--Lurie \cite{DrinfeldPrismatization,BhattLurieI,BhattLurieII} has started a flurry of results in recent years. These include at least the work of Anschütz--Heuer--Le Bras on the \emph{Hodge--Tate stack} \cite{HodgeTateStack1,HodgeTateStack2,HodgeTateStack3}, the construction of an \emph{analytic de Rham stack} by Rodríguez Camargo \cite{AnalyticDeRhamStack}, the ongoing work of Anschütz--Bosco--Hauck--Le Bras--Rodríguez Camargo--Scholze on \emph{analytic prismatisation}, the ongoing work of Bhatt--Mathew--Vologodsky--Zhang on \emph{sheared prismatisation}, and the ongoing work of Devalapurkar--Hahn--Raksit--Yuan on recovering prismatisation from their theory of \emph{even stacks} \cite{EvenStacks}.
		
		In \cite{HabiroCohomologyLecture}, Scholze proposes a construction of an \emph{analytic Habiro stack} $X^\mathrm{Hab}$, which gives rise to a cohomology theory with coefficients in an analytic version $\Hh^\mathrm{an}$ of the Habiro ring. We'll remark in~\cref{par:AlgebraicvsAnalytic} on the expected relationship between this cohomology theory and the construction of Habiro cohomology that we propose in this paper. For $q$-de Rham cohomology we don't have a global stacky approach available yet  ($p$-adically it works as a special case of prismatisation), but at least an analytic version is expected to exist (compare the discussion in \cite[\S{\chref[section]{1.3}}]{qHodge}).
	\end{numpar}
	
	Let us now explain why a stacky approach to Habiro cohomology should be much more interesting than a stacky approach to $q$-de Rham cohomology by giving a specific example which also nicely illustrates \cref{par:WhyHabiroCohomology}\cref{enum:MotivationChernSimons}.
	
	\begin{numpar}[The Habiro ring of a number field.]
		Let $F$ be a number field and let $\Delta$ be divisible by $6\operatorname{disc} F$. In \cite{HabiroRingOfNumberField}, Garoufalidis--Scholze--Wheeler--Zagier construct a certain formally étale $\Hh$-algebra $\Hh_{\Oo_F[1/\Delta]}$: the \emph{Habiro ring of the number field~$F$} (we'll recall the specific construction in \cref{subsec:HabiroRingOfNumberField}). Moreover, the authors construct a regulator map
		\begin{equation*}
			K_3(F)\longrightarrow \Pic\bigl(\Hh_{\Oo_F[1/\Delta]}\bigr)
		\end{equation*}
		and show that certain $q$-series arising from perturbative Chern--Simons theory naturally form sections of the line bundles in the image of this regulator.
		
		We would like to interpret $\Hh_{\Oo_F[1/\Delta]}$ and its formal spectrum $\Spf \Hh_{\Oo_F[1/\Delta]}$ as the \emph{Habiro cohomology} and the \emph{Habiro stack} of $\Spec \Oo_F[1/\Delta]$, respectively.%
		\footnote{Note that $\Spf \Hh_{\Oo_F[1/\Delta]}$ doesn't precisely match up with Scholze's analytic Habiro stack $(\Spec \Oo_F[1/\Delta])^\mathrm{Hab}$. See the discussion in~\cref{par:AlgebraicvsAnalytic}.}
		Already in this special case, the Habiro stack exhibits non-trivial geometry in form of line bundles with interesting sections that come from the regulator $K_3(F)\rightarrow \Pic(\Hh_{\Oo_F[1/\Delta]})$. This geometry would be completely invisible to any $q$-de Rham stack, as the regulator becomes trivial after $(q-1)$-completion!
		
		The construction of Habiro cohomology that we propose in this paper recovers $\Hh_{\Oo_F[1/\Delta]}$ (\cref{cor:ComparisonWithGSWZ2}). It also appears that the connection to Chern--Simons theory extends beyond the étale case. For example, in \cite{ExplicitClassesInHabiro}, Garoufalidis and Wheeler construct certain $(q-1)$-power series with coefficients in de Rham cohomology; these $(q-1)$-power power series are expected to be Habiro cohomology classes. While much of this is still mysterious, it suggests that something highly nontrivial is going on! 
	\end{numpar}
	
	\subsection{\texorpdfstring{$q$}{q}-Hodge filtrations and Habiro descent}
	Having justified that Question~\cref{qst:Habiro} is relevant, let us now summarise the contributions of this paper. As it turns out, Question~\cref{qst:Habiro} is closely related to the following much less esoteric question:
	\begin{numpar}[Question.]\label{qst:qHodgeFiltration}\itshape
		Is it possible to equip the $q$-de Rham complex with a $q$-deformation of the Hodge filtration?
	\end{numpar}
	To pose this question more formally, we'll introduce the following notion:
	\begin{defi}\label{def:qHodgeIntro}
		Let $R$ be an animated ring and let $\qdeRham_{R/\IZ}$ be its derived $q$-de Rham complex. A \emph{$q$-Hodge filtration} is a $(q-1)$-complete filtered module
		\begin{equation*}
			\fil_{\qHodge}^\star\qdeRham_{R/\IZ}= \Bigl(\fil_{\qHodge}^0\qdeRham_{R/\IZ}\leftarrow\fil_{\qHodge}^0\qdeRham_{R/\IZ}\leftarrow\dotsb\Bigr)
		\end{equation*}
		over the $(q-1)$-adically filtered ring $(q-1)^\star \IZ\qpower$, such that:
		\begin{alphanumerate}
			\item $\fil_{\qHodge}^\star\qdeRham_{R/\IZ}$ is a descending filtration on the derived $q$-de Rham complex. That is,\label{enum:qHodgeIntroa}
			\begin{equation*}
				\fil_{\qHodge}^0\qdeRham_{R/\IZ}\simeq \qdeRham_{R/\IZ}\,.
			\end{equation*}
			\item $\fil_{\qHodge}^\star\qdeRham_{R/\IZ}$ is a filtered $q$-deformation of the Hodge filtration on the derived de Rham complex $\deRham_{R/\IZ}$. That is,\label{enum:qHodgeIntrob}
			\begin{equation*}
				\fil_{\qHodge}^\star \qdeRham_{R/\IZ}\lotimes_{(q-1)^\star\IZ\qpower}\IZ\simeq \fil_{\Hodge}^\star\deRham_{R/\IZ}\,.
			\end{equation*}
			\item Rationally, $\fil_{\qHodge}^\star$ becomes the combined Hodge and $(q-1)$-adic filtration on $(\deRham_{R/\IZ}\lotimes_\IZ\IQ)\qpower$. That is,\label{enum:qHodgeIntroc}
			\begin{equation*}
				\bigl(\fil_{\qHodge}^\star\qdeRham_{R/\IZ}\lotimes_\IZ\IQ\bigr)_{(q-1)}^\complete\simeq \fil_{(\Hodge,q-1)}^\star\bigl(\deRham_{R/\IZ}\lotimes_\IZ\IQ\bigr)\qpower\,.
			\end{equation*}
			\item[c_p] The same holds true for any prime~$p$ if we $p$-complete first and then rationalise. That is,\label{enum:qHodgeIntrocp}
			\begin{equation*}
				\fil_{\qHodge}^\star\bigl(\qdeRham_{R/\IZ}\bigr)_p^\complete\bigl[\localise{p}\bigr]_{(q-1)}^\complete\overset{\simeq}{\longrightarrow}\fil_{(\Hodge,q-1)}^\star\bigl(\deRham_{R/\IZ}\bigr)_p^\complete\bigl[\localise{p}\bigr]\qpower\,.
			\end{equation*}
		\end{alphanumerate}
		Moreover, the equivalences from \cref{enum:qHodgeIntroa}--\cref{enum:qHodgeIntrocp} need to satisfy the obvious compatibilities (and compatibilities between compatibilities); the precise data required will be spelled out in \cref{def:qHodgeFiltration}.
		
		We also let $\cat{AniAlg}_{\IZ}^{\qHodge}$ denote the $\infty$-category of pairs $(R,\fil_{\qHodge}^\star\qdeRham_{R/\IZ})$, where $R$ is an animated ring and $\fil_{\qHodge}^\star\qdeRham_{R/\IZ}$ is a $q$-Hodge filtration as above. To any such pair, we associate the \emph{$q$-Hodge complex}
		\begin{equation*}
			\qHodge_{(R, \fil_{\smash{\qHodge}}^\star )/\IZ}\coloneqq \colimit\Bigl( \fil_{\qHodge}^0\qdeRham_{R/\IZ}\xrightarrow{(q-1)} \fil_{\qHodge}^1\qdeRham_{R/\IZ}\xrightarrow{(q-1)}\dotsb\Bigr)_{(q-1)}^\complete\,.
		\end{equation*}
		If the $q$-Hodge filtration is clear from the context, we'll often abusingly write just $\qHodge_{R/\IZ}$.
	\end{defi}
	
	\begin{rem}
		We've used the derived de Rham complex in \cref{def:qHodgeIntro} because we would like to apply the definition in cases where $R$ isn't smooth; \cref{subsec:FunctorialqHodgeIntro}. If $S$ is smooth, it doesn't matter whether we put a $q$-Hodge filtration on $\qdeRham_{S/\IZ}$, or a filtration satisfying analogous conditions on the underived $q$-de Rham complex $\qOmega_{S/\IZ}$; see \cref{rem:CompleteFiltration} for an argument.
	\end{rem}
	\begin{rem}
		The conditions from \cref{def:qHodgeIntro}\cref{enum:qHodgeIntroc} and~\cref{enum:qHodgeIntrocp} are natural to ask in view of $(\qdeRham_{R/\IZ}\lotimes_\IZ\IQ)_{(q-1)}^\complete\simeq (\deRham_{R/\IZ}\lotimes_\IZ\IQ)\qpower$ and $(\qdeRham_{R/\IZ})_p^\complete[1/p]_{(q-1)}^\complete\simeq (\deRham_{R/\IZ})_p^\complete[1/p]\qpower$; see \cref{thm:qDeRhamGlobal}\cref{enum:GlobalqDeRhamRational} and \cref{lem:RationalisationTechnicalI}. It doesn't seem to be the case that~\cref{enum:qHodgeIntrocp} follows from the other conditions and it will be a crucial assumption. 
	\end{rem}
	\begin{numpar}[Functorial $q$-Hodge filtrations?]
		With \cref{def:qHodgeIntro}, we can rephrase Question~\cref{qst:qHodgeFiltration} more formally as follows: \emph{Does the forgetful functor $\cat{AniAlg}_\IZ^{\qHodge}\rightarrow \cat{AniAlg}_\IZ$ admit a section?} The answer to this is, provably, \emph{no}! Nevertheless, the forgetful functor does have sections over surprisingly large full subcategories of $\cat{AniAlg}_\IZ$, as we'll see in \cref{subsec:FunctorialqHodgeIntro}.
		
		The general non-existence of a section of $\cat{AniAlg}_\IZ^{\qHodge}\rightarrow \cat{AniAlg}_\IZ$ is known to the experts and we'll reproduce the argument in \cref{lem:NoFunctorialqHodgeFiltration}. In the following, we'll present an essentially equivalent argument in a non-standard way, which will also serve to motivate our first main positive result.
	\end{numpar}
	\begin{numpar}[More on the $q$-Hodge complex.]\label{par:qHodgeFiltrationInCoordinates}
		Let $(S,\square)$ be a \emph{framed smooth algebra over $\IZ$}. That is, $S$ is smooth over $\IZ$ and $\square \colon \IZ[x_1,\dotsc,x_n]\rightarrow S$ is an étale map. In this case, the $q$-de Rham complex $\qOmega_{S/\IZ}$ can be represented by the explicit complex
		\begin{equation*}
			\qOmega_{S/\IZ,\square}^*\coloneqq\left(S\qpower \xrightarrow{\q\nabla}\Omega_{S/\IZ}^1\qpower\xrightarrow{\q\nabla} \dotsb \xrightarrow{\q\nabla} \Omega^n_{S/\IZ}\qpower \right)
		\end{equation*}
		defined in~\cite[\textup{\S\chref[section]{3}}]{Toulouse}. On this explicit complex, we can define a filtration $\fil_{\qHodge,\square}^\star\qOmega_{S/\IZ, \square}^*$, in which $\fil_{\qHodge,\square}^i$ is the subcomplex
		\begin{equation*}
			\left((q-1)^iS\qpower\rightarrow(q-1)^{i-1}\Omega_{S/\IZ}^1\qpower\rightarrow\dotsb\rightarrow\Omega_{S/\IZ}^i\qpower\rightarrow\dotsb\rightarrow\Omega_{S/\IZ}^n\qpower\right)\,.
		\end{equation*}
		Up to the discrepancy between $\qOmega_{S/\IZ}$ and $\qdeRham_{S/\IZ}$, which is easily fixed (see \cref{rem:CompleteFiltration}), the pair $(S,\fil_{\qHodge,\square}^\star)$ becomes an object in $\cat{AniAlg}_\IZ^{\qHodge}$.
		
		It's straightforward to check that the associated $q$-Hodge complex in the sense of \cref{def:qHodgeIntro} can be represented by the explicit complex
		\begin{equation*}
			\qHodge_{S/\IZ, \square}^*\coloneqq\Bigl(S\qpower \xrightarrow{(q-1)\q\nabla}\Omega_{S/\IZ}^1\qpower\xrightarrow{(q-1)\q\nabla} \dotsb \xrightarrow{(q-1)\q\nabla} \Omega^n_{S/\IZ}\qpower \Bigr)
		\end{equation*}
		in which all $q$-differentials in $\qOmega_{S/\IZ,\square}$ get multiplied by $(q-1)$. This \emph{coordinate-dependent $q$-Hodge complex} first shows up in Pridham's work \cite{Pridham} and was extensively studied in previous work of the author. In \cite[Theorem~\chref{4.27}]{qWitt}, we showed that for all $m\in\IN$ the cohomology
		\begin{equation*}
			\H^*\bigl(\qHodge_{S/\IZ,\square}/(q^m-1)\bigr)\cong \bigl(\qIW_m\Omega_{S/\IZ}^*\bigr)_{(q-1)}^\complete
		\end{equation*}
		agrees with the $(q-1)$-completion of a certain object $\qIW_m\Omega_{S/\IZ}^*$, which we call the \emph{$m$-truncated $q$-de Rham--Witt complex of $S$} \cite[Definition~\chref{3.13}]{qWitt}. The system $(\qIW_m\Omega_{S/\IZ}^*)_{m\in\IN}$ satisfies a similar universal property as the de Rham--Witt pro-complex and so $\qIW_m\Omega_{S/\IZ}^*$ is functorial in $S$. In particular, the cohomology $\H^*(\qHodge_{S/\IZ,\square}/(q^m-1))$ is independent of the choice of $\square$!
		
		In spite of this promising observation, we show in \cite[Theorem~\chref{5.1}]{qWitt} that it is \emph{impossible} to turn the $q$-Hodge complex into a functor of smooth $\IZ$-algebras,
		\begin{equation*}
			\qHodge_{-/\IZ}\colon \cat{Sm}_\IZ \longrightarrow \Dd\bigl(\IZ\qpower\bigr)\,,
		\end{equation*}
		in such a way that $\H^*(\qHodge_{-/\IZ}/(q^m-1))\cong (\qIW_m\Omega_{-/\IZ}^*)_{(q-1)}^\complete$ becomes functorial as well.
		
		This strange no-go result is a strong objection against the existence of a section of the forgetful functor $\cat{AniAlg}_\IZ^{\qHodge}\rightarrow \cat{AniAlg}_\IZ$ (and it can be turned into a complete proof using \cref{thm:HabiroDescentIntro}\cref{enum:qdRWComparisonIntro} below). Nevertheless, the fact that $\H^*(\qHodge_{S/\IZ, \square}^*/(q^m-1))$ is the $(q-1)$-completion of something canonical looks exactly like what we would expect to see if the $q$-Hodge complex were to descend to the Habiro ring!
		
		Our first main result says that this is indeed true: The $q$-Hodge complex, whenever it is defined, is canonically the $(q-1)$-completion of another object defined over the Habiro ring. Moreover, a derived version of the comparison with $q$-de Rham--Witt complexes holds true, even without the $(q-1)$-completion:
	\end{numpar}
	\begin{thm}[see \cref{thm:HabiroDescent}]\label{thm:HabiroDescentIntro}
		Let $\widehat{\Dd}_{(q-1)}(\IZ[q])$ and $\widehat{\Dd}_\Hh(\IZ[q])$ denote the $(q-1)$- and Habiro-complete objects \embrace{see \cref{appendix:HabiroCompletion}}, respectively, in the derived $\infty$-category of $\IZ[q]$.
		\begin{alphanumerate}
			\item The $q$-Hodge complex functor $\qHodge_{-/\IZ}\colon \cat{AniAlg}_\IZ^{\qHodge}\rightarrow \widehat{\Dd}_{(q-1)}(\IZ[q])$ admits a non-trivial symmetric monoidal factorisation\label{enum:HabiroDescentIntro}
			\begin{equation*}
				\begin{tikzcd}[column sep=huge]
					& \widehat{\Dd}_\Hh\bigl(\IZ[q]\bigr)\dar["(-)_{(q-1)}^\complete"]\\
					\cat{AniAlg}_\IZ^{\qHodge}\rar["\qHodge_{-/\IZ}"']\urar[dashed,"\qHhodge_{-/\IZ}"] & \widehat{\Dd}_{(q-1)}\bigl(\IZ[q]\bigr)
				\end{tikzcd}
			\end{equation*}
			\item For all $m\in\IN$, the quotient $\qHhodge_{-/\IZ}/(q^m-1)$ admits an exhaustive ascending filtration $ \fil_\star ^{\qIW_m\Omega}(\qHhodge_{-/\IZ}/(q^m-1))$ with associated graded\label{enum:qdRWComparisonIntro}
			\begin{equation*}
				\gr_*^{\qIW_m\Omega}\bigl(\qHhodge_{-/\IZ}/(q^m-1)\bigr)\simeq \Sigma^{-*}\qIW_m\deRham_{-/\IZ}^*\,,
			\end{equation*}
			where $\qIW_m\deRham_{-/\IZ}^*$ denotes the derived $m$-truncated $q$-de Rham--Witt complex. 
		\end{alphanumerate}
	\end{thm}
	\begin{defi}
		$\qHhodge_{-/\IZ}$ from \cref{thm:HabiroDescent}\cref{enum:HabiroDescentIntro} will be called \emph{Habiro--Hodge complex}.
	\end{defi}
	\begin{rem}
		For objects $(S,\fil_{\qHodge}^\star\qdeRham_{S/\IZ})\in\cat{AniAlg}_\IZ^{\qHodge}$ such that $S$ is smooth over $\IZ$, we'll show in \cref{prop:Letaq-1} that $\qOmega_{S/\IZ}\simeq \L\eta_{(q-1)}\qHodge_{S/\IZ}$, where $\L\eta_{(q-1)}$ denotes the Berthelot--Ogus décalage functor (see \cite[\S{\chref[section]{6}}]{BMS1} or \cite[\stackstag{0F7N}]{Stacks}). So in this case, the $q$-de Rham complex can be descended to the Habiro ring as well via $\L\eta_{(q-1)}\qHhodge_{S/\IZ}$.
		
		In general, it seems that only the $q$-Hodge complex and not the $q$-de Rham complex admits descent to the Habiro ring. Here's an informal reason why the $q$-Hodge complex is a more canonical candidate for such a descent: In the (coordinate-dependent) $q$-de Rham complex of $\IZ[x]$, the $q$-differential sends $x^m\mapsto [m]_q x^{m-1}\d x$, where $[m]_q=1+q+\dotsb+q^{m-1}$ is the $q$-analogue of $m$. This formula gives \enquote{special treatment to $q=1$}, whereas the most canonical object to descend to the Habiro ring should \enquote{treat all roots of unity equally}. So we should look for a complex with differentials that send $x^m\mapsto (q^m-1)x^{m-1}\d x$, which leads to the (coordinate-dependent) $q$-Hodge complex.
	\end{rem}
	
	\subsection{Existence results for \texorpdfstring{$q$}{q}-Hodge filtrations}\label{subsec:FunctorialqHodgeIntro}
	
	Even though $\cat{AniAlg}_\IZ^{\qHodge}\rightarrow \cat{AniAlg}_\IZ$ has no section, the $\infty$-category $\cat{AniAlg}_\IZ^{\qHodge}$ still has many interesting objects. One large class of examples can be construced from topological Hochschild homology over $\ku$; this will be the content of the companion paper \cite{qdeRhamku}.
	
	In this paper, we will describe two elementary constructions of $q$-Hodge filtrations that provide sections of $\cat{AniAlg}_\IZ^{\qHodge}\rightarrow \cat{AniAlg}_\IZ$ over fairly large full subcategories of $\cat{AniAlg}_\IZ$. Let us begin with a construction in the smooth case.
	
	\begin{numpar}[Canonical $q$-Hodge filtrations (smooth case).]\label{par:qHodgeForCurves}
		Let $S$ be smooth over~$\IZ$. The most naive idea to equip $\qOmega_{S/\IZ}$ (or rather $\qdeRham_{S/\IZ}$, but this distinction doesn't matter by \cref{rem:CompleteFiltration}) with a $q$-Hodge filtration would be to simply take the pullback
		\begin{equation*}
			\begin{tikzcd}
				\fil_{\qHodge}^\star\qOmega_{S/\IZ}\dar\rar\drar[pullback] & \qOmega_{S/\IZ}\dar\\
				\fil_{\Hodge}^\star\Omega_{S/\IZ}\rar & \Omega_{S/\IZ}
			\end{tikzcd}
		\end{equation*}
		This cannot work, of course, because in this pullback each filtration step $\fil_{\qHodge}^\star\qOmega_{S/\IZ}$ will contain all of $(q-1)\qOmega_{S/\IZ}$. In view of \cref{def:qHodgeIntro}\cref{enum:qHodgeIntroc} this is only ok for $\star\leqslant 1$.
		
		Now if $S$ has relative dimension $\dim(S/\IZ)\leqslant 1$, then the Hodge filtration $\fil_{\Hodge}^\star \Omega_{S/\IZ}^*$ is trivial in filtration degrees $\star\geqslant 2$. Consequently, if $\fil_{\qHodge}^\star\qOmega_{S/\IZ}$ is any filtered $q$-deformation of the Hodge filtration, then $\fil_{\qHodge}^\star\qOmega_{S/\IZ}$ will necessarily be given by the $(q-1)$-adic filtration $(q-1)^{\star-1}\fil_{\qHodge}^1\qOmega_{S/A}$ in filtration degrees $\geqslant 1$. We may thus define the first filtration step $\fil_{\qHodge}^1\qOmega_{S/A}$ using the pullback above and then construct the rest of the filtration as
		\begin{equation*}
			\Bigl(\qOmega_{S/\IZ}\leftarrow\fil_{\qHodge}^1\qOmega_{S/\IZ}\leftarrow (q-1)\fil_{\qHodge}^1\qOmega_{S/\IZ}\leftarrow (q-1)^2\fil_{\qHodge}^1\qOmega_{S/\IZ}\leftarrow\dotsb\Bigr)\,.
		\end{equation*}
		One can (and we will) check that this satisfies all expected properties. As we'll see in \cref{par:CanonicalqHodgeSmoothI}, a variant of this trick still works for $S$ of arbitrary dimension, as long as all primes $p\leqslant \dim(S/\IZ)$ become invertible in $S$. This will lead to the following theorem:
	\end{numpar}
	\begin{thm}[see \cref{thm:CanonicalqHodgeSmooth}]\label{thm:CanonicalqHodgeSmoothIntro}
		The forgetful functor $\cat{AniAlg}_\IZ^{\qHodge}\rightarrow \cat{AniAlg}_\IZ$ admits a section over the full subcategory $\cat{Sm}_{\IZ[\dim!^{-1}]}\subseteq \cat{AniAlg}_\IZ$ of smooth $\IZ$-algebras $S$ such that all primes $p\leqslant \dim(S/\IZ)$ become invertible in $S$.
	\end{thm}
	\begin{numpar}[Algebraic Habiro cohomology.]\label{par:AlgebraicHabiroCohomology}
		Combining \cref{thm:HabiroDescentIntro,thm:CanonicalqHodgeSmoothIntro} allows us to define canonical objects $\qHhodge_{X/\IZ}\in\Dd(X,\Hh)$ for any smooth scheme $X$ over $\IZ$ such that all primes $p\leqslant \dim(X/\IZ)$ are invertible on $X$. The sheaf cohomology of $\qHhodge_{X/\IZ}$ then deserves to be called the \emph{algebraic Habiro cohomology of $X$}. It behaves in many ways (but not all; see \cref{par:AlgebraicvsAnalytic}\cref{enum:AlgebraicvsAnalyticC} below) as one would expect from an \enquote{algebraic} theory. For example, if $X$ is smooth and proper over $\IZ[1/N]$, where $N$ is also divisible by all primes $p\leqslant \dim(X/\IZ)$, then $\R\Gamma(X,\qHhodge_{X/\IZ})$ will be a perfect complex over the Habiro-completion of $\Hh[1/N]$.
		
		Thus, up to throwing away \enquote{small primes}, algebraic Habiro cohomology provides a tentative answer to Question~\cref{qst:Habiro}.
	\end{numpar}
	\begin{numpar}[Algebraic vs.\ analytic Habiro cohomology.]\label{par:AlgebraicvsAnalytic}
		It is not yet known how algebraic Habiro cohomology relates to the sheaf cohomology of Scholze's analytic Habiro stack $X^\mathrm{Hab}$, which we would like to call \emph{analytic Habiro cohomology} for clarity. We expect algebraic and analytic Habiro cohomology to become equal after base change to a suitably completed localisation of Scholze's analytic Habiro ring $\Hh^\mathrm{an}$. Note, however, that this base change erases quite some information on either side, and there are several key differences between the algebraic and the analytic construction:
		\begin{alphanumerate}
			\item \emph{Evaluation at \enquote{small primes}.} By construction, algebraic Habiro cohomology of a smooth scheme $X$ will contain no information at primes $p\leqslant \dim(X/\IZ)$. By contrast, analytic Habiro usually does contain non-trivial information at such primes, as $N$ is \emph{not} invertible everywhere on the Habiro stack $\IZ[1/N]^\mathrm{Hab}$.
			\item \emph{Evaluation at roots of unity.} With the current construction, Scholze's analytic Habiro stack becomes the algebraic de Rham stack if $q$ is specialised to a root of unity. In particular, its cohomology will be Grothendieck's infinitesimal cohomology, which is ill-behaved in characteristic~$p$. With algebraic Habiro cohomology, evaluation at roots of unity yields $q$-de Rham--Witt cohomology by \cref{thm:HabiroDescentIntro}\cref{enum:qdRWComparisonIntro}, which is much closer to crystalline cohomology in characteristic~$p$.
			\item \emph{Stacky approach.} By construction, analytic Habiro cohomology comes with a stacky approach. For algebraic Habiro cohomology, we don't expect a stacky approach to exist. In fact, we don't even expect $\qHhodge_{X/\IZ}$ to carry an $\IE_\infty$-algebra structure! The reason goes roughly as follows: The multiplication (\enquote{cup product}) on $\qHhodge_{X/\IZ}$ should come from the diagonal embedding $\Delta\colon X\rightarrow X\times X$. Thus, for the multiplication to be defined, we need to invert all primes $p\leqslant \dim(X\times X/\IZ)=2\dim(X/\IZ)$. Similarly, for the multiplication to be homotopy-associative, we should invert all primes up to $p\leqslant 3\dim(X/\IZ)$, and to get it more and more coherent, we need to invert more and more primes.\label{enum:AlgebraicvsAnalyticC}
		\end{alphanumerate}
		In the second half of \cref{subsec:CanonicalqHodgeSmooth} we'll make the considerations from~\cref{enum:AlgebraicvsAnalyticC} precise, and we'll show a formal monoidality statement in \cref{cor:CanonicalqHodgeSmoothMonoidal}.
	\end{numpar}
	
	There's a second case in which we're able to show the existence of functorial $q$-Hodge filtrations:
	
	\begin{numpar}[Canonical $q$-Hodge filtrations (quasi-regular quotient case).]\label{par:CanonicalqHodgeQuasiregular}
		Let $R$ be a ring satisfying the following conditions:
		\begin{alphanumerate}\itshape
			\item[R] For all primes~$p$, $R$ is $p$-torsion free, the $p$-completed derived de Rham complex $(\deRham_{R/\IZ})_p^\complete$ is static, i.e.\ an actual ring concentrated in degree~$0$, and the Hodge filtration $\fil_{\Hodge}^\star(\deRham_{R/\IZ})_p^\complete$ is a descending filtration of ideals.\label{enum:RQuasiregular}
		\end{alphanumerate}
		For example, this happens in the following case (see \cref{lem:RelativelySemiperfect}): Let $B$ be a perfect $\Lambda$-ring, let $B'$ be an étale $B$-algebra, and let $R\cong B'/J$, where $J$ is generated by a Koszul-regular sequence. If $R$ is $p$-torsion free, then it will satisfy the other conditions from~\cref{enum:RQuasiregular} as well.
		
		If $R$ satisfies these assumptions, then $(\qdeRham_{R/\IZ})_p^\complete$ is static as well, and we can construct a filtration on on it in a very naive way: We define $\fil_{\qHodge}^\star(\qdeRham_{R/\IZ})_p^\complete$ to be the (non-derived!) preimage of the combined Hodge and $(q-1)$-adic filtration $\fil_{(\Hodge,q-1)}^\star(\deRham_{R/\IZ})_p^\complete[1/p]\qpower$ under
		\begin{equation*}
			\bigl(\qdeRham_{R/\IZ}\bigr)_p^\complete\longrightarrow\bigl(\qdeRham_{R/\IZ}\bigr)_p^\complete\bigl[\localise{p}\bigr]_{(q-1)}^\complete\simeq \bigl(\deRham_{R/\IZ}\bigr)_p^\complete\bigl[\localise{p}\bigr]\qpower\,.
		\end{equation*}
		A priori, there's no reason to expect that $\fil_{\qHodge}^\star(\qdeRham_{R/\IZ})_p^\complete$ would be well-behaved at all; in particular, it's usually not a $q$-deformation of the Hodge filtration $\fil_{\Hodge}^\star(\deRham_{R/\IZ})_p^\complete$. This is closely related to the non-existence of a section of $\cat{AniAlg}_\IZ^{\qHodge}\rightarrow \cat{AniAlg}_\IZ$; see \cref{exm:MainSpecialCase}. However, in the following two cases everything works:
	\end{numpar}
	\begin{thm}[see \cref{thm:qHodgeWellBehaved}]\label{thm:qHodgeWellBehavedIntro}
		With assumptions as above, suppose that one of the following two additional conditions is satisfied:
		\begin{alphanumerate}
			\item $R\cong B'/J$ as in \cref{par:CanonicalqHodgeQuasiregular} and $J$ is generated by a Koszul regular sequence of higher powers; that is, $J=(x_1^{\alpha_1},\dotsc,x_r^{\alpha_r})$, where $(x_1,\dotsc,x_r)$ is Koszul-regular and $\alpha_i\geqslant 2$ for all~$i$.\label{enum:RegularSequenceHigherPowersIntro}
			\item $R$ admits a lift to an $\IE_1$-ring spectrum $\IS_R$ such that $R\simeq \IS_R\otimes\IZ$.\label{enum:E1LiftIntro}
		\end{alphanumerate}
		Then the above filtration $\fil_{\qHodge}^\star(\qdeRham_{R/\IZ})_p^\complete$ is a $q$-deformation of the Hodge filtration
	\end{thm}
	Note that~\cref{thm:qHodgeWellBehavedIntro}\cref{enum:E1LiftIntro} is purely an existence condition; the choice of $\IS_R$ doesn't matter! We'll show in \cref{con:GlobalqHodge} that the $p$-complete filtrations $\fil_{\qHodge}^\star(\qdeRham_{R/\IZ})_p^\complete$ can be glued with the combined Hodge and $(q-1)$-adic filtration $\fil_{(\Hodge,q-1)}^\star(\deRham_{R/\IZ}\lotimes_\IZ\IQ)\qpower$ to obtain a filtration $\fil_{\qHodge}^\star\qdeRham_{R/\IZ}$ on the global derived $q$-de Rham complex. This leads to the following result:
	
	\begin{thm}[see \cref{thm:CanonicalqHodgeQuasiregular}]\label{thm:CanonicalqHodgeQuasiregularIntro}
		Let $\cat{QReg}_\IZ^{\qHodge}$ denote the category of rings $R$ that satisfy the conditions from \cref{par:CanonicalqHodgeQuasiregular}\cref{enum:RQuasiregular} and such that $\fil_{\qHodge}^\star(\qdeRham_{R/\IZ})_p^\complete$ is a $q$-deformation of the Hodge filtration for all primes~$p$. Then the forgetful functor $\cat{AniAlg}_\IZ^{\qHodge}\rightarrow \cat{AniAlg}_\IZ$ admits a section over the full subcategory $\cat{QReg}_\IZ^{\qHodge}\subseteq \cat{AniAlg}_\IZ$.
	\end{thm}
	
	\begin{numpar}[Uniqueness of sections.]
		It's natural to ask if the sections from \cref{thm:CanonicalqHodgeSmoothIntro,thm:CanonicalqHodgeQuasiregularIntro} are unique. In the quasi-regular case, it will be straightforward to see that the section we construct is terminal among all choices, and then the $q$-deformation condition from \cref{def:qHodgeIntro}\cref{enum:qHodgeIntrob} forces it to be unique.
		
		In the smooth case, we need to assume additionally that our $q$-Hodge filtrations are compatible with the morphism $\qOmega_{S/A}\rightarrow \Omega_{S/A}\qpower/(q-1)^n$ from \cref{par:CanonicalqHodgeSmoothI} below, or alternatively, that $\fil_{\qHodge}^\star(\qdeRham_{R/\IZ})_p^\complete$ acquires a $\IZ_p^\times$-action compatible with the one on $p$-completed $q$-de Rham cohomology (see \cref{rem:pTildeDeRham}). If this additional compatibility is assumed, it will be straightforward to see that the section we construct is initial among all choices, and thus unique again by \cref{def:qHodgeIntro}\cref{enum:qHodgeIntrob}.
	\end{numpar}
	
	\subsection{Organisation of this paper}
	In \cref{sec:HabiroRingsEtale}, we'll recall and generalise the construction of \emph{Habiro rings of étale extensions} from \cite{HabiroRingOfNumberField}, and we'll relate them to the rings of \emph{$q$-Witt vectors} from \cite{qWitt}. This is a special case of our more general results in \cref{sec:HabiroDescent}, but much less technical, so it will be worthwhile to spell out the étale case first.
	
	In \cref{sec:HabiroDescent}, we'll prove our main Habiro descent result. This section is long and technical. It may be helpful to read the proof of \cref{lem:qHodgeSymmetricMonoidal} first. At the end of \cref{subsec:MainResult}, we explain how the proof of \cref{thm:HabiroDescent} proceeds by generalising the arguments from the proof of \cref{lem:qHodgeSymmetricMonoidal}.
	
	In \cref{sec:FunctorialqHodge}, we'll show that despite the general non-existence result, it's possible to construct functorial $q$-Hodge filtrations on fairly large full subcategories of rings. Finally, there will be two appendices: In \cref{appendix:GlobalqDeRham}, we explain the gluing argument to obtain the global $q$-de Rham complex from the $p$-complete $q$-de Rham complex of Bhatt--Scholze. In \cref{appendix:HabiroCompletion}, we study the completion that appears in the Habiro ring and show that it behaves like the usual (derived) completion at an ideal.
	
	\begin{numpar}[Notation and conventions.]\label{par:Notations}
		Throughout the article, we freely use the language of $\infty$-categories, and we'll adopt the following conventions:
		\begin{alphanumerate}
			\item \textbf{Graded and filtered objects.} For a stable $\infty$-category $\Cc$, we let $\Gr(\Cc)$ and $\Fil(\Sp)$ denote the $\infty$-categories of \emph{graded} and \emph{\embrace{descendingly} filtered objects in $\Cc$}. The shift in graded or filtered objects will be denoted $(-)(1)$. An object with a descending filtration is typically denoted
			\begin{equation*}
				\fil^\star X=\Bigl(\dotsb\leftarrow \fil^nX\leftarrow \fil^{n+1}X\leftarrow\dotsb\Bigl)
			\end{equation*}
			and we let $\gr^*X$ denote the \emph{associated graded}, given by $\gr^nX\coloneqq \cofib(\fil^{n+1}X\rightarrow \fil^nX)$. We mostly work with filtrations that are constant in degrees $\leqslant 0$ (such as the Hodge filtration). In this case we'll abusingly write $\fil^\star X=(\fil^0X\leftarrow \fil^1 X\leftarrow\dotsb)$; this should be interpreted as the constant $\fil^0X$-valued filtration in degrees $\leqslant 0$.

			If $\Cc$ is symmetric monoidal and the tensor product $-\otimes-$ commutes with colimits in both variables, we equip $\Gr(\Cc)$ and $\Fil(\Cc)$ with their canonical symmetric monoidal structures given by Day convolution. We'll use the fact that $\Fil(\Cc)\simeq \Mod_{\IUnit_{\Gr}[t]}\Gr(\Cc)$, where $\IUnit_{\Gr}$ denotes the tensor unit in $\Gr(\Cc)$ and $t$ sits in graded degree~$-1$; see e.g.\ \cite[Proposition~\chref{3.2.9}]{RaksitFilteredCircle}. Under this equivalence, passing to the associated graded corresponds to \enquote{modding out~$t$}, i.e.\ the base change $\IUnit_{\Gr}\otimes_{\IUnit_{\Gr}[t]}-$.
			
			We say that $\fil^\star X$ is an \emph{exhaustive filtration on $X$} if $X\simeq \colimit_{n\rightarrow -\infty}\fil^nX$. We say that a filtered object $\fil^\star X$ is \emph{complete} if $0\simeq \limit_{n\rightarrow \infty}\fil^nX$. We define the \emph{completion} $\fil^\star \widehat{X}\coloneqq \limit_{n\rightarrow\infty}\cofib(\fil^{\star+n} X\rightarrow\fil^{\star}X)$. By construction, there's a pullback square
			\begin{equation*}
				\begin{tikzcd}
					\fil^\star X\rar\dar\drar[pullback] & \fil^\star \widehat{X}\dar\\
					X\rar & \widehat{X}
				\end{tikzcd}
			\end{equation*}
			We'll often refer to this by saying that \emph{every filtration is the pullback of its completion}.
			
			Sometimes we also consider \emph{ascending} filtrations. Ascendingly filtered objects will be denoted $\fil_\star X=(\dotsb\rightarrow \fil_nX\rightarrow \fil_{n+1}X\rightarrow \dotsb)$ and the associated graded by $\gr_*X$, where $\gr_nX\coloneqq \cofib(\fil^{n-1}X\rightarrow \fil^nX)$.
			
			\item \textbf{Animation.} For an ordinary ring $A$, we consider the $\infty$-category of \emph{animated $A$-algebras $\cat{AniAlg}_A$}, which is the $\infty$-category freely generated under sifted colimits (in the sense of \cite[Proposition~\chref{5.5.8.15}]{HTT}) by the category $\cat{Poly}_A$ of polynomial $A$-algebras in finitely many variables.
			
			We'll often use the fact that any functor $F\colon \cat{Poly}_A\rightarrow \Dd$ into an $\infty$-category with all sifted colimits can be uniquely extended to a sifted colimits preserving functor $\L F\colon \cat{AniAlg}_A\rightarrow \Dd$. We often call $\L F$ the \emph{animation} or the \emph{\embrace{non-abelian} derived functor of $F$}. The most important examples for us will be the $q$-de Rham complex  (which is only defined in the case where $A$ is a $\Lambda$-ring) and the Hodge-filtered de Rham complex
			\begin{equation*}
				\qOmega_{-/A}\colon \cat{Poly}_A\longrightarrow \widehat{\Dd}_{(q-1)}\bigl(A\qpower\bigr)\quad\text{and}\quad \fil_{\Hodge}^\star\Omega_{-/A}^*\colon \cat{Poly}_A\longrightarrow \Fil\Dd(A)
			\end{equation*}
			(the former is only defined if $A$ is a $\Lambda$-ring), whose animations we'll denote by $\qdeRham_{-/A}$ and $\fil_{\Hodge}^\star\deRham_{-/A}$, respectively.
			
			\item \textbf{Derived categories.} For a ring $R$, we let $\Dd(R)$ denote the derived $\infty$-category of $R$-modules. The shift functor and its inverse in $\Dd(R)$ will always be denoted by $\Sigma$ and $\Sigma^{-1}$, to avoid confusion with shifts in graded or filtered objects, as we'll frequently mix both settings.
			
			For an element $f\in R$ and an object $M\in \Dd(R)$, we let  \begin{equation*}
				M/f\coloneqq \cofib\left(f\colon M\rightarrow M\right)\,.
			\end{equation*}
			For several elements $f_1,\dotsc,f_r\in R$, we let $M/(f_1,\dotsc,f_r)\coloneqq (\dotsb(M/f_1)/f_2\dotsb)/f_r$. We warn the reader that for ordinary $R$-modules $M$, the derived quotient $M/(f_1,\dotsc,f_r)$ agrees with the usual quotient only if $(f_1,\dotsc,f_r)$ is a Koszul-regular sequence on $M$.
			
			Similarly, if $R^*$ is a graded ring, $f\in R^i$ is a homogeneous element of degree~$i$, and $M^*\in\Mod_{R^*}\Gr(\Dd(\IZ))$, we put $M^*/f\coloneqq \cofib(f\colon M(i)\rightarrow M)$ and define $M^*/(f_1,\dotsc,f_r)$ analogously. The same notation will also be used in the filtered setting, by regarding filtered objects as graded $\IUnit_{\Gr}[t]$-modules, as explained above.
			
			\item \textbf{Derived completions.} With notation as above, we define the \emph{\embrace{derived} $(f_1,\dotsc,f_r)$-adic completion} of $M$ as the limit
			\begin{equation*}
				\widehat{M}_{(f_1,\dotsc,f_r)}\coloneqq \lim_{n\geqslant 1}M/(f_1^n,\dotsc,f_r^n)
			\end{equation*}
			taken in the derived $\infty$-category $\Dd(R)$. Since the completion only depends on the ideal $I=(f_1,\dotsc,f_r)\subseteq R$, we often just write $\widehat{M}_I$ (or $(-)_I^\complete$ for longer arguments). Completion can be analogously defined in the graded or filtered setting and we'll use the same notation in these cases.
			
			In the case where $I=(f)$ is a principal ideal, the following square is always a pullback in $\Dd(R)$:
			\begin{equation*}
				\begin{tikzcd}
					M\rar\dar\drar[pullback] & \widehat{M}_f\dar\\
					M\bigl[\localise{f}\bigr]\rar & \widehat{M}_f\bigl[\localise{f}\bigr]
				\end{tikzcd}
			\end{equation*}
			This will be called a \emph{fracture square}. In the special case where $f=N$ is an integer, we'll use the term \emph{arithmetic fracture square}.
			
			For general~$I=(f_1,\dotsc,f_r)$, we let $\widehat{\Dd}_{I}(R)\subseteq \Dd(R)$, denote the full sub-$\infty$-category spanned by the \emph{$I$-complete objects}, that is, those $M$ for which $M\simeq \widehat M_{I}$. The following fact will be used countless times: If $M$ is $(f_1,\dotsc,f_r)$-complete, and the homology of $M/(f_1,\dotsc,f_r)$ vanishes in some degree~$d$, then also the homology of $M$ must vanish in degree~$d$. Analogous conclusions hold true in the graded and filtered settings.			
			
			\item \textbf{Perfectly covered $\boldsymbol\Lambda$-rings.} We call a $\Lambda$-ring \emph{perfectly} covered if there exists a faithfully flat $\Lambda$-morphism $A\rightarrow A_\infty$ into a perfect $\Lambda$-ring. Equivalently, the Adams operations $\psi^m\colon A\rightarrow A$ are all faithfully flat (see e.g.\ \cite[Remark~\chref{2.47}]{qWitt}). This condition is satisfied in many examples of interest; for example, it holds for $\IZ$, for any free $\Lambda$-ring $\IZ\left\{x_i\ \middle|\ i\in I\right\}$, and for any polynomial ring $\IZ\left[x_i\ \middle|\ i\in I\right]$ equipped with the \emph{toric $\Lambda$-structure} in which $\lambda^n(x_i)=0$ for all $n>1$.
		\end{alphanumerate}
	\end{numpar}
	
	\begin{numpar}[Acknowledgements.]
		First and foremost, I'm grateful to Peter Scholze for suggesting this question and for his support throughout the project. I'm glad that it has paid off to keep working on this question despite the initial upset: Four years after proving that the $q$-Hodge complex isn't functorial, we can now prove that it is. I would also like to thank Stavros Garoufalidis, Quentin Gazda, and Campbell Wheeler for many helpful discussions on Habiro cohomology.
		
		This work was carried out while I was a Ph.D.\ student at the MPIM/University Bonn, and I would like to thank these institutions for their hospitality. Finally, I'm grateful for the financial support provided by the DFG through Peter Scholze's Leibniz prize.
	\end{numpar}

	\newpage
	
	\section{Habiro rings of étale extensions}\label{sec:HabiroRingsEtale}

	Fix a perfectly covered $\Lambda$-ring $A$. The goal of this section is to construct a \emph{relative Habiro ring} $\Hh_{R/A}$ for any étale algebra $R$ over $A$, and to relate this construction to the theory of $q$-Witt vectors. In the case where $A=\IZ$, our construction $\Hh_{R/\IZ}$ recovers the ring $\Hh_R$ from \cite[Definition~\chref{1.1}]{HabiroRingOfNumberField}.
	
	As we'll see in \cref{sec:HabiroDescent}, the construction of $\Hh_{R/A}$ is a special case of a much more general construction. However, the general case is vastly more technical, and we hope that discussing the étale case first will make the general case easier.
	
	\subsection{A general descent principle}\label{subsec:DescentPrinciple}
	To construct $\Hh_{R/A}$, we'll first construct the completions $(\Hh_{R/A})_{\Phi_m(q)}^\complete$ for all $m\in\IN$ and then \enquote{glue them together} using a very general descent principle that we'll explain in this subsection. It will probably seem a little overkill for now, but we'll use the same descent principle again in \cref{subsec:TwistedqDeRham} to construct the twisted $q$-de Rham complexes $\qdeRham_{R/A}^{(m)}$.
	
	\begin{numpar}[Setup.]\label{par:GeneralDescent}
		Let $\Ii$ be a site whose underlying category is a partially ordered set. Let $\Dd$ be a presentable stable symmetric monoidal $\infty$-category. Suppose that for every $Z\in \Ii$ we have a full stable sub-$\infty$-category $\Dd_Z$ satisfying the following conditions:
		\begin{alphanumerate}
			\item The inclusion $\Dd_Z\subseteq \Dd$ admits a left adjoint $L_Z\colon \Dd\rightarrow \Dd_Z$.\label{enum:DZhasLeftAdjoint}
			\item Whenever $Z_1\rightarrow Z_2$ is a morphism in $\Ii$, we have $\Dd_{Z_1}\subseteq \Dd_{Z_2}$. Note that $L_{Z_1}\colon \Dd_{Z_2}\rightarrow \Dd_{Z_1}$ is still a left adjoint of this inclusion.\label{enum:DZInclusion}
			\item For all $x,y\in \Dd$ and all $Z\in\Ii$, the canonical morphism $L_Z(x\otimes y)\rightarrow L_Z(L_Z(x)\otimes y)$ is an equivalence in $\Dd$.\label{enum:DZSymmetricMonoidal}
		\end{alphanumerate}
		In this case, sending $Z\mapsto \Dd_Z$ and $(Z_1\rightarrow Z_2)\mapsto (L_{Z_1}\colon \Dd_{Z_2}\rightarrow \Dd_{Z_1})$ defines a contravariant functor
		\begin{equation*}
			\Dd_{(-)}\colon \Ii^\op\longrightarrow \CAlg(\Pr_\mathrm{st}^\L)
		\end{equation*}
		into the $\infty$-category of presentable stable symmetric monoidal $\infty$-categories. Indeed, let's ignore the symmetric monoidal structure for the moment and let $\Dd_\Ii\subseteq \Ii\times\Dd$ be the full sub-$\infty$-category spanned fibrewise by $\Dd_Z\subseteq \{Z\}\times\Dd$. By \cref{enum:DZInclusion}, $\Dd_\Ii\rightarrow \Ii$ is still a cocartesian fibration and so it defines a covariant functor $\Dd_{(-)}\colon \Ii\rightarrow \Cat_\infty$. By \cref{enum:DZhasLeftAdjoint}, this functor factors through $\Pr_\mathrm{st}^\R$. Using $\Pr_\mathrm{st}^\L\simeq (\Pr_\mathrm{st}^\R)^\op$ by \cite[Corollary~\chref{5.5.3.4}]{HTT}, we get the desired functor $\Dd_{(-)}\colon \Ii^\op\rightarrow \Pr_\mathrm{st}^\L$.
		
		To incorporate the symmetric monoidal structure, let $\Dd$ be the $\infty$-operad $\Dd^\otimes$ associated to the given symmetric monoidal structure on $\Dd$. By \cref{enum:DZSymmetricMonoidal} and \cite[Proposition~\chref{2.2.1.9}]{HA}, for all $Z\in \Ii$, the inclusion of the full sub-$\infty$-operad $\Dd_Z^\otimes\subseteq \Dd^\otimes$ spanned by $\Dd_Z$ admits a symmetric monoidal left adjoint $L_Z^\otimes\colon \Dd^\otimes\rightarrow \Dd_Z^\otimes$ which recovers $L_Z$ on underlying $\infty$-categories. Using this observation, the same argument as above can be repeated with $\Dd$ replaced by $\Dd^\otimes$.
	\end{numpar}
	
	\begin{lem}\label{lem:AbstractDescent}
		In the situation of \cref{par:GeneralDescent}, assume that covers in $\Ii$ always have finite refinements and that for any finite covering family $\{Z_i\rightarrow Z\}_{i=1,\dotsc,r}$, the functors $L_{Z_i}\colon \Dd_Z\rightarrow \Dd_{Z_i}$ are jointly conservative. Then
		\begin{equation*}
			\Dd_{(-)}\colon \Ii^\op\longrightarrow \CAlg(\Pr_\mathrm{st}^\L)
		\end{equation*}
		is a sheaf on $\Ii$. In particular, $\CAlg(\Dd_{(-)})\colon \Ii^\op\rightarrow \Pr^\L$ is a sheaf as well.
	\end{lem}
	\begin{proof}[Proof sketch]
		Everything can be checked on the level of underlying $\infty$-categories, so we can disregard the symmetric monoidal structure (but it was still essential to include the symmetric monoidal structure in the construction). By assumption, it's enough to check the sheaf property for a finite cover $\{Z_i\rightarrow Z\}_{i=1,\dotsc,r}$. If $Z_\bullet$ denotes its \v Cech nerve, we need to show $\Dd_Z\simeq \limit_{\IDelta}\Dd_{Z_\bullet}$. Since $\Ii$ is a partially ordered set, we have $Z_{i}\times_ZZ_i= Z_i$ for all $i$. It follows that the cosimplicial limit can be simplified to a limit indexed by the set $\pullbacksign^r\coloneqq \Pp(\{1,\dotsc,r\})\smallsetminus\{\emptyset\}$ of non-empty subsets of $\{1,\dotsc,r\}$, partially ordered by inclusion. Therefore, we must show
		\begin{equation*}
			\Dd_Z\simeq \limit_{S\in\pullbacksign^r}\Dd_{Z_S}\,,
		\end{equation*}
		where we put $Z_S\coloneqq Z_{i_0}\times_Z\times\dotsb\times_ZZ_{i_k}$ for every non-empty subset $S=\{i_0,\dotsc,i_k\}\subseteq \{1,\dotsc,r\}$. To prove that $\Dd_Z\rightarrow \limit_{S\in\pullbacksign^r}\Dd_{Z_S}$ is fully faithful, we have to show that 
		\begin{equation*}
			\Hom_{\Dd_Z}(x,y)\longrightarrow \Hom_{\Dd_{Z_S}}\bigl(L_{Z_S}(x),L_{Z_S}(y)\bigr)
		\end{equation*}
		is an equivalence for all $x,y\in \Dd_Z$. Rewriting $\Hom_{\Dd_{Z_S}}(L_{Z_S}(x),L_{Z_S}(y))\simeq \Hom_\Dd(x,L_{Z_S}(y))$, this reduces to showing that $y\rightarrow\limit_{S\in\pullbacksign^r}L_{Z_S}(y)$ is an equivalence. This can be checked after applying the jointly conservative functors $L_{Z_i}\colon \Dd_Z\rightarrow \Dd_{Z_i}$. After applying $L_{Z_i}$, each $L_{Z_S}(y)\Rightarrow L_{Z_{S\cup\{i\}}}$ becomes an equivalence. This easily implies $L_{Z_i}(y)\simeq \limit_{S\in\pullbacksign^r}L_{Z_i}(L_{Z_S}(y))$ (for example, by the dual of \cite[Lemma~\chref{1.2.4.15}]{HA}). Since $L_{Z_i}$ preserves finite limits, this shows that $y\rightarrow\limit_{S\in\pullbacksign^r}L_{Z_S}(y)$ is an equivalence after applying $L_{Z_i}$, and so fully faithfulness follows. The same argument shows essential surjectivity.
	\end{proof}
	\begin{rem}\label{rem:CompleteDescent}
		The quintessential example for \cref{lem:AbstractDescent} is the case where $R$ is some ring, $\Dd\coloneqq \Dd(R)$ and $\Ii$ is the partially ordered set of closed subsets $Z\subseteq \Spec R$ with quasi-compact complement. Every such $Z$ is the vanishing set of a finitely generated ideal $I$ and we define $\Dd_Z\coloneqq \widehat{\Dd}_I(R)$; note that this only depends on $Z$, not on the choice of $I$. The functors $L_Z\coloneqq (-)_I^\complete$ clearly satisfy the conditions from \cref{par:GeneralDescent}, and the condition from \cref{lem:AbstractDescent} is easily checked (see e.g.\ \cite[Lemma~\chref{2.4}]{qWitt}). Hence the descent from \cref{lem:AbstractDescent} is applicable.
	\end{rem}
	In the case that we're actually interested in, the descent diagram simplifies considerably; in particular, no coherence data needs to be provided!
	\begin{cor}\label{cor:CompleteDescent}
		Let $m\in\IN$. Suppose we're given the following data:
		\begin{alphanumerate}
			\item For all divisors $d\mid m$, a derived $\Phi_d(q)$-complete $\IE_\infty$-$A[q]$-algebra $E_d$.\label{enum:CompleteDescentLocalData}
			\item For all divisors $pd\mid m$, where $p$ is a prime, an equivalence of $\IE_\infty$-$A[q]$-algebras\label{enum:CompleteDescentGluingData}
			\begin{equation*}
				h_{d}\colon (E_{pd})_p^\complete\overset{\simeq}{\longrightarrow}(E_d)_p^\complete\,.
			\end{equation*}
		\end{alphanumerate}
		Then there exists a unique $(q^m-1)$-complete $\IE_\infty$-$A[q]$-algebra $E$ together with equivalences $E_d\simeq E_{\Phi_d(q)}^\complete$ for all $d\mid m$ such that $h_d$ becomes identified with the identity on $E_{(\Phi_d(q),\Phi_{pd}(q))}^\complete$.
	\end{cor}
	\begin{proof}
		The idea is to apply descent for $R=A[q]$ and the cover $V(q^m-1)=\bigcup_{d\mid m}V(\Phi_d(q))$. The simplifications come from the observation that many intersections are empty; see \cite[Lemma~\chref{2.1}]{qWitt} for example.
		
		For a precise argument, let $T$ be the set of positive divisors of $m$ and let $\pullbacksign^{T}\coloneqq \Pp(T)\smallsetminus\{\emptyset\}$ denote the set of non-empty subsets of $T$, partially ordered by inclusion. For every $S\subseteq T$, put $\widehat{\Dd}_S\coloneqq \widehat{\Dd}_{\left(\Phi_d(q)\ \middle|\ d\in S\right)}(A[q])$. Then \cref{lem:AbstractDescent} implies
		\begin{equation*}
			\widehat{\Dd}_{(q^m-1)}\bigl(A[q]\bigr)\simeq \limit_{S\in\pullbacksign^r}\widehat{\Dd}_S\,.
		\end{equation*}
		For every pair $(d,p)$, where $d\mid m$ is a divisor of $m$ and $p$ is a prime such that $p\nmid d$, we let $T_{d,p}\coloneqq \{d,pd,\dotsc,p^{v_p(m)}d\}\subseteq T(m)$ and write $\pullbacksign^{T_{d,p}}\subseteq \pullbacksign^{T}$ for the corresponding sub-partially ordered set. By \cite[Lemma~\chref{2.1}]{qWitt}, we have $\widehat{\Dd}_S\simeq 0$ if $S\notin \bigcup_{d,p}\pullbacksign^{T_{d,p}}$. By inspection, this means that $\widehat{\Dd}_{(-)}\colon \pullbacksign^{T}\rightarrow \CAlg(\Pr_\mathrm{st}^\L)$ is right-Kan extended from $\bigcup_{d,p}\pullbacksign^{T_{d,p}}\subseteq \pullbacksign^{T}$. Furthermore, if $S\subseteq S'$ are elements of $\pullbacksign^{T_{d,p}}$ such that $\abs{S}\geqslant 2$, then the same result tells us that the corresponding morphism $S\rightarrow S'$ is sent to the identity, as both $\widehat{\Dd}_S$ and $\widehat{\Dd}_{S'}$ agree with the full sub-$\infty$-category 
		\begin{equation*}
			\widehat{\Dd}_{(p,\Phi_d(q))}\big(A[q]\big)\subseteq \Dd\bigl(A[q]\bigr)\,.
		\end{equation*}
		Again, by inspection, this means that $\widehat{\Dd}_{(-)}\bigcup_{d,p}\pullbacksign^{T_{d,p}}\rightarrow \CAlg(\Pr_\mathrm{st}^\L)$ is right-Kan extended from $P\subseteq \bigcup_{d,p}\pullbacksign^{T_{d,p}}$, where $P$ denotes the sub-partially ordered set spanned by $T_{d,p}\in \bigcup_{d,p}\pullbacksign^{T_{d,p}(m)}$ for all $d$, $p$ (note that this includes all subsets of the form $\{d\}$, where $d$ is a divisor of $m$, as $\{d\}=T_{d,\ell}$ if $\ell$ is any prime not dividing $m$). In total, this implies $\widehat{\Dd}_{(q^m-1)}\left(A[q]\right)\simeq \limit_{S\in P}\widehat{\Dd}_S$ and thus
		\begin{equation*}
			\CAlg\left(\widehat{\Dd}_{(q^m-1)}\bigl(A[q]\bigr)\right)\simeq \limit_{S\in P}\CAlg\bigl(\widehat{\Dd}_S\bigr)\,.
		\end{equation*}
		After unravelling of definitions, an object in the limit on the right-hand side is precisely given by the data~\cref{enum:CompleteDescentLocalData} and~\cref{enum:CompleteDescentGluingData}.
	\end{proof}
	\begin{rem}\label{rem:GluingCompletions}
		In \cref{cor:CompleteDescent}, we've glued $E$ from its $\Phi_d(q)$-completions $E_{\Phi_d(q)}^\complete\simeq E_d$ for all $d\mid m$. But $E$ can also be glued from from the completed localisation $E[1/m]_{(q^m-1)}^\complete$ and the completions $E_{(p,q^m-1)}^\complete$ for all primes $p\mid m$ via the usual arithmetic fracture square (see \cref{par:Notations}). For later use, let us explain how to extract the latter from the former: If $m=p^\alpha n$, where $n$ is coprime to $p$, then
		\begin{equation*}
			E\left[\localise{m}\right]_{(q^m-1)}^\complete\simeq \prod_{d\mid m}E_d\left[\localise{m}\right]_{\Phi_d(q)}^\complete\quad\text{and}\quad E_{(p,q^m-1)}^\complete \simeq \prod_{d\mid n}(E_{p^id})_p^\complete\quad \text{for any }0\leqslant i\leqslant \alpha\,.
		\end{equation*}
		For the equivalence on the left, just observe that the factors in $(q^m-1)=\prod_{d\mid m}\Phi_d(q)$ become coprime as soon as $m$ is invertible. For the equivalence on the right, observe that after $p$-completion the $\ell$-adic gluings for $\ell\neq p$ become vacuous, so the only gluing that happens is along $(E_d)_p^\complete\simeq (E_{pd})_p^\complete\simeq \dotsb\simeq (E_{p^\alpha d})_p^\complete$ for all $d\mid n$.
	\end{rem}
	\begin{rem}\label{rem:CompleteDescentDAlg}
		\cref{cor:CompleteDescent} remains true if we replace $\IE_\infty$-$A[q]$-algebras by derived commutative $A[q]$-algebras in the sense of \cite[Example~\chref{4.3.1}]{RaksitFilteredCircle}. The proof is entirely analogous.
	\end{rem}

	\subsection{Habiro rings of étale extensions}\label{subsec:HabiroRingOfNumberField}
	In the following, we fix a perfectly covered $\Lambda$-ring $A$ as before.
	
	\begin{numpar}[Relative Habiro rings.]\label{con:RelativeHabiroRing}
		Let $R$ be an étale $A$-algebra. For all primes $p$, the $p$\textsuperscript{th} Adams operation $\psi^p\colon A\rightarrow A$ can be uniquely extended to a Frobenius lift $\phi_p\colon \widehat{R}_p\rightarrow \widehat{R}_p$. Let us denote by
		\begin{equation*}
			\phi_{p/A}\colon \bigl(\widehat{R}_p\otimes_{A,\psi^p}A\bigr)_p^\complete\overset{\simeq}{\longrightarrow}\widehat{R}_p
		\end{equation*}
		the linearised Frobenius. It is an equivalence as indicated. Indeed, this can be checked modulo $p$, where it becomes classical; see \cite[\stackstag{0EBS}]{Stacks}. We also remark that $A$ being perfectly covered implies that $A$ is $p$-torsion free (because this is true for the perfect $\Lambda$-ring $A_\infty$), and so all $p$-completions above are static.
		
		For all $m\in\IN$, let us now define a $(q^m-1)$-complete $\IE_\infty$-$A[q]$-algebra $\Hh_{R/A,  m}$ via \cref{cor:CompleteDescent}: For every $d\mid m$, let $E_d\coloneqq (R\otimes_{A,\psi^d}A)[q]_{\Phi_d(q)}^\complete$ and for every $pd\mid m$, where $p$ is a prime, let the gluing equivalence $h_d$ be the $A[q]$-linear map induced by $\phi_{p/A}$.
		
		For all $d\mid m$, \cref{cor:CompleteDescent} provides a preferred equivalence $\Hh_{R/A,  d}\simeq (\Hh_{R/A,  m})_{(q^d-1)}^\complete$. In particular, we get maps $\Hh_{R/A,  m}\rightarrow \Hh_{R/A,  d}$. The \emph{Habiro ring of $R$ relative to $A$} is then defined as the limit%
		\footnote{A pedantic remark: To even write down this limit, we need to assemble the maps $\Hh_{R/A,  m}\rightarrow \Hh_{R/A,  d}$ into a functor $\Hh_{R/A,  (-)}\colon \IN\rightarrow \CAlg \Dd(A[q])$, where $\IN$ denotes the category of natural numbers partially ordered by divisibility. With a little more effort, this functoriality can be squeezed out of \cref{cor:CompleteDescent}. Alternatively, we can take the limit over the sequential subdiagram $\{n!\}_{n\geqslant 1}$, where the existence of maps is enough. Or we could use \cref{thm:HabiroqWittComparison} to realise that we're working with ordinary rings, so there are no higher coherences to check and functoriality can be obtained by hand.}
		\begin{equation*}
			\Hh_{R/A}\coloneqq \limit_{m\in\IN}\Hh_{R/A,  m}\,.
		\end{equation*}
	\end{numpar}
	\begin{rem}
		If we were to construct $R[q]_{(q^m-1)}^\complete$ using \cref{cor:CompleteDescent}, we would take $E_d\coloneqq R[q]_{\Phi_d(q)}^\complete$, together with the identity maps on $R[q]_{(p,\Phi_d(q))}^\complete$ (instead of $\phi_{p/A}$) as gluing equivalences. Thus, there's no reason to expect that $\Hh_{R/A,  m}\simeq R[q]_{(q^m-1)}^\complete$, unless $R$ itself (rather than only its $p$-completions) admits Frobenius lifts for all prime factors $p\mid m$. In the case $A=\IZ$, a precise obstruction of this kind is shown in \cite[Corollary~\chref{2.52}]{qWitt}.
	\end{rem}
	We can now formulate the relation between $\Hh_{R/A}$ and $q$-Witt vectors relative to $A$. To this end, recall from \cite[Proposition~\chref{2.48}]{qWitt} that $\qIW_m(R/A)$ is an étale algebra over $\qIW_m(A/A)\cong A[q]/(q^m-1)$.
	\begin{thm}\label{thm:HabiroqWittComparison}
		Let $A$ be a perfectly covered $\Lambda$-ring, $R$ an $A$-algebra, and $m\in\IN$. Then
		\begin{equation*}
			\Hh_{R/A,  m}/(q^m-1)\simeq \qIW_m(R/A)\,.
		\end{equation*}
		In fact, $\Hh_{R/A,  m}$ is the unique lift of the étale $A[q]/(q^m-1)$-algebra $\qIW_m(R/A)$ to a $(q^m-1)$-complete $\IE_\infty$-algebra over $A[q]_{(q^m-1)}^\complete$. In particular, $\Hh_{R/A,  m}$ is an ordinary ring for all $m\in\IN$, and the same is true for the relative Habiro ring $\Hh_{R/A}$.
	\end{thm}
	\begin{proof}
		Let, temporarily, $W$ denote the unique lift of $\qIW_m(R/A)$ to a $(q^m-1)$-complete $\IE_\infty$-algebra over $A[q]_{(q^m-1)}^\complete$. If $p$ is prime and $pd\mid m$, then the ghost maps for the usual Witt vectors $\IW_m(A/p)$ and $\IW_m(R/p)$ satisfy $\gh_{m/d}(x)=\gh_{m/pd}(x)^p$. It follows that the ghost maps for relative $q$-Witt vectors fit into a commutative diagram
		\begin{equation*}
			\begin{tikzcd}
				\left(R\otimes_{A,\psi^{pd}}A\right)[q]/\Phi_{pd}(q)\dar & \qIW_m(R/A) \lar["\gh_{m/pd}"']\rar["\gh_{m/d}"] & \left(R\otimes_{A,\psi^d}A\right)[q]/\Phi_d(q)\dar\\
				\bigl(R/p\otimes_{A/p,\psi^{pd}}A/p\bigr)[q]/\Phi_{pd}(q)\ar[rr] & & \bigl(R/p\otimes_{A/p,\psi^d}A/p\bigr)[q]/\Phi_d(q)
			\end{tikzcd}
		\end{equation*}
		where the bottom horizontal map is induced by the relative Frobenius $R/p\otimes_{A/p,(-)^p}A/p\rightarrow R/p$. After passing to unique deformations of étale algebras everywhere, we obtain a similar diagram
		\begin{equation*}
			\begin{tikzcd}
				\left(R\otimes_{A,\psi^{pd}}A\right)[q]_{\Phi_{pd}(q)}^\complete\dar & W\lar\rar& \left(R\otimes_{A,\psi^d}A\right)[q]_{\Phi_d(q)}^\complete\dar\\
				\bigl(\widehat{R}_p\otimes_{A,\psi^{pd}}A\bigr)[q]_{(p,\Phi_{pd}(q))}^\complete\ar[rr]& & \bigl(\widehat{R}_p\otimes_{A,\psi^d}A\bigr)[q]_{(p,\Phi_d(q))}^\complete
			\end{tikzcd}
		\end{equation*}
		where the bottom horizontal map is induced by $\phi_{p/A}$ from \cref{con:RelativeHabiroRing}. By construction of $\Hh_{R/A,  m}$, this yields an $\IE_\infty$-$A[q]$-algebra map $W\rightarrow \Hh_{R/A,  m}$. As both sides are $(q^m-1)$-complete, whether this is an equivalence can be checked modulo $\Phi_d(q)$ for all $d\mid m$. By \cite[Corollary~\chref{2.51}]{qWitt} and \cref{con:RelativeHabiroRing},
		\begin{equation*}
			W/\Phi_d(q)\simeq R\otimes_{A,\psi^d}A[q]/\Phi_d(q)\simeq \Hh_{R/A,  m}/\Phi_d\,.
		\end{equation*}
		As the equivalence on the left is induced via the ghost map $\gh_{m/d}$, it is apparent from our construction that $W/\Phi_d(q)\rightarrow \Hh_{R/A,  m}/\Phi_d(q)$ is given by the chain of equivalences above. This finishes the proof that $\Hh_{R/A,  m}$ is the unique deformation of $\qIW_m(R/A)$. 
		
		Since $\Hh_{R/A,  m}$ is $(q^m-1)$-complete and becomes static modulo~$p$, we see that $\Hh_{R/A,  m}$ must be static as well. Therefore it is an ordinary ring. To conclude the same for $\Hh_{R/A}$, we've seen above that $\Hh_{R/A}/\Phi_m(q)$ is static for all $m\in\IN$. Then \cref{cor:HabiroCompleteDerivedNakayama} can be applied.
	\end{proof}
	\begin{rem}\label{rem:HabiroIsAinf?}
		By tracing through the proof of \cref{thm:HabiroqWittComparison} and checking on ghost coordinates, we see that the maps $\Hh_{R/A,  m}\rightarrow \Hh_{R/A,  d}$ from \cref{con:RelativeHabiroRing} deform the $q$-Witt vector Frobenii $F_{m/d}\colon \qIW_m(R/A)\rightarrow \qIW_d(R/A)$. Then the construction of $\Hh_{R/A}$ is reminiscent of the construction of $\IA_\inf$ from \cite[Lemma~\chref{3.2}]{BMS1}.
	\end{rem}
	
	In \cite[Definition~\chref{1.1}]{HabiroRingOfNumberField}, the Habiro ring of a number field is defined in terms of power series in $q-\zeta$, for $\zeta$ ranging through roots of unity. We'll now give a similar hands-on description of $\Hh_{R/A}$. This will imply that our construction recovers the one from \cite{HabiroRingOfNumberField}.
	
	\begin{numpar}[$p$-adic reexpansions around roots of unity.]\label{par:RootsOfUnity}
		In the following, we choose a system of roots of unity $(\zeta_m)_{m\in\IN}$ in such a way that
		\begin{equation*}
			\zeta_{mn}=\zeta_m\zeta_n\ \text{if}\ (m,n)=1\quad\text{and}\quad\zeta_{p^\alpha}=\zeta_{p^{\alpha+1}}^p\,.
		\end{equation*}
		One possible choice would be $\zeta_m\coloneqq \prod_p\mathrm{e}^{2\pi\mathrm{i}/p^{v_p(m)}}$. The conditions above are also required in \cite[\S{\chref[subsection]{1.2}}]{HabiroRingOfNumberField} and they ensure $v_p(\zeta_m-\zeta_{mp})>0$ whenever $p$ is prime, so that after $p$-completion, any power series in $(q-\zeta_m)$ can be reexpanded as power series in $(q-\zeta_{pm})$. In other words, there's a canonical zigzag
		\begin{equation*}
			\IZ[\zeta_{m}]\llbracket q-\zeta_{m}\rrbracket \longrightarrow \IZ_p[\zeta_{pm}]\llbracket q-\zeta_m\rrbracket\simeq\IZ_p[\zeta_{pm},q]_{(q-\zeta_m,q-\zeta_{pm})}^\complete \longleftarrow\IZ[\zeta_{pm}]\llbracket q-\zeta_{pm}\rrbracket\,,
		\end{equation*}
		In the situation we're interested in, we get a similar zigzag
		\begin{equation*}
			\left(R\otimes_{A,\psi^m}A\right)[\zeta_m]\llbracket q-\zeta_m\rrbracket \longrightarrow \bigl(\widehat{R}_p\otimes_{A,\psi^m}A\bigr)_p^\complete[\zeta_{pm}]\llbracket q-\zeta_{m}\rrbracket\xleftarrow{\!\phi_{p/A}\!} \left(R\otimes_{A,\psi^{pm}}A\right)[\zeta_{pm}]\llbracket q-\zeta_{pm}\rrbracket
		\end{equation*}
		where the map on the right is induced by the relative Frobenius $\phi_{p/A}$ from \cref{con:RelativeHabiroRing}, followed by a reexpansion of power series as above. We'll call the map on the left the \emph{canonical map} and the map on the right the \emph{Frobenius}.
	\end{numpar}
	\begin{lem}\label{lem:ComparisonWithGSWZ}
		The ring $\Hh_{R/A}$ agrees with following equaliser \embrace{which can be taken both in $\IE_\infty$-$A[q]$-algebras or in ordinary $A[q]$-algebras}:
		\begin{equation*}
			\Hh_{R/A}\simeq \eq\left(\prod_{m}\left(R\otimes_{A,\psi^m}A\right)[\zeta_m]\llbracket q-\zeta_m\rrbracket\overset{\operatorname{can}}{\underset{\phi_{/A}}{\doublemorphism}}\prod_{p,m}\bigl(\widehat{R}_p\otimes_{A,\psi^m}A\bigr)_p^\complete[\zeta_{pm}]\llbracket q-\zeta_m\rrbracket\right)\,.
		\end{equation*}
		Here $\operatorname{can}$ and $\phi_{/A}$ are the canonical maps and Frobenius maps described in \cref{par:RootsOfUnity}.
	\end{lem}
	\begin{proof}
		Let, temporarily, $E$ denote the derived equaliser. By construction, $\Hh_{R/A}$ can be written as a similar equaliser, with $(R\otimes_{A,\psi^m}A)[\zeta_m]\llbracket q-\zeta_m\rrbracket$ replaced by $(R\otimes_{A,\psi^m}A)[q]_{\Phi_m(q)}^\complete$. We clearly get a map of underived equalisers, hence a map $\Hh_{R/A}\rightarrow \pi_0(E)$ as $\Hh_{R/A}$ is static. Since the derived equaliser $E$ is coconnective, this yields a map $\Hh_{R/A}\rightarrow E$ as well. Since both sides are Habiro-complete in the sense of \cref{par:HabiroComplete}, whether this map is an equivalence can be checked after $(-)_\ell^\complete$ for all primes $\ell$ and after $(-\otimes_\IZ\IQ)_{\Phi_d(q)}^\complete$ for all $d\in\IN$.
		
		\emph{Proof after $\ell$-completion.} After $(-)_\ell^\complete$, all factors in $E$ with $p\neq \ell$ die, and the surviving Frobenii $\phi_{\ell/A}$ become equivalences. Similarly, in $\Hh_{R/A}$, all $p$-adic gluings for $p\neq \ell$ vanish, and the $\ell$-adic gluings become equivalences. It follows that after $\ell$-completion, the map has the form
		\begin{equation*}
			\prod_{(m,\ell)=1}\bigl(\widehat{R}_\ell\otimes_{A,\psi^m}A\bigr)_\ell^\complete[q]_{(\ell,\Phi_m(q))}^\complete\longrightarrow \prod_{(m,\ell)=1}\bigl(\widehat{R}_\ell\otimes_{A,\psi^m}A\bigr)_\ell^\complete[\zeta_m]\llbracket q-\zeta_m\rrbracket
		\end{equation*}
		So it will be enough to show that $\IZ_\ell[q]_{(\ell,\Phi_m(q))}^\complete\rightarrow \IZ_\ell[\zeta_m]\llbracket q-\zeta_m\rrbracket$ is an equivalence whenever $(m,\ell)=1$. This can be checked modulo $(\ell,\Phi_m(q))$. The left-hand side clearly becomes $\IF_\ell[q]/\Phi_m(q)\simeq \IF_\ell(\zeta_m)$ since the cyclotomic polynomial $\Phi_m(q)$ is irreducible in $\IF_\ell[q]$ if $(m,\ell)=1$. Moreover, $\Phi_m(q)$ has distinct roots in $\ov\IF_\ell$, and so $\Phi_m(q)/(q-\zeta_m)$ will be a unit in $\IF_\ell(\zeta_m)\llbracket q-\zeta_m\rrbracket$. It follows that $\IZ_\ell[\zeta_m]\llbracket q-\zeta_m\rrbracket/(\ell,\Phi_m(q))\simeq \IF_\ell(\zeta_m)$ as well. This concludes the argument after $\ell$-completion.
		
		\emph{Proof after $\Phi_d(q)$-completed rationalisation.} By \cref{con:RelativeHabiroRing}, the $\Phi_d(q)$-completion of $\Hh_{R/A}$ is $(R\otimes_{A,\psi^d}A)[q]_{\Phi_d(q)}^\complete$ and so
		\begin{equation*}
			\left(\Hh_{R/A}\otimes_\IZ\IQ\right)_{\Phi_d(q)}^\complete\simeq \bigl((R\otimes_{A,\psi^d}A)\otimes_\IZ\IQ\bigr)[q]_{\Phi_d(q)}^\complete\simeq \bigl((R\otimes_{A,\psi^d}A)\otimes_\IZ\IQ(\zeta_m)\bigr)\llbracket q-\zeta_m\rrbracket\,.
		\end{equation*}
		Here we use that $\IQ[q]_{\Phi_d(q)}^\complete\rightarrow \IQ(\zeta_m)\llbracket q-\zeta_m\rrbracket$ is an equivalence. Indeed, this can be checked modulo $\Phi_m(q)$. Since $\Phi_m(q)$ is irreducible and has distinct roots in $\ov\IQ$, the same argument as above shows that both sides become $\IQ(\zeta_m)$ modulo $\Phi_m(q)$, as desired.
		
		Let's compute $\widehat{E}_{\Phi_d(q)}$ next. Since $(R\otimes_{A,\psi^m}A)[\zeta_m]\llbracket q-\zeta_m\rrbracket$ is $\Phi_m(q)$-complete, it'll vanish upon $\Phi_d(q)$-completion unless $m/d$ is a prime power (possibly with negative exponent). Moreover, if $m/d=p^\alpha$ is a power of $p$, then the  $\Phi_d(q)$-completion of $(R\otimes_{A,\psi^m}A)[\zeta_m]\llbracket q-\zeta_m\rrbracket$ will also be $p$-complete, unless $\alpha=0$. It follows that all surviving Frobenii will become equivalences, except if their source is $(R\otimes_{A,\psi^m}A)[\zeta_d]\llbracket q-\zeta_d\rrbracket$.
		
		For all primes $p$, let $\alpha_p\coloneqq v_p(d)$ and write $d=p^{\alpha_p}d_p$. By massaging the limit using our observations so far, we find that $\widehat{E}_{\Phi_d(q)}$ sits inside a pullback diagram
		\begin{equation*}
			\begin{tikzcd}
				\widehat{E}_{\Phi_d(q)}\rar\dar\drar[pullback] & \left(R\otimes_{A,\psi^d}A\right)[\zeta_d]\llbracket q-\zeta_d\rrbracket\dar["\left(\vphantom{\phi^+}\smash{\phi_{p/A}^{\alpha_p}}\right)_p"]\\
				\prod_{p}\bigl(\widehat{R}_p\otimes_{A,\psi^{d_p}}A\bigr)_p^\complete[\zeta_{d_p}]\llbracket q-\zeta_{d_p}\rrbracket \rar & \prod_{p}\bigl(\widehat{R}_p\otimes_{A,\psi^{d_p}}A\bigr)_p^\complete[\zeta_{d}]\llbracket q-\zeta_{d_p}\rrbracket
			\end{tikzcd}
		\end{equation*}
		Observe that the bottom horizontal arrow is a split injection on underlying $\IZ[q]$-modules, because in each factor $\IZ[\zeta_{d_p}]\rightarrow \IZ[\zeta_d]$ is a split injection of abelian groups. However, as we've seen above, $\IQ[q]_{\Phi_d(q)}^\complete\simeq \IQ(\zeta_d)\llbracket q-\zeta_d\rrbracket$ contains $\zeta_d$. Thus the bottom horizontal arrow becomes an equivalence after $(-\otimes_\IZ\IQ)_{\Phi_d(q)}^\complete$. It follows that
		\begin{equation*}
			\left(E\otimes_\IZ\IQ\right)_{\Phi_d(q)}^\complete\simeq  \bigl((R\otimes_{A,\psi^d}A)\otimes_\IZ\IQ[\zeta_m]\bigr)\llbracket q-\zeta_m\rrbracket\,.
		\end{equation*}
		Thus $\Hh_{R/A}\rightarrow E$ also becomes an equivalence after $(-\otimes_\IZ\IQ)_{\Phi_d(q)}^\complete$.
	\end{proof}
	\begin{cor}\label{cor:ComparisonWithGSWZ}
		If $F$ is a number field with discriminant $\Delta$ and $R\coloneqq \Oo_F[1/\Delta]$, then $\Hh_{R/\IZ}$ agrees with the Habiro ring $\Hh_{R}$ defined in \cite[Definition~\textup{\chref{1.1}}]{HabiroRingOfNumberField}.
	\end{cor}
	\begin{proof}
		This follows immediately from \cref{lem:ComparisonWithGSWZ}.
	\end{proof}
	\begin{rem}\label{rem:HTaylorSeries}
		In the special case where $R=\IZ$, we obtain the following presentation of the ordinary Habiro ring:
		\begin{equation*}
			\Hh\simeq \eq\left(\prod_{m}\IZ[\zeta_m]\llbracket q-\zeta_m\rrbracket\overset{\operatorname{can}}{\underset{\phi_{/\IZ}}{\doublemorphism}}\prod_{p,m}\IZ_p[\zeta_{pm}]\llbracket q-\zeta_m\rrbracket\right)\,.
		\end{equation*}
		Here $\phi_{/\IZ}$ is just given by the reexpansion morphisms $\IZ[\zeta_{pm}]\llbracket q-\zeta_{pm}\rrbracket\rightarrow \IZ_p[\zeta_{pm}]\llbracket q-\zeta_m\rrbracket$ for all $p$ and all $m$. This gives precise meaning to the intuition that $\Hh$ is the \enquote{ring of power series that can be Taylor-expanded around each root of unity}. Whenever two such expansions can be compared $p$-adically, they must coincide.
	\end{rem}

	\newpage
	
	\section{Habiro descent for \texorpdfstring{$q$}{q}-Hodge complexes}\label{sec:HabiroDescent}
	
	In this section, we'll show that in those situations where a well-behaved derived $q$-Hodge complex can be defined, it descends automatically to the Habiro ring, and furthermore a derived analogue of the comparison with $q$-de Rham--Witt complexes holds true.
	
	Throughout this section, we fix a perfectly covered $\Lambda$-ring $A$.

	\begin{numpar}[Convention.]\label{conv:QuotientConvention}
		In the following we'll consider filtered modules over the filtered ring $(q^m-1)^\star A[q]$ for various $m$. For such a filtered module $ \fil^\star M$, we always let $ \fil^\star M/(q^m-1)$ denote the base change
		\begin{equation*}
			\fil^\star M/(q^m-1)\coloneqq \fil^\star M\lotimes_{(q^m-1)^\star A[q]}A
		\end{equation*}
		in filtered objects, or in other words, the quotient by $(q^m-1)$ sitting in filtration degree~$1$, not filtration degree~$0$. In particular, the $n$\textsuperscript{th} filtered piece of the quotient $\fil^\star M/(q^m-1)$ will be
		\begin{equation*}
			\cofib\left((q^m-1)\colon \fil^{n-1}M\longrightarrow \fil^nM\right)\,.
		\end{equation*}
	\end{numpar}

	\subsection{\texorpdfstring{$q$}{q}-Hodge filtrations and the \texorpdfstring{$q$}{q}-Hodge complex}\label{subsec:qHodgeFiltrations}
	
	Let us start by introducing an appropriate $\infty$-category of $A$-algebras equipped with a well-behaved $q$-deformation of the Hodge filtration. Since \cref{def:qHodgeFiltration} below is a bit of a mess, let us informally summarise the key points first: In addition to the obvious $q$-deformation condition~\cref{enum:qHodgeModq-1}, we also wish the filtration to be compatible with the rational equivalence
	\begin{equation*}
		\bigl(\qdeRham_{R/A}\lotimes_\IZ\IQ\bigr)_{(q-1)}^\complete\simeq \bigl(\deRham_{R/A}\lotimes_\IZ\IQ\bigr)\qpower\,,
	\end{equation*}
	which leads to condition~\cref{enum:qHodgeRational}. For technical reasons, we also need to require the same for the rationalisations of the $p$-completed ($q$-)de Rham complexes, which is why we have to include condition~\cref{enum:qHodgeRationalpComplete} below. These conditions need to satisfy some obvious compatibilities; recording those, we end up with the following slightly messy definition: 
	
	\begin{defi}[$q$-Hodge filtrations]\label{def:qHodgeFiltration}
		Let $R$ be an animated $A$-algebra. A \emph{$q$-Hodge filtration on $\qdeRham_{R/A}$} is a filtered $(q-1)^\star A[q]$-module
		\begin{equation*}
			\fil_{\qHodge}^\star \qdeRham_{R/A}\simeq \Bigl( \fil_{\qHodge}^0\qdeRham_{R/A}\leftarrow \fil_{\qHodge}^1\qdeRham_{R/A}\leftarrow\fil_{\qHodge}^2\qdeRham_{R/A}\leftarrow \dotsb\Bigr)\,,
		\end{equation*}
		equipped with the following data and compatibilities%
		\footnote{Since we're working with $\infty$-categories, each compatibility is again a datum that needs to be provided.}%
		:
		\begin{alphanumerate}
			\item\label{enum:qHodgeFiltrationOnqdR} An equivalence of $A[q]$-modules $\qdeRham_{R/A}\simeq  \fil_{\qHodge}^0\qdeRham_{R/A}$. In other words, we require that $\fil_{\qHodge}^\star \qdeRham_{R/A}$ defines a descending filtration on the derived $q$-de Rham complex.
			\item\label{enum:qHodgeModq-1} An equivalence of filtered $A$-modules
			\begin{equation*}
				c_{(q-1)}\colon  \fil_{\qHodge}^\star \qdeRham_{R/A}/(q-1)\overset{\simeq}{\longrightarrow} \fil_{\Hodge}^\star \deRham_{R/A}\,,
			\end{equation*}
			which in filtered degrees $\leqslant 0$ agrees with the usual equivalence $\qdeRham_{R/A}/(q-1)\simeq \deRham_{R/A}$ under the identification from~\cref{enum:qHodgeFiltrationOnqdR}. In other words, the filtration $ \fil_{\qHodge}^\star \qdeRham_{R/A}$ has to be a $(q-1)$-deformation of the Hodge filtration.
			\item\label{enum:qHodgeRational} An equivalence of filtered $(q-1)^\star (A\otimes\IQ)[q]$-modules
			\begin{equation*}
				c_\IQ\colon \bigl( \fil_{\qHodge}^\star \qdeRham_{R/A}\lotimes_\IZ\IQ\bigr)_{(q-1)}^\complete\overset{\simeq}{\longrightarrow}  \fil_{(\Hodge,q-1)}^\star \bigl(\deRham_{R/A}\lotimes_\IZ\IQ\bigr)\qpower\,,
			\end{equation*}
			where $\fil_{(\Hodge,q-1)}^\star $ denotes the $(q-1)$-completed tensor product of the Hodge filtration on $\deRham_{R/A}$ and the $(q-1)$-adic filtration on $\IQ\qpower$; in the following, we'll often call this the \emph{combined Hodge and $(q-1)$-adic filtration}. In addition, we require that $c_\IQ$ agrees in filtered degrees $\leqslant 0$ with the usual equivalence $(\qdeRham_{R/A}\lotimes_\IZ\IQ)_{(q-1)}^\complete\simeq (\deRham_{R/A}\lotimes_\IZ\IQ)\qpower$ under the identification from~\cref{enum:qHodgeFiltrationOnqdR}, and that $c_\IQ$ and $c_{(q-1)}$ from \cref{enum:qHodgeModq-1} fit into a commutative diagram
			\begin{equation*}
				\begin{tikzcd}[cramped,column sep=3.5ex]
					\fil_{\qHodge}^\star \qdeRham_{R/A}\rar\dar\ar[drr,commutes] & \fil_{\qHodge}^\star \qdeRham_{R/A}/(q-1)\rar["\simeq","c_{(q-1)}"'] &  \fil_{\Hodge}^\star \deRham_{R/A}\dar \\
					\bigl( \fil_{\qHodge}^\star \qdeRham_{R/A}\lotimes_\IZ\IQ\bigr)_{(q-1)}^\complete\rar["\simeq","c_\IQ"'] &  \fil_{(\Hodge,q-1)}^\star \bigl(\deRham_{R/A}\lotimes_\IZ\IQ\bigr)\qpower \rar &  \fil_{\Hodge}^\star \deRham_{R/A}\lotimes_\IZ\IQ
				\end{tikzcd}
			\end{equation*}
			which again must agree in filtered degrees $\leqslant 0$ with the corresponding unfiltered diagram under the identification from~\cref{enum:qHodgeFiltrationOnqdR}.
			\item[c_p]\label{enum:qHodgeRationalpComplete} For every prime~$p$, an equivalence of filtered $(q-1)^\star \widehat{A}_p[1/p]\qpower$-modules
			\begin{equation*}
				c_{\IQ_p}\colon \fil_{\qHodge}^\star \bigl(\qdeRham_{R/A}\bigr)_p^\complete\bigl[\localise{p}\bigr]_{(q-1)}^\complete\overset{\simeq}{\longrightarrow}  \fil_{(\Hodge,q-1)}^\star \bigl(\deRham_{R/A}\bigr)_p^\complete\bigl[\localise{p}\bigr]\qpower\,,
			\end{equation*}
			which is required to agree in filtered degrees $\leqslant 0$ agrees with the usual equivalence $(\qdeRham_{R/A})_p^\complete[1/p]_{(q-1)}^\complete\simeq (\deRham_{R/A})_p^\complete[1/p]\qpower$ under the identification from~\cref{enum:qHodgeFiltrationOnqdR}. In addition, we require that $c_\IQ$ and $c_{\IQ_p}$ are compatible in form of a commutative diagram
			\begin{equation*}
				\begin{tikzcd}
					\bigl( \fil_{\qHodge}^\star \qdeRham_{R/A}\lotimes_\IZ\IQ\bigr)_{(q-1)}^\complete\rar["\simeq","c_\IQ"']\dar\drar[commutes] & \fil_{(\Hodge,q-1)}^\star \bigl(\deRham_{R/A}\lotimes_\IZ\IQ\bigr)\qpower\dar\\
					\fil_{\qHodge}^\star \bigl(\qdeRham_{R/A}\bigr)_p^\complete\bigl[\localise{p}\bigr]_{(q-1)}^\complete\rar["\simeq","c_{\IQ_p}"'] &  \fil_{(\Hodge,q-1)}^\star \bigl(\deRham_{R/A}\bigr)_p^\complete\bigl[\localise{p}\bigr]\qpower
				\end{tikzcd}
			\end{equation*}
			which in filtered degrees $\leqslant0$ must agree with the usual compatibility under the identification from~\cref{enum:qHodgeFiltrationOnqdR}, and that $c_{(q-1)}$ and $c_{\IQ_p}$ fit into a commutative diagram
			\begin{equation*}
				\begin{tikzcd}[cramped,column sep=3.5ex]
					\fil_{\qHodge}^\star \bigl(\qdeRham_{R/A}\bigr)_p^\complete\rar\dar\ar[drr,commutes] & \fil_{\qHodge}^\star \bigl(\qdeRham_{R/A}\bigr)_p^\complete/(q-1)\rar["\simeq","c_{(q-1)}"'] &  \fil_{\Hodge}^\star \bigl(\deRham_{R/A}\bigr)_p^\complete\dar \\
					\fil_{\qHodge}^\star \bigl(\qdeRham_{R/A}\bigr)_p^\complete\bigl[\localise{p}\bigr]_{(q-1)}^\complete\rar["\simeq","c_{\IQ_p}"'] &  \fil_{(\Hodge,q-1)}^\star \bigl(\deRham_{R/A}\bigr)_p^\complete\bigl[\localise{p}\bigr]\qpower \rar &  \fil_{\Hodge}^\star \bigl(\deRham_{R/A}\bigr)_p^\complete\bigl[\localise{p}\bigr]
				\end{tikzcd}
			\end{equation*}
			which must agree in filtered degrees $\leqslant 0$ with the corresponding unfiltered diagram under the identification from~\cref{enum:qHodgeFiltrationOnqdR}. Finally, we require that this diagram is compatible with the diagram from \cref{enum:qHodgeRational} under the previous diagram relating $c_\IQ$ and $c_{\IQ_p}$, and that in filtered degrees $\leqslant 0$ this compatibility agrees with the usual compatibility under the identification from~\cref{enum:qHodgeFiltrationOnqdR}.
		\end{alphanumerate}
		We let $\cat{AniAlg}_A^{\qHodge}$ denote the $\infty$-category of pairs $(R,\fil_{\qHodge}^\star\qdeRham_{R/A})$, where $R$ is an animated $A$-algebra and $\fil_{\qHodge}^\star\qdeRham_{R/A}$ is a $q$-Hodge filtration on $\qdeRham_{R/A}$. Formally, the $\infty$-category $\cat{AniAlg}_A^{\qHodge}$ can be expressed as an iterated pullback of $\cat{AniAlg}_A$ and several $\infty$-categories of filtered modules; this is straightforward, but not very enlightening, so we omit the details.
	\end{defi}
	
	It is natural to ask whether $q$-Hodge filtrations can be chosen functorially. Surprisingly, this turns out to be false. 
	\begin{lem}\label{lem:NoFunctorialqHodgeFiltration}
		If $A$ is not a $\IQ$-algebra, then the forgetful functor $\cat{AniAlg}_A^{\qHodge}\rightarrow \cat{AniAlg}_A$ is not essentially surjective. In particular, it has no section, not even when restricted to the full subcategory $\cat{Sm}_A\subseteq \cat{AniAlg}_A$ of smooth $A$-algebras.
	\end{lem}
	\begin{proof}[Proof sketch]
		As far as the author is aware, this result hasn't been published, but the objection is known among the experts in the field.
		
		Let~$p$ be a prime such that $\widehat{A}_p\not\simeq 0$. Let $\widehat{A}_p\{x\}_\infty$ be the free $p$-complete perfect $\delta$-ring on a generator $x$. We'll show that the $q$-de Rham complex of $R\coloneqq \widehat{A}_p\{x\}_\infty/x$ admits no $q$-Hodge filtration. Suppose it does. Note that $(\qdeRham_{R/A})_p^\complete$ is given by the prismatic envelope
		\begin{equation*}
			\bigl(\qdeRham_{R/A}\bigr)_p^\complete\simeq \widehat{A}_p\{x\}_\infty\qpower\left\{\frac{\phi(x)}{[p]_q}\right\}_{(p,q-1)}^\complete\,.
		\end{equation*}
		In particular, it is static. Since the Hodge filtration $\fil_{\Hodge}^\star(\deRham_{R/A})_p^\complete$ is just the divided power filtration of the PD-envelope $(\deRham_{R/A})_p^\complete\simeq D_{\smash{\widehat{A}}_p\{x\}_\infty}(x)$, \cref{def:qHodgeFiltration}\cref{enum:qHodgeModq-1} implies that $\fil_{\qHodge}^\star(\qdeRham_{R/A})_p^\complete$ must also be a descending chain of submodules of $(\qdeRham_{R/A})_p^\complete$. Moreover, we see that $\fil_{\qHodge}^p(\qdeRham_{R/A})_p^\complete$ must contain an element $\widetilde{\gamma}_q(x)$ such that $\widetilde{\gamma}_q(x)\equiv x^p/p\mod (q-1)$. Using \cref{def:qHodgeFiltration}\cref{enum:qHodgeRationalpComplete}, we see that $\widetilde{\gamma}_q$ must also be contained in the ideal $(x,q-1)^p$ after completed rationalisation. But it is straightforward to check that the prismatic envelope above doesn't contain any $\widetilde{\gamma}_q(x)$ with these properties (for the details, see \cref{exm:MainSpecialCase} below).
		
		This shows that $\cat{AniAlg}_A^{\qHodge}\rightarrow \cat{AniAlg}_A$ is not essentially surjective. Hence it can't have a section, not even over $\cat{Sm}_A\subseteq \cat{AniAlg}_A$, because we could always animate to extend such a section to all of $\cat{AniAlg}_A$.
	\end{proof}
	
	\begin{rem}
		Despite the general non-existence, it's possible to construct many interesting objects of the $\infty$-category $\cat{AniAlg}_A^{\qHodge}$, and the forgetful functor $\cat{AniAlg}_A^{\qHodge}\rightarrow \cat{AniAlg}_A$ does admit sections when restricted to certain full subcategories of $\cat{AniAlg}_A$. We'll discuss several such examples in~\cref{sec:FunctorialqHodge}.
	\end{rem}

	In the remainder of this subsection, we'll study the following objects:
	
	\begin{numpar}[$q$-Hodge complexes.]\label{par:WellBehavedqHodgeComplex}
		Given a $q$-Hodge filtration $\fil^\star\qdeRham_{R/A}$ for $R$ over $A$, we can construct the \emph{$q$-Hodge complex} as
		\begin{equation*}
			\qHodge_{(R, \fil_{\smash{\qHodge}}^\star )/A}\coloneqq \colimit\Bigl( \fil_{\qHodge}^0\qdeRham_{R/A}\xrightarrow{(q-1)} \fil_{\qHodge}^1\qdeRham_{R/A}\xrightarrow{(q-1)}\dotsb\Bigr)_{(q-1)}^\complete\,.
		\end{equation*}
		If the $q$-Hodge filtration is clear from the context, we usually just write $\qHodge_{R/A}$.
	\end{numpar}
	\begin{rem}\label{rem:CompleteFiltration}
		In the above we've used the derived $q$-de Rham complex since many of our examples later on will be outside of the smooth case. But note that even if $R=S$ is smooth over $A$, the underived $q$-de Rham complex $\qOmega_{S/A}$ usually \emph{doesn't} agree with the derived $q$-de Rham complex $\qdeRham_{S/A}$, because $\Omega_{S/A}^*$ and $\deRham_{S/A}$ usually differ in characteristic~$0$. But this is not a problem. If we're given a filtration $\fil_{\qHodge}^\star \qOmega_{S/A}$ that satisfies the obvious analogues of \cref{def:qHodgeFiltration}\cref{enum:qHodgeFiltrationOnqdR}--\cref{enum:qHodgeRationalpComplete}, then its pullback along the canonical map $\qdeRham_{S/A}\rightarrow \qOmega_{S/A}$ yields a filtration $\fil_{\qHodge}^\star\qdeRham_{R/A}$ as in \cref{def:qHodgeFiltration}. Indeed, this follows from the fact that $\Omega_{S/A}^*\simeq \hatdeRham_{S/A}$ always agrees with the Hodge-completed derived de Rham complex and the fact that any filtration is the pullback of its completion (see \cref{par:Notations}).
		
		Conversely, we'll show in \cref{prop:Letaq-1} that for any  $(S,\fil_{\qHodge}^\star\qdeRham_{S/A})\in\cat{AniAlg}_A^{\qHodge}$ such that $S$ is smooth over~$A$, we have an equivalence
		\begin{equation*}
			\qOmega_{S/A}\simeq\qhatdeRham_{S/A}
		\end{equation*}
		of the underived $q$-de Rham complex and the $q$-Hodge completed derived $q$-de Rham complex. Finally, let us remark that in the definition of the $q$-Hodge complex it doesn't matter whether we use $\fil_{\qHodge}^\star\qdeRham_{R/A}$ or its completion $\fil_{\qHodge}^\star\qhatdeRham_{R/A}$, since every element in $\fil_{\qHodge}^i\qdeRham_{R/A}$ becomes divisible by $(q-1)^i$ in $\qHodge_{R/A}$ and the $q$-Hodge complex is $(q-1)$-complete.
	\end{rem}

	\begin{prop}\label{lem:qHodgeSymmetricMonoidal}
		$\cat{AniAlg}_A^{\qHodge}$ admits a canonical symmetric monoidal structure. The tensor product of two objects $(R_1, \fil_{\qHodge}^\star \qdeRham_{R_1/A})$ and $(R_2, \fil_{\qHodge}^\star \qdeRham_{R_2/A})$ is given by
		\begin{equation*}
			\left(R_1\lotimes_AR_2,\bigl( \fil_{\qHodge}^\star \qdeRham_{R_1/A}\lotimes_{(q-1)^\star A[q]} \fil_{\qHodge}^\star \qdeRham_{R_2/A}\bigr)_{(q-1)}^\complete\right)\,,
		\end{equation*}
		where in the second component we take the derived tensor as filtered modules over the filtered ring $(q-1)^\star A[q]$. Furthermore, the functor
		\begin{equation*}
			\qHodge_{-/A}\colon \cat{AniAlg}_A^{\qHodge}\longrightarrow \widehat{\Dd}_{(q-1)}\bigl(A[q]\bigr)
		\end{equation*}
		can be equipped with a canonical symmetric monoidal structure.
	\end{prop}
	
	To prove \cref{lem:qHodgeSymmetricMonoidal}, let us first construct a filtration on $\qHodge_{-/A}/(q-1)$.
	
	\begin{numpar}[The conjugate filtration.]\label{par:ConjugateFiltration}
		Let $(R, \fil_{\qHodge}^\star \qdeRham_{R/A})$ be an object in $\cat{AniAlg}_A^{\qHodge}$. Let's consider the localisation of the filtered $(q-1)^\star A[q]$-module $ \fil_{\qHodge}^\star \qdeRham_{R/A}$ at $(q-1)$:
		\begin{equation*}
			\fil_{\qHodge}^\star \qdeRham_{R/A}\bigl[\localise{q-1}\bigr]\simeq\colimit\Bigl( \fil_{\qHodge}^{\star }\qdeRham_{R/A}\xrightarrow{(q-1)} \fil_{\qHodge}^{\star +1}\qdeRham_{R/A}\xrightarrow{(q-1)}\dotso\Bigr)\,.
		\end{equation*}
		Upon completing the filtration, this filtered object becomes the $(q-1)$-adic filtration on the $q$-Hodge complex $\qHodge_{R/A}$.
		
		Before taking the colimit, the diagram above can be regarded as a bifiltered object, with one ascending (\enquote{horizontal}) filtration, given by the steps in the colimit, and one descending (\enquote{vertical}) filtration, given by the filtrations on each step $ \fil_{\qHodge}^{\star +n}\qdeRham_{R/A}$. If we pass to the associated graded in the vertical direction, we obtain
		\begin{equation*}
			\qHodge_{R/A}/(q-1)\simeq\colimit\Bigl(\gr_{\qHodge}^{0}\qdeRham_{R/A}\xrightarrow{(q-1)}\gr_{\qHodge}^{1}\qdeRham_{R/A}\xrightarrow{(q-1)}\dotso\Bigr)\,.
		\end{equation*}
		This representation as a colimit defines an exhaustive ascending filtration on $\qHodge_{R/A}/(q-1)$, which we define to be the \emph{conjugate filtration} $ \fil_\star ^\mathrm{conj}(\qHodge_{R/A}/(q-1))$.
	\end{numpar}
	\begin{lem}\label{lem:ConjugateFiltration}
		The associated graded of the conjugate filtration $ \fil_\star^\mathrm{conj}\qHodge_{R/A}/(q-1)$ is given by
		\begin{equation*}
			\gr_*^{\mathrm{conj}}\bigl(\qHodge_{R/A}/(q-1)\bigr)\simeq \Sigma^{-*}\deRham_{R/A}^*\simeq\gr_{\Hodge}^*\deRham_{R/A}\,.
		\end{equation*}
	\end{lem}
	\begin{proof}
		%
		To avoid ambiguous notation, let us identify the filtered ring $(q-1)^\star A[q]$ with the graded ring $A[\beta,t]$, where $\abs{\beta}=1$, $\abs{t}=-1$, and $\beta t=q-1$.%
		\footnote{In \cite{qdeRhamku} we'll recognise $(q-1)^\star \IZ\qpower\cong \IZ[\beta]\llbracket t\rrbracket\cong \pi_{2*}(\ku^{\h S^1})$, where $\beta\in\pi_2(\ku)$ is the Bott element and $t\in \pi_{-2}(\ku^{\h S^1})$ is a suitable complex orientation.}
		The filtered structure on $A[\beta,t]$ comes from the $A[t]$-module structure (see \cref{par:Notations}), so $t$ can be regarded as the filtration parameter and $\beta$ can be regarded as the element \enquote{$(q-1)$ sitting in degree~$1$}. If we regard $ \fil_{\qHodge}^\star\qdeRham_{R/A}$ as a graded $A[\beta,t]$-module, then
		\begin{equation*}
			\fil_{\qHodge}^\star\qdeRham_{R/A}/\beta\simeq  \fil_{\Hodge}^\star\deRham_{R/A}\quad\text{and}\quad  \fil_{\qHodge}^\star\qdeRham_{R/A}/t\simeq \gr_{\qHodge}^*\qdeRham_{R/A}
		\end{equation*}
		as graded $A[t]$- or $A[\beta]$-modules, respectively. The first equivalence follows from \cref{def:qHodgeFiltration}\cref{enum:qHodgeModq-1}, the second follows because modding out $t$ is the same as taking the associated graded (see \cref{par:Notations}). Hence also
		\begin{equation*}
			\fil_{\qHodge}^\star \qdeRham_{R/A}/(\beta,t)\simeq \gr_{\Hodge}^*\deRham_{R/A}
		\end{equation*}
		as filtered $A$-modules. Finally, by construction, we can identify $\qHodge_{R/A}/(q-1)$ with $( \fil_{\qHodge}^\star\qdeRham_{R/A}\lotimes_{A[\beta]}A[\beta^{\pm 1}])_0/(\beta t)$, where $(-)_0$ denotes the restriction of a graded object to its degree-$0$ part. Then the desired assertion follows from \cref{lem:FiltrationAbstract} below. 
	\end{proof}
	\begin{lem}\label{lem:FiltrationAbstract}
		Let $M^*$ be a graded module over the graded ring $A[\beta,t]$, where $\abs{\beta}=1$, $\abs{t}=-1$. Then $(M^*\lotimes_{A[\beta]}A[\beta^{\pm 1}])_0/(\beta t)$ admits a canonical exhaustive ascending filtration whose associated graded is $M^*/(\beta,t)$.
	\end{lem}
	\begin{proof}
		We formally get $(M^*\lotimes_{A[\beta]}A[\beta^{\pm 1}])_0/(\beta t)\simeq (M^*/t\lotimes_{A[\beta]}A[\beta^{\pm 1}])_0$. Let $\beta^{-\star}A[\beta]$ denote the ascendingly filtered graded ring
		\begin{equation*}
			\beta^{-\star}A[\beta]\coloneqq \Bigl(\dotsb\overset{\beta}{\longrightarrow}A[\beta](1)\overset{\beta}{\longrightarrow}A[\beta](0)\overset{\beta}{\longrightarrow}A[\beta](-1)\overset{\beta}{\longrightarrow}\dotsb\Bigr)\,,
		\end{equation*}
		where $A[\beta](i)$ denotes the shift of the graded object $A[\beta]$ by $i$ (to account for the fact that multiplication by $\beta$ shifts degrees). The colimit of this filtration is $\colimit \beta^{-\star}A[\beta]\simeq A[\beta^{\pm 1}]$. Hence $(M^*/t\lotimes_{A[\beta]}\beta^{-\star}A[\beta])_0$ defines an exhaustive ascending filtration on $(M^*/t\lotimes_{A[\beta]}A[\beta^{\pm 1}])_0$ (by inspection, this is also precisely how  the conjugate filtration from \cref{par:ConjugateFiltration} arises). Since the associated graded of $\beta^{-\star}A[\beta]$ is $\bigoplus_{i\in\IZ}A(-i)$, the associated graded of the filtration we've just constructed is indeed
		\begin{equation*}
			\biggl(\bigoplus_{i\in\IZ}M^*/t\lotimes_{A[\beta]}A(-i)\biggr)_0\simeq \biggl(\bigoplus_{i\in\IZ}M^*/(\beta,t)(-i)\biggr)_0\simeq M^*/(\beta,t)\,.\qedhere
		\end{equation*}
	\end{proof}
	
	\begin{proof}[Proof of \cref{lem:qHodgeSymmetricMonoidal}]
		$\cat{AniAlg}_A^{\qHodge}$ can be written as an iterated pullback of symmetric monoidal $\infty$-categories along symmetric monoidal functors, so there's a canonical way to equip it with a symmetric monoidal structure itself. The forgetful functors
		\begin{equation*}
			\cat{AniAlg}_A^{\qHodge}\longrightarrow \cat{AniAlg}_A\quad\text{and}\quad\cat{AniAlg}_A^{\qHodge}\longrightarrow \Mod_{(q-1)^\star A[q]}\bigl( \Fil \Dd(A)\bigr)_{(q-1)}^\complete
		\end{equation*}
		will then be symmetric monoidal, which shows the formula for tensor products.
		
		To construct a symmetric monoidal structure on $\qHodge_{-/A}$, we use~\cref{par:ConjugateFiltration}. Since localising is symmetric monoidal and passing to the $0$\textsuperscript{th} filtration step is lax symmetric monoidal, we get a lax symmetric monoidal structure on $\qHodge_{-/A}$. Strict symmetric monoidality can then be checked modulo $(q-1)$ because the values of $\qHodge_{-/A}$ are $(q-1)$-complete. 
		
		From the proof of \cref{lem:FiltrationAbstract} above, it is clear that $ \fil_\star^\mathrm{conj}(\qHodge_{-/A}/(q-1))$ can be equipped with a lax symmetric monoidal structure compatible with the one on $\qHodge_{-/A}/(q-1)$ (modding out $t$ or $\beta$ as well as $-\lotimes_{A[\beta]}\beta^{-\star}A[\beta]$ are symmetric monoidal and $(-)_0$ is lax symmetric monoidal). Furthermore the equivalence
		\begin{equation*}
			\gr_*^\mathrm{conj}\bigl(\qHodge_{-/A}/(q-1)\bigr)\simeq \gr_{\Hodge}^*\deRham_{-/A}
		\end{equation*}
		is an equivalence of lax symmetric monoidal functors. Strict symmetric monoidality of $ \fil_\star^\mathrm{conj}(\qHodge_{-/A}/(q-1))$ can now be checked on the associated graded, so we win since it's well-known that $\gr_{\Hodge}^*\deRham_{-/A}$ is symmetric monoidal.
	\end{proof}
	
	\subsection{The main result}\label{subsec:MainResult}
	
	We can now state the general Habiro descent result. We let $\qIW_m\Omega_{-/A}^*$ denote the $m$-truncated $q$-de Rham Witt complex from \cite[Definition~\chref{3.12}]{qWitt} and $\qIW_m\deRham_{-/A}\colon \cat{AniAlg}_A\rightarrow \Dd(A[q])$ its non-abelian derived functor.
	\begin{thm}\label{thm:HabiroDescent}
		Let $A$ be a perfectly covered $\Lambda$-ring and $\cat{AniAlg}_A^{\qHodge}$ be the $\infty$-category of animated $A$-algebras equipped with a $q$-Hodge filtration on their $q$-de Rham complex.
		\begin{alphanumerate}
			\item Let $\widehat{\Dd}_\Hh(A[q])\subseteq \Dd(A[q])$ denote the full sub-$\infty$-category of Habiro-complete objects \embrace{in the sense of \cref{par:HabiroComplete}}. Then the $q$-Hodge complex functor admits a symmetric monoidal factorisation\label{enum:HabiroDescent}
			\begin{equation*}
				\begin{tikzcd}[column sep=huge]
					& \widehat{\Dd}_\Hh\bigl(A[q]\bigr)\dar["(-)_{(q-1)}^\complete"]\\
					\cat{AniAlg}_A^{\qHodge}\rar["\qHodge_{-/A}"']\urar[dashed,"\qHhodge_{-/A}"] & \widehat{\Dd}_{(q-1)}\bigl(A[q]\bigr)
				\end{tikzcd}
			\end{equation*}
			\item For all $m\in\IN$, the quotient $\qHhodge_{-/A}/(q^m-1)$ admits an exhaustive ascending filtration $ \fil_\star ^{\qIW_m\Omega}(\qHhodge_{-/A}/(q^m-1))$ with associated graded\label{enum:qdRWComparison}
			\begin{equation*}
				\gr_*^{\qIW_m\Omega}\bigl(\qHhodge_{-/A}/(q^m-1)\bigr)\simeq \Sigma^{-*}\qIW_m\deRham_{-/A}^*\,.
			\end{equation*}
			Furthermore, $ \fil_\star ^{\qIW_m\Omega}(\qHhodge_{-/A}/(q^m-1))$ can be equipped with a canonical lax symmetric monoidal structure compatible with the one on $\qHhodge_{-/A}/(q^m-1)$, and the equivalence above is an equivalence of lax symmetric monoidal functors.
		\end{alphanumerate}
	\end{thm}
	\begin{exm}\label{exm:HabiroDescentCoordinateCase}
		If $S$ is a smooth over $A$ and $\square\colon A[x_1,\dotsc,x_n]\rightarrow S$ is an étale framing, then we can define a filtration on the coordinate-dependent $q$-de Rham complex $\qOmega_{S/A, \square}^*$ via
		\begin{equation*}
			\fil_{\qHodge, \square}^n\qOmega_{S/A, \square}^*\coloneqq (q-1)^{\max\{n-*,0\}}\qOmega_{S/A, \square}^*\,.
		\end{equation*}
		(compare the construction in~\cref{par:qHodgeFiltrationInCoordinates}). As explained in \cref{rem:CompleteFiltration}, we can take the pullback along $\qdeRham_{S/A}\rightarrow \qOmega_{S/A, \square}^*$ to get a filtration $\fil_{\qHodge,\square}^\star\qdeRham_{S/A}$ on the derived $q$-de Rham complex. It's straightforward to equip  it with the additional structure from \cref{def:qHodgeFiltration}\cref{enum:qHodgeFiltrationOnqdR}--\cref{enum:qHodgeRationalpComplete}: Just construct everything on the level of complexes and then take the pullback.
		
		Therefore, the pair $(S, \fil_{\qHodge, \square}^\star \qdeRham_{S/A})$ determines an $\IE_0$-algebra in $\cat{AniAlg}_A^{\qHodge}$. We'll explain in \cite[Remark~\chref{6.14}]{qdeRhamku} that it can be refined to an $\IE_\infty$-algebra. The derived $q$-Hodge complex associated to $(S, \fil_{\qHodge, \square}^\star \qdeRham_{S/A})$ is the coordinate-dependent $q$-Hodge complex $\qHodge_{S/A, \square}^*$. Indeed, as we've seen in \cref{rem:CompleteFiltration}, in the definition of $\qHodge_{-/A}$ it doesn't matter whether we work with the $q$-Hodge filtration on $\qdeRham_{-/A}$ or its completion. Since $ \fil_{\qHodge, \square}^\star \qOmega_{S/A, \square}^*$ is already complete, it's automatically the completion of its pullback $ \fil_{\qHodge, \square}^\star \qdeRham_{S/A}$. We conclude that the corresponding derived $q$-Hodge complex is
		\begin{equation*}
			\colimit\Bigl( \fil_{\qHodge, \square}^0\qOmega_{S/A, \square}^*\xrightarrow{(q-1)} \fil_{\qHodge, \square}^1\qOmega_{S/A, \square}^*\xrightarrow{(q-1)}\dotsb\Bigr)\cong \qHodge_{S/A, \square}^*\,,
		\end{equation*}
		as claimed.
		
		In this case, \cref{thm:HabiroDescent}\cref{enum:HabiroDescent} shows that $\qHodge_{S/A, \square}^*$ descends to an $\IE_\infty$-algebra $\qHhodge_{S/A, \square}$ in $\widehat{\Dd}_\Hh(A[q])$. As we'll see in \cref{cor:qDRWSmoothAnimation} below, $\Sigma^{-n}\qIW_m\deRham_{S/A}^n\simeq \qIW_m\Omega_{S/A}^n$ holds for all $n$. Thus, \cref{thm:HabiroDescent}\cref{enum:qdRWComparison} shows
		\begin{equation*}
			\H^*\bigl(\qHhodge_{S/A, \square}/(q^m-1)\bigr)\cong \qIW_m\Omega_{S/A}^*
		\end{equation*}
		as graded $A[q]/(q^m-1)$-modules. With a little more effort (see \cref{cor:CDGA} below), we can even get an equivalence as differential-graded $A[q]/(q^m-1)$-algebras, so we obtain an improved version of \cite[Theorem~\chref{4.27}]{qWitt}.%
		\footnote{But this theorem is being used in the proof, so we don't get a new proof.}
		
		In fact, $\qHhodge_{S/A, \square}$ can be described as an explicit complex; this was first presented in \cite[Lecture~\href{https://archive.mpim-bonn.mpg.de/id/eprint/5155/10/habiro_cohomology04.mp4}{4}]{HabiroCohomologyLecture}. To this end, equip $A[x_1,\dotsc,x_n]$ with the \emph{toric $\Lambda$-$A$-algebra structure} in which the Adams operations are given by $\psi^m(x_i)=x_i^m$ and consider the relative Habiro ring $\Hh_{S/A[x_1,\dotsc,x_n]}$. For $i=1,\dotsc,n$ let  $\gamma_i$ be the $A[q]$-algebra endomorphism of $A[x_1,\dotsc,x_n,q]$ given by $\gamma_i(x_i)=qx_i$ and $\gamma_i(x_j)=x_j$ for $j\neq i$. We wish to extend $\gamma_i$ to an automorphism of $\Hh_{S/A[x_1,\dotsc,x_n]}$. To do so, we'll extend $\gamma_i$ to each of the factors of the equaliser in \cref{lem:ComparisonWithGSWZ}. Fix $m\in\IN$ and put $S^{(m)}\coloneqq (S\otimes_{A[x_1,\dotsc,x_n],\psi^m}A[x_1,\dotsc,x_n])[\zeta_m]$. Consider the diagram
		\begin{equation*}
			\begin{tikzcd}
				A[x_1,\dotsc,x_n,\zeta_m]\llbracket q-\zeta_m\rrbracket\dar["\square"']\rar["\gamma_i"] & S^{(m)}\llbracket q-\zeta_m\rrbracket\dar\\
				S^{(m)}\llbracket q-\zeta_m\rrbracket\rar["\ov\gamma_i^{(m)}"]\urar["\gamma_i^{(m)}"{pos=0.33},dashed] & S^{(m)}
			\end{tikzcd}
		\end{equation*}
		where $\ov\gamma_i^{(m)}$ is given by the identity on the tensor factor $S$, $\ov\gamma_i^{(m)}(x_i)=\zeta_m x_i$, and $\ov\gamma_i^{(m)}(x_j)=x_j$ for $j\neq i$. By the infinitesimal lifting property of formally étale morphisms, there exists a unique dashed arrow $\gamma_i^{(m)}$ making the diagram commutative. Then $\bigl(\gamma_i^{(m)}\bigr)_{m\in\IN}$ defines the desired automorphism $\gamma_i$ of $\Hh_{S/A[x_1,\dotsc,x_n]}$ via \cref{lem:ComparisonWithGSWZ}. It's also straightforward to check that $\gamma_i\equiv \id\mod x_i$.
		
		Letting $\q\widetilde{\partial}_i\coloneqq (\gamma_i-\id)/x_i$ and $\q\widetilde{\nabla}\coloneqq \sum_i\q\widetilde{\partial}_i\d x_i$, the Koszul complex of the commuting endomorphisms $\q\widetilde{\partial}_i$,
		\begin{equation*}
			\biggl(\Hh_{S/A[x_1,\dotsc,x_n]}\overset{\q\widetilde{\nabla}}{\longrightarrow} \bigoplus_{i}\Hh_{S/A[x_1,\dotsc,x_n]}\d x_i\overset{\q\widetilde{\nabla}}{\longrightarrow}\dotsb\overset{\q\widetilde{\nabla}}{\longrightarrow}\Hh_{S/A[x_1,\dotsc,x_n]}\d x_1\dotsm\d x_n\biggr)\,,
		\end{equation*}
		is an explicit complex representing $\qHhodge_{S/A, \square}$. This can be shown by unravelling the proof of \cref{thm:HabiroDescent} (which is less horrible than it sounds).
	\end{exm}
		%
	
	\cref{exm:HabiroDescentCoordinateCase} covers in particular the case of étale $A$-algebras. In this special case, we recover a familiar construction.
	\begin{cor}\label{cor:ComparisonWithGSWZ2}
		If $R$ is étale over $A$, then $\qHhodge_{R/A}$ is the relative Habiro ring $\Hh_{R/A}$ constructed in \cref{con:RelativeHabiroRing}.
	\end{cor}
	\begin{proof}
		This is clear from the explicit presentation in \cref{exm:HabiroDescentCoordinateCase}, but it can also be shown without having to unravel the proof of \cref{thm:HabiroDescent}.
		
		We'll see in \cref{cor:qDRWSmoothAnimation} that $\qIW_m\deRham_{R/A}^n\simeq \Sigma^{-n}\qIW_m\Omega_{R/A}^n$ holds for all $n\geqslant 0$ whenever $R$ is smooth over $A$. If $R$ is étale, then combining this observation with \cref{thm:HabiroDescent}\cref{enum:HabiroDescent} and \cite[Proposition~\chref{3.31}]{qWitt} shows 
		\begin{equation*}
			\qHhodge_{R/A}/(q^m-1)\simeq \qIW_m(R/A)\simeq \Hh_{R/A}/(q^m-1)\,.
		\end{equation*}
		By uniqueness of deformations of étale extensions, these automatically lift to a unique equivalence of $\IE_\infty$-$\Hh$-algebras $(\qHhodge_{R/A})_{(q^m-1)}^\complete\simeq \Hh_{R/A,  m}$; furthermore, uniqueness also ensures that these equivalences are compatible for varying $m$. It follows that $\qHhodge_{R/A}\simeq \Hh_{R/A}$, as desired.
	\end{proof}
	The proof of \cref{thm:HabiroDescent} has many ingredients and will occupy \crefrange{subsec:TwistedqDeRham}{subsec:qHodgeHabiroDescent}. Before we get lost in the technicalities, let us already outline the main argument and point out where the missing pieces will be provided.
	\begin{proof}[Proof outline of \cref{thm:HabiroDescent}]
		The argument will proceed similarly to the proof of \cref{lem:qHodgeSymmetricMonoidal} above; in particular, for $m=1$, the arguments below will recover the proof of \cref{lem:qHodgeSymmetricMonoidal}. In \cref{subsec:TwistedqDeRham} we'll introduce \emph{twisted $q$-de Rham complexes} for all $m\in\IN$. These are $(q^m-1)$-complete $\IE_\infty$-$A[q]$-algebras  $\qdeRham_{R/A}^{(m)}$ satisfying
		\begin{equation*}
			\qdeRham_{R/A}^{(m)}/(q^m-1)\simeq \qIW_m\deRham_{R/A}
		\end{equation*}
		(see \cref{prop:TwistedqDeRhamDeformsqDRW}%
		\footnote{Informally, just as the $q$-de Rham complex is a $q$-deformation of $\deRham_{R/A}\simeq \qIW_1\deRham_{R/A}$, the twisted $q$-de Rham complexes are $q^m$-deformations of $\qIW_m\deRham_{R/A}$.}%
		). By animating the stupid filtration $\qIW_m\Omega_{-/A}^{\geqslant n,*}$, we obtain a filtration $ \fil_{\Hhodge_m}^\star \qIW_m\deRham_{-/A}$ on $\qIW_m\deRham_{-/A}$. For $m=1$, this is the Hodge filtration on $\deRham_{-/A}$; for higher $m$, it should be thought of as a $q$-Witt vector analogue of the Hodge filtration. By construction, 
		\begin{equation*}
			\gr_{\Hhodge_m}^n\qIW_m\deRham_{-/A}\simeq \qIW_m\deRham_{-/A}^n\,.
		\end{equation*}
		In \cref{subsec:TwistedqHodgeFiltration}, specifically in \cref{prop:TwistedqHodgeFiltrationqDeforms}, we'll show that for any given $q$-Hodge filtration $\fil_{\qHodge}^\star\qdeRham_{R/A}$, we can construct a filtration $ \fil_{\qHhodge_m}^\star \qdeRham_{R/A}^{(m)}$ satisfying
		\begin{equation*}
			\fil_{\qHhodge_m}^\star \qdeRham_{R/A}^{(m)}/(q^m-1)\simeq  \fil_{\Hhodge_m}^\star \qIW_m\deRham_{R/A}\,.
		\end{equation*}
		We'll also verify that $ \fil_{\qHhodge_m}^\star \qdeRham_{R/A}^{(m)}$ is lax symmetric monoidal in $(R, \fil_{\qHodge}^\star \qdeRham_{R/A})$ and that the equivalence above can be upgraded to an equivalence of lax symmetric monoidal functors $\cat{AniAlg}_A^{\qHodge}\rightarrow \Mod_{(q^m-1)^\star A[q]}( \Fil\Dd(A))$.

		With this construction, we'll build the desired Habiro descent of $\qHodge_{R/A}$ in \cref{subsec:qHodgeHabiroDescent} by mimicking the definition of the $q$-Hodge complex in \cref{par:WellBehavedqHodgeComplex}. For all $m\in\IN$, we define
		\begin{equation*}
			\qHhodge_{R/A,  m}\coloneqq \colimit\Bigl( \fil_{\qHhodge_m}^{0}\qdeRham_{R/A}^{(m)}\xrightarrow{(q^m-1)} \fil_{\qHhodge_m}^{1}\qdeRham_{R/A}^{(m)}\xrightarrow{(q^m-1)}\dotso\Bigr)_{(q^m-1)}^\complete\,.
		\end{equation*}
		In \cref{prop:HabiroDescent}, we'll show $(\qHhodge_{R/A,  m})_{(q^d-1)}^\complete\simeq \qHhodge_{R/A,  d}$ whenever $d\mid m$. It follows that $\qHhodge_{R/A}\coloneqq \limit_{m\in\IN}\qHhodge_{R/A,  m}$ determines a Habiro descent of $\qHodge_{R/A}$, thus proving \cref{thm:HabiroDescent}\cref{enum:HabiroDescent}, except for the symmetric monoidality statement. As in the proof of \cref{lem:qHodgeSymmetricMonoidal}, it's formal to construct a lax symmetric monoidal structure on $\qHhodge_{-/A}$ which reduces to the one on $\qHodge_{-/A}$ after $(q-1)$-completion; see \cref{con:HabiroDescent} for the details. Strict symmetric monoidality will then be checked in \cref{lem:HabiroDescentSymmetricMonoidal}, finishing the proof of \cref{thm:HabiroDescent}\cref{enum:HabiroDescent}.
		
		To show \cref{thm:HabiroDescent}\cref{enum:qdRWComparison}, we will mimic the arguments for the conjugate filtration (and in fact, for $m=1$, the desired filtration on $\qHhodge_{R/A}/(q-1)\simeq \qHodge_{R/A}/(q-1)$ \emph{is} the conjugate filtration). By the same argument as in \cref{par:ConjugateFiltration}, we obtain
		\begin{equation*}
			\qHhodge_{R/A,  m}/(q^m-1)\simeq \colimit\Bigl(\gr_{\qHhodge_m}^{0}\qdeRham_{R/A}^{(m)}\xrightarrow{(q^m-1)}\gr_{\qHhodge_m}^{1}\qdeRham_{R/A}^{(m)}\xrightarrow{(q^m-1)}\dotso\Bigr)\,.
		\end{equation*}
		The colimit defines an exhaustive ascending filtration on $\qHhodge_{R/A,  m}/(q^m-1)$, which we take to be our definition of $ \fil_\star ^{\qIW_m\Omega}(\qHhodge_{R/A,  m}/(q^m-1))$. The associated graded of this filtration can be determined by via \cref{lem:FiltrationAbstract} (for this we identify the filtered ring $(q^m-1)^\star A[q]$ with the graded ring $A[q,\beta,t]/(\beta t-(q^m-1))$, where $\abs{q}=0$, $\abs{\beta}=1$, and $\abs{t}=-1$): We obtain
		\begin{equation*}
			\gr_*^{\qIW_m\Omega}\bigl(\qHhodge_{R/A,  m}/(q^m-1)\bigr)\simeq \Sigma^{-*}\qIW_m\deRham_{R/A}^*\simeq \gr_{\Hhodge_m}^*\qIW_m\deRham_{R/A}\,.
		\end{equation*}
		As in the proof of \cref{lem:qHodgeSymmetricMonoidal}, the lax symmetric monoidality statements are formal, and so the proof of \cref{thm:HabiroDescent}\cref{enum:qdRWComparison} is finished.
	\end{proof}

	\subsection{Deformations of \texorpdfstring{$q$}{q}-de Rham--Witt complexes}\label{subsec:TwistedqDeRham}
	We fix a perfectly covered $\Lambda$-ring $A$ as before. We let $\psi^m$ denote its Adams operations, which we extend to a map $\psi^m\colon A[q]\rightarrow A[q]$ via $\psi^m(q)\coloneqq q^m$. We'll also frequently use the Berthelot--Ogus décalage functor $\L\eta_{[m]_q}$ (see \cite[\S{\chref[section]{6}}]{BMS1} or \cite[\stackstag{0F7N}]{Stacks}).
	
	In this subsection, we'll study \emph{twisted $q$-de Rham complexes}: For $S$ smooth over $A$, these are certain $(q^m-1)$-complete $\IE_\infty$-$A[q]$-algebras $\qOmega_{S/A}^{(m)}$, refining the $(q-1)$-complete $\IE_\infty$-$A[q]$-algebras $\L\eta_{[m]_q}\qOmega_{S/A}$ for all $m\in\IN$. The rationale behind our notation and the name \emph{twisted $q$-de Rham complexes} is as follows: If the global $q$-de Rham complex would admit Adams operations $\psi^m$ inducing equivalences
	\begin{equation*}
		\psi^m\colon \left(\qOmega_{S/A}\lotimes_{A[q],\psi^m}A[q]\right)_{(q-1)}^\complete\overset{\simeq}{\longrightarrow}\L\eta_{[m]_q}\qOmega_{S/A}\,,
	\end{equation*}
	then the corresponding twisted $q$-de Rham complex could simply be constructed as the $(q^m-1)$-completion of the \enquote{Adams-twist} $\qOmega_{S/A}\lotimes_{A[q],\psi^m}A[q]$.%
	\footnote{If $\psi^m$ is finite (for example, this holds if $A=\IZ$ or more generally if $A$ is a polynomial ring), then $\qOmega_{S/A}\lotimes_{A[q],\psi^m}A[q]$ is already $(q^m-1)$-complete.}
	However, such global Adams operations don't exist in general (this already fails if $S$ is étale over $A$, as $\Lambda$-structures usually don't extend along étale maps). The best we have is, for every prime $p$, a Frobenius $\phi_p$ on the $p$-completion $(\qOmega_{S/A})_p^\complete$. Still, these $p$-adic Frobenii are enough to construct $\qOmega_{S/A}^{(m)}$.
	\begin{numpar}[Twisted $q$-de Rham complexes]\label{con:TwistedqDeRham}
		Let $S$ be a smooth $A$-algebra. We'll construct a $(q^m-1)$-complete $\IE_\infty$-$A[q]$-algebra using \cref{cor:CompleteDescent}. In the notation of that corollary, take
		\begin{equation*}
			E_d\coloneqq \left(\L\eta_{[m/d]_q}\qOmega_{S/A}\lotimes_{A[q],\psi^d}A[q]\right)_{\Phi_d(q)}^\complete\,.
		\end{equation*}
		We must also provide $p$-adic gluing equivalences. For $p$ a prime such that $pd\mid m$, the required gluing equivalence $(E_{pd})_p^\complete\simeq (E_d)_p^\complete$ should be of the form
		\begin{equation*}
			\left(\L\eta_{[m/pd]_q}\qOmega_{S/A}\lotimes_{A[q],\psi^{pd}}A[q]\right)_{(p,\Phi_{pd}(q))}^\complete\overset{\simeq}{\longrightarrow}\left(\L\eta_{[m/d]_q}\qOmega_{S/A}\lotimes_{A[q],\psi^d}A[q]\right)_{(p,\Phi_d(q))}^\complete\,.
		\end{equation*}
		To construct this, we may replace $\L\eta_{[m/pd]_q}$ and $\L\eta_{[m/d]_q}$ by $\L\eta_{[p^\alpha]_q}$ and $\L\eta_{[p^{\alpha+1}]_q}$, where $\alpha\coloneqq v_p(m/pd)$, because the factor $[m/pd]_q/[p^\alpha]_q$ will be invertible on either side. It will thus be enough to construct an equivalence
		\begin{equation*}
			\left(\L\eta_{[p^\alpha]_q}\qOmega_{S/A}\lotimes_{A[q],\psi^p}A[q]\right)_{(p,q-1)}^\complete\overset{\simeq}{\longrightarrow} \Bigl(\L\eta_{[p^{\alpha+1}]_q}\qOmega_{S/A}\Bigr)_{(p,q-1)}^\complete
		\end{equation*}
		Now $(\L\eta_{[p^\alpha]_q}\qOmega_{S/A})_p^\complete\simeq \L\eta_{[p^\alpha]_q}(\qOmega_{S/A})_p^\complete$. Indeed, $\qOmega_{S/A}$ is $(q-1)$-complete, so $p$-completion agrees with $[p^\alpha]_q$-completion, which always commutes with $\L\eta_{[p^\alpha]_q}$ (see \cite[Lemma~\chref{6.20}]{BMS1}). Thus, we may replace $\qOmega_{S/A}$ by its $p$-completion on the left-hand side; the same argument applies to the right-hand side as well.
		
		Finally, if $(B,J)$ denotes the prism $(\widehat{A}_p\qpower,[p]_q)$ and $T\coloneqq \widehat{S}_p[\zeta_p]$, then $(\qOmega_{S/A})_p^\complete\simeq \Prism_{T/B}$, and so the desired gluing equivalence can be constructed using the general fact that the relative Frobenius 
		induces an equivalence (see \cite[Theorem~\chref{15.3}]{Prismatic})
		\begin{equation*}
			\phi_{/B}\colon\Prism_{T/B}\clotimes_{B,\phi_B}B\overset{\simeq}{\longrightarrow} \L\eta_J\Prism_{T/B}\,.
		\end{equation*}
		According to \cref{cor:CompleteDescent}, we can glue the $E_d$ for all $d\mid m$ to a $(q^m-1)$-complete $\IE_\infty$-$A[q]$-algebra $\qOmega_{S/A}^{(m)}$. This is the \emph{$m$\textsuperscript{th} twisted $q$-de Rham complex of $S$ over $A$}. Via animation, we can then define a functor
		\begin{equation*}
			\qdeRham_{-/A}^{(m)}\colon \cat{AniAlg}_A\longrightarrow \CAlg\left(\widehat{\Dd}_{(q^m-1)}\bigl(A[q]\bigr)\right)\,,
		\end{equation*}
		which agrees with $\qOmega_{-/A}^{(m)}$ on polynomial-$A$-algebras (but not on all smooth $A$-algebras, due to the usual issues in characteristic~$0$). 
	\end{numpar}
	
	The arithmetic fracture square for $\qOmega_{S/A}^{(m)}$ (in the sense of 
	\cref{par:Notations}) can be read off from the construction.
	
	\begin{lem}\label{lem:TwistedqDeRhamFractureSquare}
		Fix $m\in\IN$ and $N\neq 0$ divisible by $m$. For any prime $p\mid N$ and any divisor $d\mid m$ write $m=p^{v_p(m)}m_p$ and $d=p^{v_p(d)}d_p$, where $m_p$ and $d_p$ are coprime to $p$. Let also
		\begin{equation*}
			\phi_{p/A[q]}\colon \qOmega_{S/A}\lotimes_{A[q],\psi^p}A[q]\longrightarrow \left(\qOmega_{S/A}\right)_p^\complete
		\end{equation*}
		denote the relative Frobenius coming from the identification with prismatic cohomology. Then we have a functorial pullback square
		\begin{equation*}
			\begin{tikzcd}
				\qOmega_{S/A}^{(m)}\rar\dar\drar[pullback] & \prod_{p\mid N}\prod_{d_p\mid m_p}\left(\qOmega_{S/A}\lotimes_{A[q],\psi^{p^{\smash{v_p(m)}}d_p}}A[q]\right)_{(p,\Phi_{d_p}(q))}^\complete\dar["\left(\phi_{p/A[q]}^{v_p(m/d)}\right)_{p\mid N,\,d\mid m}"]\\
				\prod_{d\mid m}\left(\qOmega_{S/A}\lotimes_{A[q],\psi^d}A\bigl[\localise{N},q\bigr]\right)_{\Phi_d(q)}^\complete\rar & \prod_{p\mid N}\prod_{d \mid m}\Bigl(\qOmega_{S/A}\lotimes_{A[q],\psi^{d}}A[q]\Bigr)_p^\complete\bigl[\localise{p}\bigr]_{\Phi_d(q)}^\complete
			\end{tikzcd}			
		\end{equation*}
	\end{lem}
	\begin{proof}
		Using \cref{rem:GluingCompletions}, the desired pullback square can be identified with the $(q^m-1)$-completed arithmetic fracture square 
		\begin{equation*}
			\begin{tikzcd}
				\qOmega_{S/A}^{(m)}\rar\dar\drar[pullback] & \prod_{p\mid N}\bigl(\qOmega_{S/A}^{(m)}\bigr)_p^\complete\dar\\
				\qOmega_{S/A}^{(m)}\bigl[\localise{N}\bigr]_{(q^m-1)}^\complete\rar & \prod_{p\mid N}\bigl(\qOmega_{S/A}^{(m)}\bigr)_p^\complete\bigl[\localise{p}\bigr]_{(q^m-1)}^\complete
			\end{tikzcd}
		\end{equation*}
		Here we also use that $[m/d]_q$ is mapped to a unit under $\psi^d\colon A[q]\rightarrow A[1/N,q]_{\Phi_d(q)}^\complete$, so we may ignore $\L\eta_{[m/d]_q}$ for any $d\mid m$ in the bottom left corner. Similarly, we may ignore any $\L\eta_{[m/p^{v_p(m)} d]_q}$ in the top or bottom right corner.
	\end{proof}
	\begin{numpar}[Transition maps.]\label{con:TwistedqDeRhamFunctoriality}
		Whenever $n\mid m$, there's a map of $\IE_\infty$-$A[q]$-algebras
		\begin{equation*}
			\bigl(\qOmega_{S/A}^{(m)}\bigr)_{(q^n-1)}^\complete\longrightarrow \qOmega_{S/A}^{(n)}\,,
		\end{equation*}
		functorial in $S$. To construct this map, we once again appeal to the gluing procedure of \cref{cor:CompleteDescent}. On $\Phi_{d}$-completions, where $d\mid n$, the desired map is induced by the symmetric monoidal natural transformation $\L\eta_{[m/d]_q}\Rightarrow \L\eta_{[n/d]_q}$. It's straightforward to check that this is compatible with the $p$-adic gluings from \cref{con:TwistedqDeRham}. Alternatively, we can use the pullback square from \cref{lem:TwistedqDeRhamFractureSquare}: On the bottom part of the diagram, $\bigl(\qOmega_{S/A}^{(m)}\bigr)_{(q^n-1)}^\complete\rightarrow \qOmega_{S/A}^{(n)}$ is induced by projection to those factors where $d\mid n$. In the top right corner, we also need to apply the relative Frobenius $\phi_{p/A[q]}^{v_p(m/n)}$ in any factor where $d_p\mid n_p$.
		
		These maps can be assembled into a functor $\qOmega_{S/A}^{(-)}\colon \IN\rightarrow \CAlg(\widehat{\Dd}_\Hh(A[q]))$, where $\IN$ denotes the category of natural numbers partially ordered by divisibility. Furthermore, this functor is itself functorial in $S$. We'll refrain from spelling out the argument, as it would just add one more layer of technicalities. To construct the Habiro descent eventually, we only need the individual maps, not the whole functor with all its higher coherences, since any $\limit_{m\in\IN}$ can be replaced by the limit over the sequential subdiagram given by $\{n!\}_{n\geqslant 1}$.
	\end{numpar}
	\begin{rem}
		The maps $\bigl(\qOmega_{S/A}^{(m)}\bigr)_{(q^n-1)}^\complete\rightarrow \qOmega_{S/A}^{(n)}$ are usually quite far from being equivalences, as can be seen from the discrepancy between $\L\eta_{[m/d]_q}$ and $\L\eta_{[n/d]_q}$. Thus, we can form the limit
		\begin{equation*}
			\limit_{m\in\IN}\qOmega_{S/A}^{(m)}\,,
		\end{equation*}
		but it will usually be a pathological object (unless $S$ is étale over $A$, in which case we recover \cref{con:RelativeHabiroRing}). In particular, it won't be a Habiro descent of $\qOmega_{S/A}$.
	\end{rem}
	\begin{rem}\label{rem:qOmegaHabiroDescent}
		To get $\bigl(\qOmega_{S/A}^{(m)}\bigr)_{(q^n-1)}^\complete\rightarrow \qOmega_{S/A}^{(n)}$ closer to being an equivalence, a natural idea goes as follows: The Berthelot--Ogus décalage functors $\L\eta_{[m/d]_q}$ and $\L\eta_{[n/d]_q}$ come equipped with canonical filtrations (see \cite[Proposition~\chref{5.8}]{BMS2}). If these filtrations would glue to give filtrations on $\qOmega_{S/A}^{(m)}$ and $\qOmega_{S/A}^{(n)}$, we could modify $\qOmega_{S/A}^{(m)}$ and $\qOmega_{S/A}^{(n)}$ by \enquote{making elements in each filtration degree~$i$ divisible by $[m]_q^i$ and $[n]_q^i$, respectively}. It is then reasonable to hope that the map between the modifications is an equivalence after $(q^n-1)$-completion, so that in the limit we get a Habiro descent of $\qOmega_{S/A}$.
		
		However, the filtrations on $\L\eta_{[m/d]_q}$ \emph{do not} glue. To make the idea work, we need the additional datum of a $q$-Hodge filtration on $\qOmega_{S/A}$; and, we won't get a Habiro descent of $\qOmega_{S/A}$, but of $\qHodge_{S/A}$. This is precisely how we'll prove \cref{thm:HabiroDescent}. See the outline at the end of \cref{subsec:MainResult}. Also see \cref{subsec:qdeRhamHabiroDescent} for a discussion of Habiro descent for $\qOmega_{S/A}$.
	\end{rem}
	Let us now explain the relationship between $\qOmega_{S/A}^{(m)}$ and the $q$-de Rham--Witt complexes. To this end, recall from \cite[Proposition~\chref{3.17}]{qWitt} that we have a map of graded $A[q]/(q^m-1)$-algebras $F_{m/d}\colon \qIW_m\Omega_{S/A}^*\rightarrow \qIW_d\Omega_{S/A}^*$ for all divisors $d\mid m$ (the \emph{Frobenius} on $q$-de Rham--Witt complexes). This satisfies $\d \circ F_{m/d}=(m/d)\circ F_{m/d}$. Therefore, if $\overtilde{F}_{m/d}$ is given by $(m/d)^nF_{m/d}$ in degree~$n$, then
	\begin{equation*}
		\overtilde{F}_{m/d}\colon \qIW_m\Omega_{S/A}^*\longrightarrow \qIW_d\Omega_{S/A}^*
	\end{equation*}
	is a map of differential-graded $A[q]/(q^m-1)$-algebras.
	
	\begin{prop}\label{prop:TwistedqDeRhamDeformsqDRW}
		Let $A$ be a perfectly covered $\Lambda$-ring and let $S$ be a smooth $A$-algebra. There's a functorial equivalence of $\IE_\infty$-$A[q]/(q^m-1)$-algebras
		\begin{equation*}
			\qOmega_{S/A}^{(m)}/(q^m-1)\overset{\simeq}{\longrightarrow}\qIW_m\Omega_{S/A}\,.
		\end{equation*}
		Under this identification, the map $\qOmega_{S/A}^{(m)}/(q^m-1)\rightarrow \qOmega_{S/A}^{(d)}/(q^d-1)$ induced by \cref{con:TwistedqDeRhamFunctoriality} agrees with the map $\overtilde{F}_{m/d}$ above.
	\end{prop}
	\begin{proof}[Proof sketch]
		By \cite[Corollary~\chref{4.37}]{qWitt}, for any $N\neq 0$ divisible by $m$ the arithmetic fracture square for $\qIW_m\Omega_{S/A}$ has the the following form:
		\begin{equation*}
			\begin{tikzcd}[column sep=0em]
				\qIW_m\Omega_{S/A}\rar\dar["(\gh_{m/d})_{d\mid m}"']\drar[pullback] & \prod_{p\mid N}\prod_{d_p\mid m_p}\left(\Omega_{S/A}\lotimes_{A,\psi^{p^{\smash{v_p(m)}}d_p}}A[q]\right)_p^\complete/\Phi_{d_p}(q^{v_p(m)})\dar["\left(\phi_{p/A}^{v_p(m/d)}\right)_{p\mid N,\,d\mid m}"]\\
				\prod_{d\mid m}\left(\Omega_{S/A}\lotimes_{A,\psi^d}A\bigl[\localise{N},q\bigr]\right)/\Phi_d(q)\rar& \displaystyle\prod_{p\mid N}\prod_{d\mid m}\left(\Omega_{S/A}\lotimes_{A,\psi^d}A[q]\right)_p^\complete\bigl[\localise{p}\bigr]/\Phi_d(q)
			\end{tikzcd}
		\end{equation*}
		This agrees with the reduction modulo  $(q^m-1)$ of the arithmetic fracture square from \cref{lem:TwistedqDeRhamFractureSquare}. Here we note that upon reduction modulo $(q^m-1)$, every occurence of the $q$-de Rham complex $\qOmega_{S/A}$ in \cref{lem:TwistedqDeRhamFractureSquare} can be replaced by $\Omega_{S/A}$. For example, for the $d$\textsuperscript{th} factor in the bottom left corner, reduction modulo $(q^m-1)$ is the same as reduction modulo $\Phi_d(q)$, as $(q^m-1)$ and $\Phi_d(q)$ only differ by a unit in $A[1/N,q]_{\Phi_d(q)}^\complete$. Now $(q-1)$ maps to $0$ under $\psi^d\colon A[q]\rightarrow A[1/N,q]/\Phi_d(q)$, so indeed $\qOmega_{S/A}$ can be replaced by $\Omega_{S/A}$ in that corner. Similar arguments apply to the other corners.
		
		This yields the desired equivalence $\qOmega_{S/A}^{(m)}/(q^m-1)\simeq\qIW_m\Omega_{S/A}$. It's straightforward to check that this equivalence doesn't depend on the choice of $N$ (compare \cref{con:TwistedqHodgeFiltration} below).
		
		The additional assertion about $\qOmega_{S/A}^{(m)}/(q^m-1)\rightarrow \qOmega_{S/A}^{(d)}/(q^d-1)$ follows similarly by a comparison of arithmetic fracture squares (where we may now choose the same $N$). The only non-trivial step is to check that under the equivalence $(\qIW_{p^\alpha}\Omega_{S/A})_p^\complete\simeq (\Omega_{S/A}\lotimes_{A,\psi^{p^{\alpha}}}A[q]/(q^{p^\alpha}-1))_p^\complete$ the maps $\overtilde{F}_p$ and $\phi_{p/A}$ get identified. This is explained in \cite[Corollary~\chref{4.38}]{qWitt}.
	\end{proof}

	\subsection{The Nygaard filtration on \texorpdfstring{$q$}{q}-de Rham--Witt complexes}\label{subsec:Nygaard}
	In this subsection, we'll study an auxiliary filtration on $q$-de Rham--Witt complexes. Throughout \cref{subsec:Nygaard}, we fix a prime~$p$. We'll also write $\phi$ instead of $\psi^p$ for the $p$\textsuperscript{th} Adams operation of the $\Lambda$-ring $A$. We extend $\phi$ to $A[q]$ via $\phi(q)\coloneqq q^p$.
	
	Let's first recall the Nygaard filtration on $q$-de Rham cohomology.
	
	\begin{numpar}[The Nygaard filtration on $q$-de Rham cohomology.]\label{par:NygaardOnqDeRham}
		Let $S$ be a smooth $A$-algebra. By \cref{lem:TwistedqDeRhamFractureSquare}, for all $\alpha\geqslant 0$,
		\begin{equation*}
			\bigl(\qOmega_{S/A}^{(p^\alpha)}\bigr)_p^\complete\simeq\left(\qOmega_{S/A}\lotimes_{A[q],\phi^\alpha}A[q]\right)_{(p,q-1)}^\complete
		\end{equation*}
		agrees with the $\alpha$-fold Frobenius-twist of $(\qOmega_{S/A})_p^\complete$. %
		%
		%
		Since $q$-de Rham cohomology is a special case of prismatic cohomology, the general theory of Nygaard filtrations \cite[\S{\chref[section]{15}}]{Prismatic} provides a filtration $ \fil_\Nn^\star \bigl(\qOmega_{S/A}^{(p)}\bigr)_p^\complete$: It is the preimage of the filtered décalage filtration on $\L\eta_{\Phi_p(q)}(\qOmega_{S/A})_p^\complete$ under the relative Frobenius
		\begin{equation*}
			\phi_{/A[q]}\colon \bigl(\qOmega_{S/A}^{(p)}\bigr)_p^\complete\overset{\simeq}{\longrightarrow}\L\eta_{\Phi_p(q)}\left(\qOmega_{S/A}\right)_p^\complete\,.
		\end{equation*}
		Via pullback along $\phi^{\alpha-1}\colon A[q]\rightarrow A[q]$, we also get Nygaard filtrations $ \fil_\Nn^\star \bigl(\qOmega_{S/A}^{(p^\alpha)}\bigr)_p^\complete$ for all $\alpha\geqslant 2$. By construction, these Nygaard filtrations are canonically filtered $\IE_\infty$-algebras over the filtered ring $\Phi_{p^\alpha}(q)^\star A[q]$, hence over $(q^{p^\alpha}-1)^\star A[q]$ as well. By \cref{prop:TwistedqDeRhamDeformsqDRW}, we also have an equivalence
		\begin{equation*}
			\bigl(\qOmega_{S/A}^{(p^\alpha)}\bigr)_p^\complete/(q^{p^\alpha}-1)\simeq \left(\qIW_{p^\alpha}\Omega_{S/A}\right)_p^\complete\,.
		\end{equation*}
		Our goal in this subsection is to identify the image of the Nygaard filtration under this equivalence with an explicit filtration on the complex $\qIW_{p^\alpha}\Omega_{S/A}^*$.
	\end{numpar}

	\begin{numpar}[The Nygaard filtration on $q$-de Rham--Witt complexes.]\label{def:NygaardFiltration}
		Let $S$ be smooth over~$A$. The \emph{the Nygaard filtration} is the filtration $ \fil_\Nn^\star \qIW_m\Omega_{S/A}^*$ whose $n$\textsuperscript{th} term is the subcomplex $ \fil_\Nn^n\qIW_{p^\alpha}\Omega_{S/A}^*\subseteq \qIW_{p^\alpha}\Omega_{S/A}^*$ given by
		\begin{equation*}
			\left(p^{n-1}V_p\bigl(\qIW_{p^{\alpha-1}}\Omega_{S/A}^0\bigr) \rightarrow\dotsb\rightarrow p^0V_p\bigl(\qIW_{p^{\alpha-1}}\Omega_{S/A}^{n-1}\bigr)\rightarrow \qIW_{p^{\alpha}}\Omega_{S/A}^n\rightarrow\dotsb\right)\,.
		\end{equation*}
	\end{numpar}
	\begin{prop}\label{prop:NygaardComparison}
		For smooth $A$-algebras $S$, there exists a unique functorial equivalence of filtered $\IE_\infty$-$A[q]/(q^{p^\alpha}-1)$-algebras
		\begin{equation*}
			\fil_\Nn^\star \bigl(\qOmega_{S/A}^{(p^\alpha)}\bigr)_p^\complete/(q^{p^\alpha}-1)\overset{\simeq}{\longrightarrow} \fil_\Nn^\star \left(\qIW_{p^\alpha}\Omega_{S/A}\right)_p^\complete
		\end{equation*}
		\embrace{the quotient on the left-hand side is taken in accordance with Convention~\cref{conv:QuotientConvention}} which in degree~$0$ recovers the equivalence $\bigl(\qOmega_{S/A}^{(p^\alpha)}\bigr)_p^\complete/(q^{p^\alpha}-1)\simeq (\qIW_{p^\alpha}\Omega_{S/A})_p^\complete$ from \cref{par:NygaardOnqDeRham}.
	\end{prop}
	
	The proof of \cref{prop:NygaardComparison} requires several preliminary lemmas.
	
	\begin{lem}\label{lem:NygaardFrobeniusTruncationSurjection}
		Let $S$ be smooth over $A$. For all $n\geqslant 0$, the Frobenius $\overtilde{F}_p$, when restricted to $ \fil_\Nn^n\qIW_{p^\alpha}\Omega_{S/A}^*$, is divisible by $p^n$. The divided Frobenius $p^{-n}\overtilde{F}_p$ induces a map
		\begin{equation*}
			p^{-n}\overtilde{F}_p\colon \gr_\Nn^n\qIW_{p^\alpha}\Omega_{S/A}^*\longrightarrow \tau^{\leqslant n}\bigl(\qIW_{p^{\alpha-1}}\Omega_{S/A}^*/p\bigr)
		\end{equation*}
		which is surjective in degree~$n$ and an isomorphism in all other degrees.
	\end{lem}
	\begin{proof}
		It follows directly from the construction that $\overtilde{F}_p$ is divisible by $p^n$ on $ \fil_\Nn^n\qIW_{p^\alpha}\Omega_{S/A}^*$. The Verschiebung $V_p\colon \qIW_{p^{\alpha-1}}\Omega_{S/A}^*\rightarrow \qIW_{p^{\alpha}}\Omega_{S/A}^*$ satisfies $F_p\circ V_p=p$ and $\qIW_{p^{\alpha-1}}\Omega_{S/A}^*$ is degree-wise $p$-torsion free by \cite[Proposition~\chref{4.1}]{qWitt}, hence $V_p$ must be injective. It follows that
		\begin{equation*}
			p^{-n}\overtilde{F}_p\colon \gr_\Nn^n\qIW_{p^\alpha}\Omega_{S/A}^*\longrightarrow \qIW_{p^{\alpha-1}}\Omega_{S/A}^*/p
		\end{equation*}
		is an isomorphism in degrees $\leqslant n-1$. Also $\gr_\Nn^n\qIW_{p^\alpha}\Omega_{S/A}^*$ vanishes in degrees $\geqslant n+1$. In degree~$n$, the map above is given by
		\begin{equation*}
			F_p\colon \qIW_{p^{\alpha}}\Omega_{S/A}^n/V_p\longrightarrow \qIW_{p^{\alpha-1}}\Omega_{S/A}^n/p\,.
		\end{equation*}
		Since $\d\circ F_p=p(F_p\circ \d)$, this map lands in $\ker(\d\colon \qIW_{p^{\alpha-1}}\Omega_{S/A}^n/p\rightarrow \qIW_{p^{\alpha-1}}\Omega_{S/A}^{n+1}/p)$, and so $\overtilde{F}_p/p^n$ indeed factors through $\tau^{\leqslant n}(\qIW_{p^{\alpha-1}}\Omega_{S/A}^*/p)$.
		
		To finish the proof, we must show that $F_p$ maps surjectively onto this kernel. First suppose that $S=P$ is a polynomial $A$-algebra. If $\xi\in \qIW_{p^{\alpha-1}}\Omega_{P/A}^n$ satisfies $\d\xi\equiv 0\mod p$, then \cite[\chref{4.3}({\chref[section*]{15}[$d_\alpha$]})]{qWitt} shows that there exist $\omega$ and $\eta$ satisfying $\xi=F_p(\omega)+p\eta$, proving the desired surjectivity in the polynomial case. If $S$ admits an étale map $\square\colon P\rightarrow S$,  then surjectivity follows via base change along the étale map $\qIW_{p^\alpha}(P/A)\rightarrow \qIW_{p^\alpha}(S/A)$. Here we use \cite[Propositions~\chref{2.48} and~\chref{3.31}]{qWitt} as well as the observation that
		\begin{equation*}
			\d\colon \qIW_{p^{\alpha-1}}\Omega_{S/A}^n/p\longrightarrow \qIW_{p^{\alpha-1}}\Omega_{S/A}^{n+1}/p
		\end{equation*}
		is a map of $\qIW_{p^\alpha}(S/A)$-modules, as $\d\circ F_p\equiv 0\mod p$. For general $S$, we find a Zariski cover $S\rightarrow S'$ such that $S'$ admits an étale map from a polynomial $A$-algebra. Then we can again argue via base change along the étale cover $\qIW_{p^\alpha}(S/A)\rightarrow \qIW_{p^\alpha}(S'/A)$. 
	\end{proof}
	\begin{lem}\label{lem:NygaardFrobeniusKernel}
		Let $S$ be smooth over $A$. There exists canonical isomorphisms
		\begin{align*}
			\Omega_{S/A}^n\otimes_{A,\phi^\alpha}A[\zeta_{p^\alpha}]&\cong\ker\left(F_p\colon \qIW_{p^{\alpha}}\Omega_{S/A}^n\rightarrow \qIW_{p^{\alpha-1}}\Omega_{S/A}^n\right)\\
			&\cong\ker\left(F_p\colon \qIW_{p^{\alpha}}\Omega_{S/A}^n/V_p\rightarrow \qIW_{p^{\alpha-1}}\Omega_{S/A}^n/p\right)
		\end{align*}
	\end{lem}
	\begin{proof}
		We prove the second isomorphism first. For injectivity, suppose $\omega\in\qIW_{p^\alpha}\Omega_{S/A}^n$ satisfies $F_p(\omega)=0$, but is also contained in the image of $V_p$, say, $\omega=V_p(\eta)$. Then $0=F_p(\omega)=p\eta$ implies $\eta=0$ by $p$-torsion freeness, hence $\omega=0$. For surjectivity, suppose $\omega\in\qIW_{p^\alpha}\Omega_{S/A}^n$ satisfies $F_p(\omega)=p\eta$ for some $\eta$. Then $\omega-V_p(\eta)$ is contained in the kernel of $F_p$. This proves the second isomorphism.
		
		To show the first isomorphism, consider the ghost map
		\begin{equation*}
			\gh_1\colon \qIW_{p^\alpha}\Omega_{S/A}^n\longrightarrow \Omega_{S/A}^n\otimes_{A,\phi^\alpha}A[\zeta_{p^\alpha}]\,.
		\end{equation*}
		We claim that $\gh_1$ maps the kernel of $F_p$ isomorphically onto $(\zeta_p-1)(\Omega_{S/A}^n\otimes_{A,\phi^\alpha}A[\zeta_{p^\alpha}])$, which would provide the desired isomorphism, as $(\zeta_p-1)$ is a non-zerodivisor. We only need to show this claim in the case where $S=P$ is a polynomial $A$-algebra; the general case will follow by the same base change arguments as in the proof of \cref{lem:NygaardFrobeniusTruncationSurjection} above.
		
		To show injectivity, recall from \cite[Lemma~\chref{4.5}]{qWitt} that $\gh_1$ is surjective with kernel $\im V_p+\im \d V_p$. Thus, suppose $\omega\in\qIW_{p^\alpha}\Omega_{P/A}^n$ is contained both the kernel of $F_p$ and of $\gh_1$, then we may write $\omega=V_p(\eta_0)+\d V_p(\eta_1)$. Using $F_p\circ\d\circ V_p=\d$, we get  $0=F_p(\omega)=p\eta_0+\d\eta_1$. In particular, $\d\eta_1\equiv 0\mod p$. By \cite[\chref{4.3}({\chref[section*]{15}[$d_\alpha$]})]{qWitt}, $\eta_1$ can be written as $\eta_1=F_p(\xi_0)+p\xi_1$, so that $\d\eta_1=pF_p(\d\xi_0)+p\d\xi_1$. Now $p\eta_0=-\d\eta_1$ and $p$-torsion freeness imply $\eta_0=-F_p(\d\xi_0)-\d\xi_1$. Thus
		\begin{equation*}
			\omega=V_p(\eta_0)+\d V_p(\eta_1)=-V_pF_p(\d\xi_0)-V_p(\d\xi_1)+\d V_pF_p(\xi_0)+p\d V_p(\xi_1)\,.
		\end{equation*}
		Using $V_p\circ F_p=\Phi_{p^\alpha}$ and $V_p\circ \d = p(\d\circ V_p)$, we conclude $\omega=0$. This proves injectivity.
		
		Let us now show that the image is precisely $(\zeta_p-1)(\Omega_{P/A}^n\otimes_{A,\phi^\alpha}A[\zeta_{p^\alpha}])$. By $p$-torsion freeness, it's enough to check this after $p$-completion and after inverting~$p$. Once we invert~$p$, the $q$-de Rham--Witt complexes $\qIW_{p^\alpha}\Omega_{P/A}^*$ and $\qIW_{p^{\alpha-1}}\Omega_{P/A}^*$ split into products of base-changed de Rham complexes by \cite[Corollary~\chref{3.34}]{qWitt} and the assertion is clear.
		
		So let us see what happens after $p$-completion. First observe that we can replace $(\ker F_p)_p^\complete$ by the kernel of $F_p\colon (\qIW_{p^{\alpha}}\Omega_{P/A}^*)_p^\complete\rightarrow(\qIW_{p^{\alpha-1}}\Omega_{P/A}^*)_p^\complete$. Indeed, to show that the image is contained in $(\zeta_p-1)(\Omega_{P/A}^n\otimes_{A,\phi^\alpha}A[\zeta_{p^\alpha}])_p^\complete$, this is certainly sufficient. To see that all of $(\zeta_p-1)(\Omega_{P/A}^n\otimes_{A,\phi^\alpha}A[\zeta_{p^\alpha}])_p^\complete$ is hit, we may use base change \cite[Lemma~\chref{3.16}]{qWitt} and reduce to the case where $A=\IZ$. In this case we're dealing with finitely generated modules over a noetherian ring \cite[Corollary~\chref{2.39} and Proposition~\chref{3.12}({\chref[Item]{15}[$a$]})]{qWitt}, so $p$-completion commutes with kernels.
		
		In any case, we can now use \cite[Theorem~\chref{4.27}]{qWitt} to identify the $q$-de Rham--Witt Frobenius $F_p\colon (\qIW_{p^\alpha}\Omega_{P/A}^*)_p^\complete\rightarrow (\qIW_{p^{\alpha-1}}\Omega_{P/A}^*)_p^\complete$ with
		\begin{equation*}
			\H^*\bigl((\qHodge_{P/A, \square}^*)_p^\complete/(q^{p^\alpha}-1)\bigr)\longrightarrow\H^*\bigl((\qHodge_{P/A, \square}^*)_p^\complete/(q^{p^{\alpha-1}}-1)\bigr)\,,
		\end{equation*}
		where the framing $\square$ can be any choice of coordinates of the polynomial ring $P$. Also note that we can ignore the $(q-1)$-completion in the cited theorem, because everything is $p$-completed but also $(q^{p^\alpha}-1)$-torsion. In \cite[\chref{4.28}--\chref{4.30}]{qWitt} we construct a direct summand
		%
		$(\qHodge_{P/A, \square}^{*,0})_p^\complete\subseteq (\qHodge_{P/A, \square}^*)_p^\complete$ that fits into a commutative diagram
		\begin{equation*}
			\begin{tikzcd}[cramped]
				&[-1em] \H^*\bigl((\qHodge_{P/A, \square}^{*,0})_p^\complete/(q^{p^\alpha}-1)\bigr)\rar\dar["\cong"]\dlar["\gh_1"'] &[-1em]\H^*\bigl((\qHodge_{P/A, \square}^{*,0})_p^\complete/(q^{p^{\alpha-1}}-1)\bigr)\dar["\cong"]\\
				\bigl(\Omega_{P/A}^*\otimes_{A,\phi^\alpha}A[\zeta_{p^\alpha}]\bigr)_p^\complete & \bigl(\Omega_{P/A}^*\otimes_{A,\phi^\alpha}A[q]/(q^{p^\alpha}-1)\bigr)_p^\complete\lar[->>]\rar[->>] & \bigl(\Omega_{P/A}^*\otimes_{A,\phi^\alpha}A[q]/(q^{p^{\alpha-1}}-1)\bigr)_p^\complete
			\end{tikzcd}
		\end{equation*}
		It is also checked there that the complementary direct summand is sent to $0$ under $\gh_1$. It follows that the image of $\ker F_p$ under $\gh_1$ is the image of $(q^{p^{\alpha-1}}-1)(\Omega_{P/A}^*\otimes_{A,\phi^\alpha}A[q]/(q^{p^\alpha}-1))_p^\complete$ in $(\Omega_{P/A}^*\otimes_{A,\phi^\alpha}A[\zeta_{p^\alpha}])_p^\complete$, which is indeed exactly $(\zeta_p-1)(\Omega_{P/A}^n\otimes_{A,\phi^\alpha}A[\zeta_{p^\alpha}])_p^\complete$ in degree~$n$. This finishes the proof.
	\end{proof}
	\begin{cor}\label{cor:NygaardCofibreSequence}
		Let $R$ be an animated $A$-algebra and let $\qIW_{p^\alpha}\deRham_{-/A}$ denote the \embrace{$p$-completed} animations of the $q$-de Rham--Witt complex functors. For all $n\geqslant 0$ and all $\alpha\geqslant 0$, there exists a functorial divided Frobenius
		\begin{equation*}
			p^{-n}\overtilde{F}_p\colon  \gr_\Nn^n\qIW_{p^\alpha}\deRham_{R/A}\longrightarrow  \fil_n^\mathrm{conj}\bigl(\deRham_{R/A}/p\bigr)\lotimes_{A,\phi^{\alpha-1}}A[q]/\bigl(q^{p^{{\alpha-1}}}-1\bigr)\,.
		\end{equation*}
		with fibre given by $\fib(p^{-n}\overtilde{F}_p)\simeq \Sigma^{-n}\deRham_{R/A}^n\lotimes_{A,\phi^\alpha}A[\zeta_{p^\alpha}]$. Here $ \fil_\star ^\mathrm{conj}(\deRham_{R/A}/p)$ denotes the conjugate filtration on the derived de Rham complex, i.e.\ the animation of $\tau^{\leqslant \star}(\Omega_{-/A}/p)$.
	\end{cor}
	\begin{proof}
		For $S$ smooth over $A$, \cref{lem:NygaardFrobeniusTruncationSurjection,lem:NygaardFrobeniusKernel} provide a short exact sequence of complexes
		\begin{equation*}
			0\longrightarrow \Omega_{S/A}^n[-n]\otimes_{A,\phi^\alpha}A[\zeta_{p^\alpha}]\longrightarrow \gr_\Nn^n\qIW_{p^\alpha}\Omega_{S/A}\xrightarrow{p^{-n}\overtilde{F}_p}\tau^{\leqslant n}\left(\qIW_{p^{\alpha-1}}\Omega_{S/A}/p\right)\longrightarrow 0\,.
		\end{equation*}
		Using $\qIW_{p^{\alpha-1}}\Omega_{S/A}/p\simeq \Omega_{S/A}/p\lotimes_{A,\phi^{\alpha-1}}A[q]/(q^{p^{\alpha-1}}-1)$ by \cite[Proposition~\chref{4.2}]{qWitt}, this provides the desired cofibre sequence. The case of general $R$ follows by passing to animations.
	\end{proof}
	\begin{cor}\label{cor:NygaardQuasisyntomicDescent}
		For all animated $A$-algebras $R$ and all $\alpha\geqslant 0$, let us denote the animated Nygaard filtration by $ \fil_\Nn^n\qIW_{p^\alpha}\deRham_{R/A}$. Then:
		\begin{alphanumerate}
			\item $ \fil_\Nn^n(\qIW_{p^\alpha}\deRham_{R/A})_p^\complete$ satisfies quasi-syntomic descent in $R$.\label{enum:NygaardQuasisyntomicDescent}
			\item If $R$ is smooth over $A$, then $ \fil_\Nn^n(\qIW_{p^\alpha}\deRham_{R/A})_p^\complete$ agrees with its un-animated variant $\fil_\Nn^n(\qIW_{p^\alpha}\Omega_{R/A})_p^\complete$.\label{enum:NygaardSmooth}
		\end{alphanumerate}
	\end{cor}
	\begin{proof}
		It's clear that $(\qIW_{p^\alpha}\deRham_{R/A})_p^\complete\simeq (\deRham_{R/A}\lotimes_{A,\phi^\alpha}A[q]/(q^{p^\alpha}-1))_p^\complete$ satisfies quasi-syntomic descent and agrees with its un-animated variant when $R$ is smooth over $A$. To prove~\cref{enum:NygaardQuasisyntomicDescent} and~\cref{enum:NygaardSmooth}, it will thus be enough to show that $\gr_\Nn^n(\qIW_{p^\alpha}\deRham_{R/A})_p^\complete$ satisfies quasi-syntomic descent for all $n\geqslant 0$ and agrees with $\gr_\Nn^n(\qIW_{p^\alpha}\Omega_{R/A})_p^\complete$ when $R$ is smooth. Both assertions follow from \cref{cor:NygaardCofibreSequence}.
	\end{proof}
	Next we construct an analog of the fibre sequence from \cref{cor:NygaardCofibreSequence} for the other Nygaard filtration $ \fil_\Nn^\star \bigl(\qOmega_{S/A}^{(p^\alpha)}\bigr)_p^\complete/(q^{p^\alpha}-1)$. After that we'll prove \cref{prop:NygaardComparison} by carefully comparing these fibre sequences.
	\begin{lem}\label{lem:NygaardFibreSequence}
		Let $R$ be an animated $A$-algebra. For brevity, let us write
		\begin{equation*}
			\fil_{\Nn,  \qOmega}^\star \coloneqq  \fil_\Nn^\star \bigl(\qdeRham_{R/A}^{(p^\alpha)}\bigr)_p^\complete/(q^{p^\alpha}-1)
		\end{equation*}
		and let $\gr_{\Nn, \qOmega}^*$ denote the associated graded of this filtered object. Let also $\phi_{/A}$ denote the relative Frobenius on $(\deRham_{R/A})_p^\complete$. Then for all $n\geqslant 0$ there are canonical maps
		\begin{equation*}
			p^{-n}\phi_{/A}\colon \gr_{\Nn, \qOmega}^n\longrightarrow  \fil_n^\mathrm{conj}\bigl(\deRham_{R/A}/p\bigr)\lotimes_{A,\phi^{\alpha-1}}A[q]/\bigl(q^{p^{\alpha-1}}-1\bigr)
		\end{equation*}
		with fibre $\fib(p^{-n}\phi_{/A})\simeq \Sigma^{-n}(\deRham_{R/A}^n\lotimes_{A,\phi^\alpha}A[\zeta_{p^\alpha}])_p^\complete$.
	\end{lem}
	\begin{proof}
		By definition of the Nygaard filtration, the Frobenius on $q$-de Rham cohomology is divisible by $\Phi_{p^\alpha}(q)^n$ on $ \fil_\Nn^n\bigl(\qdeRham_{R/A}^{(p^\alpha)}\bigr)_p^\complete$. Therefore, for all $n\geqslant 0$ there's a commutative diagram
		\begin{equation*}
			\begin{tikzcd}[column sep=huge]
				\gr_\Nn^{n-1}\bigl(\qdeRham_{R/A}^{(p^\alpha)}\bigr)_p^\complete\rar["(q^{p^\alpha}-1)"]\dar["\Phi_{p^\alpha}(q)^{-(n-1)}\phi_{/A[q]}"',"\simeq"] & \gr_\Nn^n\bigl(\qdeRham_{R/A}^{(p^\alpha)}\bigr)_p^\complete\dar["\Phi_{p^\alpha}(q)^{-n}\phi_{/A[q]}","\simeq"']\\
				\fil_{n-1}^\mathrm{conj}\bigl(\qdeRham_{R/A}^{(p^{\alpha-1})}/\Phi_{p^\alpha}(q)\bigr)\rar["(q^{p^{\alpha-1}}-1)"] &  \fil_n^\mathrm{conj}\bigl(\qdeRham_{R/A}^{(p^{\alpha-1})}/\Phi_{p^\alpha}(q)\bigr)
			\end{tikzcd}
		\end{equation*}
		The vertical arrows are equivalences by \cite[Theorem~\chref{15.2}(2)]{Prismatic} (plus quasi-syntomic descent and passing to animations to allow for arbitrary animated $A$-algebras $R$).
		
		Now $\gr_{\Nn, \qOmega}^n$ is the cofibre of the top horizontal arrow and thus also the cofibre of the bottom horizontal arrow; we wish to compute the latter. To this end, note that
		\begin{equation*}
			\fil_n^\mathrm{conj}\bigl(\qdeRham_{R/A}^{(p^{\alpha-1})}/\Phi_{p^\alpha}(q)\bigr)/\bigl(q^{p^\alpha-1}-1\bigr)\simeq  \fil_n^\mathrm{conj}\bigl(\deRham_{R/A}/p\bigr)\lotimes_{A,\phi^{\alpha-1}}A[q]/\bigl(q^{p^{\alpha-1}}-1\bigr)\,.
		\end{equation*}
		Indeed, without the Frobenius-twists, $ \fil_n^\mathrm{conj}(\qdeRham_{R/A}/\Phi_p(q))\lotimes_{A\qpower}A\simeq  \fil_n^\mathrm{conj}(\deRham_{R/A}/p)$ follows from the base change result in \cite[Theorem~\chref{15.2}(3)]{Prismatic} plus quasi-syntomic descent, using that $-\lotimes_{A\qpower}A$ commutes with all limits. To incorporate the Frobenius twists, just take the base change along $\phi^{\alpha-1}$.
		
		As a consequence, we obtain the desired canonical map
		\begin{equation*}
			p^{-n}\phi_{/A}\colon \gr_{\Nn, \qOmega}^n\longrightarrow  \fil_n^\mathrm{conj}\bigl(\deRham_{R/A}/p\bigr)\lotimes_{A,\phi^{\alpha-1}}A[q]/\bigl(q^{p^{\alpha-1}}-1\bigr)\,.
		\end{equation*}
		By the diagram above and the Hodge--Tate comparison for prismatic cohomology (see \cite[Construction~\chref{7.6}]{Prismatic}) the fibre is indeed $\gr_n^\mathrm{conj}(\qdeRham_{R/A}^{(p^{\alpha-1})}/\Phi_{p^\alpha}(q))\simeq \Sigma^{-n}(\deRham_{R/A}^n\lotimes_{A,\phi^\alpha}A[\zeta_{p^\alpha}])_p^\complete$, as desired. 
	\end{proof}
	\begin{rem}\label{rem:NygaardFibreSequence}
		By contemplating the bottom row of the diagram in the proof above, we find that $\fib(p^{-n}\phi_{/A})\rightarrow \gr_{\Nn, \qOmega}^n$ sits inside the following diagram for all $n\geqslant 0$:
		\begin{equation*}
			\begin{tikzcd}
				&  \fil_\Nn^n\bigl(\qdeRham_{R/A}^{(p^\alpha)}\bigr)_p^\complete\dar\drar["\Phi_{p^\alpha}(q)^{-n}\phi_{/A[q]}"] &[-2em] \\
				\gr_n^\mathrm{conj}\bigl(\qdeRham_{R/A}^{(p^{\alpha-1})}/\Phi_{p^\alpha}(q)\bigr)\rar\drar["(q^{p^{\alpha-1}}-1)"'] & \gr_{\Nn, \qOmega}^n\dar &  \fil_n^\mathrm{conj}\bigl(\qdeRham_{R/A}^{(p^{\alpha-1})}/\Phi_{p^\alpha}(q)\bigr)\dlar\\
				& \gr_n^\mathrm{conj}\bigl(\qdeRham_{R/A}^{(p^{\alpha-1})}/\Phi_{p^\alpha}(q)\bigr) & 
			\end{tikzcd}
		\end{equation*}
	\end{rem}
	
	\begin{proof}[Proof of \cref{prop:NygaardComparison}]
		Thanks to \cref{cor:NygaardQuasisyntomicDescent}, we can tackle the question using quasi-syntomic descent. Let $R$ be a $p$-complete quasi-syntomic $A$-algebra which is \emph{large} in the sense of \cite[Definition~\chref{15.1}]{Prismatic}, i.e.\ there exists a surjection $\widehat{A}_p\langle x_i^{1/p^\infty}\ |\ i\in I\rangle \twoheadrightarrow R$ for some set $I$. Let $ \fil_{\Nn,  \qOmega}^\star $ and $ \fil_{\Nn, \qIW}^\star $ denote the two filtrations on $(\deRham_{R/A}\lotimes_{A,\phi^\alpha}A[q]/(q^{p^\alpha}-1))_p^\complete$ given by
		\begin{equation*}
			\fil_{\Nn,  \qOmega}^\star \coloneqq  \fil_\Nn^\star \bigl(\qdeRham_{R/A}^{(p^\alpha)}\bigr)_p^\complete/(q^{p^\alpha}-1)\quad\text{and}\quad  \fil_{\Nn, \qIW}^\star \coloneqq  \fil_\Nn^\star \left(\qIW_{p^\alpha}\deRham_{R/A}\right)_p^\complete\,.
		\end{equation*}
		Our assumptions on $R$ ensure that $(\deRham_{R/A}\lotimes_{A,\phi^\alpha}A[q]/(q^{p^\alpha}-1))_p^\complete$ is static and that $ \fil_{\Nn,  \qOmega}^\star $ is a descending filtrations by ordinary ideals. So once we've shown $ \fil_{\Nn,  \qOmega}^\star = \fil_{\Nn, \qIW}^\star $ as ideals, the comparison will automatically be functorial in $R$ (of the given form) and an equivalence of filtered $\IE_\infty$-$A[q]$-algebras. Moreover, uniqueness will also be clear. Via quasi-syntomic descent we can then recover the smooth case.
		
		To prove the proposition for $R$, we show using induction on~$n$ that $ \fil_{\Nn,  \qOmega}^n= \fil_{\Nn, \qIW}^n$ as ideals in the ring $(\deRham_{R/A}\lotimes_{A,\phi^\alpha}A[q]/(q^{p^\alpha}-1)))_p^\complete$. The case $n=0$ is clear. So assume we know $ \fil_{\Nn,  \qOmega}^n= \fil_{\Nn, \qIW}^n\eqqcolon  \fil_\Nn^n$ for some $n\geqslant 0$. Let
		\begin{equation*}
			K\coloneqq \fib\left(p^{-n}\phi_{/A}\colon  \fil_\Nn^n\rightarrow  \fil_n^\mathrm{conj}\bigl(\deRham_{R/A}/p\bigr)\lotimes_{A,\phi^{\alpha-1}}A[q]/\bigl(q^{p^{\alpha-1}}-1\bigr)\right)\,.
		\end{equation*}
		Via $ \fil_\Nn^n= \fil_{\Nn,  \qOmega}^n$ we know that $p^{-n}\phi_{/A}$ is surjective and so $K$ is static. According to \cref{cor:NygaardCofibreSequence} we have an equivalence
		\begin{equation*}
			\cofib\bigl( \fil_{\Nn, \qIW}^{n+1}\rightarrow K\bigr)\simeq \Sigma^{-n}\bigl(\deRham_{R/A}^n\lotimes_{A,\phi^\alpha}A[\zeta_{p^\alpha}]\bigr)_p^\complete\,.
		\end{equation*}
		Moreover, this equivalence can be explicitly described as follows: Consider the ghost map $\gh_1$ for $\qIW_{p^\alpha}\deRham_{R/A}$, which by \cite[Proposition~\chref{4.2}]{qWitt} just corresponds to the canonical projection 
		\begin{equation*}
			\bigl(\deRham_{R/A}\lotimes_{A,\phi^\alpha}A[q]/(q^{p^\alpha}-1)\bigr)_p^\complete\longrightarrow \bigl(\deRham_{R/A}\lotimes_{A,\phi^\alpha}A[\zeta_{p^\alpha}]\bigr)_p^\complete
		\end{equation*}
		sending $q\mapsto \zeta_{p^\alpha}$. When restricted to $ \fil_\Nn^n= \fil_{\Nn, \qIW}^n$, this lands in $( \fil_{\Hodge}^n\deRham_{R/A}\lotimes_{A,\phi^\alpha}A[\zeta_{p^\alpha}])_p^\complete$. Indeed, for smooth $A$-algebras this follows directly from \cref{def:NygaardFiltration}, as the image of $V_p$ dies under $\gh_1$; the general case follows via animation. By tracing through the proof of \cref{lem:NygaardFrobeniusKernel}, we now see that the diagram
		\begin{equation*}
			\begin{tikzcd}
				\cofib( \fil_{\Nn, \qIW}^{n+1}\rightarrow K)\dar["\simeq"'] & K\lar\rar &  \fil_\Nn^n\dar\\
				\gr_{\Hodge}^n\bigl(\deRham_{R/A}\lotimes_{A,\phi^\alpha}A[\zeta_{p^\alpha}]\bigr)_p^\complete\ar[drr,"(\zeta_p-1)"'] & &  \fil_{\Hodge}^n\bigl(\deRham_{R/A}\lotimes_{A,\phi^\alpha}A[\zeta_{p^\alpha}]\bigr)_p^\complete\dar\\
				& & \gr_{\Hodge}^n\bigl(\deRham_{R/A}\lotimes_{A,\phi^\alpha}A[\zeta_{p^\alpha}]\bigr)_p^\complete
			\end{tikzcd}
		\end{equation*}
		commutes. Thus $K$ is mapped into the submodule $(\zeta_p-1)(\gr_{\Hodge}^n\deRham_{R/A}\lotimes_{A,\phi^\alpha}A[\zeta_{p^\alpha}])_p^\complete$ and $ \fil_{\Nn, \qIW}^{n+1}$ is the fibre of this map.
		
		According to \cref{lem:NygaardFibreSequence} and the left half of the diagram from \cref{rem:NygaardFibreSequence}, for $ \fil_{\Nn,  \qOmega}^{n+1}$ we have a similar diagram:
		\begin{equation*}
			\begin{tikzcd}
				\cofib( \fil_{\Nn, \qOmega}^{n+1}\rightarrow K)\dar["\simeq"'] & K\lar\rar &  \fil_\Nn^n\dar\\
				\gr_n^\mathrm{conj}\bigl(\qdeRham_{R/A}^{(p^{\alpha-1})}/\Phi_{p^\alpha}(q)\bigr)\ar[drr,"(q^{p^{\alpha-1}}-1)"'] & & \gr_{\Nn, \qOmega}^n\dar\\
				& & \gr_n^\mathrm{conj}\bigl(\qdeRham_{R/A}^{(p^{\alpha-1})}/\Phi_{p^\alpha}(q)\bigr)
			\end{tikzcd}
		\end{equation*}
		Note that $(q^{p^{\alpha-1}}-1)$ is sent to $(\zeta_p-1)$ under $q\mapsto \zeta_{p^\alpha}$. Therefore, to show $ \fil_{\Nn,  \qOmega}^{n+1}= \fil_{\Nn, \qIW}^{n+1}$ and thus to finish the induction, it will be enough to show that the following diagram commutes; here we also use the right half of the diagram from \cref{rem:NygaardFibreSequence}:
		\begin{equation*}
			\begin{tikzcd}[column sep=10em]
				\fil_\Nn^n\bigl(\qdeRham_{R/A}^{(p^\alpha)}\bigr)_p^\complete\rar\dar["\Phi_{p^\alpha}(q)^{-n}\phi_{/A[q]}"'] &  \fil_\Nn^n\dar\\
				\fil_n^\mathrm{conj}\bigl(\qdeRham_{R/A}^{(p^{\alpha-1})}/\Phi_{p^\alpha}(q)\bigr)\dar &  \fil_{\Hodge}^n\bigl(\deRham_{R/A}\lotimes_{A,\phi^\alpha}A[\zeta_{p^\alpha}]\bigr)_p^\complete\dar\\
				\gr_n^\mathrm{conj}\bigl(\qdeRham_{R/A}^{(p^{\alpha-1})}/\Phi_{p^\alpha}(q)\bigr)\rar["\simeq","\text{Hodge--Tate comparison}"'] & \gr_{\Hodge}^n\bigl(\deRham_{R/A}\lotimes_{A,\phi^\alpha}A[\zeta_{p^\alpha}]\bigr)_p^\complete
			\end{tikzcd}
		\end{equation*}
		To show commutativity, let us first get rid of $(\alpha-1)$ Frobenius-twists (thus reducing to $\alpha=1$), as these Frobenius-twists just amount to a pullback. Moreover, commutativity can be checked after the faithfully flat base change along the map $A\rightarrow A_\infty$ into the colimit perfection of the perfectly covered $\Lambda$-ring $A$. Since everything is $p$-complete, working relative to $A_\infty$ is the same as working absolutely, so we can reduce to the case $A=\IZ$. We can then use the method from \cite[\S{\chref[section]{12}}]{Prismatic}. Let us first check commutativity in the single case $R=\IZ_p\langle x^{1/p^\infty}\rangle/x$.
		
		In this case, everything is explicit: First off, $(\qdeRham_{R/\IZ})_p^\complete$ is the ring
		\begin{equation*}
			\IZ_p\qpower\bigl\langle x^{1/p^\infty}\bigr\rangle\left\{\frac{x^p}{\Phi_p(q)}\right\}_{(p,q-1)}^\complete\simeq \biggl(\bigoplus_{i\in\IN[1/p]}\IZ_p\qpower\cdot\frac{x^i}{[\lfloor i\rfloor]_q!}\biggr)_{(p,q-1)}^\complete\,.
		\end{equation*}
		The graded piece $\gr_{\Hodge}^n(\deRham_{R/\IZ}\lotimes_{\IZ}\IZ[\zeta_p])_p^\complete$ is generated by the divided power $x^n/n!$, which is the image of the $q$-divided power $x^n/[n]_q!\in \qdeRham_{R/\IZ}$. We have
		\begin{equation*}
			\Phi_p(q)^{-n}\phi\left(\frac{x^n}{[n]_q!}\right)=\frac{x^{pn}}{[n]_{q^p}!\cdot\Phi_p(q)^n}\equiv\frac{\bigl(x^p/\Phi_p(q)\bigr)^n}{n!} \mod \Phi_p(q)
		\end{equation*}
		By \cite[Lemma~\chref{12.6}]{Prismatic}, $[n]_{q^p}!\cdot\Phi_p(q)^n$ is a unit multiple of $[pn]_q!$. This shows that $\phi(x^n/[n]_q!)$ is divisible by $\Phi_p(q)^n$ and so the image of $x^n/[n]_q!$ under $(\qdeRham_{R/\IZ})_p^\complete\rightarrow (\qdeRham_{R/\IZ}^{(p)})_p^\complete$ lies in Nygaard filtration degree~$n$. The proof of \cite[Lemma~\chref{12.7}]{Prismatic} also explains that the graded algebra $\gr_*^\mathrm{conj}(\qdeRham_{R/\IZ}/\Phi_p(q))$ is generated by divided powers of $x^p/\Phi_p(q)$ and that these generators induce the Hodge-Tate comparison. As we've seen above, said divided powers are precisely the images of $x^n/[n]_q!$, so we obtain commutativity in our special case.
		
		The method from \cite[\S{\chref[section]{12}}]{Prismatic} then shows commutativity in general: First consider the case $R=\IZ_p\langle x_1^{1/p^\infty},\dotsc,x_n^{1/p^\infty}\rangle/(x_1,\dotsc,x_n)$. This follows from the special case above by multiplicativity. Next consider the case $R=R'/(f_1,\dotsc,f_r)$, where $R'$ is a perfectoid ring and $(f_1,\dotsc,f_r)$ is a $p$-completely regular sequence. If each $f_i$ admits compatible $p$-power roots, we can reduce to the previous special case via base change. In general, by Andre's lemma \cite[Theorem~\chref{7.14}]{Prismatic}, we find a $p$-completely faithfully flat cover $R'\rightarrow R''$ such that $R''$ is perfectoid again and each $f_i$ admits compatible $p$-power roots in $R''$, so we can conclude via descent.
		
		Now assume $R$ is $p$-completely smooth over $\IZ_p$. In this case we can choose a surjection $\IZ_p\langle x_1,\dotsc,x_n\rangle\twoheadrightarrow R$ and put
		\begin{equation*}
			R_\infty\coloneqq \left(\IZ_p\bigl\langle x_1^{1/p^\infty},\dotsc,x_n^{1/p^\infty}\bigr\rangle\otimes_{\IZ_p\langle x_1,\dotsc,x_n\rangle}R\right)_p^\complete\,.
		\end{equation*}
		Using descent for $R\rightarrow R_\infty$, we only need to check the assertion for each term in the \v Cech nerve $(R_\infty^{\otimes_R\bullet})_p^\complete$. These terms are Zariski-locally of the form considered in the previous paragraph and so the smooth case follows. Finally, the case of arbitrary $R$ follows by passing to animations.
	\end{proof}
	The same slightly convoluted method of proof can be used to show the following technical lemma, which we'll need below.
	\begin{lem}\label{lem:NygaardRationalisation}
		The equivalence $(\qdeRham_{R/A})_p^\complete[1/p]_{(q-1)}^\complete\simeq (\deRham_{R/A})_p^\complete[1/p]\qpower$ upgrades uniquely to an equivalence of filtered $\IE_\infty$-$A[1/p,q]$-algebras
		\begin{equation*}
			\fil_\Nn^\star \bigl(\qdeRham_{R/A}^{(p)}\bigr)_p^\complete\bigl[\localise{p}\bigr]_{\Phi_p(q)}^\complete\overset{\simeq}{\longrightarrow} \fil_{(\Hodge,\Phi_p(q))}^\star \bigl(\deRham_{R/A}\lotimes_{A,\phi}A\bigr)_p^\complete\bigl[\localise{p},q\bigr]_{\Phi_p(q)}^\complete\,,
		\end{equation*}
		where $ \fil_{(\Hodge,\Phi_p(q))}^\star $ denotes the combined Hodge and $\Phi_p(q)$-adic filtration.
	\end{lem}
	\begin{proof}
		Let us first construct the map. It's enough to do this in the case where $R$ a $p$-complete quasi-syntomic $A$-algebra which is \emph{large} in the sense that there exists a surjection $\widehat{A}_p\langle x_i^{1/p^\infty}\ |\ i\in I\rangle \twoheadrightarrow R$ for some set~$I$. Via quasi-syntomic descent, we can then recover the case where $R$ is smooth over~$A$, and the general case follows via animation.
		
		If $R$ is as above, then $(\qdeRham_{R/A})_p^\complete[1/p]_{(q-1)}^\complete\simeq (\deRham_{R/A})_p^\complete[1/p]\qpower$ are static and so are the filtrations on them. So we have to compare two descending filtrations of a ring by ideals. It follows at once that the comparison, if it exists, must be unique, and it will automatically be compatible with the filtered $\IE_\infty$-$A[1/p,q]$-algebra structures. Moreover, to compare the two filtrations by ideals, we may base change along the faithfully flat map $A\rightarrow A_\infty$.%
		\footnote{Recall from \cref{rem:qDeRhamRationalisationTechnical} that for every fixed $n\geqslant 0$ there exists an $N$ such that the canonical map $(\qdeRham_{R/A})_p^\complete\rightarrow (\deRham_{R/A})_p^\complete[1/p]\qpower/(q-1)^n$ already factors through $p^{-N}(\deRham_{R/A})_p^\complete\qpower/(q-1)^n$. The existence of a map
			\begin{equation*}
				\fil_\Nn^\star \bigl(\qdeRham_{R/A}^{(p)}\bigr)_p^\complete\longrightarrow	 \fil_{(\Hodge,\Phi_p(q))}^\star \bigl(\deRham_{R/A}\lotimes_{A,\phi}A\bigr)_p^\complete\bigl[\localise{p},q\bigr]_{\Phi_p(q)}^\complete
			\end{equation*}	
			boils down to an inclusion of ideals. Using the observation above, this inclusion can be checked modulo powers of $p$ and $\Phi_p(q)$, and so we can use base change along $A\rightarrow A_\infty$ without having to worry about completion issues.}
		Since working relative to the perfect $\delta$-ring $A_\infty$ is equivalent to working absolutely, we may thus assume $A=\IZ$. Then we use the method from \cite[\S{\chref[section]{12}}]{Prismatic} as in the previous proof of \cref{prop:NygaardComparison}.
		
		So we only have to check the single case $R=\IZ_p\langle x^{1/p^\infty}\rangle/x$. In this case, $(\qdeRham_{R/\IZ}^{(p)})_p^\complete$ is given by a completed direct sum
		\begin{equation*}
			\bigl(\qdeRham_{R/\IZ}^{(p)}\bigr)_p^\complete\simeq\biggl(\bigoplus_{i\in\IN[1/p]}\IZ_p\qpower\cdot\frac{x^i}{[\lfloor i\rfloor]_{q^p}!}\biggr)_{(p,\Phi_p(q))}^\complete\,.
		\end{equation*}
		By definition, $ \fil_\Nn^n(\qdeRham_{R/\IZ}^{(p)})_p^\complete$ consists of those elements whose Frobenius becomes divisible by $\Phi_p(q)^n$. By inspection, these are precisely
		\begin{equation*}
			\fil_\Nn^n\bigl(\qdeRham_{R/\IZ}^{(p)}\bigr)_p^\complete\simeq \biggl(\bigoplus_{i\in\IN[1/p]}\Phi_p(q)^{\max\{n-\lfloor i\rfloor,0\}}\IZ_p\qpower\cdot\frac{x^i}{[\lfloor i\rfloor]_{q^p}!}\biggr)_{(p,\Phi_p(q))}^\complete\,.
		\end{equation*}
		After $(-)[1/p]_{\Phi_p(q)}^\complete$, this becomes the ideal $(x,\Phi_p(q))^n$, which is the $n$\textsuperscript{th} step in the combined Hodge and $\Phi_p(q)$-adic filtration on $(\deRham_{R/\IZ})_p^\complete[1/p,q]_{\Phi_p(q)}^\complete$. This finishes the discussion of the special case and thus the construction of the comparison between the two filtrations.
		
		To show that we get an equivalence, let $A$ be arbitrary again and let $R$ be any animated $A$-algebra. We'll show that both sides agree if we reduce them modulo $\Phi_p(q)$, where $\Phi_p(q)$ sits in filtration degree~$1$. Since both sides also agree in filtration degree~$0$, it will follow inductively that they agree everywhere. By construction,
		\begin{equation*}
			\fil_{(\Hodge,\Phi_p(q))}^\star \bigl(\deRham_{R/A}\lotimes_{A,\phi}A\bigr)_p^\complete\bigl[\localise{p},q\bigr]_{\Phi_p(q)}^\complete/\Phi_p(q)\simeq  \fil_{\Hodge}^\star \bigl(\deRham_{R/A}\lotimes_{A,\phi}A[\zeta_p]\bigr)_p^\complete\bigl[\localise{p}\bigr]
		\end{equation*}
		is just a base change of the Hodge filtration. So let's see what happens on the left-hand side. Since $(q-1)$ becomes invertible after $(-)[1/p]_{\Phi_p(q)}^\complete$, we may as well reduce modulo $(q^p-1)$, again sitting in filtration degree~$1$. Then \cref{prop:NygaardComparison} shows
		\begin{equation*}
			\fil_\Nn^\star \bigl(\qdeRham_{R/A}^{(p)}\bigr)_p^\complete\bigl[\localise{p}\bigr]_{\Phi_p(q)}^\complete/(q^p-1)\simeq  \fil_\Nn^\star \left(\qIW_p\deRham_{R/A}\right)_p^\complete\bigl[\localise{p}\bigr]_{\Phi_p(q)}^\complete\,.
		\end{equation*}
		We claim that the right-hand side is equivalent to $( \fil_{\Hodge}^\star \deRham_{R/A}\lotimes_{A,\phi}A[\zeta_p])_p^\complete[1/p]$ via the ghost map $\gh_1$. This may be checked in the case where $R$ is smooth over $A$, as then the general case follows via animation. As we've seen above, $(-)[1/p]_{\Phi_p(q)}^\complete$ forces $(q-1)$ to be invertible, and so all the images of $V_p$ in \cref{def:NygaardFiltration} die because they're all $(q-1)$-torsion. It follows that for $R$ smooth over $A$, the ghost map
		\begin{equation*}
			\gh_1\colon  \fil_\Nn^\star \bigl(\qIW_p\Omega_{R/A}^*\bigr)_p^\complete\bigl[\localise{p}\bigr]_{\Phi_p(q)}^\complete\overset{\simeq}{\longrightarrow}  \fil_{\Hodge}^\star \bigl(\Omega_{R/A}^*\otimes_{A,\phi}A[\zeta_p]\bigr)_p^\complete\bigl[\localise{p}\bigr]
		\end{equation*}
		is already an isomorphism on the level of complexes and so we're done.
	\end{proof}

	\subsection{The twisted \texorpdfstring{$q$}{q}-Hodge filtration}\label{subsec:TwistedqHodgeFiltration}
	
	For smooth $A$-algebras $S$, let $ \fil_{\Hhodge_m}^\star \qIW_m\Omega_{S/A}^*$ denote the stupid filtration given by $\qIW_m\Omega_{S/A}^{\smash{\geqslant} n,*}$ in degree~$n$. In general, let
	\begin{equation*}
		\fil_{\Hhodge_m}^\star \qIW_m\deRham_{-/A}\colon \cat{AniAlg}_A\longrightarrow \CAlg\Bigl( \Fil\Dd\bigl(A[q]/(q^{m}-1)\bigr)\Bigr)
	\end{equation*}
	be the animation of this functor. In this subsection, we'll show that once $\qdeRham_{R/A}$ is equipped with a $q$-Hodge filtration, the filtration $ \fil_{\Hhodge_m}^\star \qIW_m\deRham_{R/A}$ admits a canonical $q^m$-deformation
	\begin{equation*}
		\fil_{\qHhodge_m}^\star \qdeRham_{R/A}^{(m)}\,.
	\end{equation*}
	This will eventually allow us to prove \cref{thm:HabiroDescent} in \cref{subsec:qHodgeHabiroDescent} below. Let us first explain how $ \fil_{\Hhodge_m}^\star \qIW_m\deRham_{R/A}$ is related to the Nygaard filtration from \cref{subsec:Nygaard}.
	\begin{lem}\label{lem:HodgevsNygaard}
		For all smooth $A$-algebras $S$, all primes~$p$ and all $\alpha\geqslant 1$ the diagram
		\begin{equation*}
			\begin{tikzcd}
				\fil_{\Hhodge_{p^\alpha}}^n\qIW_{p^\alpha}\Omega_{S/A}^*\rar\dar\drar[pullback] &  \fil_\Nn^n\qIW_{p^\alpha}\Omega_{S/A}^*\dar["p^{-n}\overtilde{F}_p"]\\
				\fil_{\Hhodge_{p^{\alpha-1}}}^n\qIW_{p^{\alpha-1}}\Omega_{S/A}^*\rar & \qIW_{p^{\alpha-1}}\Omega_{S/A}^*
			\end{tikzcd}
		\end{equation*}
		becomes a pullback in $\Dd(A[q])$ for all $n\geqslant 0$.
	\end{lem}
	\begin{proof}
		It's enough to check that the induced map on horizontal cofibres is an equivalence. Since $ \fil_{\Hhodge_{p^\alpha}}^n\qIW_{p^\alpha}\Omega_{S/A}^*\rightarrow \fil_\Nn^n\qIW_{p^\alpha}\Omega_{S/A}^*$ is injective, the cofibre agrees with the cokernel, which is given by
		\begin{equation*}
			\left(p^{n-1}V_p\bigl(\qIW_{p^{\alpha-1}}\Omega_{S/A}^0\bigr) \rightarrow\dotsb\rightarrow p^0V_p\bigl(\qIW_{p^{\alpha-1}}\Omega_{S/A}^{n-1}\bigr)\rightarrow 0\rightarrow0\rightarrow\dotsb\right)\,.
		\end{equation*}
		Under $p^{-n}\overtilde{F}_p$, this complex is mapped isomorphically onto 
		\begin{equation*}
			\left(\qIW_{p^{\alpha-1}}\Omega_{S/A}^0 \rightarrow\dotsb\rightarrow \qIW_{p^{\alpha-1}}\Omega_{S/A}^{n-1}\rightarrow 0\rightarrow0\rightarrow\dotsb\right)\,,
		\end{equation*}
		which is the cokernel (and the cofibre) of $ \fil_{\Hhodge_{p^{\alpha-1}}}^n\qIW_{p^{\alpha-1}}\Omega_{S/A}^*\rightarrow \qIW_{p^{\alpha-1}}\Omega_{S/A}^*$
	\end{proof}
	\begin{cor}\label{cor:qDRWSmoothAnimation}
		If $S$ is smooth over $A$, then $\qIW_m\deRham_{S/A}^n\simeq \Sigma^{-n}\qIW_m\Omega_{S/A}^n$ for all $m\in\IN$ and all degrees $n\geqslant 0$.
	\end{cor}
	\begin{proof}
		It's enough to show this rationally and after $p$-completion for all primes~$p$. Rationally, \cite[Corollary~\chref{3.34}]{qWitt} shows
		\begin{equation*}
			\qIW_m\Omega_{S/A}^n\otimes_\IZ\IQ\cong \prod_{d\mid m}\left(\Omega_{S/A}^n\otimes_{A,\psi^d}(A\otimes\IQ)[\zeta_d]\right)
		\end{equation*}
		and it's well-known that the values of $\Omega_{-/A}^n$ on smooth $A$-algebras don't change under animation. After $p$-completion, \cite[Lemma~\chref{4.36}]{qWitt} allows us to restrict to the case where $m=p^\alpha$ is a prime power. Since $\qIW_{p^{\alpha}}\deRham_{S/A}^n\simeq \gr_{\Hhodge_{p^\alpha}}^n\qIW_{p^{\alpha}}\deRham_{S/A}$, it will be enough to show that the filtration $ \fil_{\Hhodge_{p^\alpha}}^\star \qIW_{p^\alpha}\Omega_{S/A}^*$ is unchanged under animation. This follows via induction on $\alpha$ from \cref{lem:HodgevsNygaard} and \cref{cor:NygaardQuasisyntomicDescent}\cref{enum:NygaardSmooth}.
	\end{proof}
	We now set out to construct the desired $q^m$-deformation of $ \fil_{\Hhodge_m}^\star \qIW_m\deRham_{R/A}$.
	\begin{numpar}[The twisted $q$-Hodge filtration ($p$-adically).]\label{con:TwistedqHodgeFiltrationpAdic}
		Let's first construct the filtration for prime powers $m=p^\alpha$ and after $p$-completion. We'll use a recursive definition. For $\alpha=0$, $\qdeRham_{R/A}^{(p^0)}\simeq \qdeRham_{R/A}$ is just the $q$-de Rham complex and we choose
		\begin{equation*}
			\fil_{\qHhodge_{p^0}}^\star \bigl(\qdeRham_{R/A}^{(p^0)}\bigr)_p^\complete\coloneqq   \fil_{\qHodge}^\star \left(\qdeRham_{R/A}\right)_p^\complete
		\end{equation*}
		to be the given $q$-Hodge filtration. For $\alpha\geqslant 1$, we consider the \enquote{rescaling} of the filtration $ \fil_{\qHhodge_{p^{\alpha-1}}}^\star \bigl(\qdeRham_{R/A}^{(p^{\alpha-1})}\bigr)_p^\complete$ by $\Phi_{p^\alpha}$, that is,
		\begin{equation*}
			\Phi_{p^\alpha}(q)^\star  \fil_{\qHhodge_{p^{\alpha-1}}}^\star \coloneqq \left( \fil_{\qHhodge_{p^{\alpha-1}}}^0\xleftarrow{\Phi_{p^\alpha}(q)} \fil_{\qHhodge_{p^{\alpha-1}}}^1\xleftarrow{\Phi_{p^\alpha}(q)}\dotsb\right)\,.
		\end{equation*}
		We also equip $\bigl(\qdeRham_{R/A}^{(p^{\alpha-1})}\bigr)_p^\complete$ with its $\Phi_{p^\alpha}(q)$-adic filtration. 
		Then we define $ \fil_{\Hhodge_{p^\alpha}}^\star \bigl(\qdeRham_{R/A}^{(p^\alpha)}\bigr)_p^\complete$ as the following pullback of filtered objects:
		\begin{equation*}
			\begin{tikzcd}
				\fil_{\qHhodge_{p^\alpha}}^\star \bigl(\qdeRham_{R/A}^{(p^\alpha)}\bigr)_p^\complete\rar\dar\drar[pullback] &  \fil_\Nn^\star \bigl(\qdeRham_{R/A}^{(p^\alpha)}\bigr)_p^\complete\dar["\phi_{p/A[q]}"]\\
				\Phi_{p^\alpha}(q)^\star  \fil_{\qHhodge_{p^{\alpha-1}}}^\star \bigl(\qdeRham_{R/A}^{(p^{\alpha-1})}\bigr)_p^\complete\rar & \Phi_{p^\alpha}(q)^\star \bigl(\qdeRham_{R/A}^{(p^{\alpha-1})}\bigr)_p^\complete
			\end{tikzcd}
		\end{equation*}
		Using this pullback diagram, we can also inductively equip $ \fil_{\qHhodge_{p^\alpha}}^\star \bigl(\qdeRham_{R/A}^{(p^\alpha)}\bigr)_p^\complete$ with the structure of a filtered module over the filtered ring $(q^{p^\alpha}-1)^\star A[q]$.
	\end{numpar}
	\begin{rem}\label{rem:PullbackModqpa-1}
		If we reduce the pullback diagram above modulo $(q^{p^\alpha}-1)$ (where we invoke Convention~\cref{conv:QuotientConvention} as usual), we obtain the pullback diagram from \cref{lem:HodgevsNygaard}. Indeed, this follows via induction on~$\alpha$, using \cref{prop:NygaardComparison}. It follows that 
		\begin{equation*}
			\fil_{\qHhodge_{p^\alpha}}^\star \bigl(\qdeRham_{R/A}^{(p^\alpha)}\bigr)_p^\complete/(q^{p^\alpha}-1)\simeq  \fil_{\Hhodge_{p^\alpha}}(\qIW_{p^\alpha}\deRham_{R/A})_p^\complete\,.
		\end{equation*}
	\end{rem}
	\begin{numpar}[Lax symmetric monoidal structure I.]\label{par:TwistedqHodgeLaxSymmetricMonoidalI}
		The functor
		\begin{equation*}
			\fil_{\Hhodge_{p^\alpha}}^\star \bigl(\qdeRham_{-/A}^{(p^\alpha)}\bigr)_p^\complete\colon \cat{AniAlg}_A^{\qHodge}\longrightarrow  \Mod_{(q^{p^\alpha}-1)^\star A[q]}\Bigl(\Fil \Dd\bigl(A[q]\bigr)\Bigr)_{(p,q-1)}^\complete
		\end{equation*}
		comes equipped with a canonical lax symmetric monoidal structure. This follows from the recursive construction. For $\alpha=0$, \cref{lem:qHodgeSymmetricMonoidal} even provides a symmetric monoidal structure. For $\alpha\geqslant 1$, we must equip the legs of the pullback in \cref{con:TwistedqHodgeFiltrationpAdic} with the structure of symmetric monoidal transformations. This is not hard. First, the Frobenius
		\begin{equation*}
			\phi_{/A[q]}\colon  \fil_\Nn^\star \bigl(\qdeRham_{R/A}^{(p^\alpha)}\bigr)_p^\complete\longrightarrow \Phi_{p^\alpha}(q)^\star \bigl(\qdeRham_{R/A}^{(p^{\alpha-1})}\bigr)_p^\complete
		\end{equation*}
		becomes a symmetric monoidal transformation by quasi-syntomic descent from the case where $R$ is a $p$-complete quasi-syntomic $A$-algebra with a surjection $\widehat{A}_p\langle x_i^{1/p^\infty}\ |\ i\in I\rangle \twoheadrightarrow R$. In this case, we're dealing with filtrations of rings by ideals, so symmetric monoidality is automatic.
		
		Second, the functor that \enquote{rescales} a filtration by $\Phi_{p^\alpha}(q)$ as in \cref{con:TwistedqHodgeFiltrationpAdic} is lax symmetric monoidal. Indeed, if we regard our filtered objects as graded modules over $\IZ[q,t]$, with the filtration parameter~$t$ in graded degree~$-1$, then rescaling corresponds to restriction along the $\IZ[q]$-linear map $\IZ[q,t]\rightarrow \IZ[q,t]$ that sends $t\mapsto \Phi_{p^\alpha}(q)t$. This is lax symmetric monoidal.
	\end{numpar}
	\begin{numpar}[Lax symmetric monoidal structure II.]\label{par:TwistedqHodgeLaxSymmetricMonoidalII}
		It follows from the construction in~\cref{con:TwistedqHodgeFiltrationpAdic} that we have a canonical map
		\begin{equation*}
			\fil_{\Hhodge_{p^\alpha}}^\star \bigl(\qdeRham_{R/A}^{(p^\alpha)}\bigr)_p^\complete\longrightarrow  \fil_{\Hhodge_{p^{\alpha-1}}}^\star \bigl(\qdeRham_{R/A}^{(p^{\alpha-1})}\bigr)_p^\complete
		\end{equation*}
		compatible with the relative Frobenius $\phi_{p/A[q]}\colon \bigl(\qdeRham_{R/A}^{(p^\alpha)}\bigr)_p^\complete\rightarrow\bigl(\qdeRham_{R/A}^{(p^{\alpha-1})}\bigr)_p^\complete$, because the \enquote{rescaling} by $\Phi_{p^\alpha}(q)$ of any filtration in non-negative degrees has a canonical map back to the original filtration. Moreover, by the discussion in \cref{par:TwistedqHodgeLaxSymmetricMonoidalI}, the map above can be canonically equipped with the structure of a symmetric monoidal transformation.
	\end{numpar}
	\begin{lem}\label{lem:TwistedqHodgepAdicCompatibilityI}
		For all primes~$p$ and all $\alpha\geqslant 0$, there exists a canonical equivalence of filtered $(q^{p^\alpha}-1)^\star A[q]$-modules
		\begin{equation*}
			\fil_{\qHhodge_{p^\alpha}}^\star \bigl(\qdeRham_{R/A}^{(p^\alpha)}\bigr)_p^\complete\bigl[\localise{p}\bigr]_{\Phi_{p^\alpha}(q)}^\complete\overset{\simeq}{\longrightarrow} \fil_{(\Hodge,\Phi_{p^\alpha}(q))}^\star \bigl(\deRham_{R/A}\lotimes_{A,\psi^{p^\alpha}}A\bigr)_p^\complete\bigl[\localise{p},q\bigr]_{\Phi_{p^\alpha}}^\complete\,,
		\end{equation*}
		where $ \fil_{(\Hodge,\Phi_{p^\alpha}(q))}^\star $ denotes the combined Hodge and $\Phi_{p^\alpha}(q)$-adic filtration. This equivalence
		is compatible with $(\qdeRham_{R/A})_p^\complete[1/p]_{(q-1)}^\complete\simeq (\deRham_{R/A})_p^\complete[1/p]\qpower$.
	\end{lem}
	\begin{proof}
		For $\alpha=0$, this is the condition from \cref{def:qHodgeFiltration}\cref{enum:qHodgeRationalpComplete}. So let $\alpha\geqslant 1$. After applying $(-)[1/p]_{\Phi_{p^\alpha}(q)}^\complete$, the polynomial $(q^{p^{\alpha-1}}-1)$ becomes invertible, and so the filtered $(q^{p^{\alpha-1}}-1)^\star A[q]$-module
		\begin{equation*}
			\fil_{\qHhodge_{p^{\alpha-1}}}^\star \bigl(\qdeRham_{R/A}^{(p^{\alpha-1})}\bigr)_p^\complete\bigl[\localise{p}\bigr]_{\Phi_{p^\alpha}(q)}^\complete
		\end{equation*}
		must be the constant filtration on $\bigl(\qdeRham_{R/A}^{(p^{\alpha-1})}\bigr)_p^\complete[1/p]_{\Phi_{p^\alpha}(q)}^\complete$. Consequently, after applying $(-)[1/p]_{\Phi_{p^\alpha}(q)}^\complete$ the bottom horizontal arrow in the pullback diagram from \cref{con:TwistedqHodgeFiltrationpAdic} becomes an equivalence and thus the top horizontal arrow becomes an equivalence too. The desired assertion then follows via base change from \cref{lem:NygaardRationalisation}.
	\end{proof}
	\begin{lem}\label{lem:TwistedqHodgepAdicCompatibilityII}
		For all primes~$p$, all $\alpha\geqslant 1$, and all $0\leqslant i\leqslant \alpha-1$, the canonical map from \cref{par:TwistedqHodgeLaxSymmetricMonoidalII} induces an equivalence of filtered $(q^{p^\alpha}-1)^\star A[q]$-modules
		\begin{equation*}
			\fil_{\qHhodge_{p^\alpha}}^\star \bigl(\qdeRham_{R/A}^{(p^\alpha)}\bigr)_p^\complete\bigl[\localise{p}\bigr]_{\Phi_{p^{\smash{i}}}(q)}^\complete\overset{\simeq}{\longrightarrow} \fil_{\qHhodge_{p^{\alpha-1}}}^\star \bigl(\qdeRham_{R/A}^{(p^{\alpha-1})}\bigr)_p^\complete\bigl[\localise{p}\bigr]_{\Phi_{p^{\smash{i}}}(q)}^\complete\,.
		\end{equation*}
	\end{lem}
	\begin{proof}
		After $(-)[1/p]_{\Phi_{p^{\smash{i}}}(q)}^\complete$, the polynomial $\Phi_{p^\alpha}(q)$ becomes invertible. Consequently, the \enquote{rescaling} of filtrations in \cref{con:TwistedqHodgeFiltrationpAdic} has no effect anymore. Moreover, it follows that the filtered $\Phi_{p^\alpha}(q)^\star A[q]$-module
		\begin{equation*}
			\fil_\Nn^\star \bigl(\qdeRham_{R/A}^{(p^\alpha)}\bigr)_p^\complete\bigl[\localise{p}\bigr]_{\Phi_{p^{\smash{i}}}(q)}^\complete
		\end{equation*}
		must be the constant filtration on $\bigl(\qdeRham_{R/A}^{(p^\alpha)}\bigr)_p^\complete\bigl[\localise{p}\bigr]_{\Phi_{p^{\smash{i}}}(q)}^\complete$. Thus, after applying $(-)[1/p]_{\Phi_{p^{\smash{i}}}(q)}^\complete$, the pullback from \cref{con:TwistedqHodgeFiltrationpAdic} collapses to the desired equivalence.
	\end{proof}
	Let us finally construct the filtration $ \fil_{\qHhodge_m}^\star \qdeRham_{R/A}^{(m)}$ in general.
	\begin{numpar}[The twisted $q$-Hodge filtration (globally)]\label{con:TwistedqHodgeFiltration}
		Choose $N\neq 0$ divisible by $m$ (we'll argue below that the choice of $N$ doesn't matter). For every divisor $d\mid m$ and every prime~$p\mid N$, write $m=p^{v_p(m)}m_p$ and $d=p^{v_p(d)}d_p$, where $m_p$ and $d_p$ are coprime to $p$. Using the animated version of \cref{lem:TwistedqDeRhamFractureSquare}, we obtain a pullback diagram
		\begin{equation*}
			\begin{tikzcd}
				\qdeRham_{R/A}^{(m)}\rar\dar\drar[pullback] & \prod_{p\mid N}\prod_{d_p\mid m_p}\left(\qdeRham_{R/A}\lotimes_{A[q],\psi^{p^{\smash{v_p(m)}}d_p}}A[q]\right)_{(p,\Phi_{d_p}(q))}^\complete\dar["\left(\phi_{p/A[q]}^{v_p(m/d)}\right)_{p\mid N,\,d\mid m}"]\\
				\prod_{d\mid m}\left(\qdeRham_{R/A}\lotimes_{A[q],\psi^d}A\bigl[\localise{N},q\bigr]\right)_{\Phi_d(q)}^\complete\rar & \prod_{p\mid N}\prod_{d \mid m}\Bigl(\qdeRham_{R/A}\lotimes_{A[q],\psi^{d}}A[q]\Bigr)_p^\complete\bigl[\localise{p}\bigr]_{\Phi_d(q)}^\complete
			\end{tikzcd}			
		\end{equation*}
		To construct $\fil_{\qHhodge_m}^\star \qdeRham_{R/A}^{(m)}$, we'll equip each factor of the pullback above with a filtration and then check that these filtrations are compatible.
		\begin{alphanumerate}
			\item On the factor $(\qdeRham_{R/A}\lotimes_{A[q],\psi^d}A[1/N,q])_{\Phi_d(q)}^\complete$ for any $d\mid m$, we put the base-changed $q$-Hodge filtration\label{enum:TwistedqHodgeFiltrationRational}
			\begin{equation*}
				\left(\fil_{\qHodge}^\star \qdeRham_{R/A}\lotimes_{A[q],\psi^d}A\bigl[\localise{N},q\bigr]\right)_{\Phi_d(q)}^\complete\,.
			\end{equation*}
			\item On the factor $(\qdeRham_{R/A}\lotimes_{A[q],\psi^{d}}A[q])_p^\complete[1/p,q]_{\Phi_d(q)}^\complete$ for any prime~$p\mid N$ and any $d\mid m$, we put again the base-changed $q$-Hodge filtration.\label{enum:TwistedqHodgeFiltrationpAdicRational}
			\item On the factor $(\qdeRham_{R/A}\lotimes_{A[q],\psi^{p^{\smash{v_p(m)}}d_p}}A[q])_{(p,\Phi_{d_p}(q))}^\complete$ for any prime~$p\mid N$ and any $d_p\mid m_p$, we put the base-changed filtration\label{enum:TwistedqHodgeFiltrationpAdic}
			\begin{equation*}
				\left( \fil_{\qHhodge_{p^{\smash{v_p(m)}}}}^\star \bigl(\qdeRham_{R/A}^{(p^{v_p(m)})}\bigr)_p^\complete\lotimes_{A[q],\psi^{d_p}}A[q]\right)_{(p,\Phi_{d_p}(q))}^\complete\,.
			\end{equation*}
		\end{alphanumerate}
		Moreover, each of these filtrations is canonically a module over the filtered ring $(q^m-1)^\star A[q]$. It's clear that \cref{enum:TwistedqHodgeFiltrationRational} and \cref{enum:TwistedqHodgeFiltrationpAdicRational} are compatible as filtered $(q^m-1)^\star A[q]$-modules. To check that \cref{enum:TwistedqHodgeFiltrationpAdic} and \cref{enum:TwistedqHodgeFiltrationpAdicRational} are compatible, we may reduce via base change to the case where $m=p^\alpha$ is a power of~$p$. From \cref{lem:TwistedqHodgepAdicCompatibilityI,lem:TwistedqHodgepAdicCompatibilityII} and our assumptions on $\fil_{\qHodge}^\star$ we deduce that both filtrations can be identified with the combined Hodge and $\Phi_{p^\alpha}(q)$-adic filtration
		\begin{equation*}
			\fil_{(\Hodge,\Phi_{p^\alpha}(q))}^\star \bigl(\deRham_{R/A}\lotimes_{A,\psi^{p^\alpha}}A\bigr)_p^\complete\bigl[\localise{p},q\bigr]_{\Phi_{p^\alpha}(q)}^\complete\,,
		\end{equation*}
		which yields the desired compatibility.
		
		Let us now argue that the choice of $N$ is irrelevant. Suppose $N\mid N'$. Then the pullback diagrams for $N'$ is obtained from the pullback square for $N$ by replacing the bottom left corner $\prod_{d\mid m}(\qdeRham_{R/A}\lotimes_{A[q],\psi^{d}}A[1/N,q])_{\Phi_d(q)}^\complete$ by the pullback square
		\begin{equation*}
			\begin{tikzcd}
				\prod_{d\mid m}\left(\qdeRham_{R/A}\lotimes_{A[q],\psi^d}A\bigl[\localise{N},q\bigr]\right)_{\Phi_d(q)}^\complete\rar\dar\drar[pullback] & \prod_{\ell\vphantom{\mid}}\prod_{d\mid m}\left(\qdeRham_{R/A}\lotimes_{A[q],\psi^d}A\bigl[\localise{N},q\bigr]\right)_{(\ell,\Phi_d(q))}^\complete\dar\\
				\prod_{d\mid m}\left(\qdeRham_{R/A}\lotimes_{A[q],\psi^d}A\bigl[\localise{N'},q\bigr]\right)_{\Phi_d(q)}^\complete \rar &\prod_{\ell\vphantom{\mid}}\prod_{d\mid m}\left(\qdeRham_{R/A}\lotimes_{A[q],\psi^d}A\bigl[\localise{N},q\bigr]\right)_\ell^\complete\bigl[\localise{\ell}\bigr]_{\Phi_d(q)}^\complete
			\end{tikzcd}
		\end{equation*}
		where the product is taken over all primes~$\ell$ such that $\ell\mid N'$ but $\ell\nmid N$. Note that for any such prime we also have $\ell\nmid m$, so each $v_\ell(m/d)=0$ and so each iterated Frobenius $\phi_{\ell/A[q]}^{v_\ell(m/d)}$ is the identity. Moreover, we see that the filtrations we put on the different factors is always $\fil_{\qHodge}^\star$, base changed along $\psi^d$. It follows that the filtrations constructed using $N$ and using $N'$ must indeed agree, as claimed. To get a canonical construction, we can let $N$ vary through a totally ordered initial sub-poset of $\IN$ (like $\{n!\}_{n\geqslant m}$) and then take the limit. This finishes the construction of $\fil_{\qHhodge_m}^\star \qdeRham_{R/A}^{(m)}$.
		
		The construction is clearly functorial. Using~\cref{par:TwistedqHodgeLaxSymmetricMonoidalI} and~\cref{par:TwistedqHodgeLaxSymmetricMonoidalII} we also see that the functor 
		\begin{equation*}
			\fil_{\qHhodge_m}^\star \qdeRham_{-/A}^{(m)}\colon \cat{AniAlg}_A^{\qHodge}\longrightarrow \Mod_{(q^m-1)^\star A[q]}\Bigl( \Fil\Dd\bigl(A[q]\bigr)\Bigr)_{(q^m-1)}^\complete
		\end{equation*}
		comes equipped with a canonical lax symmetric monoidal structure.
	\end{numpar}
	\begin{prop}\label{prop:TwistedqHodgeFiltrationqDeforms}
		For all $m\in\IN$, the equivalence $\qdeRham_{R/A}^{(m)}/(q^m-1)\simeq \qIW_m\deRham_{R/A}$ from the animated version of \cref{prop:TwistedqDeRhamDeformsqDRW} upgrades canonically to an equivalence of filtered $A[q]/(q^m-1)$-modules
		\begin{equation*}
			\fil_{\qHhodge_m}^\star \qdeRham_{R/A}^{(m)}/(q^m-1)\overset{\simeq}{\longrightarrow} \fil_{\Hhodge_m}^\star \qIW_m\deRham_{R/A}
		\end{equation*}
		\embrace{the quotient on the left-hand side is taken in accordance with Convention~\cref{conv:QuotientConvention}}.
	\end{prop}
	\begin{proof}[Proof sketch]
		We analyse the effect of $(-)/(q^m-1)$ on each of the factors in \cref{con:TwistedqHodgeFiltration}. For the factors in \cref{con:TwistedqHodgeFiltration}\cref{enum:TwistedqHodgeFiltrationpAdic}, note that $(q^m-1)$ and $\Phi_{d_p}(q^{p^{v_p(m)}})$ will only differ by a unit upon $(p,\Phi_{d_p}(q))$-adic completion. Then the argument in \cref{rem:PullbackModqpa-1} plus base change shows that after modding out $(q^m-1)$ we get
		\begin{equation*}
			\left( \fil_{\Hhodge_{p^{\smash{v_p(m)}}}}^\star \qIW_{p^{\smash{v_p(m)}}}\deRham_{R/A}\lotimes_{A[q],\psi^{d_p}}A[q]\right)_p^\complete/\Phi_{d_p}\bigl(q^{p^{v_p(m)}}\bigr)
		\end{equation*}
		It follows from \cite[Lemma~\chref{4.36}]{qWitt} that $ \fil_{\Hhodge_m}^\star (\qIW_m\deRham_{R/A})_p^\complete$ is indeed a product of factors of this form.
		
		For the factors in \cref{con:TwistedqHodgeFiltration}\cref{enum:TwistedqHodgeFiltrationRational}, note that $(q^m-1)$ and $\Phi_d(q)$ will only differ by a unit after $(-)[1/N]_{\Phi_d(q)}^\complete$. By construction, the $q$-Hodge filtration becomes the Hodge filtration modulo $(q-1)$. Thus, after base change along $\psi^d\colon A[q]\rightarrow A[q]$, we get
		\begin{equation*}
			\left(\fil_{\qHodge}^\star \qdeRham_{R/A}\lotimes_{A[q],\psi^d}A\bigl[\localise{N},q\bigr]\right)_{\Phi_d(q)}^\complete/\Phi_d(q)\simeq \fil_{\Hodge}^\star \deRham_{R/A}\lotimes_{A,\psi^d}A\bigl[\localise{N},\zeta_d\bigr]\,.
		\end{equation*}
		It follows from \cite[Corollary~\chref{3.34}]{qWitt} that $ \fil_{\Hhodge_m}^\star \qIW_m\deRham_{R/A}[1/N]$ is indeed a product of factors of this form. The same argument applies for the factors in \cref{con:TwistedqHodgeFiltration}\cref{enum:TwistedqHodgeFiltrationpAdicRational}.
	\end{proof}
	\begin{rem}
		It follows from the proof that the equivalence in \cref{prop:TwistedqHodgeFiltrationqDeforms} is, in fact, an equivalence of lax symmetric monoidal functors $\cat{AniAlg}_A^{\qHodge}\rightarrow  \Fil\Dd(A[q]/(q^m-1))$. Thus, if $(R, \fil_{\qHodge}^\star \qdeRham_{R/A})$ admits the structure of an $\IE_n$-algebra in $\cat{AniAlg}_A^{\qHodge}$ for any $0\leqslant n\leqslant\infty$, then the equivalence in \cref{prop:TwistedqHodgeFiltrationqDeforms} will be one of filtered $\IE_n$-$A[q]/(q^m-1)$-algebras.
	\end{rem}
	\begin{numpar}[Transition maps.]\label{con:TwistedqHodgeFiltrationFunctorial}
		Whenever $n\mid m$, there's a canonical map of filtered objects
		\begin{equation*}
			\fil_{\qHhodge_m}^\star \qdeRham_{R/A}^{(m)}\longrightarrow  \fil_{\qHhodge_n}^\star \qdeRham_{R/A}^{(n)}\,.
		\end{equation*}
		To construct this, we look at the factors of the pullback from \cref{con:TwistedqHodgeFiltration} (we're allowed to use the same $N$ for both $m$ and $n$). For the factors from \cref{con:TwistedqHodgeFiltration}\cref{enum:TwistedqHodgeFiltrationRational} and~\cref{enum:TwistedqHodgeFiltrationpAdicRational}, we simply project to those where $d\mid n$. For the factors from \cref{con:TwistedqHodgeFiltration}\cref{enum:TwistedqHodgeFiltrationpAdic}, we first project to those where $d_p\mid n_p$ and then we use the maps from \cref{par:TwistedqHodgeLaxSymmetricMonoidalII}, base changed along $\psi^{d_p}\colon A[q]\rightarrow A[q]$.
		
		It's clear from the construction that these maps $ \fil_{\qHhodge_m}^\star \qdeRham_{R/A}^{(m)}\rightarrow  \fil_{\qHhodge_n}^\star \qdeRham_{R/A}^{(n)}$ assemble canonically into a symmetric monoidal transformation of lax symmetric monoidal functors. With some more effort, one can also make these transformations functorial in~$n,m\in\IN$, where $\IN$ denotes the category of natural numbers partially ordered by divisibility. For our purposes, the existence of the individual maps is enough, as any $\limit_{m\in\IN}$ can be replaced by the limit over the sequential subdiagram given by $\{n!\}_{n\geqslant 1}$. We will therefore not spell out the construction of this additional functoriality.
	\end{numpar}
	
	\subsection{Habiro descent for \texorpdfstring{$q$}{q}-Hodge complexes}\label{subsec:qHodgeHabiroDescent}
	In this subsection, we'll finish the proof of \cref{thm:HabiroDescent}, following the outline that we have explained at the end of \cref{subsec:MainResult}. 
	
	\begin{numpar}[The $(q^m-1)$-complete descent.]\label{con:HabiroDescent0}
		For all $m\in\IN$, we consider the colimit
		\begin{equation*}
			\qHhodge_{R/A,m}\coloneqq \colimit\Bigl( \fil_{\qHhodge_m}^{0}\qdeRham_{R/A}^{(m)}\xrightarrow{(q^m-1)} \fil_{\qHhodge_m}^{1}\qdeRham_{R/A}^{(m)}\xrightarrow{(q^m-1)}\dotso\Bigr)_{(q^m-1)}^\complete\,.
		\end{equation*}
		In the following, we'll informally write
		\begin{equation*}
			\qHhodge_{R/A,  m}\simeq \qdeRham_{R/A}^{(m)}\left[\frac{ \fil_{\qHhodge_m}^i}{(q^m-1)^i}\ \middle|\ i\geqslant 1\right]_{(q^m-1)}^\complete
		\end{equation*}
		and we'll say that \emph{$\qHhodge_{R/A,  m}$ is given by adjoining $(q^m-1)^{-\star} \fil_{\qHhodge_m}^\star$ to $\qdeRham_{R/A}^{(m)}$}. We'll also use similar notation and terminology for related filtrations such as the Nygaard filtration or the combined Hodge and $\Phi_d(q)$-adic filtration for some $d\mid m$.
	\end{numpar}
	\begin{prop}\label{prop:HabiroDescent}
		Let $m\in\IN$. For all divisors $n\mid m$, the map from \cref{con:TwistedqHodgeFiltrationFunctorial} induces an equivalence
		\begin{equation*}
			\bigl(\qHhodge_{R/A,  m}\bigr)_{(q^n-1)}^\complete\overset{\simeq}{\longrightarrow}\qHhodge_{R/A,  n}\,. 
		\end{equation*}
		In particular, $\qHhodge_{R/A,  m}$ is a descent of $\qHodge_{R/A}$ along $\IZ[q]_{(q^m-1)}^\complete\rightarrow \IZ\qpower$.
	\end{prop}
	\begin{proof}[Proof sketch]
		Again, we look at the different factors from \cref{con:TwistedqHodgeFiltration}. Let's start with those from \cref{con:TwistedqHodgeFiltration}\cref{enum:TwistedqHodgeFiltrationRational} for some $d\mid m$. If $d\nmid n$, then $\Phi_d(q)$ and $(q^n-1)$ are coprime in $\IQ[q]$ and so the factor will die after $(q^n-1)$-completion. Therefore in $(\qHhodge_{R/A,  m})_{(q^n-1)}^\complete$ only those factors where $d\mid n$ will survive. These are precisely the factors that are also used in the construction of of $\qHhodge_{R/A,  n}$. Moreover, if $d\mid n$ then both $(q^m-1)$ and $(q^n-1)$ are unit multiples of $\Phi_d(q)$ in $\IQ[q]_{\Phi_d(q)}^\complete$, so it doesn't matter whether we adjoin $(q^m-1)^{-\star} \fil_{(\Hodge,\Phi_d(q))}^\star$ or $(q^n-1)^{-\star} \fil_{(\Hodge,\Phi_d(q))}^\star$. It follows that on the factors from \cref{con:TwistedqHodgeFiltration}\cref{enum:TwistedqHodgeFiltrationRational} we get indeed an equivalence. The same argument applies to the factors from \cref{con:TwistedqHodgeFiltration}\cref{enum:TwistedqHodgeFiltrationpAdicRational}.
		
		It remains to show that we also get an equivalence on the factors from \cref{con:TwistedqHodgeFiltration}\cref{enum:TwistedqHodgeFiltrationpAdic}. So let's consider such a factor for some prime~$p$ and some $d_p\mid m_p$. Using induction, we may assume that~$m$ and~$n$ differ only by a single prime factor. If that prime is different from~$p$, then $(q^m-1)$ and $(q^n-1)$ will differ by a unit after $(p,\Phi_{d_p}(q))$-completion and we can argue as above. So assume $n=m/p$. Via base change along $\psi^{d_p}\colon A[q]\rightarrow A[q]$, we may reduce to the case where $m=p^\alpha$ is a prime power and $n=p^{\alpha-1}$. From \cref{con:TwistedqHodgeFiltrationpAdic} we obtain a pullback diagram
		\begin{equation*}
			\begin{tikzcd}[column sep=small]
				\bigl(\qdeRham_{R/A}^{(p^\alpha)}\bigr)\left[\frac{ \fil_{\qHhodge_{p^\alpha}}^i}{(q^{p^\alpha}-1)^i}\ \middle|\ i\geqslant 0\right]_{(p,q^{p^{\alpha-1}}-1)}^\complete\rar\dar\drar[pullback] & \bigl(\qdeRham_{R/A}^{(p^\alpha)}\bigr)\left[\frac{ \fil_\Nn^i}{(q^{p^\alpha}-1)^i}\ \middle|\ i\geqslant 0\right]_{(p,q^{p^{\alpha-1}}-1)}^\complete\dar\\
				\bigl(\qdeRham_{R/A}^{(p^{\alpha-1})}\bigr)\left[\frac{ \fil_{\qHhodge_{p^{\alpha-1}}}^i}{(q^{p^{\alpha-1}}-1)^i}\ \middle|\ i\geqslant 0\right]_{(p,q^{p^{\alpha-1}}-1)}^\complete\rar & \bigl(\qdeRham_{R/A}^{(p^{\alpha-1})}\bigr)\left[\frac{1}{(q^{p^{\alpha-1}}-1)}\right]_{(p,q^{p^{\alpha-1}}-1)}^\complete
			\end{tikzcd}
		\end{equation*}
		To finish the proof, we must show that the left vertical arrow is an equivalence. Since the diagram is a pullback, it will be enough to show that the right vertical arrow is an equivalence.%
		\footnote{Also note that the bottom right corner vanishes, so it will follow that the top right corner vanishes as well. But this will be irrelevant for our argument.}

		This is now purely an assertion about the Nygaard filtration. Via base change, we may reduce to the case $\alpha=1$. This case will be shown in \cref{lem:NygaardDenominators} below.
	\end{proof}
	\begin{lem}\label{lem:NygaardDenominators}
		The relative Frobenius $\phi_{p/A[q]}\colon \bigl(\qdeRham_{R/A}^{(p)}\bigr)_p^\complete\rightarrow (\qdeRham_{R/A})_p^\complete$ induces functorial equivalences
		\begin{align*}
			\bigl(\qdeRham_{R/A}^{(p)}\bigr)_p^\complete\left[\frac{ \fil_\Nn^i}{\Phi_p(q)^i}\ \middle|\ i\geqslant 0\right]_{(p,q-1)}^\complete&\overset{\simeq}{\longrightarrow} \left(\qdeRham_{R/A}\right)_p^\complete\,,\\	\bigl(\qdeRham_{R/A}^{(p)}\bigr)_p^\complete\left[\frac{ \fil_\Nn^i}{(q^p-1)^i}\ \middle|\ i\geqslant 0\right]_{(p,q-1)}^\complete&\overset{\simeq}{\longrightarrow} \left(\qdeRham_{R/A}\right)_p^\complete\left[\frac{1}{(q-1)}\right]_{(p,q-1)}^\complete\simeq 0\,.
		\end{align*}
	\end{lem}
	\begin{proof}
		We start with the first equivalence. Since both sides are $\Phi_p(q)$-complete, it will be enough to show the equivalence modulo~$\Phi_p(q)$. The same argument as in \cref{par:ConjugateFiltration} shows
		\begin{equation*}
			\bigl(\qdeRham_{R/A}^{(p)}\bigr)_p^\complete\left[\frac{ \fil_\Nn^i}{\Phi_p(q)^i}\ \middle|\ i\geqslant 0\right]/\Phi_p(q)\simeq \colimit\Bigl(\gr_\Nn^0\xrightarrow{\Phi_p(q)}\gr_\Nn^1\xrightarrow{\Phi_p(q)}\dotsb\Bigr)\,.
		\end{equation*}
		The divided Frobenius $\Phi_p(q)^{-i}\phi_{p/A[q]}$ maps $\gr_\Nn^i$ isomorphically onto $ \fil_i^\mathrm{conj}(\qdeRham_{R/A}/\Phi_p(q))$ (by \cite[Theorem~\chref{15.2}]{Prismatic} plus quasi-syntomic descent and animation to cover all animated $A$-algebras~$R$). Since the conjugate filtration is exhaustive, this shows the first of the two claimed equivalences.
		
		For the second equivalence, note that the inclusion of the diagonal into any $\IZ_{\geqslant 0}\times \IZ_{\geqslant 0}$-shaped diagram is coinitial. Therefore, we can write
		\begin{equation*}
			\bigl(\qdeRham_{R/A}^{(p)}\bigr)_p^\complete\left[\frac{ \fil_\Nn^i}{(q^p-1)^i}\ \middle|\ i\geqslant 0\right]\simeq \colimit\left(\begin{tikzcd}
				\fil_\Nn^0\dar["(q-1)"]\rar["\Phi_p(q)"] &  \fil_\Nn^1\dar["(q-1)"]\rar["\Phi_p(q)"] & \dotsb \\
				\fil_\Nn^0\dar["(q-1)"]\rar["\Phi_p(q)"] &  \fil_\Nn^1\dar["(q-1)"]\rar["\Phi_p(q)"] & \dotsb \\
				\vdots & \vdots & \ddots
			\end{tikzcd}\right)
		\end{equation*}
		By the first equivalence, the $(p,q-1)$-completed colimit of every row in this diagram is $(\qdeRham_{R/A})_p^\complete$. If we then take the colimit in the vertical direction, the second equivalence follows and we're done
	\end{proof}
	
	\begin{numpar}[The Habiro descent]\label{con:HabiroDescent}
		Let $(R, \fil_{\qHodge}^\star \qdeRham_{R/A})$ be an object in $\cat{AniAlg}_A^{\qHodge}$. We define the \emph{Habiro--Hodge complex of $R$ over $A$} to be
		\begin{equation*}
			\qHhodge_{R/A}\coloneqq \limit_{m\in\IN}\qHhodge_{R/A,  m}\,.
		\end{equation*}
		The same argument as in the proof of \cref{lem:qHodgeSymmetricMonoidal} allows us to equip $\qHhodge_{-/A,  m}$ with a lax symmetric monoidal structure for all $m\in\IN$; thanks to \cref{con:TwistedqHodgeFiltrationFunctorial}, the equivalences from \cref{prop:HabiroDescent} will be compatible with this lax symmetric monoidal structure. It follows that there's a diagram of lax symmetric monoidal functors
		\begin{equation*}
			\begin{tikzcd}[column sep=huge]
				& \widehat{\Dd}_\Hh\bigl(A[q]\bigr)\dar["(-)_{(q-1)}^\complete"]\\
				\cat{AniAlg}_A^{\qHodge}\rar["\qHodge_{-/A}"']\urar[dashed,"\qHhodge_{-/A}"] & \widehat{\Dd}_{(q-1)}\bigl(A[q]\bigr)
			\end{tikzcd}
		\end{equation*}
	\end{numpar}
	\begin{lem}\label{lem:HabiroDescentSymmetricMonoidal}
		The lax symmetric monoidal functor $\qHhodge_{-/A}\colon \cat{AniAlg}_A^{\qHodge}\rightarrow \widehat{\Dd}_\Hh(A[q])$ is, in fact, symmetric monoidal.
	\end{lem}
	\begin{proof}
		It will be enough to show that for all $m\in\IN$ the functor
		\begin{equation*}
			\qHhodge_{-/A}/(q^m-1)\left[\bigl\{(q^d-1)^{-1}\}_{d\mid m,\, d\neq m}\right]
		\end{equation*}
		is symmetric monoidal. The same argument as in the proof of \cref{lem:qHodgeSymmetricMonoidal} allows us to equip the filtration $ \fil_\star ^{\qIW_m\Omega}(\qHhodge_{-/A}/(q^m-1))$ from \cref{thm:HabiroDescent}\cref{enum:qdRWComparison} with a lax symmetric monoidal structure. Symmetric monoidality can then be checked on the associated graded
		\begin{equation*}
			\gr_*^{\qIW_m\Omega}\bigl(\qHhodge_{-/A}/(q^m-1)\bigr)\simeq \Sigma^{-*}\qIW_m\deRham_{-/A}^*\simeq \gr_{\Hhodge_m}^*\qIW_m\deRham_{-/A}\,.
		\end{equation*}
		Thus, it would be enough to show that $ \fil_{\Hhodge_m}^\star \qIW_m\deRham_{-/A}$ is symmetric monoidal. This is not true on the nose. However, once we invert $(q^d-1)$ for all divisors $d\mid m$, $d\neq m$, we claim that the first ghost map
		\begin{equation*}
			\gh_1\colon  \fil_{\Hhodge_m}^\star \qIW_m\deRham_{-/A}\overset{\simeq}{\longrightarrow} \fil_{\Hodge}^\star \deRham_{-/A}\lotimes_{A,\psi^m}A[\zeta_m]
		\end{equation*}
		becomes an equivalence. If we can show this, we're done, since the Hodge filtration $ \fil_{\Hodge}^\star \deRham_{-/A}$ is symmetric monoidal.
		
		To prove this claim, observe that for any ordinary $R$-algebra $A$ and any $d\mid m$, $d\neq m$, the $q$-de Rham--Witt complex $\qIW_d\Omega_{R/A}^*$ is $(q^d-1)$-torsion and so it dies after inverting $(q^d-1)$. With this observation, a simple comparison of universal properties (compare the argument in \cite[Lemma~\chref{4.5}]{qWitt}) shows that
		\begin{equation*}
			\gh_1\colon \qIW_m\Omega_{R/A}^*\left[\bigl\{(q^d-1)^{-1}\}_{d\mid m,\, d\neq m}\right]\overset{\cong}{\longrightarrow} \Omega_{R/A}^*\otimes_{A,\psi^m}A\left[\zeta_m,\bigl\{(\zeta_m^d-1)^{-1}\}_{d\mid m,\, d\neq m}\right]
		\end{equation*}
		is an isomorphism of complexes. In particular, it induces an isomorphism on stupid filtrations. By passing to animations, the above claim about $\gh_1$ follows and so we're done.
	\end{proof}
	At this point, we've assembled all the ingredients to carry out the proof of \cref{thm:HabiroDescent} as outlined at the end of \cref{subsec:MainResult}, and so the proof is finally finished.
	
	\subsection{Habiro descent for \texorpdfstring{$q$}{q}-de Rham complexes}\label{subsec:qdeRhamHabiroDescent}
	
	In this short subsection, we discuss to what extent the $q$-de Rham complex $\qdeRham_{R/A}$ can or cannot be descended to the Habiro ring. Let's start with the case that works.
	\begin{prop}\label{prop:Letaq-1}
		Let $(S, \fil_{\qHodge}^\star\qdeRham_{S/A})\in \cat{AniAlg}_A^{\qHodge}$ be an object such that $S$ is a smooth $A$-algebra. Then:
		\begin{alphanumerate}
			\item $\qOmega_{S/A}\simeq \qhatdeRham_{S/A}$ is the completion of $\qdeRham_{S/A}$ at the $q$-Hodge filtration $ \fil_{\qHodge}^\star$.\label{enum:qHodgeCompletion}
			\item We have $\L\eta_{(q-1)}\qHodge_{S/A}\simeq \qOmega_{R/A}$. In particular, $\L\eta_{(q-1)}\qHhodge_{S/A}$ is a Habiro descent of $\qOmega_{S/A}$.\label{enum:Letaq-1}
		\end{alphanumerate}
	\end{prop}
	\begin{proof}
		We construct the equivalence $\qOmega_{S/A}\simeq \qhatdeRham_{S/A}$ using an arithmetic fracture square. Let us first construct an equivalence after $p$-completion for any prime~$p$. Note that the canonical maps $\qdeRham_{S/A}\rightarrow \qOmega_{S/A}$ and $\qdeRham_{S/A}\rightarrow \qhatdeRham_{S/A}$ become equivalences after $p$-completion for any prime~$p$. Indeed, this can be checked modulo $(q-1)$, where we recover the well-known fact $(\Omega_{S/A})_p^\complete\simeq (\deRham_{S/A})_p^\complete\simeq (\hatdeRham_{S/A})_p^\complete$. So we obtain the desired equivalence $(\qOmega_{S/A})_p^\complete\simeq (\qhatdeRham_{S/A})_p^\complete$.
		
		Let us now construct the equivalence rationally. We know that $\deRham_{S/A}\lotimes_\IZ\IQ\rightarrow \Omega_{S/A}\lotimes_\IZ\IQ$ identifies the right-hand side with the completion of the left-hand side at the Hodge filtration. Consequently, $(\deRham_{S/A}\lotimes_\IZ\IQ)\qpower_{(\Hodge,q-1)}^\complete\simeq (\Omega_{S/A}\lotimes_\IZ\IQ)\qpower$, which yields the desired equivalence rationally. The data from \cref{def:qHodgeFiltration}\cref{enum:qHodgeRationalpComplete} ensures that the $p$-complete and rational equivalences glue, which finishes the proof of \cref{enum:qHodgeCompletion}.
		
		To prove~\cref{enum:Letaq-1}, first observe that the natural map $\qdeRham_{S/A}\rightarrow \qHodge_{S/A}$ factors through the completion at the $q$-Hodge filtration, because each filtration step $ \fil_{\qHodge}^i$ becomes divisible by $(q-1)^i$ in $\qHodge_{S/A}$ and $\qHodge_{S/A}$ is $(q-1)$-complete. Now consider the map of filtered objects
		\begin{equation*}
			\begin{tikzcd}[column sep=2em]
				\dotsb & \qhatdeRham_{S/A}\dar["(q-1)^2"]\eqar[l] & \qhatdeRham_{S/A}\dar["(q-1)"]\eqar[l] &  \fil_{\qHodge}^0\qhatdeRham_{S/A}\dar\lar["\simeq"'] &  \fil_{\qHodge}^1\qhatdeRham_{S/A}\dar\lar & \dotsb\lar\\
				\dotsb &\qHodge_{S/A}\lar & \qHodge_{S/A}\lar["(q-1)"'] & \qHodge_{S/A}\lar["(q-1)"'] & \qHodge_{S/A}\lar["(q-1)"'] & \dotsb\lar
			\end{tikzcd}
		\end{equation*}
		We claim that the top row is the connective cover of the bottom row in the Beilinson $t$-structure. If we can prove this, then \cite[Proposition~\chref{5.8}]{BMS2} will show $\qhatdeRham_{S/A}\simeq \L\eta_{(q-1)}\qHodge_{S/A}$, hence $\qOmega_{S/A}\simeq \L\eta_{(q-1)}\qHodge_{S/A}$ by \cref{enum:qHodgeCompletion}, as desired. Since $\L\eta_{(q-1)}$ commutes with $(q-1)$-completion \cite[Lemma~\chref{6.20}]{BMS1}, we also deduce that $\L\eta_{(q-1)}\qHhodge_{S/A}$ is indeed a Habiro descent of $\qOmega_{S/A}$.
		
		To show the claim, let us first verify that the top row is indeed connective in the Beilinson $t$-structure. We must show that $\gr_{\qHodge}^n\qdeRham_{S/A}$ is concentrated in cohomological degrees $\leqslant n$ for all $n$. If $n<0$, this is clear as then $\gr_{\qHodge}^n\qdeRham_{S/A}\simeq 0$. If $n\geqslant 0$, we have a finite-length filtration
		\begin{equation*}
			0\longrightarrow \gr^0_{\qHodge}\qdeRham_{S/A}\xrightarrow{(q-1)}\gr^1_{\qHodge}\qdeRham_{S/A}\xrightarrow{(q-1)}\dotsb \xrightarrow{(q-1)}\gr^n_{\qHodge}\qdeRham_{S/A}\,.
		\end{equation*}
		The $i$\textsuperscript{th} graded piece of this filtration is $\Sigma^{-i}\Omega_{S/A}^i$ by \cref{lem:ConjugateFiltration}, which is concentrated in cohomological degree~$i$. Hence $\gr_{\qHodge}^n\qdeRham_{S/A}$ is indeed concentrated in cohomological degrees~$\leqslant n$ and so the top row is Beilinson-connective.
		
		Moreover, the argument shows that $\gr_{\qHodge}^n\qdeRham_{S/A}\rightarrow \qHodge_{S/A}/(q-1)$ induces an equivalence $\gr_{\qHodge}^n\qdeRham_{S/A}\simeq \tau^{\leqslant n}(\qHodge_{S/A}/(q-1))$. By \cite[Theorem~\chref{5.4}(2)]{BMS2}, this shows that the map from the top row to the Beilinson-connective cover of the bottom row is an equivalence on associated gradeds. As both filtered objects are complete, we're done.
	\end{proof}
	Let us now discuss what probably doesn't work.
	
	\begin{rem}
		For an arbitrary object $(R,\fil_{\qHodge}^\star \qdeRham_{R/A})$ in $\cat{AniAlg}_A^{\qHodge}$, without a smoothness assumption on~$R$, we don't know how to construct a Habiro descent of $\qdeRham_{R/A}$. A naive guess would be
		\begin{equation*}
			\limit_{m\in\IN}\qdeRham_{R/A}^{(m)}\left[\frac{ \fil_{\qHhodge_m}^i}{[m]_q^i}\ \middle|\ i\geqslant 0\right]_{(q^m-1)}^\complete\,,
		\end{equation*}
		but this object doesn't exist, since $ \fil_{\qHhodge_m}^\star\qdeRham_{R/A}^{(m)}$ is only a filtered module over the filtered ring $(q^m-1)^\star A[q]$, but not necessarily over $[m]_q^\star A[q]$.
	\end{rem}
	
	\begin{rem}
		We also don't expect that the Habiro descent of $\qOmega_{S/A}$ in \cref{prop:Letaq-1} can be constructed without the datum of a $q$-Hodge filtration $ \fil_{\qHodge}^\star\qdeRham_{S/A}$, let alone functorially in~$S$. While it seems hard to get any definite no-go theorem, let us at least explain why the most natural attempt doesn't work.
		
		In \cref{rem:qOmegaHabiroDescent}, we've explained an attempt to construct a filtration $ \fil_{\L\eta}^\star \qOmega_{S/A}^{(m)}$: Each $\L\eta_{[m/d]_q}$ carries a natural filtration via \cite[Proposition~\chref{5.8}]{BMS2}. If these filtrations could be glued to give the desired $ \fil_{\L\eta}^\star $, we could attempt to construct a Habiro descent of $\qOmega_{S/A}$ via
		\begin{equation*}
			\limit_{m\in\IN}\qOmega_{S/A}^{(m)}\left[\frac{ \fil_{\L\eta}^i}{[m]_q^i}\ \middle|\ i\geqslant 0\right]_{(q^m-1)}^\complete\,.
		\end{equation*}
		However, the filtrations on $\L\eta_{[m/d]_q}$ \emph{do not} glue. This can already be seen in the case $m=p$. In this case we have a pullback diagram
		\begin{equation*}
			\begin{tikzcd}
				\qOmega_{S/A}^{(p)}\rar\dar\drar[pullback] & \left(\qOmega_{S/A}\lotimes_{A[q],\psi^p}A[q]\right)_{[p]_q}^\complete\dar["\phi_{p/A[q]}"]\\
				\L\eta_{[p]_q}\qOmega_{S/A}\rar & \L\eta_{[p]_q}\left(\qOmega_{S/A}\right)_p^\complete
			\end{tikzcd}
		\end{equation*}
		The filtration on $\L\eta_{[1]_q}\simeq \id$ is trivial. But the trivial filtration on $(\qOmega_{S/A}\lotimes_{A[q],\psi^p}A[q])_{[p]_q}^\complete$ will not be compatible with the natural filtration on $\L\eta_{[p]_q}(\qOmega_{S/A})_p^\complete$, so gluing fails.
		
		To make the gluing work, we should instead equip $(\qOmega_{S/A}\lotimes_{A[q],\psi^p}A[q])_{[p]_q}^\complete$ with a global version of the Nygaard filtration. But such a global Nygaard filtration likely doesn't exist. To see this, let's attempt to construct it via an arithmetic fracture square. On the $p$-completion $(\qOmega_{S/A}\lotimes_{A[q],\psi^p}A[q])_{(p,[p]_q)}^\complete$, we put the usual Nygaard filtration. In view of \cref{lem:NygaardRationalisation}, on the rationalisation we should put the combined Hodge and $[p]_q$-adic filtration. But then on the $\ell$-completion $(\qOmega_{S/A}\lotimes_{A[q],\psi^p}A[q])_{(\ell,[p]_q)}^\complete$ for any prime $\ell\neq p$, we would need to put a filtration that becomes the combined Hodge and $[p]_q$-adic filtration after $(-)[1/\ell]_{[p]_q}^\complete$.
		
		It is entirely unclear (at least to the author) how to construct such a filtration, unless we're already given a $q$-Hodge filtration $ \fil_{\qHodge}^\star\qdeRham_{S/A}$. This explains the need for the additional datum of $ \fil_{\qHodge}^\star\qdeRham_{S/A}$.
	\end{rem}

	\subsection{Habiro descent in derived commutative algebras}
	Raksit \cite{RaksitFilteredCircle} has introduced $\infty$-categories of \emph{derived commutative algebras}, along with filtered, graded, and differential-graded variants. The filtrations $ \fil_{\Hhodge_m}^\star \qIW_m\deRham_{R/A}$ admit canonical filtered derived commutative $A[q]/(q^m-1)$-algebra structures (if $R$ is a polynomial $A$-algebra, these structures can be constructed on the level of complexes, then one can pass to animations) and the $\IE_\infty$-structure on the derived $q$-de Rham complex $\qdeRham_{R/A}$ can be canonically enhanced to a derived commutative $A[q]$-algebra structure (see \cref{par:DerivedCommutativeLift}).
	
	In this subsection, we sketch how \cref{thm:HabiroDescent} can be made compatible with these derived commutative structures. As a warm-up, let us instead consider $\IE_n$-monoidal structures for some $0\leqslant n\leqslant \infty$.
	\begin{numpar}[$\IE_n$-monoidal upgrade.]\label{par:EnMonoidalUpgrade}
		Let $(R, \fil_{\qHodge}^\star \qdeRham_{R/A})\in \cat{AniAlg}_A^{\qHodge}$. Suppose that the $q$-Hodge filtration $ \fil_{\qHodge}^\star \qdeRham_{R/A}$ can be equipped with the structure of an $\IE_n$-algebra in filtered $(q-1)^\star A[q]$-modules, compatible with the $\IE_\infty$-$A[q]$-algebra structure on $\qdeRham_{R/A}$. Suppose furthermore that the data from \cref{def:qHodgeFiltration}\cref{enum:qHodgeFiltrationOnqdR}--\cref{enum:qHodgeRationalpComplete} can be made compatible with this $\IE_n$-structure. Then $(R, \fil_{\qHodge}^\star \qdeRham_{R/A})$ becomes an $\IE_n$-algebra in $\cat{AniAlg}_A^{\qHodge}$.
		
		By the symmetric monoidality statement in \cref{thm:HabiroDescent}\cref{enum:HabiroDescent}, we can conclude that the Habiro--Hodge complex $\qHhodge_{R/A}$ becomes an $\IE_n$-algebra in $\widehat{\Dd}_\Hh(A[q])$. Similarly, the lax symmetric monoidality statements in \cref{thm:HabiroDescent}\cref{enum:qdRWComparison} show that $ \fil_\star ^{\qIW_m\Omega}(\qHhodge_{R/A}/(q^m-1))$ becomes a filtered $\IE_n$-algebra and the identification of its associated graded
		\begin{equation*}
			\gr_*^{\qIW_m\Omega}\bigl(\qHhodge_{R/A}/(q^m-1)\bigr)\simeq \Sigma^{-*}\qIW_m\deRham_{R/A}^*\simeq \gr_{\Hhodge_m}^*\qIW_m\deRham_{R/A}
		\end{equation*}
		becomes a graded $\IE_n$-monoidal equivalence.
	\end{numpar}
	
	\begin{numpar}[Derived commutative upgrade I.]\label{par:qHhodgeDAlg}
		Similar to \cref{par:EnMonoidalUpgrade}, suppose that $ \fil_{\qHodge}^\star \qdeRham_{R/A}$ can be equipped with the structure of a \emph{filtered derived commutative algebra} over $(q-1)^\star A[q]$, that is, an element in the slice $\infty$-category $( \Fil\DAlg_{A[q]})_{(q-1)^\star A[q]/}$, where $ \Fil\DAlg_{A[q]}$ is Raksit's $\infty$-category of filtered derived commutative $A[q]$-algebras \cite[Definition~\chref{4.3.4}]{RaksitFilteredCircle}. Suppose furthermore that this derived commutative structure is compatible with the derived commutative $A[q]$-algebra structure on $\qdeRham_{R/A}$ (see \cref{par:DerivedCommutativeLift}) and that the data from \cref{def:qHodgeFiltration}\cref{enum:qHodgeFiltrationOnqdR}--\cref{enum:qHodgeRationalpComplete} can be made compatible with the filtered derived commutative algebra structures everywhere.
		
		For example, this can be done in the special cases from \cref{exm:HabiroDescentCoordinateCase} above and \cref{con:GlobalqHodge} below. In the former case, we'll verify this in \cite[Remark~\chref{6.27}]{qdeRhamku}, in the latter case see \cref{rem:qHodgeQuasiregularDAlg}.
	\end{numpar}
	\begin{lem}
		In the situation of \cref{par:qHhodgeDAlg}, $\qHhodge_{R/A}$ admits a canonical derived commutative $A[q]$-algebra structure. Furthermore, for all $m\in\IN$, $ \fil_\star^{\qIW_m\Omega}(\qHhodge_{R/A}/(q^m-1))$ admits a filtered derived commutative $A[q]/(q^m-1)$-algebra structure, compatible with the derived commutative $A[q]/(q^m-1)$-algebra structure on $\qHhodge_{R/A}/(q^m-1)$, and the equivalence
		\begin{equation*}
			\gr_*^{\qIW_m\Omega}\bigl(\qHhodge_{R/A}/(q^m-1)\bigr)\simeq \Sigma^{-*}\qIW_m\deRham_{R/A}^*\simeq \gr_{\Hhodge_m}^*\qIW_m\deRham_{R/A}
		\end{equation*}
		from \cref{thm:HabiroDescent}\cref{enum:qdRWComparison} is an equivalence of graded derived commutative $A[q]/(q^m-1)$-algebras.
	\end{lem}
	\begin{proof}[Proof sketch]
		First note that our results about the Nygaard filtration, specifically \cref{prop:NygaardComparison} and \cref{lem:NygaardRationalisation}, also hold true as equivalences of filtered derived commutative algebras, since the proofs work in this setting as well. By tracing through \cref{con:TwistedqHodgeFiltrationpAdic}--\cref{con:TwistedqHodgeFiltrationFunctorial}, we now see that each $ \fil_{\qHhodge_m}^\star \qdeRham_{R/A}^{(m)}$ acquires a filtered derived commutative algebra structure over $(q^m-1)^\star A[q]$, and that the transition maps in \cref{con:TwistedqHodgeFiltrationFunctorial} are compatible with these structures.
		
		The construction from~\cref{con:HabiroDescent0} produces a canonical derived commutative algebra structure on $\Hh_{R/A,  m}$, because we can view the construction as a filtered localisation followed by restriction to filtered degree~$0$; compare \cref{par:ConjugateFiltration}. It is then clear from \cref{con:HabiroDescent} that $\qHhodge_{R/A}$ acquires a derived commutative $A[q]$-algebra structure. Moreover, since the filtration $ \fil_\star^{\qIW_m\Omega}(\qHhodge_{R/A}/(q^m-1))$ and the identification of its associated graded were constructed in a completely way (see the proofs of \cref{lem:ConjugateFiltration,lem:FiltrationAbstract}), they will also work on the level of derived commutative algebras. The only input this needs is that \cref{prop:TwistedqHodgeFiltrationqDeforms} holds as an equivalence of filtered derived commutative $A[q]/(q^m-1)$-algebras, which is again apparent from the constructions.
	\end{proof}
	
	But there's one more piece of structure.
	
	\begin{numpar}[Derived commutative upgrade II.]
		Since $\qIW_m\Omega_{-/A}^*$ is a functor with values in commutative differential-graded $A[q]/(q^m-1)$-algebras, we see that its animation $\Sigma^{-*}\qIW_m\deRham_{-/A}^*\simeq\gr_{\Hhodge_m}^*\qIW_m\deRham_{-/A}$ upgrades to a functor with values in Raksit's $\infty$-category $\cat{DG}_-\DAlg_{A[q]/(q^m-1)}$ of derived differential-graded $A[q]/(q^m-1)$-algebras \cite[Definition~\chref{5.1.10}]{RaksitFilteredCircle}.
		
		By transfer of structure, the associated graded $\gr_*^{\qIW_m\Omega}(\qHhodge_{R/A}/(q^m-1))$ becomes an element in $\cat{DG}_-\DAlg_{A[q]/(q^m-1)}$ as well. Via the following corollary, we can figure out what the differentials are, at least in the case where $R$ is smooth over $A$.
	\end{numpar}
	\begin{cor}\label{cor:CDGA}
		Let $(S, \fil_{\qHodge}^\star\qdeRham_{S/A})\in \cat{AniAlg}_A^{\qHodge}$ be an object such that $S$ is smooth over~$A$. Then:
		\begin{alphanumerate}
			\item  $ \fil_\star^{\qIW_m\Omega}(\qHhodge_{S/A}/(q^m-1))$ is the Whitehead filtration $\tau_{\geqslant \star}(\qHhodge_{S/A}/(q^m-1))$.\label{enum:qWOmegaFiltrationPostnikov}
			\item The equivalence from \cref{thm:HabiroDescent}\cref{enum:qdRWComparison} becomes an isomorphism of graded $A[q]/(q^m-1)$-modules \label{enum:CohomologyOfHabiro}
			\begin{equation*}
				\H^*\bigl(\qHhodge_{S/A}/(q^m-1)\bigr)\cong \qIW_m\Omega_{S/A}^*
			\end{equation*}
			\embrace{and an isomorphism of graded $A[q]/(q^m-1)$-algebras as soon as $(S, \fil_{\qHodge}^\star\qdeRham_{S/A})$ is at least an $\IE_1$-algebra in $\cat{AniAlg}_A^{\qHodge}$}.
			\item\label{enum:BocksteinDifferential} Under the isomorphism from \cref{enum:CohomologyOfHabiro}, the canonical differential on $\qIW_m\Omega_{S/A}^*$ corresponds to the Bockstein differential on $\H^*(\qHhodge_{S/A}/(q^m-1))$.
		\end{alphanumerate}
	\end{cor}
	\begin{proof}
		We've seen in \cref{cor:qDRWSmoothAnimation} that $\qIW_m\Omega_{S/A}^n\simeq \qIW_m\deRham_{S/A}^n$ for all $n$. It follows that each graded piece $\gr_n^{\qIW_m\Omega}(\qHhodge_{S/A}/(q^m-1))$ is concentrated in cohomological degree~$n$. Since $ \fil_\star^{\qIW_m\Omega}(\qHhodge_{S/A}/(q^m-1))$ is bounded below and thus complete, it has to be the Whitehead filtration. This shows~\cref{enum:qWOmegaFiltrationPostnikov} as well as the graded $A[q]/(q^m-1)$-module isomorphism from~\cref{enum:CohomologyOfHabiro}.  The isomorphism as graded $A[q]/(q^m-1)$-algebras follows from \cref{par:EnMonoidalUpgrade}.
		
		It remains to show~\cref{enum:BocksteinDifferential}. Similar to the proof of \cref{lem:ConjugateFiltration}, let us identify the filtered ring $(q^m-1)^\star A[q]$ with the graded ring $A[q,\beta,t_m]/(\beta t_m-(q^m-1))$, where $\abs{q}=0$, $\abs{\beta}=1$, and $\abs{t_m}=-1$.%
		\footnote{In \cite{qdeRhamku}, we'll recognise $\IZ[q,\beta,t_m]/(\beta t_m-(q^m-1))\cong \pi_*(\ku^{C_m})$.}
		The filtered structure comes from the $A[t_m]$-module structure. In particular, modding out~$t_m$ is the same as passing to the associated graded. Let us also regard the filtrations $ \fil_{\qHhodge_m}^\star\qdeRham_{S/A}^{(m)}$ and $ \fil_{\Hhodge_m}^\star\qIW_m\deRham_{S/A}$ as graded $A[q,\beta,t_m]/(\beta t_m-(q^m-1))$-modules $ \fil_{\qHhodge}^*$ and $ \fil_{\Hhodge}^*$. Finally, let us denote by $\beta^{-\star} A[\beta]$ the ascendingly filtered graded ring
		\begin{equation*}
			\beta^{-\star}A[\beta]\coloneqq \left(\dotsb\overset{\beta}{\longrightarrow}A[\beta](1)\overset{\beta}{\longrightarrow}A[\beta](0)\overset{\beta}{\longrightarrow}A[\beta](-1)\overset{\beta}{\longrightarrow}\dotsb\right)\,,
		\end{equation*}
		As explained in the proof of \cref{lem:FiltrationAbstract}, the filtration $ \fil_\star^{\qIW_m\Omega}(\qHhodge_{S/A}/(q^m-1))$ can be written as follows:
		\begin{equation*}
			\fil_\star^{\qIW_m\Omega}\bigl(\qHhodge_{S/A}/(q^m-1)\bigr)\simeq \bigl( \fil_{\qHhodge}^*/t_m\lotimes_{A[\beta]}\beta^{-\star}A[\beta]\bigr)_0
		\end{equation*}
		(note that $*$ on the right-hand side refers to the graded degree whereas $\star$ corresponds to the filtration degree). Now consider the Bockstein cofibre sequence for $ \fil_{\qHhodge}^*/t_m$. It fits into a commutative diagram of graded $A[q,\beta,t_m]/(\beta t_m-(q^m-1))$-modules
		\begin{equation*}
			\begin{tikzcd}
				\fil_{\qHhodge}^*(-1)/t_m\rar["t_m"] &  \fil_{\qHhodge}^*/t_m^2\rar &  \fil_{\qHhodge}^*/t_m\\
				\fil_{\qHhodge}^*/t_m\uar["\beta"]\urar["(q^m-1)"'] & &
			\end{tikzcd}
		\end{equation*}
		If we apply $(-\lotimes_{A[\beta]}A[\beta^{\pm 1}])_0$ to this diagram, the left vertical arrow becomes an equivalence and so the cofibre sequence from the top row will become equivalent to the Bockstein cofibre sequence
		\begin{equation*}
			\qHhodge_{R/A}/(q^m-1)\xrightarrow{(q^m-1)}\qHhodge_{R/A}/(q^m-1)^2\longrightarrow \qHhodge_{R/A}/(q^m-1)\,.
		\end{equation*}
		If we apply $(-\lotimes_{A[\beta]}\beta^{-\star}A[\beta])_0$ to the top row, we get a filtration on this cofibre sequence. By~\cref{enum:qWOmegaFiltrationPostnikov}, this filtration will be of the form
		\begin{equation*}
			\tau^{\leqslant \star+1}\bigl(\qHhodge_{S/A}/(q^m-1)\bigr)\longrightarrow\bigl( \fil_{\qHhodge}/t_m^2\lotimes_{A[\beta]}\beta^{-\star}A[\beta]\bigr)_0\longrightarrow\tau^{\leqslant \star }\bigl(\qHhodge_{S/A}/(q^m-1)\bigr)
		\end{equation*}
		where $\bigl( \fil_{\qHhodge}/t_m^2\lotimes_{A[\beta]}\beta^{-\star}A[\beta]\bigr)_0$ is an ascending filtration on $\qHhodge_{S/A}/(q^m-1)^2$ that lies between $\tau^{\leqslant \star}$ and $\tau^{\leqslant \star +1}$. After passing to associated gradeds, the connecting morphism will then necessarily be the usual Bockstein differential
		\begin{equation*}
			\H^*\bigl(\qHhodge_{S/A}/(q^m-1)\bigr)\longrightarrow \H^{*+1}\bigl(\qHhodge_{S/A}/(q^m-1)\bigr)\,.
		\end{equation*}
		On the other hand, the associated graded of $\beta^{-\star}A[\beta]$ is given by $\bigoplus_{i\in\IZ}A(-i)$. If we apply $(-\lotimes_{A[\beta]}\bigoplus_{i\in\IZ}A(-i))_0$ to the top row of the diagram, we get the Bockstein cofibre sequence
		\begin{equation*}
			\fil_{\Hhodge}^*(-1)/t_m\overset{t_m}{\longrightarrow} \fil_{\Hhodge}^*/t_m^2\longrightarrow \fil_{\Hhodge}^*/t_m\,,
		\end{equation*}
		because $ \fil_{\Hhodge}^*\simeq  \fil_{\qHhodge}^*/\beta$ by \cref{prop:TwistedqHodgeFiltrationqDeforms}. 
		Since $ \fil_{\Hhodge}^\star\qIW_m\deRham_{R/A}$ is the stupid filtration on the complex $\qIW_m\Omega_{R/A}^*$, the differential of $\qIW_m\Omega_{R/A}^*$ agrees with the connecting morphism for the Bockstein cofibre sequence of $ \fil_{\Hhodge}^*/t_m$. This finishes the proof of \cref{enum:BocksteinDifferential}.
	\end{proof}

	%
	%
	%
	
	\newpage
	
	\section{Functorial \texorpdfstring{$q$}{q}-Hodge filtrations}\label{sec:FunctorialqHodge}
	Fix a perfectly covered $\Lambda$-ring $A$. We've seen in \cref{lem:NoFunctorialqHodgeFiltration} that it's impossible to get a functorial $q$-Hodge filtration for all animated $A$-algebras, or even just for smooth $A$-algebras. Despite this general no-go result, we'll see in this section that functorial $q$-Hodge filtrations exist for fairly large full subcategories of $\cat{AniAlg}_A$.
	
	A further ample source of examples comes from homotopy theory and will be discussed at length in the companion paper \cite{qdeRhamku}.
	
	\subsection{Functorial \texorpdfstring{$q$}{q}-Hodge filtrations away from small primes}\label{subsec:CanonicalqHodgeSmooth}
	
	In this subsection, we'll give an elementary construction of a functorial $q$-Hodge filtration on certain smooth $A$-algebras. In the introduction (see \cref{par:qHodgeForCurves}), we've already explained the idea in the case of relative dimension~$\leqslant1$. The general case follows the same simple idea.
	%
	%
	%
	\begin{numpar}[Canonical $q$-Hodge filtrations I.]\label{par:CanonicalqHodgeSmoothI}
		Let $S$ be smooth of arbitrary dimension over $A$ and let $n$ be a positive integer such that all primes $p\leqslant n$ are invertible in $S$. This assumption ensures that the canonical map $\qOmega_{S/A}\rightarrow \Omega_{S/A}$ factors through an $\IE_\infty$-$A\qpower$-algebra map
		\begin{equation*}
			\qOmega_{S/A}\longrightarrow \Omega_{S/A}\qpower/(q-1)^n\,.
		\end{equation*}
		Indeed, by construction of the global $q$-de Rham complex (see \cref{con:GlobalqDeRham}), it's enough to check this after completion at any prime~$p$. In general, $(\qOmega_{S/A})_p^\complete\rightarrow (\Omega_{S/A})_p^\complete$ factors through $(\qOmega_{S/A})_p^\complete\rightarrow (\Omega_{S/A})_p^\complete\qpower/(q-1)^{p-1}$ by \cref{lem:RationalisationTechnicalII}. For primes $p>n$, this does what we want. For $p\leqslant n$, our assumption on $S$ ensures that $(\qOmega_{S/A})_p^\complete$ vanishes, so this case is fine too.
		
		Let us now equip $\Omega_{S/A}\qpower/(q-1)^n$ with the following filtration: We first define $\fil_{(\Hodge,q-1)}^{\star} \Omega_{S/A}\qpower\coloneqq (\fil_{\Hodge}^\star \Omega_{S/A}\lotimes_\IZ(q-1)^\star\IZ\qpower)_{(q-1)}^\complete$ to be the combined Hodge and $(q-1)$-adic filtration, as usual. We then let $\fil_{(\Hodge,q-1)}^{\star} \Omega_{S/A}\qpower/(q-1)^n$ denotes its reduction modulo~$(q-1)^n$, which we regard as an element in filtration degree~$n$.%
		\footnote{Said differently, we wish to equip $\IZ\qpower/(q-1)^n$ with the finite filtration given by $(q-1)^i\IZ\qpower/(q-1)^n$ in degree~$i$. This is \emph{not} the $(q-1)$-adic filtration in our sense, since the latter would be $\IZ\qpower/(q-1)^n$ in every degree, with transition maps given by multiplication by $(q-1)$.}
		We may then form the following pullback of filtered objects in degrees $\leqslant n$:
		\begin{equation*}
			\begin{tikzcd}
				\fil_{\qHodge,n}^{\star \leqslant n}\qOmega_{S/A}\rar\dar\drar[pullback] & \qOmega_{S/A}\dar\\
				\fil_{(\Hodge,q-1)}^{\star\leqslant n} \Omega_{S/A}\qpower/(q-1)^n\rar & \Omega_{S/A}\qpower/(q-1)^n
			\end{tikzcd}
		\end{equation*}
		Here $\fil_{(\Hodge,q-1)}^{\star\leqslant n} \Omega_{S/A}\qpower/(q-1)^n$ denotes the restriction of the to degrees $\star\leqslant n$; more precisely, we apply the truncation functor $\tau_n^*$ from \cref{lem:TruncatedFilteredObjects} below.
		
		We then wish to extend $\fil_{\qHodge,n}^{\star \leqslant n}\qOmega_{S/A}$ to degrees $\star\geqslant n+1$. Intuitively, this should be done via the $(q-1)$-adic filtration $(q-1)^{\star -n}\fil_{\qHodge,n}^n\qOmega_{S/A}$ as in \cref{par:qHodgeForCurves}. To do this formally and make the resulting filtered $(q-1)^\star A\qpower$-module structure apparent, we need to show a technical lemma.
	\end{numpar}
	
	\begin{lem}\label{lem:TruncatedFilteredObjects}
		Let $\Fil^{\geqslant 0}\Dd(\IZ)$ denote the full sub-$\infty$-categories of filtered objects that are constant in filtration degrees $\star\leqslant 0$. Let $\Fil^{[0,n]}\Dd(\IZ)\subseteq \Fil^{\geqslant 0}\Dd(\IZ)$ denote the full sub-$\infty$-category of filtered objects that also vanish in filtration degree $\star\geqslant n+1$.
		\begin{alphanumerate}
			\item The inclusion  $\Fil^{[0,n]}\Dd(\IZ)\rightarrow\Fil^{\geqslant 0}\Dd(\IZ)$ has a left adjoint $\tau_n^*$, which on objects is given by replacing all filtration degrees $\star\geqslant n+1$ by $0$. Moreover, if $\fil^\star M$ and $\fil^\star N$ are filtered $\IZ$-modules, the canonical map\label{enum:TruncationSymmetricMonoidal}
			\begin{equation*}
				\tau_n^*\left(\fil^\star M\lotimes_{\IZ}\tau_n^*(\fil^\star N)\right)\overset{\simeq}{\longrightarrow}\tau_n^*\left(\fil^\star M\lotimes_{\IZ}\fil^\star N\right)
			\end{equation*}
			is an equivalence. Consequently there's a canonical way to equip $\Fil^{[0,n]}\Dd(\IZ)$ and the functor $\tau_n^*\colon \Fil^{\geqslant 0}\Dd(\IZ)\rightarrow \Fil^{[0,n]}\Dd(\IZ)$ with symmetric monoidal structures.
			\item For any filtered $\IE_\infty$-algebra $T\in\CAlg(\Fil^{\geqslant 0}\Dd(\IZ))$, the induced symmetric monoidal functor
			\begin{equation*}
				\tau_n^*\colon \Mod_{T}\bigl(\Fil^{\geqslant 0}\Dd(\IZ)\bigr)\longrightarrow \Mod_{\tau_n^*T}\bigl(\Fil^{[0,n]}\Dd(\IZ)\bigr)
			\end{equation*}
			admits an oplax symmetric monoidal left adjoint\label{enum:TruncationLeftAdjoint}
			\begin{equation*}
				\tau_{n,!}^T\colon  \Mod_{\tau_n^*T}\bigl(\Fil^{[0,n]}\Dd(\IZ)\bigr)\longrightarrow\Mod_{T}\bigl(\Fil^{\geqslant 0}\Dd(\IZ)\bigr)
			\end{equation*}
			\embrace{if $T$ is clear from the context, we'll often just write $\tau_{n,!}$}.
			\item Let $T_1\rightarrow T_2$ be any map in $\CAlg(\Fil^{\geqslant 0}\Dd(\IZ))$ and let $\fil^\star M\in \Mod_{\tau_n^*T}(\Fil^{[0,n]}\Dd(\IZ))$. Then there's a natural equivalence\label{enum:TruncationProjectionFormula}
			\begin{equation*}
				\tau_{n,!}^{T_2}\left(\fil^\star M\otimes_{\tau_n^*T_1}\tau_n^*T_2\right)\overset{\simeq}{\longrightarrow}\tau_{n,!}^{T_1}(\fil^\star M)\otimes_{T_1} T_2\,.
			\end{equation*}
		\end{alphanumerate}
	\end{lem}
	\begin{proof}
		We start with \cref{enum:TruncationSymmetricMonoidal}. It's straightforward to see that $\tau_n^*$ exists and is given as claimed. To show the equivalence, since $\tau_n^*$ and the inclusion preserve colimits, it will be enough to check the case where $\fil^\star M\simeq \IZ(i)$ and $\fil^\star N\simeq \IZ(j)$, where $i,j\geqslant 0$. If $j\leqslant n$, then $\tau_n^*\IZ(j)\rightarrow \IZ(j)$ is an equivalence and the claim is clear. If $j\geqslant n+1$, then we must check that $\tau_n^*\IZ(i+j)\rightarrow \tau_n^*\IZ(i+n)$ is an equivalence. This is clear as both sides are just $\IZ(n)$. The final claim in~\cref{enum:TruncationSymmetricMonoidal} is general abstract nonsense about symmetric monoidal structures on localisations (see \cite[Proposition~\chref{2.2.1.9}]{HA} for example).
		
		Let us now prove \cref{enum:TruncationLeftAdjoint} and \cref{enum:TruncationProjectionFormula} simultaneously. For any map $T_1\rightarrow T_2$ in $\CAlg(\Fil^{\geqslant 0}\Dd(\IZ_p))$, the diagram
		\begin{equation*}
			\begin{tikzcd}
				\Mod_{T_2}\bigl(\Fil^{\geqslant 0}\Dd(\IZ)\bigr)\rar["\tau_n^*"]\dar & \Mod_{\tau_n^*T_2}\bigl(\Fil^{[0,n]}\Dd(\IZ)\bigr)\dar\\
				\Mod_{T_1}\bigl(\Fil^{\geqslant 0}\Dd(\IZ)\bigr)\rar["\tau_n^*"] & \Mod_{\tau_n^*T_1}\bigl(\Fil^{[0,n]}\Dd(\IZ)\bigr)
			\end{tikzcd}
		\end{equation*}
		commutes. In the special case where $T_1=\IZ$ is the filtered tensor unit and $T_2=T$, this allows us to show that $\tau_n^*\colon \Mod_T(\Fil^{\geqslant 0}\Dd(\IZ))\rightarrow \Mod_{\tau^*T}(\Fil^{[0,n]}\Dd(\IZ))$ preserves all limits and colimits. Therefore the claimed left adjoint $\tau_{n,!}^T$ exists by Lurie's adjoint functor theorem. By abstract nonsense, $\tau_{n,!}^T$ will automatically acquire an oplax symmetric monoidal structure. This shows~\cref{enum:TruncationLeftAdjoint}. By passing to left adjoints in the diagram above, we immediately obtain~\cref{enum:TruncationProjectionFormula}.
	\end{proof}
	
	\begin{numpar}[Canonical $q$-Hodge filtrations II.]\label{par:CanonicalqHodgeSmoothII}
		We resume the discussion from \cref{par:CanonicalqHodgeSmoothI}. As we know now, the pullback defining $\fil_{\qHodge,n}^{\star\leqslant n}\qOmega_{S/A}$ can be taken in $\Mod_{\tau_n^*((q-1)^\star A\qpower)}(\Fil^{[0,n]}\Dd(\IZ))$. Applying the functor $\tau_{n,!}$ from \cref{lem:TruncatedFilteredObjects}\cref{enum:TruncationLeftAdjoint}, we obtain a filtered $(q-1)^\star A\qpower$-module
		\begin{equation*}
			\fil_{\qHodge,n}^\star\qOmega_{S/A}\coloneqq \tau_{n,!}\left(\fil_{\qHodge,n}^{\star\leqslant n}\qOmega_{S/A}\right)_{(q-1)}^\complete\,.
		\end{equation*}
		We can also take the pullback along $\qdeRham_{S/A}\rightarrow \qOmega_{S/A}$ to construct $\fil_{\qHodge,n}^\star\qdeRham_{S/A}$ (in order to be in line with \cref{def:qHodgeFiltration}).
	\end{numpar}

	\begin{rem}
		If $S$ satisfies the assumptions of \cref{par:CanonicalqHodgeSmoothII} and is additionally equipped with  an étale framing $\square\colon A[x_1,\dotsc,x_n]\rightarrow S$, then there exists an equivalence of filtered $(q-1)^\star A\qpower$-modules
		\begin{equation*}
			\fil_{\qHodge,n}^\star\qOmega_{S/A}\overset{\simeq}{\longrightarrow}\fil_{\qHodge,\square}^\star\qOmega_{S/A, \square}^*
		\end{equation*}
		between the $q$-Hodge filtration from \cref{par:CanonicalqHodgeSmoothII} and the one from~\cref{exm:HabiroDescentCoordinateCase}. Indeed, we observe $\fil_{\qHodge,n}^{\star\leqslant n}\qOmega_{S/A}\simeq \tau_n^*(\fil_{\qHodge,\square}^\star\qOmega_{S/A, \square}^*)$, since both sides fit into the same pullback diagram by construction. Since $\tau_{n,!}$ was defined as a the left adjoint of $\tau_n^*$, we obtain the map above. To see that it is an equivalence, we may reduce modulo~$(q-1)$, where we get the identity on $\fil_{\Hodge}^\star\Omega_{S/A}$ by inspection and \cref{lem:CanonicalqHodgeSmooth} below.
	\end{rem}
	
	\begin{rem}\label{rem:pTildeDeRham}
		Here's another way to do the construction from \cref{par:CanonicalqHodgeSmoothI} and~\cref{par:CanonicalqHodgeSmoothII}. Fix a prime~$p$. Recall that Bhatt--Lurie \cite[Construction~\chref{4.8.3}]{BhattLurieI} have defined a \emph{$\widetilde{p}$-de Rham complex} $\pOmega_{\smash{\widehat{S}}_p/\smash{\widehat{A}}_p}$. Explicitly, it is the homotopy-fixed points of the action of $\mu_{p-1}\subseteq \IZ_p^\times$ on $(\qOmega_{S/A})_p^\complete$. Here $u\in\IZ_p^\times$ acts on the prism $(\IZ_p\qpower,[p]_q)$ via $q\mapsto q^u$, which induces an action of $\IZ_p^\times$ on $(\qOmega_{S/A})_p^\complete$ via the comparison with prismatic cohomology (\cref{thm:qDeRhamGlobal}\cref{enum:GlobalqDeRhamPrismatic}).
		
		We can then define $\fil_{\pHodge,n}^{\star\leqslant n}\pOmega_{\smash{\widehat{S}}_p/\smash{\widehat{A}}_p}$ as the pullback of the Hodge filtration along the canonical map $\pOmega_{\smash{\widehat{S}}_p/\smash{\widehat{A}}_p}\rightarrow (\Omega_{S/A})_p^\complete$ (no combined Hodge and $(q-1)$-adic filtration is needed here), extend via $\tau_{n,!}$, and then finally base change to $(q-1)^\star\IZ_p\qpower$ to define a $p$-completed $q$-Hodge filtration $\fil_{\qHodge,n}^\star(\qOmega_{S/A})_p^\complete$.
		
		These filtrations for all~$p$ can be glued with the combined Hodge and $(q-1)$-adic filtration on $(\Omega_{S/A}\lotimes_\IZ\IQ)\qpower$ to get the same filtration $\fil_{\qHodge,n}^\star\qOmega_{S/A}$ as in \cref{par:CanonicalqHodgeSmoothII}. We prefer the construction in \cref{par:CanonicalqHodgeSmoothII}, since spelling out the gluing argument is a bit of a pain.
	\end{rem}
	
	
	\begin{lem}\label{lem:CanonicalqHodgeSmooth}
		With notation as in \cref{par:CanonicalqHodgeSmoothI}, assume additionally that $\dim(S/A)\leqslant n$. Then  $\fil_{\qHodge,n}^\star\qdeRham_{S/A}$ can naturally be equipped with the structure of a $q$-Hodge filtration as in \cref{def:qHodgeFiltration}.
	\end{lem}
	\begin{proof}
		In the following, we'll regard $(q-1)$ as sitting in filtration degree~$1$, as per Convention~\cref{conv:QuotientConvention}. We first compute
		\begin{align*}
			\fil_{\qHodge,n}^\star\qOmega_{S/A}/(q-1)&\simeq \tau_{n,!}^{\IZ}\left(\fil_{\qHodge,n}^{\star\leqslant n}\qOmega_{S/A}\lotimes_{\tau_n^*((q-1)^\star\IZ\qpower)}\IZ\right)\\
			&\simeq \tau_{n,!}^\IZ\tau_n^*(\fil_{\Hodge}^{\star}\Omega_{S/A})\\
			&\simeq \fil_{\Hodge}^{\star}\Omega_{S/A}\,.
		\end{align*}
		In the first equivalence we apply \cref{lem:TruncatedFilteredObjects}\cref{enum:TruncationProjectionFormula} to $(q-1)^\star\IZ\qpower\rightarrow \IZ$. The second equivalence follows by construction. To see the third equivalence, first observe that the Hodge filtration $\fil_{\Hodge}^\star\Omega_{S/A}$ is already contained in $\Fil^{[0,n]}\Dd(\IZ)$ because we assume $\dim(S/A)\leqslant n$. Since the right adjoint of $\tau_n^*\colon \Fil^{\geqslant 0}\Dd(\IZ)\rightarrow \Fil^{[0,n]}\Dd(\IZ)$ is fully faithful, so is the left adjoint $\tau_{n,!}^\IZ$, which yields the third equivalence. Similarly,
		\begin{align*}
			\bigl(\fil_{\qHodge,n}^\star\qOmega_{S/A}\lotimes_\IZ\IQ\bigr)_{(q-1)}^\complete&\simeq \tau_{n,!}\left(\bigl(\fil_{\qHodge,n}^{\star\leqslant n}\qOmega_{S/A}\lotimes_\IZ\IQ\bigr)_{(q-1)}^\complete\right)\\
			&\simeq \tau_{n,!}\left(\bigl(\fil_{\Hodge}^\star\Omega_{S/A}\lotimes_\IZ\tau_n^*\bigl((q-1)^\star\IQ\qpower\bigr)\bigr)_{(q-1)}^\complete\right)\\ &\simeq \fil_{(\Hodge,q-1)}^\star\bigl(\Omega_{S/A}\lotimes_\IZ\IQ\bigr)\qpower\,.
		\end{align*}
		The first equivalence is \cref{lem:TruncatedFilteredObjects}\cref{enum:TruncationProjectionFormula} applied to $\IZ\rightarrow \IQ$. For the second equivalence, we apply $(-\lotimes_\IZ\IQ)_{(q-1)}^\complete$ to the pullback defining $\fil_{\qHodge,n}^\star\qOmega_{S/A}$ in \cref{par:CanonicalqHodgeSmoothII} and use the fact that $(\qOmega_{S/A}\lotimes_\IZ\IQ)_{(q-1)}^\complete\simeq (\Omega_{S/A}\lotimes_\IZ\IQ)\qpower$. The third equivalence is \cref{lem:TruncatedFilteredObjects}\cref{enum:TruncationProjectionFormula} applied to $\IZ\rightarrow (q-1)^\star\IQ\qpower$.
		
		In a completely analogus way, we obtain natural equivalences
		\begin{equation*}
			\fil_{\qHodge,n}^\star\bigl(\qOmega_{S/A}\bigr)_p^\complete\bigl[\localise{p}\bigr]_{(q-1)}^\complete\overset{\simeq}{\longrightarrow}\fil_{(\Hodge,q-1)}^\star\bigl(\Omega_{S/A}\bigr)_p^\complete\bigl[\localise{p}\bigr]\qpower
		\end{equation*}
		for all primes~$p$. Via pullback along $\qdeRham_{S/A}\rightarrow\qOmega_{S/A}$, we obtain analogous equivalences for $\fil_{\qHodge,n}^\star \qdeRham_{S/A}$. The required compatibilities from \cref{def:qHodgeFiltration} can all be induced from those for $\qdeRham_{S/A}$, and so $\fil_{\qHodge,n}^\star\qdeRham_{S/A}$ can indeed be equipped with the structure of a $q$-Hodge filtration.
	\end{proof}
	\begin{lem}
		With assumptions as in \cref{lem:CanonicalqHodgeSmooth}, $\fil_{\qHodge,n}^\star\qOmega_{S/A}$ is automatically the completion of $\fil_{\qHodge,n}^\star\qdeRham_{S/A}$.
	\end{lem}
	
	\begin{proof}
		By \cref{prop:Letaq-1}, $\qOmega_{S/A}$ is automatically the completion of $\qdeRham_{S/A}$ at the filtration $\fil_{\qHodge,n}^\star\qdeRham_{S/A}$. Since $\fil_{\qHodge,n}^\star\qdeRham_{S/A}$ is defined as the pullback of $\fil_{\qHodge,n}^\star\qOmega_{S/A}$ along $\qdeRham_{S/A}\rightarrow \qOmega_{S/A}$, the desired assertion follows.
	\end{proof}
	
	We will now make the construction from \cref{par:CanonicalqHodgeSmoothII} functorial.
	
	\begin{numpar}[Functoriality across dimensions.]
		For all non-negative integers $n$ and $d$ let $\cat{Sm}_{A[n!^{-1}]}^{\leqslant d}$ be the category of all smooth $A$-algebras $S$ of relative dimension $\dim(S/A)\leqslant d$ such that all primes $p\leqslant n$ are invertible in $S$. Then~\cref{par:CanonicalqHodgeSmoothII} and \cref{lem:CanonicalqHodgeSmooth} provide us with a functor
		\begin{equation*}
			\bigl(-,\fil_{\qHodge,n}^\star\qdeRham_{-/A}\bigr)\colon \cat{Sm}_{A[n!^{-1}]}^{\leqslant n}\longrightarrow \cat{AniAlg}_A^{\qHodge}\,.
		\end{equation*}
		We let $\cat{Sm}_{A[\dim!^{-1}]}^{\leqslant n}\subseteq \cat{Sm}_A$ be the full subcategory spanned by $\bigcup_{d\leqslant n}\cat{Sm}_{A[d!^{-1}]}^{\leqslant d}$ and we put
		\begin{equation*}
			\cat{Sm}_{A[\dim!^{-1}]}\coloneqq \bigcup_{n\geqslant 0} \cat{Sm}_{A[\dim!^{-1}]}^{\leqslant n}\,.
		\end{equation*}
		Our goal is to show that the functors above for varying~$n$ combine into a single functor defined on all of $\cat{Sm}_{A[\dim!^{-1}]}$. This will be achieved by the technical \cref{lem:SmoothPushout,lem:CanonicalqHodgeSmoothFunctorial} below.
	\end{numpar}
	\begin{lem}\label{lem:SmoothPushout}
		For all $n\geqslant 0$, the following diagram is a pushout of $\infty$-categories:
		\begin{equation*}
			\begin{tikzcd}
				\cat{Sm}_{A[(n+1)!^{-1}]}^{\leqslant n}\rar\dar\drar[pushout] & \cat{Sm}_{A[(n+1)!^{-1}]}^{\leqslant n+1}\dar\\
				\cat{Sm}_{A[\dim!^{-1}]}^{\leqslant n}\rar & \cat{Sm}_{A[\dim!^{-1}]}^{\leqslant n+1}
			\end{tikzcd}
		\end{equation*}
	\end{lem}
	\begin{proof}
		Let $\Pp$ denote the pushout. Since the diagram above commutes, we get a functor $\Pp\rightarrow \cat{Sm}_{A[\dim!^{-1}]}^{\leqslant n+1}$. This functor is clearly essentially surjective. To show that it is fully faithful, we must show that
		\begin{equation*}
			\Hom_\Pp(S_1,S_2)\overset{\simeq}{\longrightarrow} \Hom_{\cat{Sm}_{A[\dim!^{-1}]}^{\leqslant n+1}}(S_1,S_2)
		\end{equation*}
		is an equivalence for all $S_1,S_2\in \Pp$. We may assume without loss of generality that $S_1$ and $S_2$ are the images of objects in $\cat{Sm}_{A[\dim!^{-1}]}^{\leqslant n}$ or $\cat{Sm}_{A[(n+1)!^{-1}]}^{\leqslant n+1}$.
		
		By an observation of Maxime Ramzi \cite{RamziPushouts}, fully faithful functors are preserved under pushouts. Since both legs of our pushout are fully faithful, the claimed equivalence is clear if $S_1$ and $S_2$ come from the same cofactor. It remains to deal with the following two cases.
		
		\emph{Case~1: $S_1\in \cat{Sm}_{A[\dim!^{-1}]}^{\leqslant n}$ and $S_2\in \cat{Sm}_{A[(n+1)!^{-1}]}^{\leqslant n+1}$.} Observe that the fully faithful functor
		\begin{equation*}
			\cat{Sm}_{A[(n+1)!^{-1}]}^{\leqslant n}\longrightarrow \cat{Sm}_{A[\dim!^{-1}]}^{\leqslant n}
		\end{equation*}
		has a left adjoint given by localisation at $(n+1)!$. It follows formally that $\cat{Sm}_{A[\dim!^{-1}]}^{\leqslant n+1}\rightarrow\Pp$ also has a left adjoint and that the diagram of left adjoints is still a pushout (and in particular commutative). Indeed, we can simply define unit and counit by taking the pushout of the original unit and counit; the triangle identities will automatically be satisfied. Therefore, we can replace $S_1$ by $S_1[(n+1)!^{-1}]$ and thus reduce to the case where $S_1$ and $S_2$ come from the same cofactor.
		
		\emph{Case~2: $S_1\in \cat{Sm}_{A[(n+1)!^{-1}]}^{\leqslant n+1}$ and $S_2\in \cat{Sm}_{A[\dim!^{-1}]}^{\leqslant n}$.} We may additionally assume that $(n+1)!$ is not invertible in $S_2$; otherwise we would be in a case already covered. But then
		\begin{equation*}
			\Hom_{\cat{Sm}_{A[\dim!^{-1}]}^{\leqslant n+1}}(S_1,S_2)\simeq\emptyset
		\end{equation*}
		and so the map in question must be an equivalence, since only $\emptyset$ maps to $\emptyset$.
	\end{proof}
	
	\begin{lem}\label{lem:CanonicalqHodgeSmoothFunctorial}
		For all $n\geqslant 0$, in the $\infty$-category of functors $\cat{Sm}_{A[(n+1)!^{-1}]}^{\leqslant n}\rightarrow \cat{AniAlg}_A^{\qHodge}$, there exists a natural equivalence
		\begin{equation*}
			\bigl(-,\fil_{\qHodge,n}^\star\qdeRham_{-/A}\bigr)\simeq \bigl(-,\fil_{\qHodge,n+1}^\star\qdeRham_{-/A}\bigr)\,.
		\end{equation*}
	\end{lem}
	\begin{proof}
		First observe that every morphism in $\Fil^{\geqslant 0}\Dd(\IZ)$ that is sent to an equivalence by $\tau_{n+1}^*\colon \Fil^{\geqslant 0}\Dd(\IZ)\rightarrow \Fil^{[0,n]}\Dd(\IZ)$ is also sent to an equivalence by $\tau_n^*$. Since $\tau_{n+1}^*$ is a symmetric monoidal localisation, there exists a unique (up to contractible choice) symmetric monoidal functor $\tau_{n,n+1}^*$ such that
		\begin{equation*}
			\begin{tikzcd}
				\Fil^{\geqslant 0}\Dd(\IZ)\dar["\tau_{n+1}^*"']\rar["\tau_n^*"] & \Fil^{[0,n]}\Dd(\IZ)\\
				\Fil^{[0,n+1]}\Dd(\IZ)\urar[dashed,"\tau_{n,n+1}^*"'] &
			\end{tikzcd}
		\end{equation*}
		commutes. Moreover, arguing as in \cref{lem:TruncatedFilteredObjects}\cref{enum:TruncationLeftAdjoint}, we see that for any filtered $\IE_\infty$-algebra $T\in\CAlg(\Fil^{\geqslant 0}\Dd(\IZ))$, the induced symmetric monoidal functor
		\begin{equation*}
			\tau_{n,n+1}^*\colon \Mod_{\tau_{n+1}^*T}\bigl(\Fil^{[0,n+1]}\Dd(\IZ)\bigr)\longrightarrow \Mod_{\tau_n^*T}\bigl(\Fil^{[0,n]}\Dd(\IZ)\bigr)
		\end{equation*}
		admits an oplax symmetric monoidal left adjoint
		\begin{equation*}
			\tau_{n,n+1,!}\colon  \Mod_{\tau_n^*T}\bigl(\Fil^{[0,n]}\Dd(\IZ)\bigr)\longrightarrow\Mod_{\tau_{n+1}^*T}\bigl(\Fil^{[0,n+1]}\Dd(\IZ)\bigr)\,.
		\end{equation*}
		Let us now apply this in the case where $T=(q-1)^\star A\qpower$. Let $S$ be a smooth $A$-algebra such that $\dim(S/A)\leqslant n$ and all primes $p\leqslant n+1$ are invertible in $S$. Plugging the canonical projection $\fil_{(\Hodge,q-1)}^\star(\Omega_{S/A}\qpower/(q-1)^{n+1})\rightarrow \fil_{(\Hodge,q-1)}^\star(\Omega_{S/A}\qpower/(q-1)^{n})$ into the pullback from \cref{par:CanonicalqHodgeSmoothI}, we obtain a morphism 
		\begin{equation*}
			\tau_{n,n+1}^*\fil_{\qHodge,n+1}^{\star\leqslant n+1}\qOmega_{S/A}\longrightarrow \fil_{\qHodge,n}^{\star\leqslant n}\qOmega_{S/A}\,.
		\end{equation*}
		Applying $\tau_{n,!}(-)_{(q-1)}^\complete$ on both sides and using the counit $\tau_{n,n+1,!}\circ\tau_{n,n+1}^*\Rightarrow \id$, we obtain a canonical zigzag
		\begin{equation*}
			\fil_{\qHodge,n+1}^\star\qOmega_{S/A}\overset{\simeq}{\longleftarrow} \tau_{n,!}\left(\tau_{n,n+1}^*\fil_{\qHodge,n+1}^{\star\leqslant n+1}\qOmega_{S/A}\right)_{(q-1)}^\complete\overset{\simeq}{\longrightarrow} \fil_{\qHodge,n}^\star\qOmega_{S/A}
		\end{equation*}
		It is now straightforward to check that both morphisms are equivalences. Indeed, everything is filtered $(q-1)$-complete, so we may check this after reduction modulo $(q-1)$. For the outer two terms, the reduction is $\fil_{\Hodge}^\star\Omega_{S/A}$ by the calculation in the proof of \cref{lem:CanonicalqHodgeSmooth}. An analogous calculation shows that the inner term also becomes $\fil_{\Hodge}^\star\Omega_{S/A}$ and that the morphisms become the identity.
		
		The zigzag above provides a functorial equivalence $\fil_{\qHodge,n+1}^\star\qOmega_{-/A}\simeq \fil_{\qHodge,n}^\star\qOmega_{-/A}$. Taking the pullback along $\qdeRham_{S/A}\rightarrow \qOmega_{S/A}$, we get what we want.
	\end{proof}
	
	In total we've shown:
	
	\begin{thm}\label{thm:CanonicalqHodgeSmooth}
		Let $A$ be a perfectly covered $\Lambda$-ring and let $S$ be a smooth $A$-algebra such that all primes $p\leqslant \dim(S/A)$ are invertible in $S$. Then $\qdeRham_{S/A}$ admits a canonical $q$-Hodge filtration. More precisely, there exists a functor
		\begin{equation*}
			\bigl(-,\fil_{\qHodge}^\star\qdeRham_{-/A}\bigr)\colon \cat{Sm}_{A[\dim!^{-1}]}\longrightarrow \cat{AniAlg}_A^{\qHodge}
		\end{equation*}
		which is a partial section of the forgetful functor $\cat{AniAlg}_A^{\qHodge}\rightarrow \cat{AniAlg}_A$.
	\end{thm}
	\begin{proof}
		This is the quintessence of \cref{par:CanonicalqHodgeSmoothI}--\labelcref{lem:CanonicalqHodgeSmoothFunctorial}.
	\end{proof}
	
	\begin{numpar}[Monoidality.]\label{par:CanonicalqHodgeSmoothMonoidal}
		We wish to study to what extent the $q$-Hodge filtrations from \cref{par:CanonicalqHodgeSmoothII} can be equipped with multiplicative structures. To this end, it would be nice to equip the functor from \cref{thm:CanonicalqHodgeSmooth} with a symmetric monoidal structure. This is made complicated by the following issue:
		\begin{alphanumerate}\itshape 
			\item[\,!\,] $\cat{Sm}_{A[\dim!^{-1}]}$ is not closed under tensor products in $\cat{Sm}_A$ and we don't see a way of equipping it with a symmetric monoidal structure.
		\end{alphanumerate}
		To address this problem, let $\cat{Sm}_A^{\otimes}\rightarrow \cat{Fin}_*$ be the $\infty$-operad associated with the symmetric monoidal structure on $\cat{Sm}_A$. We define a sub-$\infty$-operad $\cat{Sm}_{A[\dim!^{-1}]}^{\otimes}\subseteq \cat{Sm}_A^{\otimes}$ as follows:
		\begin{alphanumerate}
			\item An object $(S_1,\dotsc,S_i)\in\cat{Sm}_A^i$ in the fibre over $\langle i\rangle \in\cat{Fin}_*$ is contained in $\cat{Sm}_{A[\dim!^{-1}]}$ if and only $S_1,\dotsc,S_i$ are all contained in $\cat{Sm}_{A[\dim!^{-1}]}$.\label{enum:SmOperadA}
			\item A morphism $(S_1,\dotsc,S_i)\rightarrow (S_1',\dotsc,S_{i'}')$ over $\alpha\colon \langle i\rangle \rightarrow \langle i'\rangle$ is contained in $\cat{Sm}_{A[\dim!^{-1}]}$ if and only if both source and target satisfy the condition from \cref{enum:SmOperadA} and the target of a cocartesian lift of $\alpha$ with source $(S_1,\dotsc,S_i)$ also satisfies the condition from \cref{enum:SmOperadA}. Equivalently, we only retain those morphisms that factor through a cocartesian lift of their image in $\cat{Fin}_*$.\label{enum:SmOperadB}
		\end{alphanumerate}
		Let us immediately warn the reader that $\cat{Sm}_{A[\dim!^{-1}]}^\otimes$ is \emph{not} the full sub-$\infty$-operad of $\cat{Sm}_A^{\otimes}$ spanned by the full subcategory $\cat{Sm}_{A[\dim!^{-1}]}\subseteq \cat{Sm}_A$, precisely because the condition from~\cref{enum:SmOperadB} yields a \emph{non-full} sub-$\infty$-operad.
		
		Below we'll sketch how to make the functor from \cref{thm:CanonicalqHodgeSmooth} into a functor of $\infty$-operads (this wouldn't work if we had used the full sub-$\infty$-operad spanned by $\cat{Sm}_{A[\dim!^{-1}]}$). Let us discuss what kind of multiplicative structures this induces on $\fil_{\qHodge}^\star\qdeRham_{S/A}$. In general, any multiplicative structure on $S$ as an object in $\cat{Sm}_{A[\dim!^{-1}]}$ will induce the same kind of multiplicative structure on $\fil_{\qHodge}^\star\qdeRham_{S/A}$. For arbitrary $S\in \cat{Sm}_{A[\dim!^{-1}]}$ there's nothing we can say. But as soon as all primes $p\leqslant 2\dim(S/A)$ are invertible in $S$, the multiplication map $S\otimes_A S\rightarrow S$ is a morphism in $\cat{Sm}_{A[\dim!^{-1}]}$, and so $S$ will have an $\IA_2$-structure in $\cat{Sm}_{A[\dim!^{-1}]}^\otimes$; that is, a homotopy-unital multiplication. If for some $r\geqslant 3$ all primes $p\leqslant r\dim(S/A)$ are invertible in $S$, then the multiplication will be $\IA_r$; that is,  coherently associative for up to $r$ factors. A similar analysis works for commutativity.
	\end{numpar}
	We'll now sketch how to make the functor from \cref{thm:CanonicalqHodgeSmooth} into a functor of $\infty$-operads. Let us temporarily fix $n\geqslant 0$.
	
	\begin{lem}\label{lem:Truncatedq-1Completion}
		Let $\Mod_{\tau_n^*((q-1)^\star A\qpower)}(\Fil^{[0,n]}\Dd(\IZ))$ be as in \cref{par:CanonicalqHodgeSmoothII} and equip the full sub-$\infty$-category of $(q-1)$-complete objects $\Mod_{\tau_n^*((q-1)^\star A\qpower)}(\Fil^{[0,n]}\Dd(\IZ))_{(q-1)}^\complete$ with the $(q-1)$-completed tensor product. Then the functor
		\begin{equation*}
			\fil_{\qHodge,n}^\star\qdeRham_{-/A}\colon \cat{Sm}_{A[n!^{-1}]}\longrightarrow \Mod_{\tau_n^*((q-1)^\star A\qpower)}\bigl(\Fil^{[0,n]}\Dd(\IZ)\bigr)_{(q-1)}^\complete
		\end{equation*}
		from \cref{par:CanonicalqHodgeSmoothII} can be equipped with a symmetric monoidal structure.
	\end{lem}
	\begin{proof}[Proof sketch]
		From the construction it's straightforward to get a lax symmetric monoidal structure. Whether it is symmetric monoidal can be checked modulo~$(q-1)$, where we reduce to the fact that $\tau_n^*(\fil_{\Hodge}^\star\deRham_{-/A})$ is symmetric monoidal.
	\end{proof}
	
	\begin{numpar}[Lax vs.\ oplax symmetric monoidal functors.]
		For every symmetric monoidal $\infty$-category with associated cocartesian fibration $\Cc^\otimes\rightarrow \cat{Fin}_*$, let $(\Cc^\otimes)^\vee\rightarrow\cat{Fin}_*^\op$ denote the dual cartesian fibration. Lax symmetric monoidal functors $\Cc\rightarrow \Dd$ are then encoded as functors $\Cc^\otimes\rightarrow \Dd^\otimes$ in $\Cat_{\infty/\cat{Fin}_*}$ that preserve cocartesian lifts of inert morphisms, whereas oplax symmetric monoidal functors are encoded as functors $(\Cc^\otimes)^\vee\rightarrow (\Dd^\otimes)^\vee$ in $\Cat_{\infty/\cat{Fin}_*^\op}$ that preserve cartesian lifts of inert morphisms.
		
		In general, the dual cartesian fibration $(\Cc^\otimes)^\vee\rightarrow \cat{Fin}_*$ has a very nice description in terms of span $\infty$-categories. This is due to Barwick--Glasman--Nardin; see \cite[\chref{1.2}]{DualFibration}. We will now apply this to the oplax symmetric monoidal structure on
		\begin{equation*}
			\bigl(-,\fil_{\qHodge,n}^\star\qdeRham_{-/A}\bigr)\colon \cat{Sm}_{A[n!^{-1}]}\longrightarrow \cat{AniAlg}_A^{\qHodge}
		\end{equation*}
		that we obtain by composing the symmetric monoidal functor from \cref{lem:Truncatedq-1Completion} with the oplax symmetric monoidal functor $\tau_{n,!}(-)_{(q-1)}^\complete$.
	\end{numpar}
	\begin{lem}\label{lem:CartesianLiftsPreserved}
		If $\varphi\colon (S_1',\dotsc,S_{i'}')\rightarrow (S_1,\dotsc,S_i)$ is a cartesian morphism in $(\cat{Sm}_{A[n!^{-1}]}^{\otimes})^\vee$ such that $S_1',\dotsc,S_{i'}'$ are all of relative dimension $\leqslant n$ over $A$, then $\varphi$ is sent to a cartesian morphism under
		\begin{equation*}
			\bigl(\cat{Sm}_{A[n!^{-1}]}^{\otimes}\bigr)^\vee\longrightarrow \bigl(\cat{AniAlg}_A^{\qHodge,\otimes}\bigr)^\vee\,.
		\end{equation*}
	\end{lem}
	\begin{proof}[Proof sketch]
		This essentially reduces to the observation that whenever a tensor product of smooth $A[n!^{-1}]$-algebras $S_1\otimes_A\dotsb\otimes_A S_i$ has relative dimension $\leqslant n$ over $A$, the $q$-Hodge filtration $\fil_{\qHodge,n}^\star \qdeRham_{S_1\otimes_A\dotsb \otimes_AS_i/A}$ will agree with
		\begin{equation*}
			\left(\fil_{\qHodge,n}^\star \qdeRham_{S_1/A}\lotimes_{(q-1)^\star A\qpower}\dotsb \lotimes_{(q-1)^\star A\qpower}\fil_{\qHodge,n}^\star \qdeRham_{S_i/A}\right)_{(q-1)}^\complete\,.
		\end{equation*}
		Indeed, this can be checked modulo $(q-1)$, where the desired claim follows using symmetric monoidality of $\fil_{\Hodge}^\star\deRham_{-/A}$.
	\end{proof}

	\begin{cor}\label{cor:CanonicalqHodgeSmoothMonoidal}
		The functor from \cref{thm:CanonicalqHodgeSmooth} underlies a functor of $\infty$-operads
		\begin{equation*}
			\cat{Sm}_{A[\dim!^{-1}]}^{\otimes}\longrightarrow \cat{AniAlg}_A^{\qHodge,\otimes}\,,
		\end{equation*}		
		which preserves all cocartesian lifts that exist in the source.
	\end{cor}
	\begin{proof}[Proof sketch]
		As in \cref{par:CanonicalqHodgeSmoothMonoidal}, we can define a sub-$\infty$-operad $\cat{Sm}_{A[n!^{-1}]}^{\leqslant n,\otimes}\subseteq \cat{Sm}_{A[n!^{-1}]}^\otimes$ given by those objects whose entries are of dimension $\leqslant n$ and those morphisms that factor through a cocartesian lift of their image in $\cat{Fin}_*$. Analogously, we can define $(\cat{Sm}_{A[n!^{-1}]}^{\leqslant n,\otimes})^\vee\subseteq (\cat{Sm}_{A[n!^{-1}]}^\otimes)^\vee$ given by those objects whose entries are of dimension $\leqslant n$ and those morphisms that factor through a cartesian lift of their image in $\cat{Fin}_*^\op$.
		
		The dualising construction from \cite[\chref{1.2}]{DualFibration} can not only be applied to cartesian fibrations, but also to $(\cat{Sm}_{A[n!^{-1}]}^{\leqslant n,\otimes})^\vee$, and it is straightforward to check that we get back $\cat{Sm}_{A[n!^{-1}]}^{\leqslant n,\otimes}$ in this case. Moreover, by \cref{lem:CartesianLiftsPreserved}, the functor
		\begin{equation*}
			\bigl(\cat{Sm}_{A[n!^{-1}]}^{\leqslant,\otimes}\bigr)^\vee\longrightarrow \bigl(\cat{AniAlg}_A^{\qHodge,\otimes}\bigr)^\vee
		\end{equation*}
		preserves all cartesian lifts that exist in the source. We may thus dualise via \cite[\chref{1.2}]{DualFibration} to obtain a functor
		\begin{equation*}
			\cat{Sm}_{A[n!^{-1}]}^{\leqslant,\otimes}\longrightarrow \cat{AniAlg}_A^{\qHodge,\otimes}\,.
		\end{equation*}
		Now the $\infty$-operad $\cat{Sm}_{A[\dim!^{-1}]}^{\otimes}$ is built from $\cat{Sm}_{A[n!^{-1}]}^{\leqslant,\otimes}$ for all $n\geqslant 0$ via a sequence of pushouts as in \cref{lem:SmoothPushout}. Combining this with a straightforward analogue of \cref{lem:CanonicalqHodgeSmoothFunctorial}, we can inductively construct the desired map of $\infty$-operads.
	\end{proof}

	\subsection{Functorial \texorpdfstring{$q$}{q}-Hodge filtrations for certain quasi-regular quotients}\label{subsec:CanonicalqHodgeQuasiregular}
	
	In this subsection, we'll explain another elementary construction of functorial $q$-Hodge filtrations. To this end, let us first fix a prime~$p$ and work in a $p$-complete setting (at the end of this subsection, we'll get back to the global case). Throughout this subsection, all ($q$-)de Rham complexes or cotangent complexes relative to $p$-complete rings will be implicitly $p$-completed.
	
	\begin{numpar}[Rings of interest.]\label{par:RingsOfInterest}
		Temporarily, $A$ will not be a perfectly covered $\Lambda$-ring, but a \emph{$p$-completely perfectly covered $\delta$-ring}, by which we mean a $p$-complete $\delta$-ring for which the map $A\rightarrow A_\infty$ into its $p$-completed colimit perfection is $p$-completely faithfully flat. Equivalently, the Frobenius $\phi\colon A\rightarrow A$ is $p$-completely flat (as being faithful is automatic). Since perfect $\delta$-rings are $p$-torsion free, it follows that $A$ must be $p$-torsion free too.
		
		Throughout, we will consider \emph{$p$-quasi-lci algebras over $A$}: These are $p$-complete rings $R$ for which the cotangent complex $\L_{R/A}$ (which, by our convention above, we always take to be implicitly $p$-completed) has $p$-complete $\Tor$-amplitude over $R$ concentrated in degree $[0,1]$. Additionally, we'll usually assume that $R/p$ is \emph{relatively semiperfect over $A$}: That is, the relative Frobenius $R/p\otimes_{A,\phi}A\twoheadrightarrow R/p$ is surjective. This forces $\Omega_{R/A}^1/p$ to vanish, so $\L_{R/A}$ will have $p$-complete $\Tor$-amplitude over $R$ concentrated in degree~$1$.
		
		An important special case are $A$-algebras of \emph{perfect-regular presentation}: These are the quotients $R\cong B/J$, where $B$ is a $p$-complete relatively perfect $\delta$-$A$-algebra, by which we mean that the relative Frobenius $\phi_{B/A}\colon (B\otimes_{A,\phi}A)_p^\complete\rightarrow B$ is an isomorphism, and $J\subseteq B$ is an ideal generated by a Koszul-regular sequence. We'll sometimes refer to $B/J$ as a \emph{perfect-regular presentation of $R$}.
	\end{numpar}
	The reason for restricting to rings $R$ as above is the following lemma.
	\begin{lem}\label{lem:RelativelySemiperfect}
		Let $R$ be a $p$-torsion free $A$-algebra such that $\L_{R/A}$ has $p$-complete $\Tor$-amplitude over $R$ concentrated in degree~$1$.
		\begin{alphanumerate}
			\item The de Rham complex $\deRham_{R/A}$, its Hodge-completion $\hatdeRham_{R/A}$, every degree in the completed Hodge filtration $\fil_{\Hodge}^\star\hatdeRham_{R/A}$, and the $q$-de Rham complex $\qdeRham_{R/A}$ are all static and $p$-torsion free.\label{enum:deRhamStatic}
			\item The un-completed Hodge filtration $\fil_{\Hodge}^\star\deRham_{R/A}$ is static in every degree if and only if $R/p$ is relatively semiperfect over $A$.\label{enum:HodgeStatc}
		\end{alphanumerate}
	\end{lem}
	\begin{proof}
		To show that every degree in the completed Hodge filtration is static and $p$-torsion free, just observe that the same is true for the associated graded $\gr_{\Hodge}^*\hatdeRham_{R/A}\simeq \Sigma^{-*}\bigwedge^*\L_{R/A}$, because our assumption on $R$ guarantees that $\Sigma^{-1}\L_{R/A}$ is a $p$-completely flat module over the $p$-torsion free ring $R$. To show that the $(q-1)$-complete object $\qdeRham_{R/A}$ is static and $p$-torsion free, it will be enough to show the same for $\qdeRham_{R/A}/(q-1)\simeq \deRham_{R/A}$. Now all assertions about $\deRham_{R/A}$ and its Hodge filtration can be checked after base change along the $p$-completely faithfully flat map $A\rightarrow A_\infty$.
		
		So let us put $R_\infty\coloneqq (R\otimes_A A_\infty)_p^\complete$ and consider $\deRham_{R_\infty/A_\infty}$ and let $\ov R_\infty\coloneqq R_\infty/p$. Since $A_\infty$ is a perfect $\delta$-ring, $\L_{A_\infty/\IZ_p}\simeq 0$, so we may as well consider $\deRham_{R_\infty/\IZ_p}$. To see that $\deRham_{R_\infty/\IZ_p}$ is static and $p$-torsion free, it suffices to check that its modulo $p$ reduction $\deRham_{R_\infty/\IZ_p}/p\simeq \deRham_{\ov R_\infty/\IF_p}$ is static. The latter admits an ascending exhaustive filtration, the \emph{conjugate filtration}, whose associated graded $\Sigma^{-*}\bigwedge^*\L_{\ov R_\infty/\IF_p}\simeq \Sigma^{-*}\bigwedge^*\L_{R_\infty/\IZ_p}/p$ is static in every degree since $\Sigma^{-1}\L_{R_\infty/\IZ_p}$ is $p$-completely flat over the $p$-torsion free ring $R_\infty$. This shows that $\deRham_{\ov R_\infty/\IF_p}$ is indeed static and we've finished the proof of \cref{enum:deRhamStatic}.
		
		For~\cref{enum:HodgeStatc}, we've already seen that $\deRham_{R_\infty/\IZ_p}$ and the associated graded of the Hodge filtration are static and $p$-torsion free in every degree. Hence $\fil_{\Hodge}^\star\deRham_{R_\infty/\IZ_p}$ is degree-wise static if and only if it consists of sub-modules of $\deRham_{R_\infty/\IZ_p}$, which must be $p$-torsion free too. Thus $\fil_{\Hodge}^\star\deRham_{R_\infty/\IZ_p}$ is degree-wise static if and only if the same is true for $\fil_{\Hodge}^\star\deRham_{R_\infty/\IZ_p}/p\simeq \fil_{\Hodge}^\star\deRham_{\ov R_\infty/\IF_p}$. In the case where $\ov R_\infty$ is semiperfect, this holds by \cite[Proposition~\chref{8.14}]{BMS2}. Conversely, assume $\fil_{\Hodge}^\star\deRham_{\ov R_\infty/\IF_p}$ is degree-wise static. If $\fil_\Nn^\star\W\deRham_{\ov R_\infty/\IF_p}$ denotes the Nygaard filtration on the derived de Rham--Witt complex, then 
		\begin{equation*}
			\fil_\Nn^n\W\deRham_{\ov R_\infty/\IF_p}/p\fil_\Nn^{n-1}\W\deRham_{\ov R_\infty/\IF_p}\simeq \fil_{\Hodge}^n\deRham_{\ov R_\infty/\IF_p}
		\end{equation*}
		holds for all $n$ by deriving \cite[Lemma~\chref{8.3}]{BMS2}. Inductively it follows that $\W\deRham_{R_\infty/\IZ_p}$ and each step in its Nygaard filtration must be static too. By definition, $\fil_\Nn^n\W\deRham_{\ov R_\infty/\IF_p}$ is the fibre of
		\begin{equation*}
			\W\deRham_{\ov R_\infty/\IF_p}\overset{\phi}{\longrightarrow}\W\deRham_{\ov R_\infty/\IF_p}\longrightarrow \W\deRham_{\ov R_\infty/\IF_p}/p^n\,,
		\end{equation*}
		so this composition must be surjective for all $n$. Then $\phi\colon \W\deRham_{\ov R_\infty/\IF_p}\rightarrow\W\deRham_{\ov R_\infty/\IF_p}$ must be surjective as well. Since $\W\deRham_{\ov R_\infty/\IF_p}/p\simeq \deRham_{\ov R_\infty/\IF_p}\rightarrow \ov R_\infty$ is surjective by our assumption that $\fil_{\Hodge}^1\deRham_{\ov R_\infty/\IF_p}$ is static, we conclude that the Frobenius on $\ov R_\infty$ must be surjective too.
	\end{proof}
	\begin{rem}
		In the case where $R\cong B/J$ is of perfect-regular presentation over $A$, everything can be made explicit: $\deRham_{R/A}\simeq D_B(J)$ is the ($p$-completed) PD-envelope of $J$, the Hodge filtration is just the PD-filtration, and the $q$-de Rham complex $\qdeRham_{R/A}$ is the corresponding $q$-PD-envelope in the sense of \cite[Lemma~\chref{16.10}]{Prismatic}.
	\end{rem}
	\begin{rem}
		There exist $p$-complete $\IZ_p$-algebras whose cotangent complex has $p$-complete $\Tor$-amplitude concentrated in degree~$1$, but whose reduction modulo $p$ is not semiperfect. For example, if $p\geqslant 3$, the $\IF_p$-algebra constructed in \cite{SemiperfectCounterexample} can be lifted in a straightforward way to a $p$-complete $\IZ_p$-algebra with this property.
	\end{rem}
	
	Let us now define a $q$-Hodge filtration for rings $R$ as in \cref{lem:RelativelySemiperfect}.
	
	\begin{con}\label{con:qHodge}
		Suppose $R$ is a $p$-torsion free quasi-lci $A$-algebra such that $R/p$ is relatively semiperfect over $A$. By \cref{lem:RationalisedqCrystalline}, after rationalisation, $\deRham_{R/A}$ and $\qdeRham_{R/A}$ are related via a functorial equivalence
		\begin{equation*}
			\qdeRham_{R/A}\bigl[\localise{p}\bigr]_{(q-1)}^\complete\simeq \deRham_{R/A}\bigl[\localise{p}\bigr]\qpower\,.
		\end{equation*}
		By \cref{lem:RelativelySemiperfect}, both sides are static rings. Let us equip the right-hand side with the combined Hodge and $(q-1)$-adic filtration $\fil_{(\Hodge,q-1)}^\star \deRham_{R/A}[1/p]\qpower$ as in \cref{def:qHodgeFiltration}\cref{enum:qHodgeRationalpComplete}. This is a descending filtration by ideals. 
		
		We now construct $\fil_{\qHodge}^\star\qdeRham_{R/A}$ as the $1$-categorical (!) preimage of this filtration under $\qdeRham_{R/A}\rightarrow\deRham_{R/A}[1/p]\qpower$; in other words, as the pullback
		\begin{equation*}
			\begin{tikzcd}
				\fil_{\qHodge}^\star \qdeRham_{R/A}\rar\dar\drar[pullback] & \fil_{(\Hodge,q-1)}^\star \deRham_{R/A}\bigl[\localise{p}\bigr]\qpower\dar\\
				\qdeRham_{R/A}\rar &  \deRham_{R/A}\bigl[\localise{p}\bigr]\qpower
			\end{tikzcd}
		\end{equation*}
		taken in the $1$-category of filtered $(q-1)^\star A\qpower$-modules. We remark that $\fil_{\qHodge}^\star\qdeRham_{R/A}$ will be a descending filtration of ideals in the static ring $\qdeRham_{R/A}$, hence it's automatically a filtered $\IE_\infty$-algebra over $(q-1)^\star A\qpower$.
		
		Let us also remark that the canonical projection $\qdeRham_{R/A}\rightarrow \deRham_{R/A}$ induces a (necessarily unique) filtered map
		\begin{equation*}
			\fil_{\qHodge}^\star \qdeRham_{R/A}\longrightarrow \fil_{\Hodge}^\star \deRham_{R/A}\,.
		\end{equation*}
		Indeed, to see this, we must check that $\fil_{\Hodge}^\star \deRham_{R/A}$ is the preimage of $\fil_{\Hodge}^\star\deRham_{R/A}[1/p]$ under $\deRham_{R/A}\rightarrow \deRham_{R/A}[1/p]$. Since any filtration is the preimage of its completion, we may further replace the Hodge filtration $\fil_{\Hodge}^\star\deRham_{R/A}[1/p]$ by its completion $\fil_{\Hodge}^\star\deRham_{R/A}[1/p]_{\Hodge}^\complete$. To check that $\fil_{\Hodge}^\star\deRham_{R/A}$ is the preimage, it will thus be enough to check that the map on associated gradeds is injective. Now
		\begin{equation*}
			\Sigma^{-n}\bigwedge^n\L_{R/A}\rightarrow \Sigma^{-n}\bigwedge^n \L_{R/A}\bigl[\localise{p}\bigr]
		\end{equation*}
		will be injective for all $n\geqslant 0$, because $\Sigma^{-n}\bigwedge^n\L_{R/A}$ is a $p$-completely flat module over the $p$-torsion free ring $R$ and thus $p$-torsion free itself.
	\end{con}
	In general, the $q$-Hodge filtration from \cref{con:qHodge} will be nonsense. But it does behave as desired in the following cases:
	
	\begin{thm}\label{thm:qHodgeWellBehaved}
		Let $A$ be a $p$-completely perfectly covered $\delta$-ring and let $R$ be a $p$-torsion free quasi-lci $A$-algebra such that $R/p$ is relatively semiperfect over $A$. Suppose that one of the following two additional assumptions is satisfied:
		\begin{alphanumerate}
			\item There exists a perfect-regular presentation $R\cong B/J$, where the ideal $J\subseteq B$ is generated by a Koszul-regular sequence of higher powers, that is, a Koszul-regular sequence $(x_1^{\alpha_1},\dotsc,x_r^{\alpha_r})$ with $\alpha_i\geqslant 2$ for all $i$.\label{enum:RegularSequenceHigherPowers}
			\item The ring $R_\infty\coloneqq (R\otimes_A A_\infty)_p^\complete$ admits a lift to a $p$-complete connective  $\IE_1$-ring spectrum $\IS_{R_\infty}$ satisfying $R_\infty\simeq \IS_{R_\infty}\otimes_{\IS_p}\IZ_p$.\label{enum:qHodgeE1Lift}
		\end{alphanumerate}
		Then $\fil_{\qHodge}^\star\qdeRham_{R/A}$ is a $q$-deformation of $\fil_{\Hodge}^\star\deRham_{R/A}$ in the sense that the canonical map from \cref{con:qHodge} induces an equivalence
		\begin{equation*}
			\fil_{\qHodge}^\star\qdeRham_{R/A}/(q-1)\overset{\simeq}{\longrightarrow} \fil_{\Hodge}^\star\deRham_{R/A}\,.
		\end{equation*}
		Here we take the quotient in filtered $(q-1)^\star A\qpower$-modules, with $(q-1)$ regarded as an element in filtration degree~$1$.
	\end{thm}
	\begin{rem}\label{rem:BurklundsMagic}
		For primes $p>2$, \cref{thm:qHodgeWellBehaved}\cref{enum:qHodgeE1Lift} implies \cref{enum:RegularSequenceHigherPowers}. Indeed, if we put $B_\infty\coloneqq (B\otimes_A A_\infty)_p^\complete$, then $B_\infty$ is a perfect $\delta$-ring and so it lifts uniquely to a connective $p$-complete $\IE_\infty$-ring spectrum. We can then use Burklund's theorem on $\IE_n$-structures on quotients \cite[Theorem~\chref{1.5}]{BurklundMooreSpectra} to construct an $\IE_1$-structure on
		\begin{equation*}
			\IS_{R_\infty}\coloneqq\IS_{B_\infty}/\bigl(x_1^{\alpha_1},\dotsc,x_r^{\alpha_r}\bigr)\,.
		\end{equation*}
		More precisely, since $p>2$, each $\IS_{B_\infty}/x_i$ admits a right-unital multiplication (the relevant obstruction $Q_1(x_i)$ is $2$-torsion), and so Burklund's result provides $\IE_1$-structures on $\IS_{B_\infty}/x_i^{\alpha_i}$ in $\Mod_{\IS_{B_\infty}}(\Sp)$, of which we can take the tensor product.
		
		For $p=2$, $\IS_{B_\infty}/x_i^2$ still admits a right-unital multiplication (see \cite[Remark~\chref{5.5}]{BurklundMooreSpectra}) and so the same argument shows that \cref{thm:qHodgeWellBehaved}\cref{enum:qHodgeE1Lift} implies \cref{enum:RegularSequenceHigherPowers} if all $\alpha_i$ are even and $\geqslant 4$. It is somewhat surprising that \cref{thm:qHodgeWellBehaved}\cref{enum:RegularSequenceHigherPowers} is true without this additional restriction at $p=2$.
	\end{rem}
	
	Before we prove \cref{thm:qHodgeWellBehaved}, let us discuss two examples.
	
	\begin{exm}\label{exm:MainSpecialCase}
		Let $A\coloneqq \IZ_p\{x\}_p^\complete$ be the free $p$-complete $\delta$-ring on a generator $x$ and let $R\coloneqq\IZ_p\{x\}_p^\complete/x^\alpha$ for some $\alpha\geqslant 1$. Then \cref{thm:qHodgeWellBehaved}\cref{enum:RegularSequenceHigherPowers} will apply as soon as $\alpha\geqslant2$, but not for $\alpha=1$. So let's see what goes wrong for $\alpha=1$ and how higher powers (or divine intervention?) fix the issue.
		
		In the case at hand, $\deRham_{R/A}$ and $\qdeRham_{R/A}$ are the usual PD-envelope and the $q$-PD envelope
		\begin{equation*}
			D_\alpha\coloneqq \IZ_p\{x\}\left\{\frac{\phi(x^\alpha)}{p}\right\}_p^\complete\quad\text{and}\quad\q D_\alpha\coloneqq \IZ_p\{x\}\qpower\left\{\frac{\phi(x^\alpha)}{[p]_q}\right\}_{(p,q-1)}^\complete\,,
		\end{equation*}
		respectively. If the $q$-Hodge filtration were to be a $q$-deformation of the Hodge filtration, then $\fil_{\qHodge}^p\q D_\alpha$ would need to contain a lift $\widetilde{\gamma}_q(x^\alpha)$  of the divided power $\gamma(x^\alpha)\coloneqq x^{\alpha p}/p\in \fil_{\Hodge}^pD_\alpha$. Certainly, $\q D_\alpha$ itself contains such a lift; namely, the $q$-divided power
		\begin{equation*}
			\gamma_q(x^\alpha)\coloneqq\frac{\phi(x^\alpha)}{[p]_q}-\delta(x^\alpha)\,.
		\end{equation*}
		The problem is that $\gamma_q(x^\alpha)$ is usually not contained in $\fil_{\qHodge}^p\q D_\alpha$. So for the $q$-Hodge filtration to be a $q$-deformation of the Hodge filtration, it must be possible to modify $\gamma_q(x^\alpha)$ by elements from $(q-1)\q D_\alpha$ to get an element in $\fil_{\qHodge}^p\q D_\alpha$. As we'll see momentarily, this is impossible for $\alpha=1$, but it works for $\alpha\geqslant 2$.
		
		By definition, $\fil_{\qHodge}^\star \qdeRham_{R/A}$ is the preimage of the combined Hodge and $(q-1)$-adic filtration on $D_\alpha[1/p]\qpower$. Since every filtration is the preimage of its completion, we may replace the latter by its completion, which is the $(x^\alpha,q-1)$-adic filtration on $\IQ_p\langle \delta(x),\delta^2(x),\dotsc\rangle\llbracket x,q-1\rrbracket$. So our task is to modify $\gamma_q(x^\alpha)$ by elements from the $(q-1)\q D_\alpha$ such that the result is contained in the ideal $(x^\alpha,q-1)^p\subseteq \IQ_p\langle \delta(x),\delta^2(x),\dotsc\rangle\llbracket x,q-1\rrbracket$.
		
		Write $[p]_q=pu+(q-1)^{p-1}$, where $u\equiv 1\mod q-1$. In particular, $u$ is a unit in $\q D_\alpha$. In $\IQ_p\langle \delta(x),\delta^2(x),\dotsc\rangle\llbracket x,q-1\rrbracket$, we can rewrite $\gamma_q(x^\alpha)$ as
		\begin{equation*}
			\frac{x^{\alpha p}}{[p]_q}+\left(\frac{p}{[p]_q}-1\right)\delta(x^\alpha)=\frac{x^{\alpha p}}{[p]_q}+\left((u^{-1}-1)-u^{-2}\frac{(q-1)^{p-1}}{p}+\mathrm{O}\bigl((q-1)^p\bigr)\right)\delta(x^\alpha)\,.
		\end{equation*}
		Here $\mathrm{O}((q-1)^p)$ denotes \enquote{error terms} which are divisible by $(q-1)^p$. Observe that these error terms are contained in $(x^\alpha,q-1)^p$, so we can safely ignore them. Also $x^{\alpha p}/[p]_q$ is clearly contained in $(x^\alpha,q-1)^p$. The term $(u^{-1}-1)\delta(x^\alpha)$ is contained in $(q-1)\q D_\alpha$, so we can just kill it. This leaves the term $u^{-2}(q-1)^{p-1}\delta(x^\alpha)/p$.
		
		If $\alpha=1$, there's nothing we can do: No modification by elements from $(q-1)\q D_\alpha$ will ever get rid of a non-integral multiple of  $\delta(x)$, as $\delta(x)$ is a polynomial variable in $\IZ_p\{x\}$. This shows that for $\alpha=1$, the $q$-Hodge filtration on $\q D_\alpha$ is \emph{not} a $q$-deformation of the Hodge filtration. For $\alpha=2$, however, we have $\delta(x^2)=2x^p\delta(x)+p\delta(x)^2$. Now the term $2x^p\delta(x)u^{-2}(q-1)^{p-1}/p$ is contained in $(x^2,q-1)^p$ and so
		\begin{equation*}
			\widetilde{\gamma}_q(x^\alpha)\coloneqq \gamma_q(x^\alpha)-(u^{-1}-1)\delta(x^2)+u^{-2}(q-1)^{p-1}\delta(x)^2
		\end{equation*}
		is contained in $\fil_{\qHodge}^p\q D_\alpha$ and satisfies $\widetilde{\gamma}_q(x^\alpha)\equiv x^{2p}/p\mod q-1$, as desired. For $\alpha\geqslant 3$, we can similarly decompose $\delta(x^\alpha)$ into a multiple of $x^{p(\alpha-1)}$ and a multiple of $p$.
		
		This explains what goes wrong at $\alpha=1$ and how the objection is resolved for $\alpha\geqslant2$. In the latter case, it is possible to continue the analysis above and construct for all $n\geqslant 1$ a lift of the divided power $x^{\alpha n}/n!$ that lies in in $\fil_{\qHodge}^n\qdeRham_{R/A}$. This is explained in \cite[\S{\chref[subsection]{3.2}}]{qHodge} and leads to an elementary proof of \cref{thm:qHodgeWellBehaved}\cref{enum:RegularSequenceHigherPowers}. 
	\end{exm}
	\begin{exm}
		An example for \cref{thm:qHodgeWellBehaved}\cref{enum:qHodgeE1Lift} that is not covered by \cref{thm:qHodgeWellBehaved}\cref{enum:RegularSequenceHigherPowers} is the case $A\cong \IZ_p[x]_p^\complete$, with $\delta$-structure defined by $\delta(x)=0$, and $R\cong A/(x-1)\cong \IZ_p$. Then $A$ lifts to the $p$-complete $\IE_\infty$-ring spectrum $\IS_p[x]_p^\complete$ and $A\rightarrow R$ lifts to an $\IE_\infty$-map $\IS_p[x]_p^\complete\rightarrow \IS_p$. Base changing along $\IS_p[x]_p^\complete\rightarrow \IS_p[x^{1/p^\infty}]_p^\complete$ yields a lift of $R_\infty$, even as an $\IE_\infty$-ring. In this case, $\qdeRham_{R/A}$ is the $q$-PD envelope
		\begin{equation*}
			\q D\coloneqq \IZ_p[x]\qpower\left\{\frac{x^p-1}{[p]_q}\right\}_{(p,q-1)}^\complete\,.
		\end{equation*}
		It can be shown that this ring contains elements of the form $(x-1)(x-q)\dotsm (x-q^{n-1})/[n]_q!$ for all $n\geqslant 1$ (see \cite[Lemma~\chref{6.11}]{qdeRhamku} for an argument). After completed rationalisation, these elements are visibly contained in the ideal $(x-1,q-1)^n$. Hence they belong to $\fil_{\qHodge}^n \qdeRham_{R/A}$ and lift the usual divided powers.
	\end{exm}
	
	Let us now prove \cref{thm:qHodgeWellBehaved}. We start with a simple observation, which says that only surjectivity is critical.
	
	\begin{lem}\label{lem:qHodgeInjective}
		Let $R$ be a $p$-torsion free $p$-quasi-lci $A$-algebra such that $R/p$ is relatively semiperfect over~$A$. Then the canonical map from \cref{con:qHodge} induces a degree-wise injection
		\begin{equation*}
			\fil_{\qHodge}^\star\qdeRham_{R/A}/(q-1)\longhookrightarrow \fil_{\Hodge}^\star\deRham_{R/A}\,.
		\end{equation*}
	\end{lem}
	\begin{proof}
		We need to check 
		\begin{equation*}
			(q-1)\fil_{\qHodge}^{n-1}\qdeRham_{R/A}=\fil_{\qHodge}^n\qdeRham_{R/A}\cap (q-1)\qdeRham_{R/A}
		\end{equation*}
		for all~$n$. This immediately reduces to the analogous assertion for the combined Hodge and $(q-1)$-adic filtration on $\deRham_{R/A}[1/p]\qpower$, which is straightforward to check.
	\end{proof}
	
	The $q$-Hodge filtration from \cref{con:qHodge} enjoys a general flat base change property. This will allow us to reduce the proof of \cref{thm:qHodgeWellBehaved} to the case where $A$ is perfect.
	
	\begin{lem}\label{lem:qHodgeBaseChange}
		Let $R$ be a $p$-torsion free $p$-quasi-lci $A$-algebra such that $R/p$ is relatively semiperfect over~$A$. Let $A\rightarrow A'$ be a $p$-completely flat morphism of $\delta$-rings, where $A'$ is also $p$-completely perfectly covered, and put $R'\coloneqq (R\otimes_AA')_p^\complete$. Then the canonical map
		\begin{equation*}
			\bigl(\fil_{\qHodge}^\star\qdeRham_{R/A}\lotimes_{A}A'\bigr)_{(p,q-1)}^\complete \overset{\simeq}{\longrightarrow} \fil_{\qHodge}^\star\qdeRham_{R'/A'}
		\end{equation*}
		is an equivalence.
	\end{lem}
	\begin{proof}
		This is not completely automatic since we have to be careful with completions. Fix~$n$. By \cref{rem:qDeRhamRationalisationTechnical}, the canonical map $\qdeRham_{R/A}\rightarrow (\deRham_{R/A}\otimes_\IZ\IQ)\qpower/(q-1)^n$ already factors through $p^{-N}\deRham_{R/A}\qpower/(q-1)^n$ for sufficiently large $N$. Since $\fil_{\qHodge}^n\qdeRham_{R/A}$ contains $(q-1)^n\qdeRham_{R/A}$, we can also express it as a pullback of $A\qpower$-modules
		\begin{equation*}
			\begin{tikzcd}
				\fil_{\qHodge}^n\qdeRham_{R/A}\rar\dar\drar[pullback] & \qdeRham_{R/A}\dar\\
				p^{-N}\fil_{(\Hodge,q-1)}^n\deRham_{R/A}\qpower/(q-1)^n\rar & p^{-N}\deRham_{R/A}\qpower/(q-1)^n
			\end{tikzcd}
		\end{equation*}
		(here the combined Hodge and $(q-1)$-adic filtration $\fil_{(\Hodge,q-1)}^\star\deRham_{R/A}\qpower/(q-1)^n$ is constructed as in \cref{par:CanonicalqHodgeSmoothI} above).
		
		It will be enough to show that the pullback is preserved $(-\lotimes_{A}A')_{(p,q-1)}^\complete$. To this end, let $P$ denote the derived pullback (that is, the pullback taken in the derived $\infty$-category $\Dd(A\qpower)$) and recall that derived tensor products preserve derived pullbacks. It is then enough to check that $(\H_{-1}(P)\lotimes_{A}A')_{(p,q-1)}^\complete$ is static. We claim that $\H_{-1}(P)$ is $(q-1)^n$-torsion and $p^m$-torsion for sufficiently large $m$. Believing this for the moment, $p$-complete flatness of $A\rightarrow A'$ guarantees that $\H_{-1}(P)\lotimes_AA'$ is static. Since it is also $p^m$- and $(q-1)^n$-torsion, the completion doesn't change anything and we're done.
		
		To prove the claim, observe that the cokernel of $\qdeRham_{R/A}\rightarrow p^{-N}\deRham_{R/A}$ must clearly be $p^{N}$-torsion. Hence the cokernel of the right vertical map
		\begin{equation*}
			\qdeRham_{R/A}\longrightarrow p^{-N}\deRham_{R/A}\qpower/(q-1)^n
		\end{equation*}
		is $p^{nN}$-torsion and also $(q-1)^n$-torsion. Since $\H_{-1}(P)$ is a quotient of that cokernel (explicitly the quotient by the bottom left corner of the pullback diagram), we conclude that $\H_{-1}(P)$ is $p^{nN}$-torsion and $(q-1)^n$-torsion too, as desired.
	\end{proof}
	\begin{proof}[Proof of \cref{thm:qHodgeWellBehaved}]
		By \cref{lem:qHodgeInjective}, we only need to check surjectivity. By \cref{lem:qHodgeBaseChange}, we can check this after the $p$-completely faithfully flat base change $A\rightarrow A_\infty$ and thus assume that $A$ is perfect. Since in this case $\qdeRham_{R/A}\simeq \qdeRham_{R/\IZ_p}$, we may further reduce to the case $A=\IZ_p$. Then part~\cref{enum:qHodgeE1Lift} is a special case of \cite[Theorem~\chref{4.17}]{qdeRhamku}.
		
		To show surjectivity in the case of~\cref{enum:RegularSequenceHigherPowers}, it'll be enough to show that for each of the generators of $J=(x_1^{\alpha_1},\dotsc,x_r^{\alpha_r})$ and all $n\geqslant 0$, the $n$-fold iterated divided power $\gamma^{(n)}(x_i^{\alpha_i})$ admits a lift which lies in the $(p^n)$\textsuperscript{th} step of the $q$-Hodge filtration. This reduces the problem to the universal case $A=\IZ_p\{x\}$ and $R=\IZ_p\{x\}/x^\alpha$ for $\alpha\geqslant 2$. In this case the desired lifts have been constructed in \cite[Lemma~\chref{3.17}]{qHodge}.
	\end{proof}
	
	To finish this subsection, we'll extract a global construction from the above. From now on, we cancel the assumptions from \cref{par:RingsOfInterest} and return to our usual notation, where $A$ is a perfectly covered $\Lambda$-ring.
	
	\begin{con}\label{con:GlobalqHodge}
		Let $R$ be an $A$-algebra such that for all primes $p$, $R$ is $p$-torsion free, the $p$-completion $\widehat{R}_p$ is $p$-quasi-lci over $\widehat{A}_p$, and $R/p$ is relatively semiperfect over $\widehat{A}_p$. We construct $\fil_{\qHodge}^\star\qdeRham_{R/A}$ as the pullback
		\begin{equation*}
			\begin{tikzcd}
				\fil_{\qHodge}^\star\qdeRham_{R/A}\rar\dar\drar[pullback] & \prod_p\fil_{\qHodge}^\star\qdeRham_{\smash{\widehat{R}}_p/\smash{\widehat{A}}_p}\dar\\
				\fil_{(\Hodge,q-1)}^\star\bigl(\deRham_{R/A}\lotimes_\IZ\IQ\bigr)\qpower\rar & \fil_{(\Hodge,q-1)}^\star\biggl(\prod_p\deRham_{\smash{\widehat{R}}_p/\smash{\widehat{A}}_p}\lotimes_\IZ\IQ\biggr)\qpower
			\end{tikzcd}
		\end{equation*}
		taken in the $\infty$-category filtered $\IE_\infty$-algebras over $(q-1)^\star A\qpower$. To see that the right vertical map in the pullback exists, observe that we're dealing with two filtrations by submodules, so there's only a set-level condition to check, which follows directly from the definition of $\fil_{\qHodge}^\star\qdeRham_{\smash{\widehat{S}}_p/\smash{\widehat{A}}_p}$.
	\end{con}
	\begin{thm}\label{thm:CanonicalqHodgeQuasiregular}
		Let $A$ be a perfectly covered $\Lambda$-ring and let $\cat{QReg}_A^{\qHodge}$ be the category of all $A$-algebras $R$ such that for all primes~$p$, $R$ is $p$-torsion free, the $p$-completion $\widehat{R}_p$ is $p$-quasi-lci over $\widehat{A}_p$, $R/p$ is relatively semiperfect over $\widehat{A}_p$, and the canonical morphism from \cref{con:qHodge} induces an equivalence
		\begin{equation*}
			\fil_{\qHodge}^\star\qdeRham_{\smash{\widehat{R}}_p/\smash{\widehat{A}}_p}/(q-1)\overset{\simeq}{\longrightarrow}\fil_{\Hodge}^\star\deRham_{\smash{\widehat{R}}_p/\smash{\widehat{A}}_p}\,.
		\end{equation*}
		Then \cref{con:GlobalqHodge} determines a functor
		\begin{equation*}
			\bigl(-,\fil_{\qHodge}^\star\qdeRham_{-/A}\bigr)\colon \cat{QReg}_{A}^{\qHodge}\longrightarrow \CAlg\bigl(\cat{AniAlg}_A^{\qHodge}\bigr)\,,
		\end{equation*}
		which is a partial section of the forgetful functor $\CAlg(\cat{AniAlg}_A^{\qHodge})\rightarrow \cat{AniAlg}_A$.
	\end{thm}
	\begin{proof}[Proof sketch]
		Let us construct the required data from \cref{def:qHodgeFiltration}. In degree~$0$, the pullback square from \cref{con:GlobalqHodge} becomes the one from \cref{con:GlobalqDeRham}, which provides the datum from \cref{def:qHodgeFiltration}\cref{enum:qHodgeFiltrationOnqdR}. If we reduce the pullback from \cref{con:GlobalqHodge} modulo~$(q-1)$, we'll get the arithmetic fracture square for $\fil_{\Hodge}^\star\deRham_{R/A}$ by our assumptions on~$R$. This provides the data from \cref{def:qHodgeFiltration}\cref{enum:qHodgeModq-1}. Similarly, if we apply $(-\lotimes_\IZ\IQ)_{(q-1)}^\complete$ or $(-)_p^\complete[1/p]_{(q-1)}^\complete$ to the pullback, we get the data from \cref{def:qHodgeFiltration}\cref{enum:qHodgeRational} and~\cref{enum:qHodgeRationalpComplete}.
		
		So $(R,\fil_{\qHodge}^\star\qdeRham_{R/A})$ can be made into an object of $\cat{AniAlg}_A^{\qHodge}$. Since all constructions above can also be done on the level of filtered $\IE_\infty$-algebras, it immediately upgrades to an object of $\CAlg(\cat{AniAlg}_A^{\qHodge})$. Finally, all steps of the construction can easily be made functorial in~$R$. To this end, one writes $\CAlg(\cat{AniAlg}_A^{\qHodge})$ as an iterated pullback of $\CAlg(-)$ of various symmetric monoidal $\infty$-categories of filtered objects. We know how to make $\fil_{(\Hodge,q-1)}^\star(\deRham_{R/A}\lotimes_\IZ\IQ)\qpower$ functorial; in the other factors of the iterated pullback, the objects in question will be $1$-categorical in nature, so all functorialities and compatibilities can easily be constructed by hand.
	\end{proof}
	\begin{rem}
		Thanks to \cref{thm:qHodgeWellBehaved}, it's easy to write down objects of $\cat{QReg}_A^{\qHodge}$. For example, it contains the category $\cat{QReg}_A^{\lightning\!}$ of $A$-algebras $R$ which are $p$-torsion free for all primes~$p$ and can be written in the form $R\cong B/J$, where $B$ is a relatively perfect $\Lambda$-$A$-algebra (by which we mean that the relative Adams operations $\psi_{B/A}^m\colon B\otimes_{A,\psi^m }A\rightarrow B$ are isomorphisms) and $J\subseteq B$ is an ideal generated by a Koszul-regular sequence of higher powers, that is, a Koszul-regular sequence $(x_1^{\alpha_1},\dotsc,x_r^{\alpha_r})$ with $\alpha_i\geqslant 2$ for all~$i$.
	\end{rem}
	\begin{rem}\label{rem:qHodgeQuasiregularDAlg}
		We can not only equip $\fil_{\qHodge}^\star\qdeRham_{R/A}$ with a filtered $\IE_\infty$-algebra structure, but even with the structure of a filtered derived commutative $(q-1)^\star A\qpower$-algebra as in \cref{par:qHhodgeDAlg}, and the various compatibilities all respect this structure.
	\end{rem}
	\begin{numpar}[Monoidality.]
		Similar to \cref{par:CanonicalqHodgeSmoothMonoidal}, the functor from \cref{thm:CanonicalqHodgeQuasiregular} can be equipped with an $\infty$-operad structure. To this end, let
		\begin{equation*}
			\cat{QReg}_A^{\qHodge,\otimes}\subseteq \cat{AniAlg}_A^{\otimes}
		\end{equation*}
		be the non-full sub-$\infty$-operad spanned by those objects whose entries are all contained in $\cat{QReg}_A^{\qHodge}$ and those morphisms that factor through a cocartesian lift of its image in $\cat{Fin}_*$ (compare the construction of $\cat{Sm}_{A[\dim!^{-1}]}^{\otimes}$ in \cref{par:CanonicalqHodgeSmoothMonoidal}).
		
		Note that $\cat{QReg}_A^{\qHodge,\otimes}\rightarrow \cat{Fin}_*$ is not a cocartesian fibration, because $\cat{QReg}_A^{\qHodge}$ is not closed under tensor products in $\cat{AniAlg}_A$. The problem is that $R_1\lotimes_A R_2$ might not be static or not $p$-torsion free for some prime~$p$. As we'll see momentarily, this is the only obstruction.
	\end{numpar}
	\begin{lem}\label{lem:qHodgeQuasiregularSymmetricMonoidal}
		Let $R_1,R_2\in\cat{QReg}_A^{\qHodge}$ and put $R\coloneqq R_1\lotimes_AR_2$.
		\begin{alphanumerate}
			\item If $R$ is static and $p$-torsion free for all primes~$p$, then also $R\in\cat{QReg}_A^{\qHodge}$.\label{enum:RStatic}
			\item In the situation from \cref{enum:RStatic} the canonical map\label{enum:qHodgeQuasiregularSymmetricMonoidal}
			\begin{equation*}
				\Bigl(\fil_{\qHodge}^\star\qdeRham_{R_1/A}\lotimes_{(q-1)^\star A\qpower}\fil_{\qHodge}^\star\qdeRham_{R_2/A}\Bigr)_{(q-1)}^\complete\overset{\simeq}{\longrightarrow}\fil_{\qHodge}^\star\qdeRham_{R/A}
			\end{equation*}
			is an equivalence of filtered $\IE_\infty$-algebras over $(q-1)^\star A\qpower$.
		\end{alphanumerate}
	\end{lem}
	\begin{proof}
		Let $p$ be any prime. Using $\L_{R/A}\simeq (\L_{R_1/A}\lotimes_A R_2)\oplus (R_1\lotimes_A\L_{R_2/A})$, it's clear that $\widehat{R}_p$ is again $p$-quasi-lci over $\widehat{A}_p$. Similarly, $R/p$ will still be relatively semiperfect over $\widehat{A}_p$. To show $R\in\cat{QReg}_A^{\qHodge}$, it remains to verify that
		\begin{equation*}
			\fil_{\qHodge}^\star\qdeRham_{\smash{\widehat{R}}_p/\smash{\widehat{A}}_p}/(q-1)\overset{\simeq}{\longrightarrow}\fil_{\Hodge}^\star\deRham_{\smash{\widehat{R}}_p/\smash{\widehat{A}}_p}\,.
		\end{equation*}
		is an equivalence. By \cref{lem:qHodgeInjective}, only surjectivity needs to be checked. But since we have $\fil_{\Hodge}^\star \deRham_{\smash{\widehat{R}}_p/\smash{\widehat{A}}_p}\simeq (\fil_{\Hodge}^\star \deRham_{\smash{\widehat{R}}_{1,p}/\smash{\widehat{A}}_p}\lotimes_A \fil_{\Hodge}^\star \deRham_{\smash{\widehat{R}}_{2,p}/\smash{\widehat{A}}_p})_p^\complete$, surjectivity for $R$ follows from the analogous assertions for $R_1$ and $R_2$. This shows~\cref{enum:RStatic}.
		
		To show~\cref{enum:qHodgeQuasiregularSymmetricMonoidal}, we can reduce both sides modulo~$(q-1)$ and then once again reduce to the well-known fact that $\fil_{\Hodge}^\star\deRham_{-/A}$ is symmetric monoidal.
	\end{proof}
	\begin{cor}\label{cor:CanonicalqHodgeQuasiRegularMonoidal}
		The functor from \cref{thm:CanonicalqHodgeQuasiregular} underlies a functor of $\infty$-operads
		\begin{equation*}
			\cat{QReg}_{A}^{\qHodge,\otimes}\longrightarrow \CAlg\bigl(\cat{AniAlg}_A^{\qHodge}\bigr)^{\otimes}\,,
		\end{equation*}		
		which preserves all cocartesian lifts that exist in the source. In particular, when we restrict to the full subcategory $\cat{QReg}_A^{\qHodge,\flat}\subseteq \cat{QReg}_A^{\qHodge}$ spanned by those $R$ that are flat over $A$, the functor from \cref{thm:CanonicalqHodgeQuasiregular} is symmetric monoidal.
	\end{cor}
	\begin{proof}[Proof sketch]
		To construct the functor of $\infty$-operads, we repeat the argument from the proof of \cref{thm:CanonicalqHodgeQuasiregular}: Write $\CAlg(\cat{AniAlg}_A^{\qHodge})^\otimes$ as an iterated pullback of $\CAlg(-)^{\otimes}$ of various symmetric monoidal $\infty$-categories of filtered objects. For $\fil_{(\Hodge,q-1)}^\star(\deRham_{-/A}\lotimes_\IZ\IQ)\qpower$ we know what to do, for all other factors of the iterated pullbacks the objects in question are $1$-categorical in nature, so everything can be constructed by hand.
		
		That all existing cocartesian lifts are preserved boils down to \cref{lem:qHodgeQuasiregularSymmetricMonoidal}\cref{enum:qHodgeQuasiregularSymmetricMonoidal}. Finally, if $R_1,R_2\in\cat{QReg}_A^{\qHodge}$ are flat over~$A$, then $R_1\lotimes_A R_2$ will be static and $p$-torsion free for all~$p$, so \cref{lem:qHodgeQuasiregularSymmetricMonoidal}\cref{enum:RStatic} implies that the full sub-$\infty$-operad of $\cat{QReg}_A^{\qHodge,\otimes}$ spanned by $\cat{QReg}_A^{\qHodge,\flat}$ will be a cocartesian fibration. Since our map preserves all cocartesian lifts, we deduce that we indeed get a symmetric monoidal functor
		\begin{equation*}
			\bigl(-,\fil_{\qHodge}^\star\qdeRham_{-/A}\bigr)\colon \cat{QReg}_A^{\qHodge,\flat}\longrightarrow \CAlg\bigl(\cat{AniAlg}_A^{\qHodge}\bigr)\,.\qedhere
		\end{equation*}
	\end{proof}

	\newpage
	\appendix
	\renewcommand{\thetheorem}{\thesection.\arabic{theorem}}
	\renewcommand{\SectionPrefix}{Appendix~}
	
	\section{The \texorpdfstring{$q$}{q}-de Rham complex}\label{appendix:GlobalqDeRham}
	Let $p$ be a prime. In \cite[\S{\chref[section]{16}}]{Prismatic}, Bhatt and Scholze construct a functorial $(p,q-1)$-complete $q$-de Rham complex relative to any $q$-PD pair $(D,I)$. This verifies Scholze's conjecture \cite[Conjecture~\href{https://arxiv.org/pdf/1606.01796\#theorem.3.1}{3.1}]{Toulouse} after $p$-completion, but leaves open the global case. There are (at least) two strategies to tackle the global case:
	\begin{alphanumerate}
		\item One can glue the global $q$-de Rham complex from its $p$-completions and its rationalisation using an arithmetic fracture square.\label{enum:StrategyA}
		\item Following Kedlaya \cite[\S\href{https://kskedlaya.org/prismatic/sec_q-global.html}{29}]{Kedlaya}, one can construct the global $q$-de Rham complex as the cohomology of a global $q$-crystalline site.\label{enum:StrategyB}
	\end{alphanumerate}
	Strategy~\cref{enum:StrategyA} is what Bhatt and Scholze originally had in mind, but they never published the argument. It is essentially straightforward, but not entirely trivial. Since all our global constructions proceed similarly by gluing $p$-completions and rationalisations, 
	it will be worthwhile to fill in the missing details of strategy~\cref{enum:StrategyA}. Our goal is to show the following theorem.
	\begin{thm}\label{thm:qDeRhamGlobal}
		Let $A$ be a $\Lambda$-ring that is $p$-torsion free for all primes $p$. Then there exists a functor
		\begin{equation*}
			\qOmega_{-/A}\colon \cat{Sm}_A\longrightarrow \CAlg\left(\widehat{\Dd}_{(q-1)}\bigl(A\qpower\bigr)\right)
		\end{equation*}
		from the $\infty$-category of smooth $A$-algebras into the $\infty$-category of $(q-1)$-complete $\IE_\infty$-algebras over $A\qpower$, satisfying the following properties:
		\begin{alphanumerate}
			\item $\qOmega_{-/A}/(q-1)\simeq \Omega_{-/A}^*$ agrees with the usual de Rham complex functor. In other words, the $q$-de Rham complex $\qOmega_{-/A}$ is a $q$-deformation of the de Rham complex $\Omega_{-/A}^*$.\label{enum:GlobalqDeRhamDeformsDeRham}
			\item For all primes~$p$, the $p$-completion\label{enum:GlobalqDeRhamPrismatic}
			\begin{equation*}
				(\qOmega_{-/A})_p^\complete\simeq \Prism_{(-)^{(p)}[\zeta_p]/\smash{\widehat{A}}_p\qpower}
			\end{equation*}
			agrees with prismatic cohomology relative to the $q$-de Rham prism $(\widehat{A}_p\qpower,[p]_q)$. Here we denote the $p$-adic Frobenius twist by $(-)^{(p)}\coloneqq (-\otimes_{A,\psi^p}A)_p^\complete$.
			\item After rationalisation, $(\qOmega_{-/A}\lotimes_\IZ\IQ)_{(q-1)}^\complete\simeq (\Omega_{-/A}\lotimes_\IZ\IQ)\qpower$ becomes the trivial $q$-deformation.\label{enum:GlobalqDeRhamRational}
			\item For every framed smooth $A$-algebra $(S,\square)$, the underlying object of $\qOmega_{S/A}$ in the derived $\infty$-category of $A\qpower$ can be represented as\label{enum:GlobalqDeRhamCoordinates}
			\begin{equation*}
				\qOmega_{S/A}\simeq \qOmega_{S/A, \square}^*\,,
			\end{equation*}
			where $\qOmega_{S/A, \square}^*$ denotes the coordinate-dependent $q$-de Rham complex as in~\cite[\textup{\S\chref[section]{3}}]{Toulouse}.
		\end{alphanumerate}
		Moreover, if $A\rightarrow A'$ is a map of $\Lambda$-rings such that $A'$ is also $p$-torsion free for all primes~$p$, there's a canonical base change equivalence
		\begin{equation*}
			\bigl(\qOmega_{-/A}\lotimes_{A}A'\bigr)_{(q-1)}^\complete\overset{\simeq}{\longrightarrow}\qOmega_{(-\otimes_AA')/A'}\,.
		\end{equation*}
		Modulo $(q-1)$ this reduces to the usual base change equivalence of the de Rham complex.
	\end{thm}
	\begin{rem}\label{rem:DerivedCommutativeRingsqDeRham}
		It will be apparent from our proof of \cref{thm:qDeRhamGlobal} (and we'll give a precise argument in \cref{par:DerivedCommutativeLift}) that the $q$-de Rham complex functor lifts canonically to a functor
		\begin{equation*}
			\qOmega_{-/A}\colon \cat{Sm}_A\longrightarrow \bigl(\DAlg_{A\qpower}\bigr)_{(q-1)}^\complete
		\end{equation*}
		into $(q-1)$-complete objects of the \emph{the $\infty$-category of derived commutative $A\qpower$-algebras} $\DAlg_{A\qpower}$ as defined in \cite[Definition~\chref{4.2.22}]{RaksitFilteredCircle}.
	\end{rem}
	\begin{numpar}[Convention.]\label{conv:ImplicitCompletion}
		Throughout \cref{appendix:GlobalqDeRham}, to increase readability and avoid excessive use of completions, all ($q$-)de Rham complexes or cotangent complexes relative to a $p$-complete ring will be implicitly $p$-completed.
	\end{numpar}
	
	\subsection{Rationalised \texorpdfstring{$q$}{q}-crystalline cohomology}\label{subsec:RationalisedqDeRham}
	
	Fix a prime $p$. Then $(\widehat{A}_p\qpower,(q-1))$ is a $q$-PD pair as in \cite[Definition~\chref{16.1}]{Prismatic} and so we can use $q$-crystalline cohomology to construct a functorial $(p,q-1)$-complete $q$-de Rham complex $\qOmega_{S/\smash{\widehat{A}}_p}$ for every $p$-completely smooth $\widehat{A}_p$-algebra $S$. We let $\qdeRham_{-/\smash{\widehat{A}}_p}$ denote its non-abelian derived functor (or \emph{animation}), which is now defined for all $p$-complete animated $\widehat{A}_p$-algebras. Observe that animation leaves the values on $p$-completely smooth $\widehat{A}_p$-algebras unchanged, as can be seen modulo $(p,q-1)$, where it reduces to a well-known fact about derived de Rham cohomology in characteristic~$p$.
	
	Our first goal is to show that after rationalisation derived $q$-de Rham cohomology is just a base change of derived de Rham cohomology relative to $\widehat{A}_p$. In coordinates, such an equivalence was already constructed in \cite[Lemma~\chref{4.1}]{Toulouse} (see \cref{par:RationalisedqDeRham} for a review), but here we need a different argument: We want a coordinate-independent equivalence, so we have to work with the definition of the $q$-de Rham complex via $q$-crystalline cohomology.
	\begin{lem}\label{lem:RationalisedqCrystalline}
		For all $p$-complete animated $\widehat{A}_p$-algebras $R$ there is a functorial equivalence of $\IE_\infty$-$(\widehat{A}_p\otimes_{\IZ}\IQ)\qpower$-algebras
		\begin{equation*}
			\bigl(\qdeRham_{R/\smash{\widehat{A}}_p}\lotimes_{\IZ}\IQ\bigr)_{(q-1)}^\complete\simeq \bigl(\deRham_{R/\smash{\widehat{A}}_p}\lotimes_{\IZ}\IQ\bigr)\qpower\,.
		\end{equation*}
	\end{lem}
	\begin{proof}
		By passing to non-abelian derived functors, it's enough to construct such a functorial equivalence for $p$-completely smooth $\widehat{A}_p$-algebras $S$. In this case, we can identify derived ($q$-)de Rham and ($q$-)crystalline cohomology:
		\begin{equation*}
			\qdeRham_{S/\smash{\widehat{A}}_p}\simeq \R\Gamma_\qcrys\bigl(S/\widehat{A}_p\qpower\bigr)\quad\text{and}\quad \deRham_{S/\smash{\widehat{A}}_p}\simeq \R\Gamma_\crys\bigl(S/\widehat{A}_p\bigr)\,.
		\end{equation*}
		To construct the desired identification between $q$-crystalline and crystalline cohomology after rationalisation, let $P\twoheadrightarrow S$ be a surjection from a $p$-completely ind-smooth $\delta$-$\widehat{A}_p$-algebra. Extend the $\delta$-structure on $P$ to $P\qpower$ via $\delta(q)\coloneqq 0$. Let $J$ be the kernel of $P\twoheadrightarrow S$ and let $D\coloneqq D_P(J)$ be its $p$-completed PD-envelope. Finally, let $\q D$ denote the corresponding $q$-PD-envelope as defined in \cite[Lemma~\chref{16.10}]{Prismatic}. It will be enough to construct a functorial equivalence
		\begin{equation*}
			\bigl(\q D\otimes_{\IZ}\IQ\bigr)_{(q-1)}^\complete\simeq \bigl(D\otimes_{\IZ}\IQ\bigr)\qpower\,.
		\end{equation*}
		If $D^\circ$ denotes the un-$p$-completed PD-envelope of $J$, then $P\rightarrow \q D\rightarrow (\q D\otimes_{\IZ}\IQ)_{(q-1)}^\complete$ uniquely factors through $D^\circ \rightarrow (\q D\otimes_{\IZ}\IQ)_{(q-1)}^\complete$. The tricky part is to show that this map extends over the $p$-completion. Since $D^\circ$ is $p$-torsion free, its $p$-completion agrees with $D^\circ\llbracket t\rrbracket/(t-p)$. By \cref{lem:RationalisationTechnicalII} below, for every fixed $n\geqslant 0$, every $p$-power series in $D^\circ$ converges in the $p$-adic topology on $(\q D\otimes_{\IZ}\IQ)/(q-1)^n$, so we indeed get our desired extension $D\rightarrow (\q D\otimes_{\IZ}\IQ)_{(q-1)}^\complete$.
		
		Extending further, we get a map $(D\otimes_{\IZ}\IQ)\qpower\rightarrow (\q D\otimes_{\IZ}\IQ)_{(q-1)}^\complete$ of the desired form. Whether this is an equivalence can be checked modulo $(q-1)$ by the derived Nakayama lemma. Then the base change property from \cite[Lemma~\chref{16.10}(3)]{Prismatic} finishes the proof---up to verifying convergence for $p$-power series in $D^\circ$.
	\end{proof}
	
	To complete the proof of \cref{lem:RationalisedqCrystalline}, we need to prove two technical lemmas about ($q$\nobreakdash-)di\-vided powers. Let's fix the following notation: According to \cite[Lemmas~\chref{2.15} and~\chref{2.17}]{Prismatic}, we may uniquely extend the $\delta$-structure from $\q D$ to $(\q D\otimes_{\IZ}\IQ)_{(q-1)}^\complete$. We still let $\phi$ and $\delta$ denote the extended Frobenius and $\delta$-map. Furthermore, we denote by
	\begin{equation*}
		\gamma(x)=\frac{x^p}{p}\quad\text{and}\quad \gamma_q(x)=\frac{\phi(x)}{[p]_q}-\delta(x)
	\end{equation*}
	the maps defining a PD-structure and a $q$-PD structure, respectively. 
	Note that $\gamma(x)$ and $\gamma_q(x)$ make sense for all $x\in(\q D\otimes_{\IZ}\IQ)_{(q-1)}^\complete$ since $p$ and $[p]_q$ are invertible.
	
	\begin{lem}\label{lem:RationalisationTechnicalI}
		With notation as above, the following is true for the self-maps $\delta$ and $\gamma_q$ of $(\q D\otimes_{\IZ}\IQ)_{(q-1)}^\complete$:
		\begin{alphanumerate}
			\item For all $n\geqslant 1$ and all $\alpha\geqslant 1$, the map $\delta$ sends $(q-1)^n\q D$ into itself, and $p^{-\alpha}(q-1)^n\q D$ into $p^{-(p\alpha+1)}(q-1)^n\q D$.\label{enum:qDdelta}
			\item For all $n\geqslant 1$ and all $\alpha\geqslant 1$, the map $\gamma_q$ sends $(q-1)^n\q D$ into $(q-1)^{n+1}\q D$, and $p^{-\alpha}(q-1)^n\q D$ into $p^{-(p\alpha+1)}(q-1)^{n+1}\q D$.\label{enum:qDgammaq}
		\end{alphanumerate}
	\end{lem}
	\begin{proof}
		Let's prove \cref{enum:qDdelta} first. Let $x=p^{-\alpha}(q-1)^ny$ for some $y\in \q D$. Since $\q D$ is flat over $\IZ_p\qpower$ and thus is $p$-torsion free, we can compute
		\begin{equation*}
			\delta(x)=\frac{\phi(x)-x^p}{p}=\frac{(q^p-1)^n\phi(y)}{p^{\alpha+1}}-\frac{(q-1)^{pn}y^p}{p^{p\alpha+1}}\,.
		\end{equation*}
		As $q^p-1$ is divisible by $q-1$, the right-hand side lies in $p^{-(p\alpha+1)}(q-1)^{n}\q D$. If $\alpha=0$, then the right-hand side must also be contained in $\q D$. But $\q D\cap p^{-1}(q-1)^n\q D=(q-1)^n\q D$ by flatness again. This proves both parts of \cref{enum:qDdelta}. Now for \cref{enum:qDgammaq}, we first compute
		\begin{equation*}
			\gamma_q(q-1)=\frac{\phi(q-1)}{[p]_q}-\delta(q-1)=-(q-1)^2\sum_{i=2}^{p-1}\frac{1}{p}\binom{p}{i}(q-1)^{i-2}\,.
		\end{equation*}
		Hence $\gamma_q(q-1)$ is divisible by $(q-1)^2$. 
		In the following, we'll repeatedly use the relation $\gamma_q(xy)=\phi(y)\gamma_q(x)-x^p\delta(y)$ from \cite[Remark~\chref{16.6}]{Prismatic} repeatedly. First off, it shows that 
		\begin{equation*}
			\gamma_q\bigl((q-1)^nx\bigr)=\phi\bigl((q-1)^{n-1}x\bigr)\gamma_q(q-1)-(q-1)^p\delta\bigl((q-1)^{n-1}x\bigr)\,.
		\end{equation*}
		It follows from \cref{enum:qDdelta} that $\delta((q-1)^{n-1}x)$ and $\phi((q-1)^{n-1}x)$ are divisible by $(q-1)^{n-1}$. Hence $\gamma_q((q-1)^nx)$ is indeed divisible by $(q-1)^{n+1}$. Moreover, we obtain
		\begin{equation*}
			\gamma_q\bigl(p^{-\alpha}(q-1)^nx\bigr)=\phi(p^{-\alpha})\gamma_q\bigl((q-1)^nx\bigr)-(q-1)^{np}x^p\delta(p^{-\alpha})\,.
		\end{equation*}
		Now $\phi(p^{-\alpha})=p^{-\alpha}$ and $\delta(p^{-\alpha})$ is contained in $p^{-(p\alpha+1)}\q D$, hence $\gamma_q(p^{-\alpha}(q-1)^nx)$ is contained in $p^{-(p\alpha+1)}(q-1)^n\q D$. This finishes the proof of \cref{enum:qDgammaq}.
	\end{proof}
	\begin{lem}\label{lem:RationalisationTechnicalII}
		Let $x\in J$. For every $n\geqslant 1$, there are elements $y_0,\dotsc,y_n\in \q D$ such that $y_0$ admits $q$-divided powers in $\q D$ and
		\begin{equation*}
			\gamma^{(n)}(x)=y_0+\sum_{i=1}^np^{-2(p^{i-1}+\dotsb+p+1)}(q-1)^{(p-2)+i}y_i
		\end{equation*}
		holds in $\q D\otimes_{\IZ}\IQ$, where $\gamma^{(n)}=\gamma\circ \dotsb\circ\gamma$ denotes the $n$-fold iteration of $\gamma$.
	\end{lem}
	\begin{proof}
		We use induction on $n$. For $n=1$, we compute
		\begin{equation*}
			\gamma(x)=\frac{x^p}{p}=\gamma_q(x)+\frac{[p]_q-p}{p}\bigl(\gamma_q(x)+\delta(x)\bigr)\,.
		\end{equation*}
		Note that $x$ admits $q$-divided powers in $\q D$ since we assume $x\in J$. Then $\gamma_q(x)$ admits $q$-divided powers again by \cite[Lemma~\chref{16.7}]{Prismatic}. Moreover, writing $[p]_q=pu+(q-1)^{p-1}$, we find that $([p]_q-p)/p=(u-1)+p^{-1}(q-1)^{p-1}$. Then $(u-1)(\gamma_q(x)+\delta(x))$ admits $q$-divided powers since $u\equiv 1\mod (q-1)$. This settles the case $n=1$. We also remark that the above equation for $\gamma(x)$ remains true without the assumption $x\in J$ as long as the expression $\gamma_q(x)$ makes sense.
		
		Now assume $\gamma^{(n)}$ can be written as above. We put $z_i=p^{-2(p^{i-1}+\dotsb+p+1)}(q-1)^{(p-2)+i}y_i$ for short, so that $\gamma^n(x)=y_0+z_1+\dotsb+z_n$. Recall the relations
		\begin{equation*}
			\gamma_q(a+b)=\gamma_q(a)+\gamma_q(b)+\sum_{i=1}^{p-1}\frac1p\binom pia^ib^{p-i}\,,\ \delta(a+b)=\delta(a)+\delta(b)-\sum_{i=1}^{p-1}\frac1p\binom pia^ib^{p-i}\,.
		\end{equation*}
		The first relation implies that $\gamma_q(y_0+z_1+\dotsb+z_n)$ is equal to $\gamma_q(y_0)+\gamma_q(z_1)+\dotsb+\gamma_q(z_n)$ plus a linear combination of terms of the form $y_0^{\alpha_0}z_1^{\alpha_1}\dotsm z_n^{\alpha_n}$ with $0\leqslant \alpha_i<p$ and $\alpha_0+\dotsb+\alpha_n=p$. Now $\gamma_q(y_0)$ admits $q$-divided powers again. Moreover, \cref{lem:RationalisationTechnicalI}\cref{enum:qDgammaq} makes sure that each $\gamma_q(z_i)$ is contained in $p^{-2(p^i+\dotsb+p+1)}(q-1)^{(p-2)+i+1}\q D$. It remains to consider monomials $y_0^{\alpha_0}z_1^{\alpha_1}\dotsm z_n^{\alpha_n}$. Put $m\coloneqq \max\left\{i\ \middle|\  \alpha_i\neq 0\right\}$. If $\alpha_0=p-1$, then all other $\alpha_i$ must vanish except $\alpha_m=1$. In this case, the monomial is contained in $p^{-2(p^{m-1}+\dotsb+p+1)}(q-1)^{(p-2)+m}\q D$. If $\alpha_0<p-1$, we get at least one more factor $(q-1)$ and the monomial $y_0^{\alpha_0}z_1^{\alpha_1}\dotsm z_n^{\alpha_n}$ is contained in $p^{-2(p^m+\dotsb+p+1)}(q-1)^{(p-2)+m+1}\q D$.
		
		A similar analysis, using the second of the above relations as well as \cref{lem:RationalisationTechnicalI}\cref{enum:qDdelta}, shows that $(u-1)\delta(y_0+z_1+\dotsb+z_n)$ and $p^{-1}(q-1)^{p-1}\delta(y_0+z_1+\dotsb+z_n)$ can be decomposed into a bunch of terms, each of which is either a multiple of $(q-1)$ in $\q D$, so that it admits $q$-divided powers, or contained in $p^{-2(p^i+\dotsb+p+1)}(q-1)^{i+1}\q D$ for some $1\leqslant i\leqslant n+1$. We conclude that
		\begin{equation*}
			\gamma^{(n+1)}(x)=\gamma_q\bigl(\gamma^{(n)}(x)\bigr)+\frac{[p]_q-p}{p}\left(\gamma_q\bigl(\gamma^{(n)}(x)\bigr)+\delta\bigl(\gamma^{(n)}(x)\bigr)\right)
		\end{equation*}
		can be written in the desired form.
	\end{proof}
	The following remark is irrelevant for our proof of \cref{thm:qDeRhamGlobal}, but it is occasionally useful for technical arguments.
	\begin{rem}\label{rem:qDeRhamRationalisationTechnical}
		There's also an analogue of \cref{lem:RationalisationTechnicalII} with the roles of $D$ and $\q D$ reversed. For every $x\in J$ and $n\geqslant 1$, there's an infinite sequence $y_0,y_1,\dotsc,\in D$ such that $y_0$ admits divided powers and
		\begin{equation*}
			\gamma_q^{(n)}(x)=y_0+\sum_{i\geqslant 1}p^{-2(p^{i-1}+\dotsb+1)}(q-1)^{(p-2)+i}y_i
		\end{equation*}
		holds in $(D\otimes_\IZ\IQ)\qpower$. The proof is very similar to \cref{lem:RationalisationTechnicalII}: We write
		\begin{equation*}
			\gamma_q(x)=\left(\gamma(x)+\frac{[p]_q-p}{p}\delta(x)\right)\frac{p}{[p]_q}
		\end{equation*}
		and $[p]_q=pu+(q-1)^{p-1}$. Then we use induction on $n\geqslant 1$. For the inductive step, we first check that the operations $\gamma(-)$, $(u-1)\delta(-)$ and $p^{-1}(q-1)^{p-1}\delta(-)$ all preserve expressions of the desired form. Then we observe that $u$ is a unit in $\IZ_p\qpower$ and so multiplication by $p/[p]_q=u^{-1}\sum_{i\geqslant 0}p^{-i}u^{-i}(q-1)^{(p-1)i}$ also preserves expressions of the desired form.
	\end{rem}
	\begin{numpar}[The equivalence on $q$-de Rham complexes.]\label{par:RationalisedqDeRham}
		Suppose we're given a $p$-completely smooth $\widehat{A}_p$-algebra $S$ together with a $p$-completely étale framing $\square\colon \widehat{A}_p\langle T_1,\dotsc,T_d\rangle\rightarrow S$. In this case, the $q$-crystalline cohomology can be computed as a $q$-de Rham complex
		\begin{equation*}
			\R\Gamma_\qcrys\bigl(S/\widehat{A}_p\qpower\bigr)\simeq \qOmega_{S/\smash{\widehat{A}}_p, \square}^*
		\end{equation*}
		by \cite[Theorem~\chref{16.22}]{Prismatic}. Similarly, it's well-known that the crystalline cohomology is given by the ordinary de Rham complex $\Omega_{S/\smash{\widehat{A}}_p}^*$ (recall that according to Convention~\cref{conv:ImplicitCompletion}, all ($q$-)de Rham complexes of the $p$-complete ring $S$ will implicitly be $p$-completed). In this case, an explicit isomorphism of complexes
		\begin{equation*}
			\bigl(\qOmega_{S/\smash{\widehat{A}}_p, \square}^*\otimes_{\IZ}\IQ\bigr)_{(q-1)}^\complete\overset{\cong}{\longrightarrow} \bigl(\Omega_{S/\smash{\widehat{A}}_p}^*\otimes_{\IZ}\IQ\bigr)\qpower
		\end{equation*}
		can be constructed as explained in \cite[Lemma~\chref{4.1}]{Toulouse}: One first observes that, after rationalisation, the partial $q$-derivatives $\q\partial_i$ can be computed in terms of the usual partial derivative $\partial_i$ via the formula
		\begin{equation*}
			\q\partial_i=\left(\frac{\log(q)}{q-1}+\sum_{n\geqslant 2}\frac{\log(q)^n}{n!(q-1)}(\partial_iT_i)^{(n-1)}\right)\partial_i\,;
		\end{equation*}
		see \cite[Lemma~\chref{12.4}]{BMS1}. Here $\log(q)$ refers to the usual Taylor series for the logarithm around $q=1$. Noticing that the first factor is an invertible automorphism, one can then appeal to the following general fact: If $M$ is an abelian group together with commuting endomorphisms $g_1,\dotsc,g_d$ and commuting automorphisms $h_1,\dotsc,h_d$ such that $h_i$ commutes with $g_j$ for $i\neq j$ one always has an isomorphism $\Kos^*(M,(g_1,\dotsc,g_d))\cong \Kos^*(M,(h_1g_1,\dotsc,h_dg_d))$ of Koszul complexes.\footnote{We don't require $h_i$ to commute with $g_i$ (and it's not true in the case at hand).}
		
		We would like to show that this explicit isomorphism is compatible with the one constructed in \cref{lem:RationalisedqCrystalline}. To this end, let's put ourselves in a slightly more general situation: Instead of a $p$-completely étale framing $\square$ as above, let's assume we're given a surjection $P\twoheadrightarrow S$ from a $p$-completely ind-smooth $\widehat{A}_p$-algebra $P$, which is in turn equipped with a $p$-completely ind-étale framing $\square\colon \widehat{A}_p\langle x_i\ |\ i\in I\rangle\rightarrow P$ for some (possible infinite) set $I$. Then $\widehat{A}_p\langle x_i\ |\ i\in I\rangle$ carries a $\delta$-$\widehat{A}_p$-algebra structure characterised by $\delta(x_i)=0$ for all $i\in I$. By \cite[Lemma~\chref{2.18}]{Prismatic}, this extends uniquely to a $\delta$-$\widehat{A}_p$-algebra structure on $P$. If $J$ denotes the kernel of $P\twoheadrightarrow S$, we can form the usual PD-envelope $D\coloneqq D_P(J)_p^\complete$ and the $q$-PD-envelope $\q D$ as before. Furthermore, we let $\breve{\Omega}_{D/\smash{\widehat{A}}_p}^*$ and $\q\breve{\Omega}_{\q D/\smash{\widehat{A}}_p, \square}^*$ denote the usual PD-de Rham complex and the $q$-PD-de Rham complex from \cite[Construction~\chref{16.20}]{Prismatic}, respectively (both are implicitly $p$-completed).
	\end{numpar}
	\begin{lem}\label{lem:RationalisedqDeRham}
		With notation as above, there is again an explicit isomorphism of complexes
		\begin{equation*}
			\bigl(\q\breve{\Omega}_{\q D/\smash{\widehat{A}}_p, \square}^*\otimes_{\IZ}\IQ\bigr)_{(q-1)}^\complete\overset{\cong}{\longrightarrow} \bigl(\breve{\Omega}_{D/\smash{\widehat{A}}_p}^*\otimes_{\IZ}\IQ\bigr)\qpower\,.
		\end{equation*}
	\end{lem}
	\begin{proof}
		This follows from the same recipe as in \cref{par:RationalisedqDeRham}, provided we can show that the formula for $\q\partial_i$ in terms of $\partial_i$ remains true under the identification $(\q D\otimes_{\IZ}\IQ)_{(q-1)}^\complete\cong (D\otimes_{\IZ}\IQ)\qpower$ from the proof of \cref{lem:RationalisedqCrystalline}. But for every fixed~$n$, the images of the diagonal maps in the diagram
		\begin{equation*}
			\begin{tikzcd}[column sep=0pt]
				& \left(P\otimes_{\IZ}\IQ\right)\qpower\dlar\drar & \\
				\left(\q D\otimes_{\IZ}\IQ\right)/(q-1)^n\ar[rr,"\cong"] & & \left(D\otimes_{\IZ}\IQ\right)\qpower/(q-1)^n
			\end{tikzcd}
		\end{equation*}
		are dense for the $p$-adic topology and for elements of $(P\otimes_{\IZ} \IQ)\qpower$ the formula is clear.
	\end{proof}
	\begin{lem}\label{lem:CompatibilityCheck}
		With notation as above, the following diagram commutes:
		\begin{equation*}
			\begin{tikzcd}[column sep=huge]
				\left(\R\Gamma_\qcrys\bigl(S/\widehat{A}_p\qpower\bigr)\lotimes_{\IZ}\IQ\right)_{(q-1)}^\complete\rar["\simeq","{(\labelcref{lem:RationalisedqCrystalline})}"']\dar["\simeq"'] & \left(\R\Gamma_\crys\bigl(S/\widehat{A}_p\bigr)\lotimes_{\IZ}\IQ\right)\qpower\dar["\simeq"]\\
				\bigl(\q\breve{\Omega}_{\q D/\smash{\widehat{A}}_p, \square}^*\otimes_{\IZ}\IQ\bigr)_{(q-1)}^\complete \rar["\cong","{(\labelcref{lem:RationalisedqDeRham})}"'] & \bigl(\breve{\Omega}_{D/\smash{\widehat{A}}_p, \square}^*\otimes_{\IZ}\IQ\bigr)\qpower
			\end{tikzcd}
		\end{equation*}
		Here the left vertical arrow is the quasi-isomorphism from \cite[Theorem~\textup{\chref{16.22}}]{Prismatic} and the right vertical arrow is the usual quasi-isomorphism between crystalline cohomology and PD-de Rham complexes.
	\end{lem}
	\begin{proof}
		Let $P^\bullet$ be the degreewise $p$-completed \v Cech nerve of $\widehat{A}_p\rightarrow P$ and let $J^\bullet\subseteq P^\bullet$ be the kernel of the augmentation $P^\bullet\twoheadrightarrow S$. Let $D^\bullet\coloneqq D_{P^\bullet}(J^\bullet)_p^\complete$ be the PD-envelope and let $\q D^\bullet$ be the corresponding $q$-PD-envelope. Finally, form the cosimplicial complexes
		\begin{equation*}
			M^{\bullet,*}\coloneqq \breve{\Omega}_{D^\bullet/\smash{\widehat{A}}_p}^*\quad\text{and}\quad\q M^{\bullet,*}\coloneqq  \q\breve{\Omega}_{\q D^\bullet/\smash{\widehat{A}}_p, \square}^*\,.
		\end{equation*}
		In the proof of \cite[Theorem~\textup{\chref{16.22}}]{Prismatic} it's shown that the totalisation $\Tot(\q M^{\bullet,*})$ of $\q M^{\bullet,*}$ is quasi-isomorphic to the $0$\textsuperscript{th} column $\q M^{0,*}\cong \q\breve{\Omega}_{\q D/\smash{\widehat{A}}_p, \square}^*$, but also to the totalisation of the $0$\textsuperscript{th} row $\Tot(\q M^{\bullet,0})\cong \Tot(\q D^\bullet)$. This provides the desired quasi-isomorphism
		\begin{equation*}
			\q\breve{\Omega}_{\q D/\smash{\widehat{A}}_p, \square}^*\simeq \Tot(\q M^{\bullet,*})\simeq \Tot(\q D^\bullet)\simeq \R\Gamma_\qcrys\bigl(S/\widehat{A}\qpower\bigr)\,.
		\end{equation*}
		In the exact same way, the quasi-isomorphism $\breve{\Omega}_{D/\smash{\widehat{A}}_p}^*\simeq \R\Gamma_\crys(S/\widehat{A}_p)$ is constructed using the cosimplicial complex $M^{\bullet,*}$ in \cite[\stackstag{07LG}]{Stacks}. Applying \cref{lem:RationalisedqDeRham} column-wise gives an isomorphism of cosimplicial complexes $(\q M^{\bullet,*}\otimes_{\IZ}\IQ)_{(q-1)}^\complete\cong (M^{\bullet,*}\otimes_{\IZ}\IQ)\qpower$. On $0$\textsuperscript{th} columns, this is the isomorphism from \cref{lem:RationalisedqDeRham}, whereas on $0$\textsuperscript{th} rows it is the isomorphism from \cref{lem:RationalisedqCrystalline}. This proves commutativity of the diagram.
	\end{proof}
	
	\subsection{The global \texorpdfstring{$q$}{q}-de Rham complex}
	From now on, we no longer work in a $p$-complete setting, but we keep Convention~\cref{conv:ImplicitCompletion}. 
	
	\begin{numpar}[Doing \cref{subsec:RationalisedqDeRham} for all primes at once.]\label{par:AllPrimesAtOnce}
		Fix $n$ and put $N_n\coloneqq \prod_{\ell\leqslant n}\ell^{2(\ell^{n-1}+\dotsb+\ell+1)}$, where the product is taken over all primes $\ell\leqslant n$. Now fix an arbitrary prime $p$ and let $P$, $D$, and $\q D$ be as in \cref{subsec:RationalisedqDeRham}. We've verified that the map $P\rightarrow \q D\rightarrow \q D/(q-1)^n\otimes_\IZ\IQ$ admits a unique continuous extension
		\begin{equation*}
			\begin{tikzcd}
				P\rar\dar & \q D/(q-1)^n\otimes_\IZ\IQ\\
				D\urar[dashed] &
			\end{tikzcd}
		\end{equation*}
		But in fact, \cref{lem:RationalisationTechnicalII} shows that this extension already factors through $N_n^{-1}\q D/(q-1)^n$, no matter how our implicit prime $p$ is chosen. This observation allows us to construct canonical maps $\dR_{\smash{\widehat{R}}_p/\smash{\widehat{A}}_p}\rightarrow N_n^{-1}\qdeRham_{\smash{\widehat{R}}_p/\smash{\widehat{A}}_p}/(q-1)^n$ for all animated rings $R$ and all $n\geqslant 0$. Taking the product over all $p$ and the limit over all $n$ allows us to construct a map
		\begin{equation*}
			\biggl(\prod_p\qdeRham_{\smash{\widehat{R}}_p/\smash{\widehat{A}}_p}\lotimes_{\IZ}\IQ\biggr)_{(q-1)}^\complete\overset{\simeq}{\longleftarrow} \biggl(\prod_{p}\deRham_{\smash{\widehat{R}}_p/\smash{\widehat{A}}_p}\lotimes_{\IZ}\IQ\biggr)\qpower\,.
		\end{equation*}
		compatible with the one from \cref{lem:RationalisedqCrystalline}. 
		This map is an equivalence as indicated, as one immediately checks modulo $q-1$. 
	\end{numpar}
	\begin{con}\label{con:GlobalqDeRham}
		For all smooth $A$-algebras $S$, we construct the \emph{$q$-de Rham complex of $S$ over $A$} as the pullback
		\begin{equation*}
			\begin{tikzcd}
				\qOmega_{S/A}\rar\dar\drar[pullback] & \prod_{p}\qOmega_{\smash{\widehat{S}}_p/\smash{\widehat{A}}_p}\dar\\
				\bigl(\Omega_{S/A}^*\lotimes_{\IZ}\IQ\bigr)\qpower\rar & \biggl(\prod_{p}\Omega_{{\widehat{S}}_p/{\widehat{A}}_p}^*\lotimes_\IZ\IQ\biggr)\qpower
			\end{tikzcd}
		\end{equation*}
		Here the right vertical map is the one constructed in \cref{par:AllPrimesAtOnce} above.
	\end{con}
	\begin{proof}[Proof of \cref{thm:qDeRhamGlobal}]
		We've constructed $\qOmega_{S/A}$ in \cref{con:GlobalqDeRham}. Functoriality is clear since all constituents of the pullback are functorial and so are the arrows between them. Modulo $(q-1)$, the pullback reduces to the usual arithmetic fracture square for $\Omega_{R/A}$, proving \cref{enum:GlobalqDeRhamDeformsDeRham}. By construction, $(\qOmega_{S/A})_p^\complete\simeq \qOmega_{\smash{\widehat{S}}_p/\smash{\widehat{A}}_p}$, and so \cref{enum:GlobalqDeRhamPrismatic} follows from \cite[Theorem~\chref{16.18}]{Prismatic}. Part~\cref{enum:GlobalqDeRhamRational} follows again from the construction.
		
		For \cref{enum:GlobalqDeRhamCoordinates}, suppose $S$ is equipped with an étale framing $\square\colon A[x_1,\dotsc,x_d]\rightarrow S$. The same argument as in \cref{par:RationalisedqDeRham} provides an isomorphism $(\qOmega_{S/A, \square}^*\otimes_{\IZ}\IQ)_{(q-1)}^\complete\cong (\Omega_{S/A}^*\otimes_{\IZ}\IQ)\qpower$. The compatibility check from \cref{lem:CompatibilityCheck} now allows us to identify the pullback square for $\qOmega_{S/A}$ with the usual arithmetic fracture square for the complex $\qOmega_{S/A, \square}^*$, completed at $(q-1)$. This shows $\qOmega_{S/A}\simeq \qOmega_{S/A, \square}^*$, as desired.
		
		For the additional assertion, it's clear from the construction that a base change morphism 
		\begin{equation*}
			\bigl(\qOmega_{-/A}\lotimes_{A}{A'}\bigr)_{(q-1}^\complete\longrightarrow \qOmega_{(-\otimes_AA')/A'}
		\end{equation*}
		exists and that it reduces modulo $(q-1)$ to the usual base change equivalence for the de Rham complex. In particular, it must be an equivalence as well. This finishes the proof.
	\end{proof}
	\begin{numpar}[Upgrade to derived commutative $A\qpower$-algebras.]\label{par:DerivedCommutativeLift}
		Let us explain how to lift the $q$-de Rham complex to a functor
		\begin{equation*}
			\qOmega_{-/A}\colon \cat{Sm}_A\longrightarrow \bigl(\DAlg_{A\qpower}\bigr)_{(q-1)}^\complete
		\end{equation*}
		into the $\infty$-category of $(q-1)$-complete derived commutative $A\qpower$-algebras. The key observation is that all limits and colimits in derived commutative $A\qpower$-algebras can be computed on the level of underlying $\IE_\infty$-$A\qpower$-algebras by \cite[Proposition~\chref{4.2.27}]{RaksitFilteredCircle}. Thus, by compatibility with pullbacks, it'll be enough to lift the three components of the pullback from \cref{con:GlobalqDeRham} to derived commutative $A\qpower$-algebras. By compatibility with cosimplicial limits, it'll be enough to construct functorial cosimplicial realisations of $\Omega_{S/A}$, $\Omega_{\smash{\widehat{S}}_p/\smash{\widehat{A}}_p}$, and $\qOmega_{\smash{\widehat{S}}_p/\smash{\widehat{A}}_p}$.
		
		For the latter two, the comparison with ($q$-)crystalline cohomology easily provides such realisations. But the same trick works just as well for $\Omega_{S/A}$: Let $P\twoheadrightarrow S$ be any surjection from an ind-smooth-$A$-algebra (which can be chosen functorially; for example, take $P\coloneqq A[\{T_s\}_{s\in S}]$), form the \v Cech nerve $P^\bullet$ of $A\rightarrow P$, let $J^\bullet\subseteq P^\bullet$ be the kernel of the augmentation $P^\bullet\twoheadrightarrow S$, and let $D^\bullet\coloneqq D_{P^\bullet}(J^\bullet)$ be its PD-envelope. Then $\Omega_{S/A}\simeq \Tot D_{P^\bullet}(J^\bullet)$ holds by a straightforward adaptation of the proof of \cite[Theorem~\chref{16.22}]{Prismatic}: Namely, one considers the cosimplicial complex
		\begin{equation*}
			M^{\bullet,*}\coloneqq \breve{\Omega}_{D^\bullet/A}^*
		\end{equation*}
		and checks that each column $M^{i,*}$ is quasi-isomorphic to $M^{0,*}$ (this is the Poincaré lemma) and that each row $M^{\bullet,j}$ for $j>0$ is nullhomotopic (e.g.\ by \cite[Tag~\chref{07L7}]{Stacks} applied to the cosimplicial ring $D^\bullet$).
		
		In fact, this argument can be used to show something even better: Since the de Rham complex $\Omega_{S/A}^*$ and its PD-variants $\breve{\Omega}_{D_{P^\bullet}(J^\bullet)/A}^*$ are commutative differential-graded $A$-algebras, they define elements in Raksit's $\infty$-category $\cat{DG}_-\DAlg_{A}$ \cite[Definition~\chref{5.1.10}]{RaksitFilteredCircle}, which gives another construction of a derived commutative algebra structure on $\Omega_{S/A}$. But the argument above shows that $\Omega_{S/A}\simeq \Tot D_{P^\bullet}(J^\bullet)$ holds true as derived commutative $A$-algebras.
	\end{numpar}
	
	\begin{numpar}[Derived global $q$-de Rham complexes.]\label{par:DerivedGlobalqDeRham}
		We let $\qdeRham_{-/A}$ denote the animation of $\qOmega_{-/A}$. For all animated $A$-algebras $R$, we call $\qdeRham_{R/A}$ the \emph{derived $q$-de Rham complex of $R$ over $A$}. By construction, it sits inside a pullback square
		\begin{equation*}
			\begin{tikzcd}
				\qdeRham_{R/A}\rar\dar\drar[pullback] & \prod_{p}\qdeRham_{\smash{\widehat{R}}_p/\smash{\widehat{A}}_p}\dar\\
				\bigl(\deRham_{R/A}\lotimes_{\IZ}\IQ\bigr)\qpower\rar & \biggl(\prod_{p}\deRham_{\smash{\widehat{R}}_p/\smash{\widehat{A}}_p}\lotimes_\IZ\IQ\biggr)\qpower
			\end{tikzcd}
		\end{equation*}
		where the right vertical map again comes from \cref{par:AllPrimesAtOnce}. It's still true that $\qdeRham_{-/A}/(q-1)\simeq \deRham_{-/A}$ and that $\qdeRham_{-/A}$ lifts canonically to $(q-1)$-complete derived commutative $A\qpower$-algebras (this follows immediately from compatibility with colimits as explained in \cref{par:DerivedCommutativeLift}).
		
		However, in contrast to the $p$-complete situation, it's no longer true that the values on smooth $A$-algebras remain unchanged under animation (only the values on polynomial algebras do). In fact, this already fails for the derived de Rham complex in characteristic~$0$. If $\qdeRham_{R/A}$ can be equipped with a $q$-deformation of the Hodge filtration, this problem can be fixed by considering the $q$-Hodge-completed derived $q$-de Rham complex $\qhatdeRham_{R/A}$ (see \cref{prop:Letaq-1}\cref{enum:qHodgeCompletion}).
	\end{numpar}

	\newpage
	
	\section{Habiro-completion}\label{appendix:HabiroCompletion}
	In this appendix we'll study the \emph{Habiro completion functor} $(-)_\Hh^\complete\coloneqq\limit_{m\in\IN}(-)_{(q^m-1)}^\complete$ and show that it behaves for all practical purposes like completion at a finitely generated ideal. We'll also study Habiro completion in the setting of solid condensed mathematics.
	
	
	In the following, we'll use the notion of \emph{killing an idempotent algebra}, which is nicely reviewed in \cite[Lecture~\href{https://youtu.be/38PzTzCiMow?list=PLx5f8IelFRgGmu6gmL-Kf_Rl_6Mm7juZO&t=5523}{13}]{AnalyticStacks}.
	
	\begin{numpar}[Habiro-complete spectra.]\label{par:HabiroComplete}
		Following Manin \cite[\S{\chpageref[0.2]{3}}]{Manin}, let us denote the localisation $\IZ\left[q^{\pm 1},\{(q^m-1)^{-1}\}_{m\in\IN}\right]$ by $\Rr$ and let $\IS_\Rr\coloneqq \IS\left[q^{\pm 1},\{(q^m-1)^{-1}\}_{m\in\IN}\right]$ be its obvious spherical lift. Then $\IS_\Rr$ is an idempotent algebra over $\IS[q^{\pm 1}]$ and we define the \emph{$\infty$-category of Habiro-complete spectra}
		\begin{equation*}
			\Mod_{\IS_\Hh}(\Sp)_\Hh^\complete\subseteq \Mod_{\IS[q^{\pm 1}]}(\Sp)
		\end{equation*}
		to be the full sub-$\infty$-category obtained by killing the idempotent $\IS_\Rr$. That is, $\Mod_{\IS_\Hh}(\Sp)_\Hh^\complete$ consists of those $M\in \Mod_{\IS[q^{\pm 1}]}(\Sp)$ such that $\Hom_{\IS[q^{\pm 1}]}(\IS_\Rr,M)\simeq 0$.
		
		It'll be apparent from \cref{lem:HabiroComplete} below that the inclusion $\Mod_{\IS_\Hh}(\Sp)_\Hh^\complete\subseteq \Mod_{\IS[q^{\pm 1}]}(\Sp)$ has a left adjoint $(-)_\Hh^\complete\coloneqq\limit_{m\in\IN}(-)_{(q^m-1)}^\complete$ which we call \emph{Habiro-completion}. When applied to the tensor unit, we obtain the \emph{spherical Habiro ring}
		\begin{equation*}
			\IS_\Hh\coloneqq \limit_{m\in\IN}\IS[q]_{(q^m-1)}^\complete\,.
		\end{equation*}
		Note that $q$ is already a unit in $\IS_\Hh$, so it doesn't matter whether we complete $\IS[q]$ or $\IS[q^{\pm 1}]$. We let $-\cotimes_{\IS_\Hh}-$ denote the Habiro-completed tensor product in $\Mod_{\IS_\Hh}(\Sp)_\Hh^\complete$. We also let $\widehat{\Dd}(\Hh)\subseteq \Dd(\IZ[q^{\pm 1}])$ denote the full sub-$\infty$-category of Habiro-complete objects and denote its completed tensor product by $-\clotimes_\Hh-$.
	\end{numpar}
	\begin{lem}\label{lem:HabiroComplete}
		For a $\IS[q^{\pm 1}]$-module spectrum $M$, the following conditions are equivalent.
		\begin{alphanumerate}
			\item $M$ is Habiro-complete.\label{enum:HabiroComplete}
			\item $\Hom_{\IS[q^{\pm 1}]}(\IS_\Rr,M)\simeq 0$.\label{enum:HabiroCompleteKillingSRr}
			\item The canonical $\IS[q^{\pm 1}]$-module morphism
			\begin{equation*}
				M\longrightarrow \limit_{n\geqslant 1}M/(q;q)_n\simeq \limit_{m\in\IN}M_{(q^m-1)}^\complete
			\end{equation*}
			is an equivalence. Here $(a;q)_n\coloneqq (1-a)(1-aq)\dotsm(1-aq^{n-1})$ denotes the $q$-Pochhammer symbol, as usual.\label{enum:HabiroCompletion}
			\item All homotopy groups $\pi_n(M)$, $n\in\IZ$, are Habiro-complete.\label{enum:HabiroCompleteHomotopyGroups}
		\end{alphanumerate}
	\end{lem}
	\begin{proof}
		The proof is analogous to \cite[\stackstag{091P}]{Stacks}. Equivalence of~\cref{enum:HabiroComplete} and~\cref{enum:HabiroCompleteKillingSRr} follows by definition of what it means to kill the idempotent $\IS_\Rr$. Condition~\cref{enum:HabiroCompleteKillingSRr} is equivalent to $M\simeq \Hom_{\IS[q^{\pm 1}]}(\fib(\IS[q^{\pm 1}]\rightarrow \IS_\Rr),M)$. Writing
		\begin{equation*}
			\fib\bigl(\IS[q^{\pm 1}]\rightarrow \IS_\Rr\bigr)\simeq \Sigma^{-1}\colimit\left(\IS[q^{\pm 1}]/(q;q)_1\xrightarrow{(1-q^2)}\IS[q^{\pm 1}]/(q;q)_2\xrightarrow{(1-q^3)}\dotsb\right)
		\end{equation*}
		we see that this condition is equivalent to $M\simeq \limit_{n\geqslant 1}M/(q;q)_n$, thus \cref{enum:HabiroCompleteKillingSRr} $\Leftrightarrow$ \cref{enum:HabiroCompletion}. Finally, to show \cref{enum:HabiroComplete} $\Leftrightarrow$ \cref{enum:HabiroCompleteHomotopyGroups}, consider the Postikov filtration $\tau_{\geqslant \star}(M)$. This allows us to define a descending filtration on $\Hom_{\IS[q^{\pm 1}]}(\IS_\Rr,M)$ via
		\begin{equation*}
			\fil^\star \Hom_{\IS[q^{\pm 1}]}(\IS_\Rr,M)\coloneqq \Hom_{\IS[q^{\pm 1}]}\bigl(\IS_\Rr,\tau_{\geqslant \star}(M)\bigr)\,.
		\end{equation*}
		This filtration is complete, because $0\simeq \limit_{n\rightarrow \infty}\tau_{\geqslant n}(M)$ can be pulled into $\Hom_{\IS[q^{\pm 1}]}(\IS_\Rr,-)$. To show that the filtration is exhaustive, we need to check that $M\simeq \colimit_{n\rightarrow -\infty}\tau_{\geqslant n}(M)$  can similarly be pulled into $\Hom_{\IS[q^{\pm 1}]}(\IS_\Rr,-)$. This works because $\IS_\Rr$ is connective, whereas the cofibres $\cofib(\tau_{\geqslant n}(M)\rightarrow M)\simeq \tau_{\leqslant n-1}(M)$ become more and more coconnective as $n\rightarrow -\infty$.
		
		Since each $\pi_n(M)$ is already a $\IZ[q^{\pm 1}]$-module, the associated graded of this filtration is given by
		\begin{equation*}
			\gr^n\Hom_{\IS[q^{\pm 1}]}(\IS_\Rr,M)\simeq \Hom_{\IS[q^{\pm 1}]}\bigl(\IS_\Rr,\Sigma^n\pi_n(M)\bigr)\simeq \Sigma^n\RHom_{\IZ[q^{\pm 1}]}\bigl(\Rr,\pi_n(M)\bigr)\,.
		\end{equation*}
		Now $\Rr$ has a two-term resolution by free $\IZ[q^{\pm 1}]$-modules. For example, take
		\begin{equation*}
			0\longrightarrow \bigoplus_{i\geqslant 0}\IZ[q^{\pm 1}]\longrightarrow \bigoplus_{i\geqslant 0}\IZ[q^{\pm 1}]\longrightarrow \Rr\longrightarrow 0\,,
		\end{equation*}
		where the first arrow sends $(a_i)_{i\geqslant 0}\mapsto (a_i-(q;q)_ia_{i-1})_{i\geqslant 0}$ (with $a_{-1}\coloneqq 0$) and the second arrow sends $(a_i)_{i\geqslant 0}\mapsto\sum_{i\geqslant 0}a_i/(q;q)_i$. It follows that $\Sigma^n\RHom_{\IZ[q^{\pm 1}]}(\Rr,\pi_n(M))$ is concentrated in homological degrees $[n-1,n]$. Combined with the fact that the filtration is complete and exaustive%
		\footnote{Alternatively, observe that the spectral sequence associated to the filtered spectrum $ \fil^\star \Hom_{\IS[q^{\pm 1}]}(\IS_\Rr,M)$ collapses on the $E^2$-page.}%
		, we obtain short exact sequences
		\begin{equation*}
			0\longrightarrow \Ext_{\IZ[q^{\pm 1}]}^1\bigl(\Rr,\pi_{n+1}(M)\bigr)\longrightarrow\pi_n\Hom_{\IS[q^{\pm 1}]}(\IS_\Rr,M)\longrightarrow \Hom_{\IZ[q^{\pm 1}]}\bigl(\Rr,\pi_n(M)\bigr)\longrightarrow 0
		\end{equation*}
		for all $n\in\IZ$. Therefore, $\Hom_{\IS[q^{\pm 1}]}(\IS_\Rr,M)$ vanishes if and only if $\RHom_{\IZ[q^{\pm 1}]}(\Rr,\pi_n(M))$ vanishes for all $n\in\IZ$, which proves that $M$ is Habiro-complete if and only if each $\pi_n(M)$ is.
	\end{proof}
	We have the following \enquote{derived Nakayama lemma}.
	\begin{lem}\label{lem:HabiroCompleteDerivedNakayama}
		Let $M$ be a Habiro-complete spectrum. If $M/\Phi_m(q)\simeq 0$ for all $m\in\IN$, then $M\simeq 0$. If $M$ is an ordinary $\IZ[q^{\pm 1}]$-module, the same conclusion is already true if the quotients are taken in the underived sense.
	\end{lem}
	\begin{proof}
		By the usual derived Nayama lemma, if $M/\Phi_m(q)\simeq 0$, then $M_{\Phi_m(q)}^\complete\simeq 0$, hence $M_{(q^m-1)}^\complete\simeq 0$. By \cref{lem:HabiroComplete}\cref{enum:HabiroCompletion}, this implies $M\simeq 0$. Now suppose $M$ is an ordinary $\IZ[q^{\pm 1}]$-module such that the underived quotients $M/\Phi_m(q)$ vanish for all $m\in\IN$. We argue as in \cite[\stackstag{09B9}]{Stacks}. The assumption implies that multiplication by $(q;q)_n$ is surjective on $M$ for all $n\geqslant 1$. It follows that the underived limit of
		\begin{equation*}
			\left(M\xleftarrow{(q;q)_1}M\xleftarrow{(q;q)_2}M\xleftarrow{(q;q)_3}\dotsb\right)
		\end{equation*}
		is non-zero. Then the derived limit is non-zero as well, which forces $\Hom_{\IS[q^{\pm 1}]}(\IS_\Rr,M)\not\simeq 0$, so $M$ is not Habiro-complete.
	\end{proof}
	\begin{cor}\label{cor:HabiroCompleteDerivedNakayama}
		Let $M$ be a Habiro-complete spectrum and fix $n\in\IZ$. If $\pi_n(M/\Phi_m(q))\cong 0$ for all $m\in\IN$, then already $\pi_n(M)\cong 0$.
	\end{cor}
	\begin{proof}
		The underived quotient $\pi_n(M)/\Phi_m(q)$ is a sub-$\IZ[q^{\pm 1}]$-module of $\pi_n(M/\Phi_m(q))$, so if $\pi_n(M/\Phi_m(q))$ vanishes, then the underived quotient $\pi_n(M)/\Phi_m(q)$ vanishes as well. If this is happens for all $m\in\IN$, \cref{lem:HabiroCompleteDerivedNakayama} implies $\pi_n(M)\cong 0$, because $\pi_n(M)$ is Habiro-complete by \cref{lem:HabiroComplete}\cref{enum:HabiroCompleteHomotopyGroups}.
	\end{proof}
	\begin{rem}
		In \cref{lem:HabiroCompleteDerivedNakayama} and \cref{cor:HabiroCompleteDerivedNakayama}, we could equally well replace $\{\Phi_m(q)\}_{m\in\IN}$ by $\{(q^m-1)\}_{m\in\IN}$, or $\{(q;q)_n\}_{n\geqslant 1}$, or any set of polynomials in which each $\Phi_m(q)$ occurs as a factor at least once.
	\end{rem}
	
	
	To finish this appendix, we'll show that bounded below Habiro-complete objects are closed under the solid tensor product. To this end, let us first briefly review the (solid) condensed formalism of Clausen--Scholze \cite{AnalyticStacks}.
	
	\begin{numpar}[Solid condensed recollections.]\label{par:CondensedRecollections}
		Let $\Cond(\Sp)$ denote the $\infty$-category of \emph{\embrace{light} condensed spectra}, that is, hypersheaves of spectra on the site of light profinite sets as defined by Clausen and Scholze \cite{AnalyticStacks}. The evaluation at the point $(-)(*)\colon \Cond(\Sp)\rightarrow \Sp$ admits a fully faithful symmetric monoidal left adjoint $(\underline{-})\colon \Sp\rightarrow\Cond(\Sp)$, sending a spectrum $X$ to the \emph{discrete} condensed spectrum $\underline{X}$.
		
		One can develop a theory of \emph{solid condensed spectra} along the lines of \cite[Lectures~\href{https://www.youtube.com/watch?v=bdQ-_CZ5tl8&list=PLx5f8IelFRgGmu6gmL-Kf_Rl_6Mm7juZO}{5}--\href{https://www.youtube.com/watch?v=KKzt6C9ggWA&list=PLx5f8IelFRgGmu6gmL-Kf_Rl_6Mm7juZO}{6}]{AnalyticStacks}. Let $\Null\coloneqq \cofib(\IS[\{\infty\}]\rightarrow\IS[\IN\cup\{\infty\}])$ be the free condensed spectrum on a null sequence. Let $\sigma\colon \Null\rightarrow \Null$ be the endomorphism induced by the shift map $(-)+1\colon \IN\cup\{\infty\}\rightarrow \IN\cup\{\infty\}$. Recall that a condensed spectrum $M$ is called \emph{solid} if
		\begin{equation*}
			1-\sigma^*\colon \Hhom_\IS(\Null,M)\overset{\simeq}{\longrightarrow}\Hhom_\IS(\Null,M)
		\end{equation*}
		is an equivalence, where $\Hhom_\IS$ denotes the internal Hom in $\Cond(\Sp)$. We let $\Sp_\solid\subseteq \Cond(\Sp)$ denote the full sub-$\infty$-category of solid condensed spectra. Then $\Sp_\solid$ is closed under all limits and colimits. This implies that the inclusion $\Sp_\solid\subseteq \Cond(\Sp)$ admits a left adjoint $(-)^\solid\colon \Cond(\Sp)\rightarrow \Sp_\solid$. It satisfies $(M\otimes N)^\solid\simeq (M^\solid\otimes N)^\solid$, which allows us to endow $\Sp_\solid$ with a symmetric monoidal structure, called the \emph{solid tensor product}, via $M\soltimes N\coloneqq (M\otimes N)^\solid$.
	\end{numpar}
	
	\begin{numpar}[Habiro-complete solid condensed spectra.]\label{par:HabiroCompleteCondensed}
		We can also define Habiro-complete objects and Habiro completion inside $\Mod_{\IS[q^{\pm 1}]}(\Sp_\solid)$. To every ordinary Habiro-complete spectrum~$M$, we can associate a Habiro-complete solid condensed spectrum by taking the condensed Habiro-completion of the associated discrete condensed spectrum $\underline{M}$. By abuse of notation, this Habiro-complete solid condensed spectrum will be denoted $M$ again, and then \enquote{$M\mapsto M$} defines a fully faithful functor
		\begin{equation*}
			\Mod_{\IS_\Hh}(\Sp)_\Hh^\complete\longrightarrow \Mod_{\IS_\Hh}(\Sp_\solid)\,,
		\end{equation*}
		which is still fully faithful, since it's straightforward to check that the unit is still an equivalence.
	\end{numpar}
	\begin{lem}\label{lem:SolidTensorProductHabiroComplete}
		The solidified tensor product $-\soltimes_{\IS_\Hh}-$ preserves bounded below Habiro-complete objects. In particular, the fully faithful functor $\Mod_{\IS_\Hh}(\Sp)_\Hh^\complete\rightarrow \Mod_{\IS_\Hh}(\Sp_\solid)$ from \cref{par:HabiroCompleteCondensed} is symmetric monoidal when restricted to bounded below objects.
	\end{lem}
	\begin{proof}[Proof sketch]
		The proof is analogous to the proof that the solid tensor product preserves bounded below $p$-complete objects (see \cite[Lecture~\href{https://youtu.be/KKzt6C9ggWA?list=PLx5f8IelFRgGmu6gmL-Kf_Rl_6Mm7juZO&t=3288}{6}]{AnalyticStacks} or \cite[Proposition~\chref{A.3}]{BoscoRationalHodge}), but let us still sketch the argument.
		
		First we claim that $\IS_\Hh$ is idempotent in $\Mod_{\IS[q^{\pm 1}]}(\Sp_\solid)$. Indeed, each stage of the limit $\IS_\Hh\simeq \limit_{n\geqslant 1}\IS[q^{\pm 1}]/(q;q)_n$ is a finite direct sum of copies of $\IS$. Limits of this form interact well with the solid tensor product (as $\prod_\IN\IS\soltimes\prod_\IN\IS\simeq\prod_{\IN\times\IN}\IS$) and we obtain
		\begin{equation*}
			\IS_\Hh\soltimes\IS_\Hh\simeq \limit_{m,n\geqslant 1}\Bigl(\IS[q_1^{\pm 1}]/(q_1;q_1)_m\soltimes \IS[q_2^{\pm 1}]/(q_2;q_2)_n\Bigr)\simeq \limit_{m\in\IN}\IS[q_1,q_2]_{(q_1^m-1,q_2^m-1)}^\complete\,.
		\end{equation*}
		Taking the solidified tensor product over $\IS[q^{\pm 1}]$ instead amounts to identifying $q_1$ and $q_2$, which implies $\IS_\Hh\soltimes_{\IS[q^{\pm 1}]}\IS_\Hh\simeq \IS_\Hh$, as desired. A similar argument shows $\prod_\IN\IS\soltimes\IS_\Hh\simeq \prod_\IN\IS_\Hh$, so $\Mod_{\IS_\Hh}(\Sp_\solid)$ is compactly generated by shifts of $\prod_\IN\IS_\Hh$.
		
		Now let $M$ and $N$ be bounded below and Habiro-complete. We wish to show that $M\soltimes_{\IS_\Hh}N$ is Habiro-complete again. Using that Habiro-completion is a countable limit and thus commutes with $\omega_1$-filtered colimits, we can reduce to the case where $M$ and $N$ are the Habiro-completions of countable direct sums of the form $\bigoplus_{n\in\IN}\prod_{I_n}\IS_\Hh$, where each $I_n$ is countable as well. For ease of notation, let us assume $\abs{I_n}=1$ for all $n$; the argument in the general case is exactly the same. The Habiro completion of $\bigoplus_{n\in\IN}\IS_\Hh$ can be written as
		\begin{equation*}
			\biggl(\bigoplus_{n\in\IN}\IS_\Hh\biggr)_\Hh^\complete\simeq \colimit_{\substack{f\colon \IN\rightarrow \IN,\\
					f(n)\rightarrow \infty}}\prod_{n\in\IN}(q;q)_{f(n)}\IS_\Hh\,,
		\end{equation*}
		where the colimit is taken over all functions $f\colon \IN\rightarrow \IN$ such that $f(n)\rightarrow \infty$ as $n\rightarrow\infty$. It follows that
		\begin{equation*}
			M\soltimes_{\IS_\Hh}N\simeq \colimit_{\substack{f,\,g\colon\IN\rightarrow \IN,\\f(n),\,g(n)\rightarrow \infty}}\prod_{(m,n)\in\IN\times\IN}(q;q)_{f(m)}(q;q)_{g(n)}\IS_\Hh\,.		
		\end{equation*}
		Observe that $(q;q)_{f(m)}(q;q)_{g(n)}$ divides $(q;q)_{f(m)+g(n)}$, because $q$-binomial coefficients are polynomials in $\IZ[q]$. Moreover, for every $h\colon \IN\times\IN\rightarrow \IN$ such that $h(m,n)\rightarrow \infty$ as $m+n\rightarrow \infty$ there exist $f,g\colon \IN\rightarrow \IN$ such that $f(n),g(n)\rightarrow \infty$ and $h(m,n)\geqslant f(m)+g(n)$ for all $m$, $n$. By the same argument as for $p$-completions, it follows that the colimit above can be rewritten as
		\begin{equation*}
			M\soltimes_{\IS_\Hh}N\simeq \colimit_{\substack{h\colon \IN\times\IN\rightarrow \IN,\\ h(m,n)\rightarrow \infty}}\prod_{(m,n)\in\IN\times\IN}(q;q)_{h(m,n)}\IS_\Hh\simeq \biggl(\bigoplus_{m\in\IN}\IS_\Hh\soltimes_{\IS_\Hh}\bigoplus_{n\in\IN}\IS_\Hh\biggr)_\Hh^\complete\,.\qedhere
		\end{equation*}
	\end{proof}

	\newpage 
	\renewcommand{\SectionPrefix}{}
	
	\renewcommand{\bibfont}{\small}
	\printbibliography
\end{document}